\documentclass[oneside,final,10pt]{ucr}
\usepackage{hyperref}
\usepackage{amssymb}
\usepackage{bm}
\usepackage{amsmath,amscd,amsfonts}
\usepackage{pictexwd,dcpic}
\usepackage{mathrsfs}
\usepackage[all,cmtip]{xy}
\usepackage[dvips]{graphicx}
\usepackage{graphics}
\usepackage{subfigure}
\usepackage{flafter}
\usepackage{sw20uctd}

\newtheorem{theorem}{Theorem}

\newtheorem{thm}[equation]{Theorem}

\newtheorem{conjecture}[theorem]{Conjecture}
\newtheorem{corollary}[equation]{Corollary}

\newtheorem{definition}[equation]{Definition}
\newtheorem{example}[equation]{Example}

\newtheorem{lemma}[equation]{Lemma}
\newtheorem{notation}[theorem]{Notation}
\newtheorem{problem}[theorem]{Problem}
\newtheorem{proposition}[equation]{Proposition}
\newtheorem{remark}[equation]{Remark}

\newenvironment{proof}[1][Proof]{\textbf{#1.} }{\ \rule{0.5em}{0.5em}}
\begin{document}



\title{Deformations of Compact Holomorphic Poisson Manifolds and Algebraic Poisson Schemes}
\author{Chunghoon Kim}
\degreemonth{March}
\degreeyear{2014}
\degree{Doctor of Philosophy}
\chair{Professor Ziv Ran}
\othermembers{Professor Wee Liang Gan\\
Professor David Rush}
\numberofmembers{3}
\field{Mathematics}
\campus{Riverside}

\maketitle
\copyrightpage{}
\approvalpage{}

\degreesemester{Winter}

\begin{frontmatter}

\begin{acknowledgements}

I thank Professor Ziv Ran for being my advisor and providing me with advice, in particular, in the preparation of the part \ref{part2} of the thesis which is based on his paper \cite{Ran00}. I thank Professor Wee Liang Gan for helpful conversations, his generosity and his help as a graduate advisor. Especially I thank Professor Wee Liang Gan for bringing Namikawa's paper \cite{Nam09} to my attention. Namikawa's paper \cite{Nam09} was crucial in this thesis. My work began when I started thinking about the holomorphic version of deformations of Poisson schemes presented in $\cite{Nam09}$. I was inspired by Kodaira's informal explanation on deformations of complex structures, saying that {\it ``deformations of a complex manifold $M$ is the glueing of the same polydisks $U_j$ via different identification''} parametrized by a variable $t$. The double complex in the proof of Proposition 8 in Namikawa's paper \cite{Nam09}, gave me the idea that deformations of a holomorphic Poisson manifold $(M,\Lambda)$ is the glueing of the Poisson polydisks $(U_j,\Lambda_j(t))$ via different identification parametrized by a variable $t$. That led me to define a Poisson analytic family on the basis of Kodaira-Spencer.
I thank Professor David Rush for serving on my dissertation committee and his wonderful courses he gave. I thank Professor Gerhard Gierz for financial support in my final year in the PhD program.

My thesis would not have existed without great books, articles, and papers from senior mathematicians. I am deeply indebted to the following books and articles: Kodaira's  book \cite{Kod05},  Sernesi's book \cite{Ser06}, Manetti's lecture notes \cite{Man04}, Hartshorne's book \cite{Har10}, and Schlessinger and Lichtenbaum's paper \cite{Sch67}.
During summer vacation in 2012 in Korea, I spent everyday reading Kodaira's monumental book ``complex manifolds and deformations of complex structures'' (\cite{Kod05}) with Ravel's G major piano concerto 2nd movement and Thomas Mann's Doctor Faustus. That was the most happiest in my whole graduate years.

I have been fortunate to have great friends here in UCR. I thank John Dusel for being my friend. I learned a lot from him on the importance of writing. I truly appreciate his support, encouragement, and advice. I will miss our numerous conversations. I thank Adam Navas for being my friend. Probably he is the one whom I most hung out with in my graduate years. I will miss the times I spent with him and I appreciate his tremendous help. I don't want to imagine how my life in UCR would be without him. I thank Jorge Perez for his genuine kindness. He is the kindest person I have ever met in my whole life. His existence gave me a lot of thoughts on human nature, and I believe that he will find his meaning of life. I thank Jason Payne for being my friend. After he left UCR, sometimes I missed our conversations on mathematics, and I really appreciate his occasional emails. I also thank Daniel Majcherek and Oliver Thistlethwaite for spending our graduate years together as good friends. I thank Dr. Sung Rak Choi for his kindness and his help in my early graduate years while he was a postdoc here in UCR. 

I thank my family whose constant encouragement and support helped me to finish my thesis. Lastly, I thank Professor Yongnam Lee who encouraged me to be a mathematician when I was an undergraduate student. I deeply appreciate his help and support.

\end{acknowledgements}

\begin{dedication}
\null\vfil
{\large
\begin{center}

{\it In memory of my grandfather Joo Young Kim, whose unconditional love will remain forever in my heart.}
\end{center}}
\vfil\null
\end{dedication}

\begin{abstract}

 In this thesis, we study deformations of compact holomorphic Poisson manifolds and algebraic Poisson schemes. Deformations of compact holomorphic Poisson manifolds is based on Kodaira-Spencer's analytic deformation theory of compact complex manifolds, and deformations of algebraic Poisson schemes is based on Grothendieck's algebraic deformation theory of algebraic schemes. The only difference is that we deform an additional structure, namely `Poisson structures'  as well as underlying complex structures or algebraic structures in a family of compact holomorphic Poisson manifolds or in a family of algebraic Poisson schemes. Hence when we ignore Poisson structures, the underlying deformation theory  is same to ordinary deformation theory in the sense of Kodaira-Spencer, and Grothendieck.

In the part \ref{part1} of the thesis, we study deformations of compact holomorphic Poisson manifolds. We define a concept of a family of holomorphic Poisson manifolds, called a Poisson analytic family on the basis of Kodaira-Spencer's complex analytic family. We use the truncated holomorphic Poisson cohomology to study infinitesimal deformations of holomorphic Poisson manifolds and define a Poisson Kodaira Spencer map. We deduce the integrability condition. We study the `theorem of existence' for holomorphic Poisson structures.

In the part \ref{part2} of the thesis, we describe a differential graded Lie algebra governing infinitesimal Poisson deformations of a compact holomorphic Poisson manifold $(X,\Lambda_0)$ over a local artinian $\mathbb{C}$-algebra with the residue $\mathbb{C}$. We study an universal Poisson deformation of $(X,\Lambda_0)$ when $HP^1(X,\Lambda_0)=0$. 

In the part \ref{part3} of the thesis, we study deformations of algebraic Poisson schemes. We focus on infinitesimal Poisson deformations of an algebraic Poisson scheme $(X,\Lambda_0)$ over a local artinian $k$-algebra with the residue $k$, where $k$ is a algebraically closed field. We study first order Poisson deformation, obstruction and Poisson deformation functor $PDef_{(X,\Lambda_0)}$. By following \cite{Sch67}, we construct a Poisson contangent complex for a Poisson $k$-algebra homomorphism $A\to B$ and a Poisson $B$-module $M$ and define $PT^i(B/A,M)$.
\end{abstract}

\tableofcontents

\end{frontmatter}
\addcontentsline{toc}{chapter}{Preface}
\chapter*{Preface}

In this thesis, we study deformations of compact holomorphic Poisson manifolds \footnote{For general information of Poisson geometry, see Appendix \ref{appendixa}. A holomorphic Poisson manifold is a complex manifold such that its structure sheaf is a sheaf of Poisson algebras. For more details of deformations of compact holomorphic Poisson manifolds, see the part \ref{part1} of the thesis.} and algebraic Poisson schemes.\footnote{A Poisson algebraic scheme is an algebraic scheme over an algebraically closed field $k$ such that its structure sheaf is a sheaf of Poisson algebras. For more details of the definition of Poisson schemes and deformations of algebraic Poisson schemes, see the part \ref{part3} of the thesis.} Deformations of compact holomorphic Poisson manifolds is based on Kodaira-Spencer's deformation theory of compact complex manifolds, and deformations of algebraic Poisson schemes is based on Grothendieck's deformation theory of algebraic schemes. The only difference is that we deform an additional structure, namely `Poisson' structures' in a family of compact holomorphic Poisson manifolds or algebraic Poisson schemes. Hence when we ignore Poisson structures, the underlying deformation theory  is same to ordinary deformation theory in the sense of Kodaira-Spencer, and Grothendieck. The relationship between deformations of compact complex manifolds and deformations of algebraic schemes is well described in the Introduction of Sernesi's book \cite{Ser06}. I will briefly explain their relationship described in \cite{Ser06}, and then extend their relationship to the relationship between deformations of compact holomorphic Poisson manifolds, and algebraic Poisson schemes in the following.

Given a compact complex manifold $X$, a family of deformations of $X$ is a commutative diagram of holomorphic maps between complex manifolds
\begin{center}
$\xi:$$\begin{CD}
X @>>> \mathcal{X}\\
@VVV @VV\pi V\\
\star @>>> B
\end{CD}$
\end{center}
with $\pi$ proper and smooth, B connected and where $\star$ denotes the singleton space. We denote by $\mathcal{X}_t$ the fibre $\pi^{-1}(t),t\in B$. We call $(\mathcal{X},B,\pi)$ a complex analytic family. Kodaira and Spencer started studying small deformations of $X$ in a complex analytic family by defining, for every tangent vector $\frac{\partial}{\partial t}\in T_{t_0}B$, the derivative of the family along $\frac{{\partial}}{\partial t}\in T_{t_0}B$ as an element 
\begin{align*}
\frac{\partial \mathcal{X}_t}{\partial t}\in H^1(X,\Theta)
\end{align*}
which gives a Kodaira Spencer map $\kappa:T_{t_0}B\to H^1(X,\Theta)$. They investigated the problem of classifying all small deformations of $X$, by constructing a ``complete family" of deformations of $X$ which roughly means that every small deformation of $X$ is induced from the complete family. More precisely, they established the following theorems.

\begin{thm}[Theorem of Existence]
Let $X$ be a compact complex manifold and suppose $H^2(X,\Theta)=0$. Then there exists a complex analytic family $(\mathcal{X},B,\pi)$ with $0\in B\subset \mathbb{C}^m$ satisfying the following conditions:
\begin{enumerate}
\item $\pi^{-1}(0)=M$
\item $\rho_0:\frac{\partial}{\partial t}\to \left(\frac{\partial M_t}{\partial t}\right)_{t=0}$ with $M_t=\pi^{-1}(t)$ is an isomorphism of $T_0(B)$ onto $H^1(M,\Theta):T_0(B)\xrightarrow{\rho_0} H^1(M,\Theta)$.
\end{enumerate}
\end{thm}

\begin{thm}[Theorem of Completeness]
Let $(\mathcal{X},B,\pi)$ be a complex analytic family and $\pi^{-1}(0)=X$. If $\rho_0:T_0 B\to H^1(X,\Theta)$ is surjective, the complex analytic family $(\mathcal{X},B,\pi)$ is compete at $0\in B$.
\end{thm}

By combing these two theorems, we get

\begin{corollary}
If $H^2(X,\Theta)=0$, then there exists a complete family of deformations of $X$ whose Kodaira Spencer map is an isomorphism. If moreover, $H^0(X,\Theta)=0$, then such complete family is universal.
\end{corollary}

Later Kuranishi generalized this result without assumptions on $H^2(X,\Theta)=0$ by relaxing the definition of a family of deformations of $X$ in a way that $B$ is allowed to be an analytic space.

On the other hand,  Grothendieck's algebraic deformation theory is to algebraically formalize Kodaira-Spencer's analytic deformation theory. Let $X$ be an algebraic scheme over $k$, where $k$ is an algebraically closed field. A local deformation, or a local family of deformations of $X$ is a commutative diagram

\begin{center}
$\xi:$$\begin{CD}
X @>>> \mathcal{X}\\
@VVV @VV\pi V\\
Spec(k) @>>> S
\end{CD}$
\end{center}
where $\pi$ is a flat, $S=Spec(A)$ where $A$ is a local $k$-algebra with residue field $k$, and $X$ is identified with the fibre over the closed point. We can define a deformation functor
\begin{align*}
Def_X:\mathcal{A}^*\to (Sets)
\end{align*}
defined by $Def_X(A)=\{\text{local deformations of $X$ over $Spec(A)$}\}/(\text{isomorphisms})$, where $\mathcal{A}^*$ is the category of noetherian local $k$-algebras with the residue $k$. To study the question of representability of the functor $Def_X$ by some noetherian local $k$-algebra $\mathcal{O}$, the approach of Grothendieck was to formalize the method of Kodaira and Spencer, which consists in a formal construction followed by a proof of convergence. One of main problems is on prorepresentabiliy of $Def_X:\bold{Art}\to (Sets)$, where $\bold{Art}$ is the category of local artinian $k$-algebras with residue $k$.

As I said before, deformation theories of holomorphic Poisson manifolds and algebraic Poisson schemes are based on deformation theories of compact complex manifolds and algebraic schemes. The main difference is that we simply put one more structure on complex analytic families or algebraic families, namely ``Poisson structures''. So deformations of compact holomorphic Poisson manifolds, or algebraic Poisson schemes mean that we deform not only underlying complex or algebraic structures, but also Poisson structures.  I will explain small deformations of compact holomorphic Poisson manifolds. Given a holomorphic Poisson manifold $(X,\Lambda_0)$, a family of deformations of $(X,\Lambda_0)$ is a commutative diagram of holomorphic maps between a holomorphic Poisson manifold $(\mathcal{X},\Lambda)$ and a complex manifold $B$
\begin{center}
$\xi:$$\begin{CD}
(X,\Lambda_0) @>>> (\mathcal{X},\Lambda)\\
@VVV @VV\pi V\\
\star @>>> B
\end{CD}$
\end{center}
with $\pi$ is proper and smooth, $B$ is connected and where $\star$ denotes the singleton space. We denote $(\mathcal{X}_t,\Lambda_t)$ the fiber $\pi^{-1}(t), t\in B$ which is a compact holomorphic Poisson submanifold of $(\mathcal{X},\Lambda)$. We call $(\mathcal{X},\Lambda, \pi, B)$ a Poisson analytic family. As in a complex analytic family, we can define, for every tangent vector $\frac{\partial}{\partial t}\in T_{t_0} B$, the derivative of the family along $\frac{\partial}{\partial t}$ as an element
\begin{equation*}
\frac{\partial (\mathcal{X}_t,\Lambda_t)}{\partial t}\in HP^2(X,\Lambda_0)
\end{equation*}
which gives a linear map
\begin{equation*}
\varphi:T_{t_0} B\to HP^2(X,\Lambda_0)
\end{equation*}
called the Poisson Kodaira Spencer map of the family $(\mathcal{X},\Lambda,\pi, B)$. We can also define the concept of a complete family as in deformations of compact complex manifolds. I was interested in the problem of classifying all small deformations of $(X,\Lambda_0)$, by constructing a ``complete family" of deformations of $(X,\Lambda_0)$, but by some technical issues, I believe that I only proved the theorem of existence for holomorphic Poisson structures.
\begin{thm}[Theorem of Existence for  holomorphic Poisson structures]\label{theorem of existence}
Let $(M,\Lambda_0)$ be a compact holomorphic Poisson manifold satisfying some assumption and suppose that $HP^3(M,\Lambda_0)=0$. Then there exists a Poisson analytic family $(
\mathcal{X},\Lambda,B,\pi)$ with $0\in B\subset \mathbb{C}^m$ satisfying the following conditions:
\begin{enumerate}
\item $\pi^{-1}(0)=(X,\Lambda_0)$
\item $\varphi_0:\frac{\partial}{\partial t}\to\left(\frac{\partial (M_t,\Lambda_t)}{\partial t}\right)_{t=0}$ with $(\mathcal{X}_t,\Lambda_t)=\pi^{-1}(t)$ is an isomorphism of $T_0(B)$ onto $HP^2(M,\Lambda_0):T_0 B\xrightarrow{\rho_0} HP^2(M,\Lambda_0)$.
\end{enumerate}
\end{thm}

\begin{conjecture}[Theorem of Completeness for holomorphic Poisson structures]\

 Let $(\mathcal{X},\Lambda, B,\pi)$ be a Poisson analytic family and $\pi^{-1}(0)=(X,\Lambda_0)$. If $\varphi_0:T_0 B\to HP^2(X,\Lambda_0)$ is surjective, then the Poisson analytic family $(\mathcal{X},\Lambda, B,\pi)$ is complete at $0\in B$. 
\end{conjecture}
By combining the theorem and the conjecture, we get
\begin{corollary}
If $HP^3(X,\Lambda_0)=0$, then there exists a complete family of deformations of $(X,\Lambda_0)$ whose Poisson Kodaira Spencer map is an isomorphism. Moreover if $HP^1(X,\Lambda_0)=0$, then such complete family is universal.
\end{corollary}
The natural question is the existence of Kuranishi family for deformations of a holomorphic Poisson manifold. 
\begin{conjecture}
A complete family of deformations of $(X,\Lambda_0)$ such that the Poisson Kodaira Spencer map is an isomorphism exists without assumptions on $HP^3(X,\Lambda_0)=0$ provided the base $B$ is allowed to be an analytic space.
\end{conjecture}
 While I worked on deformations of holomorphic Poisson structures, the reason why I focused on ``theorem of existence", ``theorem of completeness" and``construction of Kuranishi family" for holomorphic Poisson structures is that I wanted to extend the relationship between Kodaira Spencer's analytic deformation theory and Grothendieck's algebraic deformation theory to the relationship between analytic Poisson deformation theory and algebraic deformation theory of Poisson schemes as presented in the book \cite{Ser06}. Now I will explain deformations of algebraic Poisson schemes. Deformations of Poisson schemes have already been studied by Namikawa (\cite{Nam09}), Ginzburg, and Kaledin (\cite{Gin04}). It seems that Ginzburg and Kaledin (\cite{Gin04}) defined firstly deformations of Poisson schemes in the context of Grothendieck's deformation theory. Let's fix an algebraically closed field $k$ and consider an Poisson algebraic $k$-scheme $(X,\Lambda_0)$. A local Poisson deformation or a local family of Poisson deformations of $(X,\Lambda_0)$ is a cartesian diagram

\begin{center}
$\xi:$$\begin{CD}
(X,\Lambda_0) @>>> (\mathcal{X},\Lambda)\\
@VVV @VV\pi V \\
Spec(k) @>>> S
\end{CD}$
\end{center}
where $\pi$ is a flat morphism, $S=Spec\,A$ and $(\mathcal{X},\Lambda)$ is a Poisson $S$-scheme via $\pi$ where $A$ is a local $k$-algebra with residue field $k$, and the Poisson $k$-scheme $(X,\Lambda_0)$ is identified with the fiber over the closed point. Similarly we can define a Poisson deformation functor
\begin{align*}
Def_X:\mathcal{A}^*\to (Sets)
\end{align*}
defined by $Def_X(A)=\{\text{local Poisson deformations of $X$ over $Spec(A)$}\}/(\text{isomorphisms})$, where $\mathcal{A}^*$ is the category of noetherian local $k$-algebras with the residue $k$. We can consider analogous problems coming from classical deformation theory of algebraic schemes.

 I have been guided by the analytic and algebraic deformation theory originating from Kodaira-Spencer and Grothendieck in the context of Poisson category through books, articles and papers from senior mathematicians.  My thesis on deformations of compact holomorphic Poisson manifolds and algebraic Poisson schemes is the sophistication of this general picture.

\part{Deformations of compact holomorphic Poisson manifolds}\label{part1}

In the first part of the thesis, we study deformations of holomorphic Poisson structures in the framework of Kodaira and Spencer's deformation theory of complex analytic structures (\cite{Kod58},\cite{Kod60}). The main difference from Kodaira and Spencer's deformation theory is that  for deformations of a holomorphic Poisson manifold, we deform not only its complex structures, but also holomorphic Poisson structures. We thoroughly apply Kodaira and Spencer's ideas to holomorphic Poisson category.

 Kodaira and Spencer's main idea of deformations of complex analytic structures is as follows \cite[p.182]{Kod05}. A $n$-dimensional compact complex manifold $M$\footnote{In this thesis, we assume that a complex manifold is connected} is obtained by glueing domains $U_1,...,U_n$ in $\mathbb{C}^n:M=\cup_{j=1}^n U_j$ where $\mathfrak{U}=\{U_j|j=1,...,n\}$ is a locally finite open covering of $M$, and that each $U_j$ is a polydisk:
 \begin{align*}
 U_j=\{z_j\in \mathbb{C}^n||z_j^1|<1,...,|z_j^n|<1\}
 \end{align*}
and for $p\in U_j\cap U_k$, the coordinate transformation
\begin{align*}
f_{jk}:z_k\to z_j=(z_j^1,...,z_j^n)=f_{jk}(z_k)
\end{align*}
transforming the local coordinates $z_k=(z_k^1,...,z_k^n)=z_k(p)$ into the local coordinates $z_j=(z_j^1,...,z_j^n)=z_j(p)$ is  biholomorphic. According to Kodaira,

 \begin{quote}
 \textit{$``$A deformation of $M$ is considered to be the glueing of the same polydisks $U_j$ via different identification. In other words, replacing $f_{jk}^{\alpha}(z_k) $ by the functions $f_{jk}^{\alpha}(z_k,t)=f^{\alpha}_{jk}(z_k,t_1,...,t_m),$ $ f_{jk}(z_k,0)=f_{jk}^{\alpha}(z_k)$ of $z_k$, and the parameter $t=(t_1,...,t_m)$, we obtain deformations $M_t$ of $M=M_0$ by glueing the polydiscks $U_1,...,U_n$ by identifying $z_k\in U_k$ with $z_j=f_{jk}(z_k,t)\in U_j$"}
\end{quote}

A $n$-dimensional compact holomorphic Poisson manifold $M$ is a compact complex manifold such that the structure sheaf $\mathcal{O}_M$ is a sheaf of Poisson algebras.(See Appendix \ref{appendixa}) The holomorphic Poisson structure is encoded in a holomorphic section (a holomorphic bivector field) $\Lambda \in H^0(M,\wedge^2 \Theta_M)$ with $[\Lambda,\Lambda]=0$.\footnote{We denote by $T=T_M$  the holomorphic tangent bundle of $M$, by $\Theta_M$ the sheaf of holomorphic vector fields on $M$, by $T^*=T_{M}^*$ by the dual bundle of  $T_M$, by $\bar{T}=\bar{T}_M$ the anti holomorphic tangent bundle, by $\bar{T}^*=\bar{T}^*_M$ the dual bundle of $\bar{T}^*_M$, by $T_{\mathbb{C}} M=T\oplus \bar{T}$ the complexified tangent bundle, and by $T_{\mathbb{C}}^* M=T^* \oplus \bar{T}^*$ the dual bundle of $T_{\mathbb{C}} M$ and the bracket $[-,-]$ is the Schouten bracket. See Appendix \ref{appendixc}.} In the sequel a holomorphic Poisson manifold will be denoted by $(M,\Lambda)$. For deformations of a holomorphic Poisson manifold $(M,\Lambda)$, we use the ideas of Kodaira and Spencer. A $n$-dimensional holomorphic Poisson manifold is obtained by glueing the domains $U_1,...,U_n$ in $\mathbb{C}^n$: $M=\bigcup_{j=1}^n U_j$ where $\mathfrak{U}=\{U_j|j=1,...,n\}$ is a locally finite open covering of $M$ and each $U_j$ is a polydisk 
 \begin{align*}
 U_j=\{z_j\in \mathbb{C}^n||z_j^1|<1,...,|z_j^n|<1\}
 \end{align*}
equipped with a holomorphic bivector fields $\Lambda_j=\sum_{\alpha,\beta=1}^n g_{\alpha\beta}^j(z_j) \frac{\partial}{\partial z_j^{\alpha}}\wedge \frac{\partial}{\partial z_j^{\beta}}$\footnote{In this thesis, we always assume that $g_{\alpha\beta}^j(z)=-g_{\beta\alpha}^j(z)$} with $[\Lambda_j,\Lambda_j]=0$ on $U_j$ and for $p\in U_j\cap U_k$, the coordinate transformation
\begin{align*}
f_{jk}:z_k\to z_j=(z_j^1,...,z_j^n)=f_{jk}(z_k)
\end{align*}
transforming the local coordinates $z_k=(z_k^1,...,z_k^n)=z_k(p)$ into the local coordinates $z_j=(z_j^1,...,z_j^n)=z_j(p)$ is  a biholomorphic Poisson map.\footnote{For the definition of Poisson map, See Appendix \ref{appendixa}.} 

 Deformations of a holomorphic Poisson manifold $(M,\Lambda)$ is the glueing of the Poisson polydisks $(U_j,\Lambda_j(t))$ parametrized by $t$ via different identification. That is, replacing $f_{jk}^{\alpha}(z_k)$ by the functions $f_{jk}^{\alpha}(z_k,t) ( f_{jk}(z_k,0)=f_{jk}^{\alpha}(z_k)$ of $z_k$), replacing $\Lambda_j=\sum_{\alpha,\beta=1}^n g_{\alpha\beta}^j(z_j) \frac{\partial}{\partial z_j^{\alpha}}\wedge \frac{\partial}{\partial z_j^{\beta}}$ by $\Lambda_j(t)=\sum_{\alpha,\beta=1}^n g_{\alpha\beta}^j(z_j,t) \frac{\partial}{\partial z_j^{\alpha}}\wedge \frac{\partial}{\partial z_j^{\beta}}$ with $[\Lambda_j(t),\Lambda_j(t)]=0$ and $\Lambda_j(0)=\Lambda_j$, and the parmeter $t=(t_1,...,t_m)$, we obtain defomrations $(M_t,\Lambda_t)$ by gluing the Poisson polydisks $(U_1,\Lambda_1(t)),...,(U_n,\Lambda_n(t))$ by identifying $z_k\in U_k$ with $z_j=f_{jk}(z_k,t)\in U_j$. The work on deformations of holomorphic Poisson structures is based on this fundamental idea.

In chapter \ref{chapter1}, we define a family of compact holomorphic Poisson manifolds, called a Poisson analytic family in the framework of Kodaira-Spencer deformation theory. In other words, when we ignore Poisson structures, a family of compact holomorphic Poisson manifolds is just a family of compact complex manifolds in the sense of Kodaira and Spencer. So deformations of holomorphic Poisson manifolds means that we deform complex structures as well as Poisson structures. And we show that infinitesimal deformation of a holomorphic Poisson manifold $(M,\Lambda)$ in a Poisson analytic family is encoded in the truncated holomorphic Poisson cohomology. More precisely, an infinitesimal deformation is realized as an element in the second hypercohomology group\footnote{We adopt the notation from \cite{Nam09} for the expression of the truncated holomorphic Poisson cohomology groups}  $HP^2(M,\Lambda)$ of a complex of sheaves $0\to \Theta_M\to \wedge^2 \Theta_M\to \cdots\to \wedge^n \Theta_M\to 0$ induced by $[\Lambda,-]$. Analogously to deformations of complex structure, we define so called Poisson Kodaira Spencer map where the Kodaira Spencer map is realized as a component of the Poisson Kodaira Spencer map. We define a concept of a trivial family, locally trivial family, rigidity and pullback family, and raise some questions that I cannot answer at this stage. 

In chapter \ref{chapter2}, we study the integrability condition for a Poisson analytic family. Kodaira showed that given a family of deformations of  a compact complex manifold $M$, locally the family is represented by a $C^{\infty} (0,1)$ vectors $\varphi(t)$ with $\varphi(0)=0$ satisfying $\bar{\partial}\varphi(t)-\frac{1}{2}[\varphi(t),\varphi(t)]=0$. And we show that given a family of deformations of a holomorphic Poisson manifold $(M,\Lambda)$, locally the family is represented by a $C^{\infty} (0,1)$ vectors $\varphi(t)$ with $\varphi(0)=0$ and a $C^{\infty}$ bivectors $\Lambda(t)$ with $\Lambda(0)=\Lambda$ satisfying $[\Lambda(t),\Lambda(t)]=0, \bar{\partial} \Lambda(t)-[\Lambda(t),\varphi(t)]=0$, and $\bar{\partial}\varphi(t)-\frac{1}{2}[\varphi(t),\varphi(t)]=0$. Replacing $\varphi(t)$ by $-\varphi(t)$, the integrability condition becomes $\bar{\partial}(\varphi(t)+\Lambda(t))+\frac{1}{2}[\varphi(t)+\Lambda(t),\varphi(t)+\Lambda(t)]=0$ which is a solution of the Maurer Cartan equation of the differential graded Lie algebra $(\mathfrak{g},\bar{\partial},[-,-])$.(See Appendix \ref{appendixc}). But we have another differential graded Lie algebra structure on $\mathfrak{g}$.(See Proposition \ref{d}) If we take $\Lambda'(t)=\Lambda(t)-\Lambda$. Then we have $\Lambda'(0)=0$ and the integrability condition is equivalent to $L(\varphi(t)+\Lambda'(t))+\frac{1}{2}[\varphi(t)+\Lambda'(t),\varphi(t)+\Lambda'(t)]=0$ where $L=\bar{\partial}+[\Lambda,-]$.\footnote{We remark that the integrability condition was proved in more general context in the language of generalized complex geometry (See \cite{Gua11}). As $H^1(M,\Theta)$ is realized as a subspace of the generalized second cohomology group of a complex manifold $M$, $HP^2(M,\Lambda)$ is realized as a subspace of the generalized second cohomology group of a holomorphic Poisson manifold $(M,\Lambda)$. In this thesis, we deduce the integrability condition by following Kodaira's original approach, that is,   by starting from a concept of a geometric family (a Poisson analytic family). } Then $\varphi(t)+\Lambda'(t)$ is a solution of the Mauer Cartan equation of the differential graded Lie algebra $(\mathfrak{g},L,[-,-])$. In the part II of the thesis, we show that the differential graded Lie algebra $(\mathfrak{g},L,[-,-])$ is a differential graded Lie algebra governing the holomorphic Poisson deformations of $(M,\Lambda)$ in the language of functor of Artin rings.

In chapter \ref{chapter3}, under some  assumption, we establish an analogous theorem to the following theorem of Kodaira and Spencer.(\cite{Kodaira58},\cite{Kod05} p.270)
\begin{thm}[Theorem of Existence]
Let $M$ be a compact complex manifold and suppose $H^2(M,\Theta)=0$. Then there exists a complex analytic family $(\mathcal{M},B,\omega)$ with $0\in B\subset \mathbb{C}^m$ satisfying the following conditions:
\begin{enumerate}
\item $\omega^{-1}(0)=M$
\item $\rho_0:\frac{\partial}{\partial t}\to \left(\frac{\partial M_t}{\partial t}\right)_{t=0}$ with $M_t=\omega^{-1}(t)$ is an isomorphism of $T_0(B)$ onto $H^1(M,\Theta):T_0(B)\xrightarrow{\rho_0} H^1(M,\Theta)$.
\end{enumerate}
\end{thm}
Similiary, under the assumption (\ref{assumption}), we prove the theorem of existence for deformations of holomorphic Poisson structures.(See Theorem \ref{theorem of existence}) 
\begin{thm}[Theorem of Existence for  holomorphic Poisson structures]\label{theorem of existence}
Let $(M,\Lambda_0)$ be a compact holomorphic Poisson manifold satisfying $($\ref{assumption}$)$ and suppose that $HP^3(M,\Lambda_0)=0$. Then there exists a Poisson analytic family $(
\mathcal{M},\Lambda,B,\omega)$ with $0\in B\subset \mathbb{C}^m$ satisfying the following conditions:
\begin{enumerate}
\item $\omega^{-1}(0)=(M,\Lambda_0)$
\item $\varphi_0:\frac{\partial}{\partial t}\to\left(\frac{\partial (M_t,\Lambda_t)}{\partial t}\right)_{t=0}$ with $(M_t,\Lambda_t)=\omega^{-1}(t)$ is an isomorphism of $T_0(B)$ onto $HP^2(M,\Lambda_0):T_0 B\xrightarrow{\rho_0} HP^2(M,\Lambda_0)$.
\end{enumerate}
\end{thm}
Our proof is rather formal. We throughly follow the Kuranishi methods presented in \cite{Mor71}. The reason for the assumption is to apply their methods in the holomorphic Poisson context and my unfamiliarity with the analytic properties of the operator $\bar{\partial}+[\Lambda,-]$.  I do not know that we could relax the assumption (\ref{assumption}).  Lastly, based on Kuranishi's lecture notes \cite{Kur71}, we define a Poisson analytic family over a complex space, a concept of pullback family, and complete family. We pose a problem on the existence of Kuranishi family in holomorphic Poisson context.  I could not access to this problem for my unfamiliarity with analysis behind the operator $\bar{\partial} +[\Lambda,-]$.

\chapter{Poisson analytic families}\label{chapter1}

\section{Families of holomorphic Poisson manifolds}

\begin{definition}$($compare $\cite{Kod05}$ p.59$)$ \label{test}
Suppose given a domain $B\in \mathbb{C}^m$ , and a set $\{(M_t,\Lambda_t)|t \in B\}$ of holomorphic Poisson manifolds $(M_t,\Lambda_t)$, depending on $t\in B$. We say that $\{(M_t,\Lambda_t)|t\in B\}$ is a family of compact holomorphic Poisson manifolds or a Poisson analytic family of compact holomorphic Poisson manifolds if $(\mathcal{M},\Lambda)$ is a holomorphic Poisson manifold and a holomorphic map $\pi:\mathcal{M}\to B$ satisfies the following properties
\begin{enumerate}
\item $\pi^{-1}(t)$ is a compact holomorphic Poisson submanifolds of $(\mathcal{M},\Lambda)$.
\item $(M_t,\Lambda_t)=\pi^{-1}(t)(M_t$ has the induced Poisson holomorphic structure $\Lambda_t$ from $\Lambda)$.
\item The rank of Jacobian of $\pi$ is equal to $m$ at every point of $\mathcal{M}$.
\end{enumerate}
\end{definition}

Then we can choose a system of local complex coordinates $\{z_1,...,z_j,...\},z_j:p\to z_j(p)$, and coordinate polydisks $U_j$ with respect to $z_j$ satisfying the following conditions.
\begin{enumerate}
\item $z_j(p)=(z_j^1(p),...,z_j^n(p),t_1,...,t_m),(t_1,...,t_m)=\omega(p)$
\item $\mathcal{U}=\{\mathcal{U}_j|j=1,2,...\}$ is locally finite.
\end{enumerate}
Then 
\begin{align*}
\{p\mapsto (z_j^1(p),...,z_j^n(p))| \mathcal{U}_j \cap M_t\ne \emptyset\}
\end{align*}
gives a system of local complex coordinates on $M_t$. In terms of these coordinates, $\omega$ is the projection given by
\begin{align*}
\omega:(z_j^1,...,z_j^n,t_1,...,t_m)\to (t_1,...,t_m).
\end{align*}
For $j,k$ with $U_j\cap U_k\ne \emptyset$, we denote the coordinate transformations from $z_k$ to $z_j$ by
\begin{align*}
f_{jk}:(z_k^1,...,z_k^n,t)\to (z_j^1,...,z_j^n,t)=f_{jk}(z_k^1,...,z_k^n,t)
\end{align*}
Note that $t_1,...,t_m$ as part of local coordinates on $\mathcal{M}$ do not change under these coordinate transformations. Thus $f_{jk}$ is given by
\begin{align*}
z_j^{\mathcal{\alpha}}=f_{jk}(z_k^1,...,z_k^n,t_1,...,t_m),\,\,\,\,\,\,\alpha=1,...,n.
\end{align*}

We now discuss the holomorphic Poisson structures. Since $M_t\hookrightarrow \mathcal{M}$ is a holomorphic Poisson manifold induced from $\Lambda$ and $\mathcal{M}=\cup M_t$, the Poisson structure $\Lambda$ on $\mathcal{M}$ can be expressed in terms of local coordinates as $\Lambda=g_{\alpha \beta}(z_j^1,...,z_j^n,t)\frac{\partial}{\partial{z_j^{\alpha}}}\wedge \frac{\partial}{\partial{z_j^{\beta}}}$ on $U_j$. For fixed $t^0$, the holomorphic Poisson structure $\Lambda_{t^0}$ on $M_{t_0}$ is given by $g_{\alpha \beta}(z_j^1,...,z_j^n,t^0)\frac{\partial}{\partial{z_j^{\alpha}}}\wedge \frac{\partial}{\partial{z_j^{\beta}}}$ by restricting $\Lambda$ to $M_{t^0}$ and $g_{\alpha\beta}(z_j,t)$ is  holomorphic with respect to $z_j$.


Of course, the definition can be extended to the case $B$ is an arbitrary complex manifold. We also could define a family of compact holomorphic Poisson manifolds in the following way.

\begin{definition}$($compare $\cite{Uen99}$ p.2$)$
A holomorphic map $\pi$ of $\mathcal{M}$ onto $B$ is called a family of compact holomorphic Poisson manifolds or a Poisson analytic family of compact holomorphic poisson manifolds if theres a holomorphic Poisson manifold $(\mathcal{M},\Lambda)$ satisfying the following conditions
\begin{enumerate}
\item $\pi$ is proper. In other words, inverse image of a compact set in $B$ is compact.
\item $\pi$ is a submersion. In other words, for each point $x\in \mathcal{M}$, $d\pi_x:T_x \mathcal{M} \to T_{\pi(x)} B$ is surjective.\\
$($The above two conditions imply that $\pi^{-1}(t)$ is a complex submanifold of $\mathcal{M}$ for each $t\in B$.$)$
\item $\pi^{-1}(t)$ is a holomorphic Poisson submanifold of $(\mathcal{M},\Lambda)$ for each $t\in B$
\item $\pi^{-1}(t)$ is connected.
\end{enumerate}
\end{definition}

\begin{example}[complex tori]$($$\cite{Kod58}$ p.408$)$
Let $S$ be the space of $n\times n$ matrices $s=(s_{\beta}^{\alpha})$ with $|\mathfrak{J} s | >0$, where $\alpha$ denotes the row index and $\beta$ the column index. For each matrix $s\in S$ we define an $n\times 2n$ matrix $\omega(s)=(\omega_j^{\alpha}(s))$ by
\begin{equation} \label{chunghoon}
\omega_j^{\alpha}(s)=
\begin{cases}
\delta_j^{\alpha},\,\,\,\,\,\,\,\,\,\, $\text{for $1\leq j\leq n$}$\\
s_j^{\alpha},\,\,\,\,\,\,\,\,\ $\text{for $j=n+\beta, 1\leq \beta \leq n$}$
\end{cases}
\end{equation}

Let $\mathbb{C}^n$ be the space of $n$ complex variables $z=(z^1,...,z^{\alpha},...,z^n)$ and let $G$ be the discontinuous abelian group of analytic automorphisms of $\mathbb{C}^n\times S$ generated by
\begin{align*}
g_j:(z,s)\to (z+\omega_j(s),s),\,\,\,\,\, j=1,...,2n,
\end{align*}
where $\omega_j(s)=(\omega_j^1(s),...,\omega_j^{\alpha}(s),...,\omega_j^n(s))$ is th $j$-th column vector of $\omega(s)$. The factor space
\begin{align*}
\mathscr{B}=\mathbb{C}^n\times S/G
\end{align*}
is obviously a complex manifold and the canonical projection $\mathbb{C}^m\times S\to S$ induces a regular map $\omega:\mathscr{B}\to S$ such that $B_s=\omega^{-1}(s)$ is a complex torus of complex dimension $n$ with the periods $\omega_j(s)$ $(j=1,...,2n)$ Then $\mathscr{B}=\{B_s,s\in S\}$ forms a complex analytic family of complex tori.

We would like to describe a $G$-invariant holomorphic bicvector field of the form $\sum_{i,j} f(z,s)\frac{\partial}{\partial z_i}\wedge \frac{\partial}{\partial z_j}$ on $\mathbb{C}^n\times S$. Then this induces a holomorphic bivector field on $\mathscr{B}$. 
Any element of $g\in G$ is of the form $g:(z,s)\to (z+m_1\omega_1(s)+\cdots +m_n\omega_{2n}(s),s)$. Hence for $\sum_{i,j} f_{ij}(z,s)\frac{\partial}{\partial z_i}\wedge \frac{\partial}{\partial z_j}$ to be invariant vector field, we have $f_{ij}(z,s)=f_{ij}(z+m_1\omega_1(s)+\cdots +m_n\omega_{2n} (s),s)$ for any inters $m_1,...,m_{2n}$. This means that $f(z,s)$ is independent of $z$. So $f(z,s)=f(s)$. Hence an invariant bivector field is of the from $\Lambda=\sum_{i,j}f_{ij}(s)\frac{\partial}{\partial z_i}\wedge \frac{\partial}{\partial z_j}$. Since $f_{ij}(s)$ are independent of $z$, we have $[\Lambda,\Lambda]=0$. So $(\mathscr{B},\Lambda)$ is a Poisson analytic family.

\end{example}

\begin{example}[Hirzebruch-Nagata surface]$($$\cite{Uen99}$ p.13$)$
Take two $\mathbb{C}\times \mathbb{P}^1$ and write the coordinates as $(u,(\xi_0:\xi_1), (v,(\eta_0:\eta_1))$, respectively, where $(\xi_0:\xi_1),(\eta_0:\eta_1)$ are the homogeneous coordinates of $\mathbb{P}^1$.
Patch $\mathbb{C}\times \mathbb{P}^1$ together by relation

\begin{equation}
\begin{cases}
u=1/v, \\
(\xi_0:\xi_1)=(\eta_0:v^m\eta_1)
\end{cases}
\end{equation}

Then we obtain a two dimensional compact  complex manifold $F_m$. The complex manifold $F_m$ is called the Hirzebruch-Nagata surface.
Now we deform the patching by introducing  a new patching relation with parameter $t\in \mathbb{C}$.

\begin{equation}
\begin{cases}
u=1/v, \\
(\xi_0:\xi_1)=(\eta_0:v^m\eta_1+tv^k\eta_0), m-2\leq 2k \leq  m
\end{cases}
\end{equation}
and patching two $\mathbb{C}\times \mathbb{P}^1$ by the relation, we obtain a surface $S_t$ for each $t\in \mathbb{C}$. By the relation we have $S_0=F_m$ and $\omega:\mathcal{S}=\{S_t\}_{t\in\mathbb{C}} \to \mathbb{C}$ is an complex analytic family. 

We  put a holomorphic Poisson structure $\Lambda$ on $\mathcal{S}$ so that $(\mathcal{S},\Lambda)$ is a Poisson analytic family. 

For one $\mathbb{C}\times \mathbb{P}^1$ with coordinate $(u,(\xi_0:\xi_1))$, we have two affine covers, namely, $\mathbb{C}\times \mathbb{C}$ and $\mathbb{C}\times \mathbb{C}$. They are glued via  $\mathbb{C}\times (\mathbb{C}-\{0\})$ and  $\mathbb{C}\times (\mathbb{C}-\{0\})$ by $(u,x=\frac{\xi_1}{\xi_0})\mapsto (u,y=\frac{\xi_0}{\xi_1})=(u,\frac{1}{x})$. Similary for another $\mathbb{C}\times \mathbb{P}^1$ they are glued via $\mathbb{C}\times (\mathbb{C}-\{0\})$ and $\mathbb{C}\times (\mathbb{C}-\{0\})$ and $\mathbb{C}\times (\mathbb{C}-\{0\})$ via $(v,w=\frac{\eta_1}{\eta_0})\mapsto (v,z)=(v,\frac{1}{w}=\frac{\eta_0}{\eta_1})$. We put holomorphic Poisson structures $\Lambda$ with $[\Lambda,\Lambda]=0$ on each patches and show that they are glued via the above relations to give a global bivector field on $\mathcal{S}$. On $(u,x)$ coordinate, we give $g(t)x^2\frac{\partial}{\partial u}\wedge \frac{\partial}{\partial x}$. On $(u,y)$ coordinate, we give $-g(t)\frac{\partial}{\partial u}\wedge \frac{\partial}{\partial y}$. On $(v,w)$ coordinate, we give $-g(t)v^{2k-m+2}(wv^{m-k}+t)^2\frac{\partial}{\partial v}\wedge\frac{\partial}{\partial w}$. And on $(v,z)$ coordinate, we give $g(t)v^{2k-m+2}(v^{m-k}+tz)^2\frac{\partial}{\partial v}\wedge \frac{\partial}{\partial z}$. In the following  picture, we have
{\tiny{
\begin{center}
$\begin{CD}
(u,x)=(\frac{1}{v},v^mw+tv^k)=(\frac{1}{v},\frac{v^m+tv^kz}{z}),g(t)x^2\frac{\partial}{\partial u}\wedge \frac{\partial}{\partial x}@<<< (v,w),-g(t)v^{2k-m+2}(wv^{m-k}+t)^2\frac{\partial}{\partial v}\wedge\frac{\partial}{\partial w}\\
@VVV @VVV \\
(u,y)=(u,\frac{1}{x})=(\frac{1}{v},\frac{z}{v^m+tv^kz})=(\frac{1}{v},\frac{1}{v^mw+tv^k}),-g(t)\frac{\partial}{\partial u}\wedge \frac{\partial}{\partial y}@<<< (v,z=\frac{1}{w}),g(t)v^{2k-m+2}(v^{m-k}+tz)^2\frac{\partial}{\partial v}\wedge \frac{\partial}{\partial z}
\end{CD}$
\end{center}}}

we have $\frac{\partial}{\partial x}=-\frac{1}{x^2}\frac{\partial}{\partial y}$, $\frac{\partial}{\partial v}=-\frac{1}{v^2}\frac{\partial}{\partial u}+(mv^{m-1}w+ktv^{k-1})\frac{\partial}{\partial x}=-\frac{1}{v^2}\frac{\partial}{\partial u}+\frac{mv^{m-1}+ktv^{k-1}z}{z}\frac{\partial}{\partial x}=-\frac{1}{v^2}\frac{\partial}{\partial u}+\frac{-mv^{m-1}-ktv^{k-1}}{(v^m+tv^k)^2}\frac{\partial}{\partial y}$.

$\frac{\partial}{\partial w}=v^m\frac{\partial}{\partial x}=\frac{-v^m}{(v^mw+tv^k)^2}\frac{\partial}{\partial y}$, and $\frac{\partial}{\partial z}=\frac{-v^m}{z^2}\frac{\partial}{\partial x}=\frac{v^m}{(v^m+tv^kz)^2}\frac{\partial}{\partial y}$, $\frac{\partial}{\partial w}=-\frac{1}{w^2}\frac{\partial}{\partial z}$.

Now we show that they are glued. 
\begin{enumerate}
\item $(u,x)$ and $(u,y)$. $g(t)x^2\frac{\partial}{\partial u}\wedge \frac{\partial}{\partial x}=g(t)x^2(-\frac{1}{x^2})\frac{\partial}{\partial u}\wedge \frac{\partial}{\partial y}=-g(t)\frac{\partial}{\partial u}\wedge \frac{\partial}{\partial y}$
\item $(v,w)$ and $(v,z)$. $-g(t)v^{2k-m+2}(wv^{m-k}+t)^2\frac{\partial}{\partial v}\wedge\frac{\partial}{\partial w}=-g(t)v^{2k-m+2}(wv^{m-k}+t)^2(\frac{-1}{w^2})\frac{\partial}{\partial v}\wedge\frac{\partial}{\partial z}=g(t)v^{2k-m+2}(v^{m-k}+tz)^2\frac{\partial}{\partial v}\wedge \frac{\partial}{\partial z}$

\item $(u,x)$ and $(v,w)$. $-g(t)v^{2k-m+2}(wv^{m-k}+t)^2\frac{\partial}{\partial v}\wedge\frac{\partial}{\partial w}=-g(t)v^{2k-m+2}(wv^{m-k}+t)^2(-v^{m-2})\frac{\partial}{\partial u}\wedge\frac{\partial}{\partial x}=g(t)x^2\frac{\partial}{\partial u}\wedge \frac{\partial}{\partial x}$
\item $(u,y)$ and $(v,z)$. $g(t)v^{2k-m+2}(v^{m-k}+tz)^2\frac{\partial}{\partial v}\wedge \frac{\partial}{\partial z}=g(t)v^{2k-m+2}(v^{m-k}+tz)^2\frac{v^m}{(v^m+tv^kz)^2}(-\frac{1}{v^2})\frac{\partial}{\partial u}\wedge \frac{\partial}{\partial y}=-g(t)\frac{\partial}{\partial u}\wedge \frac{\partial}{\partial y}$
\item $(u,x)$ and $(v,z)$. $g(t)v^{2k-m+2}(v^{m-k}+tz)^2\frac{\partial}{\partial v}\wedge \frac{\partial}{\partial z}=g(t)v^{2k-m+2}(v^{m-k}+tz)^2\frac{v^{m-2}}{z^2}\frac{\partial}{\partial u}\wedge \frac{\partial}{\partial x}=g(t)(\frac{v^{m}}{z}+tv^k)^2\frac{\partial}{\partial u}\wedge \frac{\partial}{\partial x}=g(t)x^2\frac{\partial}{\partial u}\wedge \frac{\partial}{\partial x}$
\item $(u,y)$ and $(v,w)$. $-g(t)v^{2k-m+2}(wv^{m-k}+t)^2\frac{\partial}{\partial v}\wedge\frac{\partial}{\partial w}=-g(t)v^{2k-m+2}(wv^{m-k}+t)^2\frac{v^{m-2}}{(v^mw+tv^k)^2}\frac{\partial}{\partial u}\wedge\frac{\partial}{\partial y}=-g(t)\frac{\partial}{\partial u}\wedge \frac{\partial}{\partial y}$

\end{enumerate}

So $(\mathcal{S},\Lambda)$ is a Poisson analytic family.
\end{example}

\begin{example}[Hopf surfaces]
By a Hopf surface we mean any complex manifold homeomorphic to $S^1\times S^3$. Let $W=\mathbb{C}^2-\{0\}$.
\begin{thm}\label{w}
For every Hopf surface $X$ there exist numbers $m\in \mathbb{N},a,b,t, \in \mathbb{C}$ satisfying
\begin{equation*}
0<|a|\leq |b| <1\,\,\,\,\, \text{and}\,\,\,\,\, (b^m-a)t=0
\end{equation*}
such that $X$ is biholomorphic to $W/<\gamma>$, where $\gamma$ is an automorphism of $W$ given by 
\begin{equation*}
\gamma(z_1,z_2)=(az_1+tz_2^m,bz_2)
\end{equation*}
Conversely, for any $m,a,b,t$ as above, the corresponding group $<\gamma>$ acts freely and properly discontinuous on $W$ and the complex manifold $W/<\gamma>$ is a Hopf surface.
\end{thm}
\begin{proof}
See $\cite{Kod65}$.
\end{proof}

We construct an one parameter Poisson analytic family of general Hopf surfaces.
An automorphism of $W\times \mathbb{C}$ given by
\begin{equation*}
g:(z_1,z_2,t)\to(az_1+tz_2^m,bz_2,t)
\end{equation*}
where $0<|a|\leq |b| <1$ and $b^m-a=0$$($i.e $a=b^m$$)$, generates an infinitely cyclic group $G$, which properly discontinuous and fixed point free by Theorem \ref{w}. Hence $\mathcal{M}=W\times \mathbb{C}/G$ is a complex manifold. Since the projection of $W\times \mathbb{C}$ to $\mathbb{C}$ commutes with $g$, it induces a holomorphic map $\omega$ of $\mathcal{M}$ to $\mathbb{C}$. Clearly the rank of the Jacobian matrix of $\omega$ is equal to 1. Thus $(\mathcal{M},\mathbb{C},\omega)$ is a complex analytic family with $\omega^{-1}(t)=W/G_t=M_t$.

We give a holomorphic Poisson structure on $\mathcal{M}$. A holomorphic bivector field on $\mathcal{M}$ is induced from a $G$-invariant holomorphic bivector field on $W\times \mathbb{C}$. In what follows we write $(z',t')=(z_1',z_2',t')$ instead of $(a^n z_1+na^{n-1}t z_2^m,b^n z_2,t)$. In this notation we have 
\begin{equation*}
g^n:(z_1,z_2,t)\to(z_1',z_2',t'),
\end{equation*}

We consider a $G$-invariant holomorphic bivector field on $W\times \mathbb{C}$ of the form $f(z_1,z_2,t)\frac{\partial}{\partial z_1}\wedge\frac{\partial}{\partial z_2}$, where $f(z_1,z_2,t)$ is a holomorphic function on $W\times \mathbb{C}$. Since
\begin{equation*}
\frac{\partial}{\partial z_1}=a^n\frac{\partial}{\partial z_1'},\,\,\,\,\, \frac{\partial}{\partial z_2}=mna^{n-1}t z_2^{m-1}\frac{\partial}{\partial z_1'}+b^n\frac{\partial}{\partial z_2'}
\end{equation*}
the bivector field $f(z_1,z_2,t)\frac{\partial}{\partial z_1}\wedge\frac{\partial}{\partial z_2}$ is transformed by $g^n$ into the bivector field 
\begin{equation*}
f(z_1,z_2,t)a^n b^n\frac{\partial}{\partial z_1'}\wedge\frac{\partial}{\partial z_2'}
\end{equation*}
Since $f(z_1,z_2,t)\frac{\partial}{\partial z_1}\wedge\frac{\partial}{\partial z_2}$ is $G$-invariant, we have

\begin{equation*}
f(z_1',z_2',t')=f(z_1,z_2,t)a^nb^n
\end{equation*}
By Hartog's theorem, holomorphic function $f(z_1,z_2,t)$ on $W\times \mathbb{C}$ are extended to holomorphic function on $\mathbb{C}^2\times \mathbb{C}$. Therefore we may assume that $f(z_1,z_2,t)$ is holomorphic on all $\mathbb{C}^2\times \mathbb{C}$. We have
\begin{equation*}
f(z_1,z_2,t)=\frac{1}{a^n b^n}f(a^n z_1+na^{n-1}t z_2^m,b^n z_2,t)=\frac{1}{b^{n(m+1)}}f(b^{nm} z_1+nb^{m(n-1)}t z_2^m,b^n z_2,t)
\end{equation*}
Consequently, since $0<|b |<1$, letting
\begin{equation*}
f(z_1,z_2,t)=\sum_{i,j,k}^{+\infty} c_{ijk} {z_1}^i {z_2}^j t^k
\end{equation*}
be the power series expansion of $f(z_1,z_2,t)$, we have
\begin{align*}
f(z_1,z_2,t)&=\lim_{n\to +\infty} \frac{1}{b^{n(m+1)}} \sum_{i,j,k} c_{ijk}(b^{nm} z_1+nb^{m(n-1)}t z_2^m)^i(b^n z_2)^j t^k\\
                 & =(c_{0(m+1)0}+c_{0(m+1)1}t+c_{0(m+1)2}t^2+\cdots)z_2^{m+1}
\end{align*}

Hence $(\mathcal{M},([(c_{0(m+1)0}+c_{0(m+1)1}t+c_{0(m+1)2}t^2+\cdots)z_2^{m+1}]\frac{\partial}{\partial z_1}\wedge \frac{\partial}{\partial z_2})$ is a Poisson analytic family of Hopf surfaces. For each $t$, we have a holomorphic Poisson structure $[(c_{0(m+1)0}+c_{0(m+1)1}t+c_{0(m+1)2}t^2+\cdots)z_2^{m+1}]\frac{\partial}{\partial z_1}\wedge \frac{\partial}{\partial z_2}$ on a Hopf surface $M_t$.

\end{example}

\section{Infinitesimal deformation}

\subsection{Infinitesimal deformation and truncated holomorphic Poisson cohomology}\

In this section, we show that for a Poisson analytic family $(\mathcal{M},\Lambda,B,\omega)$, the infinitesimal deformation of a holomorphic Poisson manifold $(M_t,\Lambda_t)$ with dimension $n$ is captured by the second Hyperohomology group of complex of sheaves $0\to \Theta_{M_t}\to \wedge^2 \Theta_{M_t}\to \cdots \to \wedge^n \Theta_{M_t}\to 0$ induced by $[\Lambda_t,-]$\footnote{For the definition, see Appendix \ref{appendixb}} analogously to how the infinitesimal deformation of a complex manifold $M_t$ is captured by the first cohomology group $H^1(M_t,\Theta_t)$.

Let $(M,\Lambda)$ be a holomorphic Poisson manifold and consider the complex of sheaves 
\begin{align*}
0\to \Theta_M\xrightarrow{[\Lambda,-]}\wedge^2 \Theta_M\xrightarrow{[\Lambda,-]}\cdots \xrightarrow{[\Lambda,-]} \wedge^n \Theta_M\to 0
\end{align*}
where $\Theta_M$ is the holomorphic tangent sheaf. Let $\mathcal{U}=\{U_j\}$ be sufficiently fine open covering of $M$ such that $U_j=\{z_j\in \mathbb{C}^n||z_j^{\alpha}|<r_j^{\alpha},\alpha=1,...,n\}$. Then we can compute the hypercohomology group of the above complex of sheaves by the following \u{C}ech resolution.(See Appendix \ref{appendixb})
\begin{center}
$\begin{CD}
@A[\Lambda,-]AA\\
C^0(\mathcal{U},\wedge^3 \Theta_M)@>-\delta>>\cdots\\
@A[\Lambda,-]AA @A[\Lambda,-]AA\\
C^0(\mathcal{U},\wedge^2 \Theta_M)@>\delta>> C^1(\mathcal{U},\wedge^2 \Theta_M)@>-\delta>>\cdots\\
@A[\Lambda,-]AA @A[\Lambda,-]AA @A[\Lambda,-]AA\\
C^0(\mathcal{U},\Theta_M)@>-\delta>>C^1(\mathcal{U},\Theta_M)@>\delta>>C^2(\mathcal{U},\Theta_M)@>-\delta>>\cdots\\
@AAA @AAA @AAA @AAA \\
0@>>>0 @>>> 0 @>>> 0@>>> \cdots
\end{CD}$
\end{center}

\begin{definition}
We say that the $i$-th truncated holomorphic Poisson cohomology group\footnote{In \cite{Wei99}, holomorphic Poisson cohomology for a holomorphic Poisson manifold $(M,\Lambda)$ is defined by the $i$-th hypercohomology group of complex of sheaves $\mathcal{O}_M\to \Theta_M\to \wedge^2 \Theta_M \to \cdots \to\wedge^n \Theta_M\to 0$ induced by $[\Lambda,-]$. However  since there is no role of the structure sheaf $\mathcal{O}_M$ in deformations of holomorphic Poisson manifolds, we truncate the complex of sheaves. See also \cite{Nam09}. } of a holomorphic Poisson manifold $(M,\Lambda)$ is the $i$-th hypercohomology group associated with the complex of sheaves $0\to \Theta_M\xrightarrow{[\Lambda,-]} \wedge^2 \Theta_M\xrightarrow{[\Lambda,-]} \cdots \xrightarrow{[\Lambda,-]} \wedge^n \Theta_M\to 0$ where $\Theta_M$ is the holomorphic tangent sheaf, and is denoted by $HP^i(X,\Lambda)$\footnote{We adopt the notation from \cite{Nam09}. By general philosophy of deformation theory, it might be natural to shift the grading after truncation so that the $0$-th cohomology group corresponds to infinitesimal automorphisms, the first cohomology group corresponds to infinitesimal deformations and third cohomology group corresponds to obstructions. However, we follow \cite{Nam09}. So I put $0\to 0\to 0\cdots$  on the bottom of the complex.}
\end{definition}

 Now we relate the 2nd  truncated holomorphic Poisson cohomology group $HP^2(M_t,\Lambda_t)$ to the infinitesimal deformation of $(M_t,\Lambda_t)$ in a Poisson  analytic family $(\mathcal{M},\Lambda,B,\pi)$ for each $t$. Let $t_0\in B$ and choose a sufficiently small polydisk $\Delta$ with $t_0\in \Delta \subset B$. Then $\omega^{-1}(\Delta)=\mathcal{M}_{\Delta}=\bigcup_{j=1}^l \mathcal{U}_j$ where   $\mathcal{U}_j:=U_j\times \Delta$ such that $U_j$ is a polydisk, and $(z_j,t)\in U_j\times \Delta$ and $(z_k,t)\in U_k\times \Delta$  are the same point on $\mathcal{M}_{\Delta}$ if 
 \begin{align*}
 z_j^{\alpha}=f_{jk}^{\alpha}(z_k,t),\,\,\, \alpha=1,...,n
 \end{align*}
$z_j^{\alpha}=f_{jk}^{\alpha}(z_k^1,...,z_k^n,t_1,...,t_m)$ 
is a holomorphic transition function in $z_k^1,...,z_k^n,t_1,...,t_m$. And on each local complex coordinate system $U_j\times \Delta$, $\Lambda$ can be expressed as $\sum_{\alpha,\beta} g_{\alpha \beta}^{j}(z,t)\frac{\partial}{\partial z_{j}^{\alpha}}\wedge \frac{\partial}{\partial z_{j}^{\beta}}$ where $g^{j}_{\alpha \beta}(z,t)$ is a holomorphic function on $U_{j}\times \Delta$ and $g_{\alpha \beta}^j(f_{jk}^1(z_k,t),...,f_{jk}^n(z_k,t))=\sum_{r,s} g_{rs}^k(z_k,t)\frac{\partial f_{jk}^{\alpha}}{\partial z_k^r}\frac{\partial f_{jk}^{\beta}}{\partial z_k^s}$ on $U_j\times \Delta\cap U_k\times \Delta$. We will denote $\mathcal{U}_j^t:=U_j\times t$. For each $t\in \Delta$, $\mathcal{U}^t=\{\mathcal{U}_j^t\}$ be an open covering of $M_t$. Let $\frac{\partial}{\partial t}=\sum_{\lambda=1}^m c_{\lambda}\frac{\partial}{\partial t_{\lambda}}$, $c_{\lambda}\in \mathbb{C}$, of $B$. We show that 

\begin{proposition}\label{gg}
\begin{align*}
( \{\theta_{jk}(t)=\sum_{\alpha=1}^n \frac{\partial f_{jk}^{\alpha}(z_k,t)}{\partial t}\frac{\partial}{\partial z_j^{\alpha}}\},\{\Lambda_j(t)=\sum_{\alpha,\beta} \frac{\partial g_{\alpha \beta}^{j}(z,t)}{\partial t}\frac{\partial}{\partial z_{j}^{\alpha}}\wedge \frac{\partial}{\partial z_{j}^{\beta}}\})\in C^1(\mathcal{U}^t,\Theta_{M_t})\oplus C^0(\mathcal{U}^t,\wedge^2 \Theta_{M_t})
\end{align*}
define a 2-cocycle and call its cohomology class $\in HP^2(M_t,\Lambda_t)$ the infinitesimal $($Poisson$)$ deformation along $\frac{\partial}{\partial t}$ and this expression is independent of the choice of system of local coordinates.
\end{proposition}
\begin{proof}
First, $\delta(\{\theta_{jk}(t)\})=0$ (See $\cite{Kod05}$ p.201). Second, since $[\sum_{\alpha,\beta} g_{\alpha \beta}^{j}(z,t)\frac{\partial}{\partial z_{j}^{\alpha}}\wedge \frac{\partial}{\partial z_{j}^{\beta}},\sum_{\alpha,\beta} g_{\alpha \beta}^{j}(z,t)\frac{\partial}{\partial z_{j}^{\alpha}}\wedge \frac{\partial}{\partial z_{j}^{\beta}}]=0$, by taking the derivative with repect to $t$,  we have $[\sum_{\alpha,\beta} \frac{\partial g_{\alpha \beta}^{j}(z,t)}{\partial t}\frac{\partial}{\partial z_{j}^{\alpha}}\wedge \frac{\partial}{\partial z_{j}^{\beta}},\sum_{\alpha,\beta} g_{\alpha \beta}^{j}(z,t)\frac{\partial}{\partial z_{j}^{\alpha}}\wedge \frac{\partial}{\partial z_{j}^{\beta}}]+[\sum_{\alpha,\beta} g_{\alpha \beta}^{j}(z,t)\frac{\partial}{\partial z_{j}^{\alpha}}\wedge \frac{\partial}{\partial z_{j}^{\beta}},\sum_{\alpha,\beta} \frac{\partial g_{\alpha \beta}^{j}(z,t)}{\partial t}\frac{\partial}{\partial z_{j}^{\alpha}}\wedge \frac{\partial}{\partial z_{j}^{\beta}}]=2[\sum_{\alpha,\beta} g_{\alpha \beta}^{j}(z,t)\frac{\partial}{\partial z_{j}^{\alpha}}\wedge \frac{\partial}{\partial z_{j}^{\beta}},\sum_{\alpha,\beta} \frac{\partial g_{\alpha \beta}^{j}(z,t)}{\partial t}\frac{\partial}{\partial z_{j}^{\alpha}}\wedge \frac{\partial}{\partial z_{j}^{\beta}}]=0$. It remains to show that $\delta(\{\Lambda_j(t)\})+[\Lambda_t,\{\theta_{jk}\}]=0$. More precisely, on $\mathcal{U}_{jk}^t$, we show that $\Lambda_{k}(t)-\Lambda_{j}(t)+[\Lambda_t,\theta_{jk}(t)]=0$. In other words,
\begin{align*}
(*)\sum_{r,s=1}^n \frac{\partial g^k_{rs}}{\partial t}\frac{\partial}{\partial z^{r}_k}\wedge\frac{\partial}{\partial z_k^{s}}-\sum_{\alpha,\beta=1}^n \frac{\partial g^j_{\alpha \beta}}{\partial t}\frac{\partial}{\partial z^{\alpha}_j}\wedge\frac{\partial}{\partial z_j^{\beta}}+[\sum_{r,s=1}^n g_{rs}^{j}(z,t)\frac{\partial}{\partial z_{j}^{r}}\wedge \frac{\partial}{\partial z_{j}^{s}},\sum_{c=1}^n \frac{\partial f_{jk}^{c}(z_k,t)}{\partial t}\frac{\partial}{\partial z_j^{c}}]=0
\end{align*}
We note that since $z_j^{\alpha}=f_{jk}^{\alpha}(z_k^1,...,z_k^n,t_1,...,t_m)$ for $\alpha=1,...,n$, $\frac{\partial}{\partial z_k^{r}}=\sum_{a=1}^{n}\frac{\partial f_{jk}^a}{\partial z_k^{r}}\frac{\partial}{\partial z_j^a}$ for $r=1,...,n$. Hence

\begin{align*}
\sum_{r,s=1}^n \frac{\partial g^k_{rs}}{\partial t}\frac{\partial}{\partial z^{r}_k}\wedge\frac{\partial}{\partial z_k^{s}}=\sum_{r,s,a,b=1}^n \frac{\partial g_{rs}^k}{\partial t}\frac{\partial f_{jk}^a}{\partial z_k^r}\frac{\partial f_{jk}^b}{\partial z_k^s}\frac{\partial}{\partial z_j^a}\wedge \frac{\partial}{\partial z_j^b}\\
\end{align*}
\begin{align*}
&[\sum_{r,s=1}^n g_{rs}^{j}(z,t)\frac{\partial}{\partial z_{j}^{r}}\wedge \frac{\partial}{\partial z_{j}^{s}},\sum_{c=1}^n \frac{\partial f_{jk}^{c}(z_k,t)}{\partial t}\frac{\partial}{\partial z_j^{c}}]=\sum_{r,s,c=1}^n [g_{rs}^{j}(z,t)\frac{\partial}{\partial z_{j}^{r}}\wedge \frac{\partial}{\partial z_{j}^{s}},\frac{\partial f_{jk}^{c}(z_k,t)}{\partial t}\frac{\partial}{\partial z_j^{c}}]\\
&=\sum_{r,s,c=1}^n [g_{rs}^j \frac{\partial}{\partial z_j^r},\frac{\partial f_{jk}^c}{\partial t} \frac{\partial}{\partial z_j^c}]\wedge \frac{\partial}{\partial z_j^s}-g_{rs}^j[\frac{\partial}{\partial z_j^s},\frac{\partial f_{jk}^c}{\partial t}\frac{\partial}{\partial z_j^c}]\wedge \frac{\partial}{\partial z_j^r}\\
&=\sum_{r,s,c=1}^n g_{rs}^j\frac{\partial}{\partial z_j^r}\left(\frac{\partial f_{jk}^c}{\partial t}\right) \frac{\partial}{\partial z_j^c}\wedge \frac{\partial}{\partial z_j^s}-\frac{\partial f_{jk}^c}{\partial t}\frac{\partial g_{rs}^j}{\partial z_j^c}\frac{\partial}{\partial z_j^r}\wedge \frac{\partial}{\partial z_j^s}+g_{rs}^j\frac{\partial}{\partial z_j^s}\left(\frac{\partial f_{jk}^c}{\partial t}\right)\frac{\partial}{\partial z_j^r}\wedge \frac{\partial}{\partial z_j^c}
\end{align*}
By considering the coefficients of $\frac{\partial}{\partial z_j^a}\wedge \frac{\partial}{\partial z_j^b}$, $(*)$ is equivalent to
\begin{align*}
(**)\sum_{r,s=1}^n \frac{\partial g_{rs}^k}{\partial t}\frac{\partial f_{jk}^a}{\partial z_k^r}\frac{\partial f_{jk}^b}{\partial z_k^s}-\frac{\partial g_{ab}^j}{\partial t}-\sum_{c=1}^n \frac{\partial g_{ab}^j}{\partial z_j^c}\frac{\partial f_{jk}^c}{\partial t}+\sum_{c=1}^n g_{cb}^j\frac{\partial}{\partial z_j^c}\left(\frac{\partial f_{jk}^a}{\partial t}\right)+g_{ac}^j\frac{\partial}{\partial z_j^c}\left(\frac{\partial f_{jk}^b}{\partial t}\right)=0
\end{align*}

On the other hand, since $\sum_{\alpha,\beta=1}^n g_{\alpha \beta}^j\frac{\partial}{\partial z_j^{\alpha}}\wedge \frac{\partial}{\partial z_j^{\beta}}=\sum_{r,s=1}^n g_{rs}^k\frac{\partial}{\partial z_k^{r}}\wedge \frac{\partial}{\partial z_k^{s}}=\sum_{r,s,a,b=1}^n g_{rs}^k \frac{\partial f_{jk}^a}{\partial z_k^r}\frac{\partial f_{jk}^b}{\partial z_k^s} \frac{\partial}{\partial z_j^a}\wedge\frac{\partial }{\partial z_j^b}$ on $\mathcal{U}_j^t \cap \mathcal{U}_k^t\ne \emptyset$, we have
\begin{align*}
g_{ab}^j(f_{jk}^1(z_k,t),...,f_{jk}^n(z_k,t),t_1,...,t_m)=\sum_{r,s=1}^n g_{rs}^k \frac{\partial f_{jk}^a}{\partial z_k^r}\frac{\partial f_{jk}^b}{\partial z_k^s}.\\
\end{align*}
By taking the derivative with respect to $t$, we have
\begin{align*}
\frac{\partial g_{ab}^j}{\partial z_j^1}\frac{\partial f_{jk}^1}{\partial t}+\cdots+\frac{\partial g_{ab}^j}{\partial z_j^n}\frac{\partial f_{jk}^n}{\partial t}+\frac{\partial g_{ab}^j}{\partial t}=\sum_{r,s=1}^n \frac{\partial g_{rs}^k}{\partial t}\frac{\partial f_{jk}^a}{\partial z_k^r}\frac{\partial f_{jk}^b}{\partial z_k^s}+g_{rs}^k(\frac{\partial}{\partial z_k^r}\left(\frac{\partial f_{jk}^a}{\partial t}\right)\frac{\partial f_{jk}^b}{\partial z_k^s}+\frac{\partial f_{jk}^a}{\partial z_k^r}\frac{\partial}{\partial z_k^s}\left(\frac{\partial f_{jk}^b}{\partial t}\right) )
\end{align*}
Hence $(**)$ is equivalent to
\begin{align*}
\sum_{c=1}^n g_{cb}^j\frac{\partial}{\partial z_j^c}\left(\frac{\partial f_{jk}^a}{\partial t}\right)+g_{ac}^j\frac{\partial}{\partial z_j^c}\left(\frac{\partial f_{jk}^b}{\partial t}\right)=\sum_{r,s=1}^n g_{rs}^k(\frac{\partial}{\partial z_k^r}\left(\frac{\partial f_{jk}^a}{\partial t}\right)\frac{\partial f_{jk}^b}{\partial z_k^s}+\frac{\partial f_{jk}^a}{\partial z_k^r}\frac{\partial}{\partial z_k^s}\left(\frac{\partial f_{jk}^b}{\partial t}\right) )
\end{align*}
Indeed,
{\small{\begin{align*}
\sum_{c=1}^n g_{cb}^j\frac{\partial}{\partial z_j^c}\left(\frac{\partial f_{jk}^a}{\partial t}\right)+g_{ac}^j\frac{\partial}{\partial z_j^c}\left(\frac{\partial f_{jk}^b}{\partial t}\right)=\sum_{r,s,c=1}^n g_{rs}^k\frac{\partial f_{jk}^c}{\partial z_k^r}\frac{\partial f_{jk}^b}{\partial z_k^s}\frac{\partial}{\partial z_j^c}\left(\frac{\partial f_{jk}^a}{\partial t}\right)+g_{rs}^k\frac{\partial f_{jk}^a}{\partial z_k^r}\frac{\partial f_{jk}^c}{\partial z_k^s}\frac{\partial}{\partial z_j^c}\left(\frac{\partial f_{jk}^b}{\partial t}\right)\\
\sum_{r,s=1}^n g_{rs}^k(\frac{\partial}{\partial z_k^r}\left(\frac{\partial f_{jk}^a}{\partial t}\right)\frac{\partial f_{jk}^b}{\partial z_k^s}+\frac{\partial f_{jk}^a}{\partial z_k^r}\frac{\partial}{\partial z_k^s}\left(\frac{\partial f_{jk}^b}{\partial t}\right) )=\sum_{r,s,c=1}^n g_{rs}^k\frac{\partial f_{jk}^c}{\partial z_k^r}\frac{\partial}{\partial z_j^c}\left(\frac{\partial f_{jk}^a}{\partial t}\right)\frac{\partial f_{jk}^b}{\partial z_k^s}+g_{rs}^k\frac{\partial f_{jk}^a}{\partial z_k^r}\frac{\partial f_{jk}^c}{\partial z_k^s}\frac{\partial}{\partial z_j^c}\left(\frac{\partial f_{jk}^b}{\partial t}\right)
\end{align*}}}
It remains to show that $(\Lambda(t),\theta(t))$ is independent of the choice of systems of local coordinates. We can show the infinitesimal deformation does not change under the refinement of the open covering (See $\cite{Kod05}$ page 190). Since we can choose a common refinement for two system of local coordinates, it is sufficient to show that given two local coordinates $x_j=(z_j,t)$ and $u_j=(w_j,t)$ on each $\mathcal{U}_j$, the infinitesimal deformation $(\eta(t),\Lambda'(t))$ with respect to $\{u_j\}$ coincides with $(\theta(t),\Lambda(t))$ with respect to $\{x_j\}$. Let 
\begin{align*}
w_j^{\alpha}=g_j^{\alpha}(z_j^1,...,z_j^n,t)
\end{align*}
the coordinate transformation from $(z_j,t)$ to $(w_j,t)$ which is holomorphic in $z_j^1,...,z_j^n$.
So we have $\frac{\partial}{\partial z_j^r}=\sum_{a} \frac{\partial g_j^a}{\partial z_j^r}\frac{\partial}{\partial w_j^a}$. And let 
\begin{align*}
\theta_j(t)=\sum_{\alpha=1}^n \frac{\partial g_j^{\alpha}(z_j,t)}{\partial t} \frac{\partial }{\partial w_j^{\alpha}}, \,\,\,\,\,w_j^{\alpha}=g_j^{\alpha}(z_j,t),
\end{align*}
Then we claim that $(\theta_{jk}(t),\Lambda_j(t))-(\eta_{jk}(t),\Lambda'_j(t))=\theta_k(t)-\theta_j(t)-[\Lambda, \theta_j(t)]=-\delta(-\theta(t))+[\Lambda,-\theta(t)]$. Since $\delta(\theta_j(t))=\{\theta_{jk}(t)\}-\{\eta_{jk}(t)\}$(see page 192). We only need to see $\Lambda_j(t)-\Lambda'_j(t)+[\Lambda,\theta_j(t)]=0$. Equivalently,
{\small{\begin{align*}
\sum_{r,s} \frac{\partial \Lambda^{rs}_{j}(z_j,t)}{\partial t}\frac{\partial }{\partial z_j^r}\wedge \frac{\partial}{\partial z_j^s}-\sum_{\alpha,\beta} \frac{\partial \Lambda^{'\alpha\beta}_j(w_j,t)}{\partial t}\frac{\partial}{\partial w_j^{\alpha}}\wedge \frac{\partial}{\partial w_j^{\beta}}+[\sum_{\alpha,\beta} \Lambda^{'\alpha\beta}_j(w_j,t)\frac{\partial}{\partial w_j^{\alpha}}\wedge \frac{\partial}{\partial w_j^{\beta}},\sum_c \frac{\partial g_j^{c}(z_j,t)}{\partial t} \frac{\partial }{\partial w_j^{c}}]=0
\end{align*}}}
But the computation is essentially same to the above.
\end{proof}

\begin{definition}[(holomorphic) Poisson Kodaira-Spencer map]
Let $(\mathcal{M},\Lambda,B,\pi)$ be a family of compact holomorphic Poisson manifolds, where $B$ is a domain of $\mathbb{C}^m$, and $(z,t)$ its system of local coordinates. Then each $(z_j,t)$ on $\mathcal{U}_j$ is a local complex coordinate system of the complex manifold $\mathcal{M}$. And in the local complex coordinate system $\Lambda$ can be expressed as $\sum_{\alpha,\beta} g_{\alpha \beta}^{j}(z,t)\frac{\partial}{\partial z_{j}^{\alpha}}\wedge \frac{\partial}{\partial z_{j}^{\beta}}$ where $g^{j}_{\alpha \beta}(z,t)$ is a holomorphic function on $\mathcal{U}_{j}$. For a tangent vector $\frac{\partial}{\partial t}=\sum_{\lambda=1}^{m} c_{\lambda}\frac{\partial}{\partial t_{\lambda}},c_{\lambda} \in \mathbb{C}$, of $B$, we put 
\begin{align*}
\frac{\partial \Lambda_t}{\partial t}=\sum_{\alpha,\beta}\left[\sum_{\lambda=1}^{m}c_{\lambda}\frac{\partial g_{\alpha \beta}^{j}(z,t)}{\partial t_{\lambda}}\right] \frac{\partial}{\partial z_{j}^{\alpha}}\wedge \frac{\partial}{\partial z_{j}^{\beta}}
\end{align*}
The $($holomorphic$)$ Poisson Kodaira-Spencer map is a $\mathbb{C}$-linear map of 
\begin{align*}
\varphi_t:T_t(B) &\to HP^2(M_t,\Lambda_t)\\
\frac{\partial}{\partial t} &\mapsto \left[\rho_t\left(\frac{\partial}{\partial t}\right)\left(=\frac{\partial{M}_t}{\partial t}\right), \frac{\partial{\Lambda_t}}{\partial t}\right]=\frac{\partial (M_t,\Lambda_t)}{\partial t}
\end{align*}
where $\rho_t:T_t(B)\to H^1(M_t,\Theta_t)$ is the Kodaira-Spencer map. $($See $\cite{Kod05}$ p.201$)$
\end{definition}

\subsection{Tirivial, locally trivial family and rigidity}\

\begin{definition}
Two Poisson analytic families $(\mathcal{M},\Lambda, B,\pi)$ and $(\mathcal{N},\Lambda', B,\pi')$ are equivalent if there is a biholomorphic Poisson map $\Phi$ of $(\mathcal{M},\Lambda)$ onto $(\mathcal{N},\Lambda')$ such that $\pi=\pi'\circ \Phi$. 
Then $(M_t,\Lambda_t)$ and $(N_t,\Lambda_t)$ are biholomorphic Poisson map. 
\end{definition}

\begin{definition}
A Poisson analytic family $(\mathcal{M},\Lambda,,B,\pi)$ is called trivial if it is equivalent to $(M\times B,\Lambda_{t_0}\oplus 0,B,\pi')$\footnote{For the definition of product of holomorphic Poisson manifolds, see Appendix \ref{appendixa}. Here we consider $B$ as a holomorphic Poisson manifold with trivial Poisson structure $0$} with $M=\pi^{-1}(t^0)$ where $t^0$ is some point of $B$. Similarly we define the local triviality of $(\mathcal{M},\Lambda,B,\pi):$ for each $t\in B$, there exists a neighborhood $\Delta$ of $t$ such that $(\mathcal{M}_{\Delta},\Lambda_{\Delta}, \Delta,\pi)$ is trivial.\footnote{Let $(\mathcal{M},\Lambda,B,\pi)$ be a Poisson analytic family. Let $\Delta$ be an open set of $B$. Then ($\mathcal{M}_{\Delta}=\pi^{-1}(\Delta),\Lambda|_{M_{\Delta}},\Delta,\pi|_{\mathcal{M}_{\Delta}})$ is an Poisson analytic family. We denote the family by $(\mathcal{M}_{\Delta},\Lambda_{\Delta},\Delta,\pi)$}
\end{definition}

The following problem is an analogue of a question from deformations of complex structures.
\begin{problem}
If $dim\,HP^2(M_t,\Lambda_t)$ is independent of $t\in B$, and $\varphi_t=0$ identically, then is the holomorphic Poisson family $(\mathcal{M},\Lambda,B,\pi)$ locally trivial $?$\footnote{For the question of deformations of complex structures, we can find the proof in $\cite{Kod05}$. But I could not access to this problem for unfamiliarity of analysis}
\end{problem}

\begin{definition}
We say that a compact holomorphic Poisson manifold $(M,\Lambda_0)$ is rigid if, for any Poisson analytic family $(\mathcal{M},\Lambda,B,\pi)$ such that $M_{t_0}=M$, we can find a neighborhood $\Delta$ of $t_0$ such that $M_t=M_{t_0}$ for $t\in \Delta$. More precisely, $(\mathcal{M}_{\Delta},\Lambda_{\Delta},\Delta,\pi)$ is Poisson biholomorphic to $(M_{t_0}\times \Delta, \Lambda_{t_0}\oplus 0,\Delta, pr)$ where we consider $\Delta$ as a trivial holomorphic Poisson manifold $(\Delta,0)$ and $pr$ is the second projection.
\end{definition}

The following problem is an analogue of a question from deformations of complex structures.

\begin{problem}
If $HP^2(M,\Lambda_0)=0$, is $(M,\Lambda_0)$ is rigid $?$\footnote{I verified that we can use Kodaira's methods presented in \cite{Mor71}. Actually the proof is the special case of theorem of completeness(See $\cite{Kod05}$). But I could not prove the inductive step in Kodaira's methods. However, in the part \ref{part3} of the thesis, we prove that for a nonsingular Poisson variety $(X,\Lambda_0)$ over an algebraically closed field $k$, if $HP^2(X,\Lambda_0)=0$, then $(X,\Lambda_0)$ is rigid in algebraic Poisson deformations (see Proposition \ref{3rigid}). I could not find any example with $HP^2(M,\Lambda)=0$. Even for complex projective plane $\mathbb{P}_{\mathbb{C}}^2$ with any holomorphic Poisson structure $\Lambda$, $HP^2(\mathbb{P}_{\mathbb{C}}^2,\Lambda)\ne 0$. See \cite{Pin11}.}
\end{problem}

\subsection{Change of parameter}(compare $\cite{Kod05}$ p.205)
Suppose given a Poisson analytic family $\{(M_t,\Lambda_t)|(M_t,\Lambda_t)=\omega^{-1}(t),t\in B\}=(\mathcal{M},\Lambda,B,\pi)$ of compact holomorphic Poisson manifolds, where $B$ is a domain of $\mathbb{C}^m$. Let $D$ be a domain of $\mathbb{C}^r$ and $h:s\to t=h(s),s\in D$, a holomorphic map of $D$ into $B$. Then by changing the parameter from $t$ to $s$, we construct a Poisson analytic family $\{(M_{h(t)},\Lambda_{h(t)})|s\in D\}$ on the parameter space $D$ in the following way.

Let $\mathcal{M}\times_B D:=\{(p,s)\in \mathcal{M}\times B|\omega(p)=h(s)\}$. Then we have the following commutative diagram
\begin{center}
$\begin{CD}
\mathcal{M}\times_B D @>p>> \mathcal{M}\\
@V\pi VV @VV\omega V\\
D @>h>> B
\end{CD}$
\end{center}
Since $\omega$ is a submersion, $\mathcal{M}\times_B D$ is a complex submanifold of $\mathcal{M}\times D$ and $\pi$ is a submersion. So $(\mathcal{M}\times_B D,D,\pi)$ is a complex analytic family in the sense of Kodaira and Spencer and we have $\pi^{-1}(s)=M_{h(s)}$. We show that it is naturally a Poisson analytic family such that $\pi^{-1}(s)=(M_{h(s)},\Lambda_{h(s)})$. Note that $D$ can be considered as a holomorphic Poisson manifold with trivial Poisson structure. In other words, $(D,0)$ is a holomorphic Poisson manifold. Then $(\mathcal{M}\times D,\Lambda\oplus 0)$ is a holomorphic Poisson manifold.\footnote{For the definition of product of two holomorphic Poisson manifolds, see Appendix \ref{appendixa}} We show that $\mathcal{M}\times_B D$ is a holomorphic Poisson submanifold of $(\mathcal{M}\times D,\Lambda\oplus 0)$. We check locally by applying Proposition \ref{box} (3). Let $(p_0,s_0)\in \mathcal{M}\times_B D$. Taking a sufficiently small coordinate polydisk $\Delta$ with $h(s_0)\in \Delta$, we represent $(\mathcal{M}_{\Delta},\Lambda_{\Delta})=\omega^{-1}(\Delta)$ in the form of
\begin{align*}
(\mathcal{M}_{\Delta},\Lambda_{\Delta})=(\bigcup_{j=1}^l U_j\times \Delta, \sum_{\alpha,\beta} g_{\alpha\beta}^j(z_j,t)\frac{\partial}{\partial z_j^{\alpha}}\wedge\frac{\partial}{\partial z_j^{\beta}})
\end{align*}
where each $U_j$ is a polydisk independent of $t$, and $(z_j,t)\in U_j\times \Delta$ and $(z_k,t)\in U_k\times \Delta$ are the same point on $\mathcal{M}_{\Delta}$ if $z_j^{\alpha}=f_{jk}^{\alpha}(z_k,t), \alpha=1,...,n$. Let $E$ be a sufficiently small polydisk of $D$ with $s_0\in E$ and $h(E)\subset \Delta$. Then we can represent $\mathcal{M}\times D$ locally in the form of 
\begin{align*}
(\mathcal{M}_{\Delta}\times E,\Lambda_{\Delta}\oplus 0)=(\bigcup_{j=1}^l U_j\times \Delta\times E, \sum_{\alpha,\beta} g_{\alpha\beta}^j(z_j,t)\frac{\partial}{\partial z_j^{\alpha}}\wedge\frac{\partial}{\partial z_j^{\beta}})
\end{align*}
where $(z_j,t,s)\in U_j\times \Delta\times E$ and $(z_k,t,s)\in U_k\times \Delta\times E$ are the same point on $\mathcal{M}_{\Delta}\times E$ if $z_j=f_{jk}(z_k,t)$.
And we can represent $\mathcal{M}\times_B D$ locally in the form of
\begin{align*}
\bigcup_{j=1}^l U_j\times G_E \subset \mathcal{M}\times \Delta
\end{align*}
where $G_E=\{(h(s),s)|s\in E\}\subset \Delta\times E$ and $(z_j,h(s),s)\in U_j\times G_E$ and $(z_k,h(s),s)\in U_k\times G_E$ are the same points if $z_j=f_{jk}(z_k,h(s))$. We note that at $(p_0,s_0)\in \mathcal{M}\times_B D$, we have $(\Lambda\oplus 0)_{(p_0,s_0)}=\sum_{\alpha,\beta} g_{\alpha\beta}^j(p_0,h(s_0))\frac{\partial}{\partial z_j^{\alpha}}|_{p_0}\wedge\frac{\partial}{\partial z_j^{\beta}}|_{p_0}\in \wedge^2 T_{\mathcal{M}\times_B D}$. Hence $\mathcal{M}\times_B D$ is a holomorphic Poisson submanifold of $(\mathcal{M}\times D,\Lambda \oplus 0)$. Since $i:\mathcal{M}\times_B D\hookrightarrow \mathcal{M}\times D$ is a Poisson map and $\mathcal{M}\times D \to\mathcal{M}$ is a Poisson map, $p:\mathcal{M}\times_B D\to \mathcal{M}$ is a Poisson map.

Since $G_E$ is biholomorphic to $E$. The holomorphic Poisson manifold $\mathcal{M}\times_B D$ is represented locally by the form
\begin{align*}
(\bigcup_{j=1}^l U_j\times E, \sum_{\alpha,\beta} g_{\alpha\beta}^j(z_j,h(s))\frac{\partial}{\partial z_j^{\alpha}}\wedge\frac{\partial}{\partial z_j^{\beta}})
\end{align*}
where $(z_k,s)\in U_k\times E$ and $(z_j,s)\in U_j\times E$ are the same points if $z_j=f_{jk}(z_k,h(s))$.

\begin{definition}
The Poisson analytic family $(\mathcal{M}\times_B D,D, (\Lambda\oplus 0)|_{\mathcal{M}\times_B D},\pi)$ is called the Poisson analytic family induced from $(\mathcal{M},B,\Lambda,\omega)$ by the holomorphic map $h:D\to B$.
\end{definition}

Now we consider the change of variable formula in the infinitesimal deformations.

\begin{thm}
For any tangent vector $\frac{\partial}{\partial s}=c_1\frac{\partial}{\partial s_1}+\cdots +c_r\frac{\partial}{\partial s_r}\in T_s(D)$, the infinetesimal holomorphic poisson deformation of $(M_{h(s)},\Lambda_{h(s)})$ along $\frac{\partial}{\partial s}$ is given by
\begin{align*}
\frac{\partial(M_{h(s)},\Lambda_{h(s)})}{\partial{s}}=(\sum_{\lambda=1}^{m} \frac{\partial t_{\lambda}}{\partial s} \frac{\partial M_t}{\partial t_{\lambda}},\sum_{\lambda=1}^{m} \frac{\partial t_{\lambda}}{\partial s}\frac{\partial{\Lambda_t}}{\partial t_{\lambda}})
\end{align*}
\end{thm}

\begin{proof}
We put
\begin{align*}
\theta_{\lambda jk}(t)&=\sum_{\alpha=1}^{n}\frac{\partial{f_{jk}^{\alpha}(z_k,t_1,...,t_m)}}{\partial{t_{\lambda}}}\frac{\partial}{\partial{z_j^{\alpha}}},\\
\eta_{jk}(s)&=\sum_{\alpha=1}^{n}\frac{\partial{f_{jk}^{\alpha}(z_k, h(s))}}{\partial{s}}\frac{\partial}{\partial{z_j^{\alpha}}},\\
\Lambda_{\lambda j}(t)&=\sum_{\alpha,\beta =1}^n\frac{\partial{g^{j}_{\alpha \beta}(z_j,t_1,...,t_m)}}{\partial{t_{\lambda}}}\frac{\partial}{\partial{z_j^{\alpha}}}\wedge \frac{\partial}{\partial{z_j^{\beta}}},\\
\Lambda_{j}(s)&=\sum_{\alpha, \beta=1}^n\frac{\partial{g^{j}_{\alpha \beta}(z_j,h(s))}}{\partial{s}}\frac{\partial}{\partial{z_j^{\alpha}}}\wedge \frac{\partial}{\partial{z_j^{\beta}}}
\end{align*}
$\frac{\partial{(M_t,\Lambda_t)}}{\partial{t_{\lambda}}}$ is the cohomology class of the 1-cocycle $(\{\theta_{\lambda jk}(t)\},\{\Lambda_{j}(s)\})$, and $\frac{\partial{(M_{h(s)}, \Lambda_{h(s)})}}{\partial{s}}$ is that of $(\{\eta_{jk}(s)\},\{\Lambda_{j}(s)\})$. Since $h(s)=(t_1,...,t_m)$, we have
\begin{align*}
\frac{\partial{f_{jk}^{\alpha}(z_k, h(s))}}{\partial{s}}&=\sum_{\lambda=1}^{m}\frac{\partial{t_{\lambda}}}{\partial{s}}\frac{\partial{f_{jk}^{\alpha}(z_k,t_1,...,t_m)}}{\partial{t_{\lambda}}},\\
\frac{\partial{g^{j}_{\alpha \beta}(z_j,h(s))}}{\partial{s}}&=\sum_{l=1}^{r}c_l\frac{\partial{{g^{j}_{\alpha \beta}(z_j,h(s))}}}{\partial{s_l}}=\sum_{l=1}^r \sum_{\lambda=1}^m c_l\frac{\partial{t_{\lambda}}}{\partial{s_l}}\frac{\partial{{g^{j}_{\alpha \beta}(z_j,t_1,...,t_m)}}}{\partial{t_\lambda}}=\sum_{\lambda=1}^{m}\frac{\partial{t_{\lambda}}}{\partial{s}}\frac{\partial{g^j_{\alpha \beta}(z_j,t_1,...,t_m)}}{\partial{t_{\lambda}}}
\end{align*}
Hence we get the theorem.
\end{proof}

At this point, we discuss a concept of completeness in deformations of holomorphic Poisson manifolds. We define a complete family.

\begin{definition}
Let $(\mathcal{M},\Lambda,B, \omega)$ be a Poisson analytic family of compact holomorphic Poisson manifolds, and $t^0\in B$. Then $(\mathcal{M},\Lambda,B,\omega)$ is called complete at $t^0\in B$ if for any Poisson analytic family $(\mathcal{N},\Lambda',D,\pi)$ such that $D$ is a domain of $\mathbb{C}^l$ containing $0$ and that $\pi^{-1}(0)=\omega^{-1}(t^0)$, there are a sufficiently small domain $\Delta$ with $0\in \Delta\subset D$, and a holomorphic map $h:s\to t=h(s)$ with $h(0)=t^0$ such that $(\mathcal{N}_{\Delta},{\Lambda'}_{\Delta},\Delta,\pi)$ is the Poisson analytic family induced from $(\mathcal{M},\Lambda,B,\omega)$ by $h$ where $(\mathcal{N}_{\Delta},{\Lambda'}_{\Delta})=\pi^{-1}(\Delta)$.
\end{definition}

The following problem is an analogue of theorem of completeness from deformations of complex structures.

\begin{problem}[Theorem of Completeness for deformations of holomorphic Poisson manifolds]
If $\varphi_0:T_0 B\to HP^2(M,\Lambda_0)$ is surjective, is the Poisson analytic family $(\mathcal{M},\Lambda, B,\omega)$ complete at $0\in B?$\footnote{I verified that for this problem, we can use Kodaira's methods presented in $\cite{Kod05}$. But I could not prove the inductive step in the Poisson direction.}
\end{problem}

\chapter{Integrability condition}\label{chapter2}


In a family $(\mathcal{M},B,\Lambda)$ of deformations of  a complex manifold $M$, the deformations near $M$ are represented by a $C^{\infty}$ vector $(1,0)$-form $\varphi(t) \in A^{0,1}(M,T_M)$ on $M$ with $\varphi(0)=0$, $\bar{\partial} \varphi(t)-\frac{1}{2}[\varphi(t),\varphi(t)]=0$ where $t \in \Delta$ a sufficiently small polydisk in $B$. In this chapter, we show that in a family $(\mathcal{M},B,\Lambda,\pi)$ of deformations of a holomorphic Poisson manifold $(M,\Lambda_0)$, the deformations near $(M,\Lambda_0)$ are represented by $C^{\infty}$ vector $(1,0)$-form $\varphi(t)\in A^{0,1}(M,T_M)$ and $C^{\infty}$ bivector $\Lambda(t)\in A^{0,0}(M,\wedge^2 T_M)$ with $\varphi(0)=0$, $\Lambda(0)=\Lambda_0$ and $\bar{\partial}(\varphi(t)+\Lambda(t))+\frac{1}{2}[\varphi(t)+\Lambda(t),\varphi(t)+\Lambda(t)]=0$. To deduce the integrability condition, we follow Kodaira's approach ($\cite{Kod05}$ section \S 5.3 (b) page 259) in the context of holomorphic Poisson deformations.

\section{Preliminaries}

Let $(\mathcal{M}, \Lambda, B,\omega)$ be a Poisson analytic family of compact Poisson holomorphic manifolds, and put $(M_t,\Lambda_t)=\omega^{-1}(t)$ where $B$ is a domain of $\mathbb{C}^m$ containing the origin $0$. Define $|t|=max_{\lambda}|t_{\lambda}|$ for $t=(t_1,...,t_m)\in \mathbb{C}^m$, and let $\Delta=\Delta_r =\{t\in \mathbb{C}^m||t|<r\}$ the polydisk of radius $r>0$. If we take a sufficiently small $\Delta \subset B,\mathcal{M}_{\Delta}=\omega^{-1}(\Delta)$ is represented in the form
\begin{align*}
\mathcal{M}_{\Delta}=\bigcup_j U_j\times \Delta
\end{align*}
We denote a point of $U_j$ by $\xi_j=(\xi_j^1,...,\xi_j^n)$ and its holomorphic Poisson structure $\Lambda_j=g_{\alpha \beta}^j(\xi_j,t) \frac{\partial}{\partial \xi_j^{\alpha}}\wedge \frac{\partial}{\partial \xi_j^{\beta}}$ on $U_j\times \Delta$. For simplicity we assume that $U_j=\{\xi_j\in \mathbb{C}^m||\xi_j|<1\}$ where $|\xi|=mat_a|\xi_j^a|$. $(\xi_j,t)\in U_j\times \Delta$ and $(\xi_k,t)\in U_k\times \Delta$ are the same point on $\mathcal{M}_{\Delta}$ if $\xi_j^{\alpha}=f_{jk}^{\alpha}(\xi_k,t)$, $\alpha=1,...,n$ where $f_{jk}^{\alpha}(\xi_k,t)$ is a poisson holomorphic map of $\xi_{k}^1,...,\xi_k^n,t_1,...,t_m$, defined on $U_k\times \Delta \cap U_j\times \Delta$ and we have $g_{\alpha \beta}^j(f_{jk}^1(\xi_k,t),...,f_{jk}^n(\xi_k,t))=\sum_{r,s} g_{rs}^k(\xi_k,t)\frac{\partial f_{jk}^{\alpha}}{\partial \xi_k^r}\frac{\partial f_{jk}^{\beta}}{\partial \xi_k^s}$. So $(M_t,\Lambda_t)=\cup_j (U_j,\Lambda_j(t))$ is a compact holomorphic Poisson manifold obtained by glueing a finite number of Poisson polydiscks $(U_1,\Lambda_1(t)),...,(U_j,\Lambda_j(t)),...$ by identifying $\xi_j\in U_j$ and $\xi_k\in U_k$ if $\xi_j=f_{jk}(\xi_k,t)$, and that holomorphic Poisson structure of $(M_t,\Lambda_t)$ varies since the manner of glueing and holomorphic Poisson structure vary with $t$. We note that by $\cite{Kod05}$ Theorem 2.3, when we ignore complex structures and Poisson structures $M_t$ for any $t\in \Delta$ is a diffeomorphic to $M_0$ as differentiable manifolds.

By $\cite{Kod05}$ Theorem 2.5, if we take a sufficiently small $\Delta$, there is a diffeomorphism $\Psi$ of $M\times \Delta$ onto $\mathcal{M}_{\Delta}$ as differentiable manifolds such that $\omega\circ \Psi$ is the projection $M\times \Delta \to \Delta$, where we put $M=M_0$. If we denote a point of $M$ by $z$, we have
\begin{align*}
\omega\circ \Psi(z,t)=t,\,\,\,\,\, t\in \Delta.
\end{align*}
 
$\Psi$ is the identity of $M=M\times 0$ onto $M=M_0\subset \mathcal{M}_{\triangle}$, namely $\Psi(z,0)=z$. Put $\Psi(z,t)=(\xi,t)=(\xi_j,t)$ for $\Psi(z,t)\in U_j\times \triangle$. Then each component $\xi_j^{\alpha}=\xi_j^{\alpha}(z,t)$, $\alpha=1,...,n$, of $\xi_j=(\xi_j^1,...,\xi_j^n)$ is a $C^{\infty}$ function: 
\begin{align*}
\Psi(z,t)=(\xi_j^1(z,t),...,\xi_j^n(z,t),t_1,...,t_m).
\end{align*}
If we identify $\mathcal{M}_{\Delta}=\Psi(M\times \Delta)$ with $M_0\times \Delta$ via $\Psi$, $(\mathcal{M}_{\Delta},\Lambda)$ is considered as a holomorphic Poisson structure defined on the $C^{\infty}$ manifold $M\times \Delta$ by the system of local coordinates
\begin{align*}
\{(\xi_j,t)|j=1,2,3,...\},\,\,\,\,\, (\xi_j,t)=(\xi_j^1(z,t),...,\xi_j^n(z,t),t_1,...,t_m).
\end{align*}
and local holomorphic Poisson structures on $U_j\times \Delta$
\begin{align*}
\{\sum_{\alpha,\beta} g_{\alpha \beta}^j(\xi_j(z,t),t)\frac{\partial}{\partial \xi_j^{\alpha}}\wedge \frac{\partial}{\partial \xi_j^{\beta}}|j=1,2,3,...\}
\end{align*}


Let $(z^1,...,z^n)$ be arbitrary local complex coordinates of a point $z$ of $M_0$.
\begin{align*}
\xi_j^{\alpha}(z,t)=\xi_j^{\alpha}(z_1,...,z_n,t_1,...,t_m),\,\,\,\,\, \alpha=1,...,n,
\end{align*}
are $C^{\infty}$ functions of the complex variables $z^1,...,z^n,t_1,...,t_m$. Since for $t=0$, both $(\xi_j^1(z,0),...,\xi_j^n(z,0))$ and $(z_1,...,z_n)$ are local complex coordinates on the complex holomorphic manifold $M_0=M$, $\xi_j^{\alpha}(z,0)$ are holomorphic functions of $z_1,...,z_n$, and
\begin{align*}
det\left(\frac{\partial \xi_j^{\alpha}(z,0)}{\partial z_{\lambda}}\right)_{\alpha,\lambda=1,...,n}\ne 0
\end{align*}
Hence, if we take $\Delta$ sufficiently small, it follows that 
\begin{align*}
det\left(\frac{\partial \xi_j^{\alpha}(z,t)}{\partial z_{\lambda}}\right)_{\alpha,\lambda=1,...,n}\ne 0
\end{align*}
for any $t\in \Delta$.

With this preparation, we identify the holomorphic Poisson deformations near $(M,\Lambda_0)$, where $M=M_0$ in the analytic family $(\mathcal{M},\Lambda,B,\Delta)$ with $\varphi(t)+\Lambda(t)$ where $\varphi(t)$ is a $C^{\infty}$ vector $(0,1)$ form and $\Lambda(t)$ is a $C^{\infty}$ bivector on $M$.

\subsection{Identification of the deformations of complex structures with $\varphi(t)$}\

We consider $\bar{\partial} \xi_j^{\alpha}(z,t)=\sum \bar{\partial}\xi_j^{\alpha}(z,t)d\bar{z}^v$. The domain $\mathcal{U}_j=\Psi^{-1}(U_j\times \triangle)$ of $\xi_j^{\alpha}(z,t)$ is a domain of $M\times \triangle$.
 
 Since $det\left(\frac{\partial \xi_j^{\alpha}(z,t)}{\partial z_{\lambda}}\right)_{\alpha,\lambda=1,...,n}\ne 0$, we define a $(0,1)$-form $\varphi^{\lambda}_j(z,t)=\sum_{v=1}^n \varphi^{\lambda}_{jv}(z,t)d\bar{z_v}$ in the following way:
 
\begin{equation*}
\left(
\begin{matrix}
\varphi_j^1(z,t)\\
\vdots \\
\varphi_j^n(z,t)
\end{matrix}
\right)
:=
\left(
\begin{matrix}
\frac{\partial \xi_j^1}{\partial z_1} & \dots & \frac{\partial \xi_j^1}{\partial z_n}\\
\vdots & \vdots\\
\frac{\partial \xi_j^n}{\partial z_1} & \dots & \frac{\partial \xi_j^n}{\partial z_n}
\end{matrix}
\right)^{-1}
\left(
\begin{matrix}
\bar{\partial} \xi_j^1\\
\vdots \\
\bar{\partial} \xi_j^n
\end{matrix}
\right)
\end{equation*}

Then we have

\begin{equation*}
\left(
\begin{matrix}
\frac{\partial \xi_j^1}{\partial z_1} & \dots & \frac{\partial \xi_j^1}{\partial z_n}\\
\vdots & \vdots\\
\frac{\partial \xi_j^n}{\partial z_1} & \dots & \frac{\partial \xi_j^n}{\partial z_n}
\end{matrix}
\right)
\left(
\begin{matrix}
\varphi_j^1(z,t)\\
\vdots \\
\varphi_j^n(z,t)
\end{matrix}
\right)
=
\left(
\begin{matrix}
\bar{\partial} \xi_j^1\\
\vdots \\
\bar{\partial} \xi_j^n
\end{matrix}
\right)
\end{equation*}
which is equivalent to
\begin{equation*}
\left(
\begin{matrix}
\frac{\partial \xi_j^1}{\partial z_1} & \dots & \frac{\partial \xi_j^1}{\partial z_n}\\
\vdots & \vdots\\
\frac{\partial \xi_j^n}{\partial z_1} & \dots & \frac{\partial \xi_j^n}{\partial z_n}
\end{matrix}
\right)
\left(
\begin{matrix}
\varphi_{j1}^1 & \dots & \varphi_{jn}^1\\
\vdots & \vdots\\
\varphi_{j1}^n & \dots & \varphi_{jn}^n
\end{matrix}
\right)
\left(
\begin{matrix}
d\bar{z_1}\\
\vdots \\
d\bar{z_n}
\end{matrix}
\right)
=
\left(
\begin{matrix}
\frac{\partial \xi_j^1}{\partial \bar{z_1}} & \dots & \frac{\partial \xi_j^1}{\partial \bar{z_n}}\\
\vdots & \vdots\\
\frac{\partial \xi_j^n}{\partial \bar{z_1}} & \dots & \frac{\partial \xi_j^n}{\partial \bar{z_n}}
\end{matrix}
\right)
\left(
\begin{matrix}
d\bar{z_1}\\
\vdots \\
d\bar{z_n}
\end{matrix}
\right)
\end{equation*}
In other words, we have $(0,1)$-forms
\begin{align*}
\varphi^{\lambda}_j(z,t)=\sum_{v=1}^n \varphi^{\lambda}_{jv}(z,t)d\bar{z_v}
\end{align*}
for each $\lambda=1,...,n$, such that
\begin{align*}
\bar{\partial}\xi_j^{\alpha}(z,t)=\sum_{\lambda=1}^{n} \varphi_j^{\lambda}(z,t)\frac{\partial \xi_j^{\alpha}(z,t)}{\partial z_{\lambda}},\,\,\,\,\, \alpha=1,...,n
\end{align*}
The coefficients $\varphi_{jv}^{\alpha}(z,t)$ are $C^{\infty}$ functions on $\mathcal{U}_j$.
\begin{lemma}\label{c}
On $\mathcal{U}_j \cap \mathcal{U}_k$, we have
\begin{align*}
\sum_{\lambda=1}^n \varphi_j^{\lambda}(z,t)\frac{\partial}{\partial z_{\lambda}}=\sum_{\lambda=1}^n \varphi_k^{\lambda}(z,t)\frac{\partial}{\partial z_{\lambda}}
\end{align*}
\end{lemma}
\begin{proof}
See $\cite{Kod05}$ p.262.
\end{proof}

If for $(z,t)\in \mathcal{U}_j$, we define
\begin{align}\label{b}
\varphi(z,t):=\sum_{\lambda=1}^n \varphi_j^{\lambda}(z,t) \frac{\partial}{\partial z_{\lambda}}=\sum_{v,\lambda}  \varphi_v^{\lambda}(z,t) d\bar{z_v} \frac{\partial}{\partial z_{\lambda}}
\end{align}
 By Lemma \ref{c}, $\varphi(t)=\varphi(z,t)$ is a $C^{\infty}$ vector $(0,1)$-form on $M$ for every $t\in\triangle$.

Then since $\bar{\partial} \xi_j^{\alpha}(z,0)=0$, and $det\left(\frac{\partial \xi_j^{\alpha}(z,t)}{\partial z_{\lambda}}\right)_{\alpha,\lambda=1,...,n}\ne 0$, we have $\varphi(0)=0$. $\varphi(t)$ satisfies $\bar{\partial}\varphi(t)-\frac{1}{2}[\varphi(t),\varphi(t)]=0$\footnote{For the proof, see $\cite{Kod05}$ p.263,p.265} and we have the following theorem.
 
\begin{thm}\label{text}
If we take a sufficiently small polydisk $\Delta$, then for $t\in \Delta$, al local $C^{\infty}$ function $f$ on $M$ is holomorphic  with respect to the complex structure $M_t$ if and only if $f$ satisfies the equation
\begin{align*}
(\bar{\partial}-\varphi(t))f=0
\end{align*}
\end{thm}
\begin{proof}
See $\cite{Kod05}$ Theorem 5.3 p.263.
\end{proof}

\subsection{Identification of the deformations of Poisson structures with $\Lambda(t)$}\

For, on each $U_j\times \Delta$, the holomorphic Poisson structure $\sum_{\alpha,\beta}g_{\alpha,\beta}^j(\xi_j,t) \frac{\partial}{\partial \xi_j^{\beta}}\wedge \frac{\partial}{\partial \xi_j^{\beta}}$, there exists the unique bivector field  $\Lambda'=\sum_{r,s} f_{rs}^j(z,t)\frac{\partial}{\partial z_r}\wedge \frac{\partial}{\partial z_s}$ on  $\mathcal{U}_j=\Psi^{-1}(U_j\times \Delta)$ such that $\sum_{r,s} f_{rs}^j(z,t)\frac{\partial \xi_j^{\alpha}}{\partial z_r}\frac{\partial \xi_j^{\beta}}{\partial z_s}=g_{\alpha\beta}^j(\xi_j(z,t),t)$.
Indeed, since $det\left(\frac{\partial \xi_j^{\alpha}(z,t)}{\partial z_{\lambda}}\right)_{\alpha,\lambda=1,...,n}\ne 0$,
we set
\begin{equation*}
\left(
\begin{matrix}
f_{11}^j(z,t)& \dots & f_{1n}^j(z,t)\\
\vdots & \vdots &\vdots\\
f_{n1}^j(z,t)& \dots & f_{nn}^j(z,t)
\end{matrix}
\right)
:=
\left(
\begin{matrix}
\frac{\partial \xi_j^1}{\partial z_1} & \dots & \frac{\partial \xi_j^1}{\partial z_n}\\
\vdots & \vdots &\vdots\\
\frac{\partial \xi_j^n}{\partial z_1} & \dots & \frac{\partial \xi_j^n}{\partial z_n}
\end{matrix}
\right)^{-1}
\left(
\begin{matrix}
g_{11}^j(\xi_j(z,t))& \dots & g_{1n}^j(\xi_j(z,t))\\
\vdots & \vdots &\vdots\\
g_{n1}^j(\xi_j(z,t)) & \dots & g_{nn}^j(\xi_j(z,t))
\end{matrix}
\right)
\left(
\begin{matrix}
\frac{\partial \xi_j^1}{\partial z_1} & \dots & \frac{\partial \xi_j^n}{\partial z_1}\\
\vdots & \vdots &\vdots\\
\frac{\partial \xi_j^1}{\partial z_n} & \dots & \frac{\partial \xi_j^n}{\partial z_n}
\end{matrix}
\right)^{-1}
\end{equation*}
Then we have the unique $C^{\infty}$ bivector field $\Lambda_j':=\sum_{r,s} f_{rs}^j(z,t)\frac{\partial}{\partial z_r}\wedge \frac{\partial}{\partial z_s}$ on $\mathcal{U}_j$

\begin{lemma}\label{e}
On $\mathcal{U}_j\cap \mathcal{U}_k$, we have $f_{rs}^j(z,t)=f_{rs}^k(z,t)$. 
 
\end{lemma}

\begin{proof}
We first note the following identities.
\begin{equation*}
\left(
\begin{matrix}
\frac{\partial \xi_j^1}{\partial z_1} & \dots & \frac{\partial \xi_j^1}{\partial z_n}\\
\vdots & \vdots\\
\frac{\partial \xi_j^n}{\partial z_1} & \dots & \frac{\partial \xi_j^n}{\partial z_n}
\end{matrix}
\right)
\left(
\begin{matrix}
f_{11}^j(z,t)& \dots & f_{1n}^j(z,t)\\
\vdots & \vdots & \vdots\\
f_{n1}^j(z,t) & \dots & f_{nn}^j(z,t)
\end{matrix}
\right)
\left(
\begin{matrix}
\frac{\partial \xi_j^1}{\partial z_1} & \dots & \frac{\partial \xi_j^n}{\partial z_1}\\
\vdots & \vdots &\vdots\\
\frac{\partial \xi_j^1}{\partial z_n} & \dots & \frac{\partial \xi_j^n}{\partial z_n}
\end{matrix}
\right)
=\left(
\begin{matrix}
g_{11}^j(\xi_j(z,t))& \dots & g_{1n}^j(\xi_j(z,t))\\
\vdots & \vdots &\vdots\\
g_{n1}^j(\xi_j(z,t)) & \dots & g_{nn}^j(\xi_j(z,t))
\end{matrix}
\right)
\end{equation*}
\begin{equation*}
\left(
\begin{matrix}
\frac{\partial \xi_k^1}{\partial z_1} & \dots & \frac{\partial \xi_k^1}{\partial z_n}\\
\vdots & \vdots &\vdots\\
\frac{\partial \xi_k^n}{\partial z_1} & \dots & \frac{\partial \xi_k^n}{\partial z_n}
\end{matrix}
\right)
\left(
\begin{matrix}
f_{11}^k(z,t)& \dots & f_{1n}^k(z,t)\\
\vdots & \vdots &\vdots\\
f_{n1}^k(z,t) & \dots & f_{nn}^k(z,t)
\end{matrix}
\right)
\left(
\begin{matrix}
\frac{\partial \xi_k^1}{\partial z_1} & \dots & \frac{\partial \xi_k^n}{\partial z_1}\\
\vdots & \vdots &\vdots \\
\frac{\partial \xi_k^1}{\partial z_n} & \dots & \frac{\partial \xi_k^n}{\partial z_n}
\end{matrix}
\right)
=\left(
\begin{matrix}
g_{11}^k(\xi_k(z,t))& \dots & g_{1n}^k(\xi_k(z,t))\\
\vdots & \vdots &\vdots\\
g_{n1}^j(\xi_k(z,t)) & \dots & g_{nn}^k(\xi_k(z,t))
\end{matrix}
\right)
\end{equation*}
\begin{equation*}
\left(
\begin{matrix}
\frac{\partial \xi_j^1}{\partial \xi_k^1} & \dots & \frac{\partial \xi_j^1}{\partial \xi_k^n}\\
\vdots & \vdots &\vdots\\
\frac{\partial \xi_j^n}{\partial \xi_k^1} & \dots & \frac{\partial \xi_j^n}{\partial \xi_k^n}
\end{matrix}
\right)
\left(
\begin{matrix}
g_{11}^k(\xi_k(z,t))& \dots & g_{1n}^k(\xi_k(z,t))\\
\vdots & \vdots &\vdots\\
g_{n1}^k(\xi_k(z,t)) & \dots & g_{nn}^k(\xi_k(z,t))
\end{matrix}
\right)
\left(
\begin{matrix}
\frac{\partial \xi_j^1}{\partial \xi_k^1} & \dots & \frac{\partial \xi_j^n}{\partial \xi_k^1}\\
\vdots & \vdots &\vdots \\
\frac{\partial \xi_j^1}{\partial \xi_k^n} & \dots & \frac{\partial \xi_j^n}{\partial \xi_k^n}
\end{matrix}
\right)
=\left(
\begin{matrix}
g_{11}^j(\xi_j(z,t))& \dots & g_{1n}^j(\xi_j(z,t))\\
\vdots & \vdots &\vdots\\
g_{n1}^j(\xi_j(z,t)) & \dots & g_{nn}^j(\xi_j(z,t))
\end{matrix}
\right)
\end{equation*}

Since $\frac{\partial \xi_j^q}{\partial z_p}=\sum_{r=1}^n \frac{\partial \xi_j^q}{\partial \xi_k^r}\frac{\partial \xi_k^r}{\partial z_p}$, we have 
\begin{equation*}
\left(
\begin{matrix}
\frac{\partial \xi_j^1}{\partial \xi_k^1} & \dots & \frac{\partial \xi_j^1}{\partial \xi_k^n}\\
\vdots & \vdots &\vdots\\
\frac{\partial \xi_j^n}{\partial \xi_k^1} & \dots & \frac{\partial \xi_j^n}{\partial \xi_k^n}
\end{matrix}
\right)
\left(
\begin{matrix}
\frac{\partial \xi_k^1}{\partial z_1} & \dots & \frac{\partial \xi_k^1}{\partial z_n}\\
\vdots & \vdots &\vdots\\
\frac{\partial \xi_k^n}{\partial z_1} & \dots & \frac{\partial \xi_k^n}{\partial z_n}
\end{matrix}
\right)
=\left(
\begin{matrix}
\frac{\partial \xi_j^1}{\partial z_1} & \dots & \frac{\partial \xi_j^1}{\partial z_n}\\
\vdots & \vdots &\vdots\\
\frac{\partial \xi_j^n}{\partial z_1} & \dots & \frac{\partial \xi_j^n}{\partial z_n}
\end{matrix}
\right)
\end{equation*}
\begin{equation*}
\left(
\begin{matrix}
\frac{\partial \xi_k^1}{\partial z_1} & \dots & \frac{\partial \xi_k^n}{\partial z_1}\\
\vdots & \vdots &\vdots\\
\frac{\partial \xi_k^1}{\partial z_n} & \dots & \frac{\partial \xi_k^n}{\partial z_n}
\end{matrix}
\right)
\left(
\begin{matrix}
\frac{\partial \xi_j^1}{\partial \xi_k^1} & \dots & \frac{\partial \xi_j^n}{\partial \xi_k^1}\\
\vdots & \vdots &\vdots \\
\frac{\partial \xi_j^1}{\partial \xi_k^n} & \dots & \frac{\partial \xi_j^n}{\partial \xi_k^n}
\end{matrix}
\right)
=\left(
\begin{matrix}
\frac{\partial \xi_j^1}{\partial z_1} & \dots & \frac{\partial \xi_j^n}{\partial z_1}\\
\vdots & \vdots &\vdots\\
\frac{\partial \xi_j^1}{\partial z_n} & \dots & \frac{\partial \xi_j^n}{\partial z_n}
\end{matrix}
\right)
\end{equation*}
Since $det\left(\frac{\partial \xi_j^{\alpha}(z,t)}{\partial z_{\lambda}}\right)_{\alpha,\lambda=1,...,n}\ne 0$, we have $f_{rs}^j(z,t)=f_{rs}^k(z,t)$.
\end{proof}

If for $(z,t)\in \mathcal{U}_j$, we define 
\begin{align}\label{f}
\Lambda(z,t):=\sum_{r,s} f_{rs}^j(z,t)\frac{\partial}{\partial z_r}\wedge \frac{\partial}{\partial z_s}.
\end{align}
By Lemma \ref{e}, $\Lambda(t)=\Lambda(z,t)$ is a $C^{\infty}$ bivector field on $M$ for every $t\in \Delta$.

\begin{thm}\label{1thm}
If we take a sufficiently small polydiesk $\Delta$, then for the Poisson structure $\sum_{\alpha,\beta}g_{\alpha,\beta}^j(\xi_j,t) \frac{\partial}{\partial \xi_j^{\beta}}\wedge \frac{\partial}{\partial \xi_j^{\beta}}$ on $U_j\times \Delta$ for each $j$, there exists the unique bivector field  $\Lambda_j'=\sum_{r,s} f_{rs}^j(z,t)\frac{\partial}{\partial z_r}\wedge \frac{\partial}{\partial z_s}$ on $\mathcal{U}_j$ satisfying
\begin{enumerate}
\item $\sum_{r,s} f_{rs}^j(z,t)\frac{\partial \xi_j^{\alpha}}{\partial z_r}\frac{\partial \xi_j^{\beta}}{\partial z_s}=g_{\alpha\beta}^j(\xi_j(z,t),t)$
\item $\Lambda_j'$ are glued together to define a $C^{\infty}$ bivector field $\Lambda'$ on $M\times \Delta$ 
\item for each $j$, $[\Lambda_j',\Lambda_j']=0$. Hence we have $[\Lambda',\Lambda']=0$
\end{enumerate}
\end{thm}

We need the following lemma to prove the theorem.

\begin{lemma}\label{formula}
If $\rho=\sum_{p,q} \sigma_{pq}\frac{\partial }{\partial z_p}\wedge \frac{\partial}{\partial z_q}$, then 
$[\sigma,\sigma]=0$ is equivalent to 
\begin{align*}
\sum_{l=1}^n \sigma_{lk}\frac{\partial \sigma_{ij}}{\partial z_l}+\sigma_{li}\frac{\partial \sigma_{jk}}{\partial z_l}+\sigma_{lj}\frac{\partial \sigma_{ki}}{\partial z_l}=0
\end{align*}
for each $1\leq i,j,k \leq n$.
\end{lemma}

\begin{proof}[Proof of Theorem \ref{1thm}]
We have already showed $(1)$ and $(2)$. It remains to show $(3)$. We note that $[\sum_{\alpha,\beta} g_{\alpha\beta}^j(\xi_j,t) \frac{\partial}{\partial \xi_j^{\alpha}}\wedge\frac{\partial}{\partial \xi_j^{\beta}},\sum_{\alpha,\beta} g_{\alpha\beta}^j(\xi_j,t) \frac{\partial}{\partial \xi_j^{\alpha}}\wedge\frac{\partial}{\partial \xi_j^{\beta}}]=0$ and $g_{\alpha\beta}^j(\xi_j(z,t),t)=\sum_{a,b} f_{ab}^j(z,t)\frac{\partial \xi_j^{\alpha}}{\partial z_a}\frac{\partial \xi_j^{\beta}}{\partial z_b}$ is holomorphic with respect to $\xi_j=(\xi_j^{\alpha}),\alpha=1,...,n$. In the following, for simplicity, we denote $\xi_j^{\alpha}(z_j,t)$ by $\xi_{\alpha}$ and $f_{ab}^j(z,t)$ by $f_{ab}$. By Lemma \ref{formula}, we have

\begin{align*}
0=&\sum_{a,b,c,d,l} f_{ab}\frac{\partial \xi_l}{\partial z_a}\frac{\partial \xi_k}{\partial z_b}\frac{\partial}{\partial \xi_l}\left(f_{cd}\frac{\partial \xi_i}{\partial z_c}\frac{\partial \xi_j}{\partial z_d}\right)+f_{ab}\frac{\partial \xi_l}{\partial z_a}\frac{\partial \xi_i}{\partial z_b}\frac{\partial}{\partial \xi_l}\left(f_{cd}\frac{\partial \xi_j}{\partial z_c}\frac{\partial \xi_k}{\partial z_d}\right)+f_{ab}\frac{\partial \xi_l}{\partial z_a}\frac{\partial \xi_j}{\partial z_b}\frac{\partial}{\partial \xi_l}\left(f_{cd}\frac{\partial \xi_k}{\partial z_c}\frac{\partial \xi_i}{\partial z_d}\right)\\
+&\sum_{a,b,c,d,l} f_{ab}\frac{\partial \bar{\xi}_l}{\partial z_a}\frac{\partial \xi_k}{\partial z_b}\frac{\partial}{\partial \bar{\xi}_l}\left(f_{cd}\frac{\partial \xi_i}{\partial z_c}\frac{\partial \xi_j}{\partial z_d}\right)+f_{ab}\frac{\partial \bar{\xi}_l}{\partial z_a}\frac{\partial \xi_i}{\partial z_b}\frac{\partial}{\partial \bar{\xi}_l}\left(f_{cd}\frac{\partial \xi_j}{\partial z_c}\frac{\partial \xi_k}{\partial z_d}\right)+f_{ab}\frac{\partial \bar{\xi}_l}{\partial z_a}\frac{\partial \xi_j}{\partial z_b}\frac{\partial}{\partial \bar{\xi}_l}\left(f_{cd}\frac{\partial \xi_k}{\partial z_c}\frac{\partial \xi_i}{\partial z_d}\right)\\
=&\sum_{a,b,c,d,l}f_{ab}\frac{\partial \xi_l}{\partial z_a}\frac{\partial \xi_k}{\partial z_b}\frac{\partial f_{cd}}{\partial \xi_l}\frac{\partial \xi_i}{\partial z_c}\frac{\partial \xi_j}{\partial z_d}+f_{ab}\frac{\partial \xi_l}{\partial z_a}\frac{\partial \xi_k}{\partial z_b}f_{cd}\frac{\partial}{\partial \xi_l}\left(\frac{\partial \xi_i}{\partial z_c}\right)\frac{\partial \xi_j}{\partial z_d}+f_{ab}\frac{\partial \xi_l}{\partial z_a}\frac{\partial \xi_k}{\partial z_b}f_{cd}\frac{\partial \xi_i}{\partial z_c}\frac{\partial}{\partial \xi_l}\left(\frac{\partial \xi_j}{\partial z_d}\right)\\
+&\sum_{a,b,c,d,l}f_{ab}\frac{\partial \xi_l}{\partial z_a}\frac{\partial \xi_i}{\partial z_b}\frac{\partial f_{cd}}{\partial \xi_l}\frac{\partial \xi_j}{\partial z_c}\frac{\partial \xi_k}{\partial z_d}+f_{ab}\frac{\partial \xi_l}{\partial z_a}\frac{\partial \xi_i}{\partial z_b}f_{cd}\frac{\partial}{\partial \xi_l}\left(\frac{\partial \xi_j}{\partial z_c}\right)\frac{\partial \xi_k}{\partial z_d}+f_{ab}\frac{\partial \xi_l}{\partial z_a}\frac{\partial \xi_i}{\partial z_b}f_{cd}\frac{\partial \xi_j}{\partial z_c}\frac{\partial}{\partial \xi_l}\left(\frac{\partial \xi_k}{\partial z_d}\right)\\
+&\sum_{a,b,c,d,l}f_{ab}\frac{\partial \xi_l}{\partial z_a}\frac{\partial \xi_j}{\partial z_b}\frac{\partial f_{cd}}{\partial \xi_l}\frac{\partial \xi_k}{\partial z_c}\frac{\partial \xi_i}{\partial z_d}+f_{ab}\frac{\partial \xi_l}{\partial z_a}\frac{\partial \xi_j}{\partial z_b}f_{cd}\frac{\partial}{\partial \xi_l}\left(\frac{\partial \xi_k}{\partial z_c}\right)\frac{\partial \xi_i}{\partial z_d}+f_{ab}\frac{\partial \xi_l}{\partial z_a}\frac{\partial \xi_j}{\partial z_b}f_{cd}\frac{\partial \xi_k}{\partial z_c}\frac{\partial}{\partial \xi_l}\left(\frac{\partial \xi_i}{\partial z_d}\right)\\
+&\sum_{a,b,c,d,l}f_{ab}\frac{\partial \bar{\xi}_l}{\partial z_a}\frac{\partial \xi_k}{\partial z_b}\frac{\partial f_{cd}}{\partial \bar{\xi}_l}\frac{\partial \xi_i}{\partial z_c}\frac{\partial \xi_j}{\partial z_d}+f_{ab}\frac{\partial \bar{\xi}_l}{\partial z_a}\frac{\partial \xi_k}{\partial z_b}f_{cd}\frac{\partial}{\partial \bar{\xi}_l}\left(\frac{\partial \xi_i}{\partial z_c}\right)\frac{\partial \xi_j}{\partial z_d}+f_{ab}\frac{\partial \bar{\xi}_l}{\partial z_a}\frac{\partial \xi_k}{\partial z_b}f_{cd}\frac{\partial \xi_i}{\partial z_c}\frac{\partial}{\partial \bar{\xi}_l}\left(\frac{\partial \xi_j}{\partial z_d}\right)\\
+&\sum_{a,b,c,d,l}f_{ab}\frac{\partial \bar{\xi}_l}{\partial z_a}\frac{\partial \xi_i}{\partial z_b}\frac{\partial f_{cd}}{\partial \bar{\xi}_l}\frac{\partial \xi_j}{\partial z_c}\frac{\partial \xi_k}{\partial z_d}+f_{ab}\frac{\partial \bar{\xi}_l}{\partial z_a}\frac{\partial \xi_i}{\partial z_b}f_{cd}\frac{\partial}{\partial \bar{\xi}_l}\left(\frac{\partial \xi_j}{\partial z_c}\right)\frac{\partial \xi_k}{\partial z_d}+f_{ab}\frac{\partial \bar{\xi}_l}{\partial z_a}\frac{\partial \xi_i}{\partial z_b}f_{cd}\frac{\partial \xi_j}{\partial z_c}\frac{\partial}{\partial \bar{\xi}_l}\left(\frac{\partial \xi_k}{\partial z_d}\right)\\
+&\sum_{a,b,c,d,l}f_{ab}\frac{\partial \bar{\xi}_l}{\partial z_a}\frac{\partial \xi_j}{\partial z_b}\frac{\partial f_{cd}}{\partial \bar{\xi}_l}\frac{\partial \xi_k}{\partial z_c}\frac{\partial \xi_i}{\partial z_d}+f_{ab}\frac{\partial \bar{\xi}_l}{\partial z_a}\frac{\partial \xi_j}{\partial z_b}f_{cd}\frac{\partial}{\partial \bar{\xi}_l}\left(\frac{\partial \xi_k}{\partial z_c}\right)\frac{\partial \xi_i}{\partial z_d}+f_{ab}\frac{\partial \bar{\xi}_l}{\partial z_a}\frac{\partial \xi_j}{\partial z_b}f_{cd}\frac{\partial \xi_k}{\partial z_c}\frac{\partial}{\partial \bar{\xi}_l}\left(\frac{\partial \xi_i}{\partial z_d}\right)\\
=&\sum_{a,b,c,d}f_{ab}\frac{\partial f_{cd}}{\partial z_a}\frac{\partial \xi_k}{\partial z_b}\frac{\partial \xi_i}{\partial z_c}\frac{\partial \xi_j}{\partial z_d}+f_{ab}\frac{\partial \xi_k}{\partial z_b}f_{cd}\frac{\partial^2 \xi_i}{\partial z_a\partial z_c}\frac{\partial \xi_j}{\partial z_d}+f_{ab}\frac{\partial \xi_k}{\partial z_b}f_{cd}\frac{\partial \xi_i}{\partial z_c}\frac{\partial^2 \xi_j}{\partial z_a\partial z_d}\\
+&\sum_{a,b,c,d}f_{ab}\frac{\partial f_{cd}}{\partial z_a}\frac{\partial \xi_i}{\partial z_b}\frac{\partial \xi_j}{\partial z_c}\frac{\partial \xi_k}{\partial z_d}+f_{ab}\frac{\partial \xi_i}{\partial z_b}f_{cd}\frac{\partial^2 \xi_j}{\partial z_a\partial z_c}\frac{\partial \xi_k}{\partial z_d}+f_{ab}\frac{\partial \xi_i}{\partial z_b}f_{cd}\frac{\partial \xi_j}{\partial z_c}\frac{\partial^2 \xi_k}{\partial z_a\partial z_d}\\
+&\sum_{a,b,c,d}f_{ab}\frac{\partial f_{cd}}{\partial z_a}\frac{\partial \xi_j}{\partial z_b}\frac{\partial \xi_k}{\partial z_c}\frac{\partial \xi_i}{\partial z_d}+f_{ab}\frac{\partial \xi_j}{\partial z_b}f_{cd}\frac{\partial^2 \xi_k}{\partial z_a\partial z_c}\frac{\partial \xi_i}{\partial z_d}+f_{ab}\frac{\partial \xi_j}{\partial z_b}f_{cd}\frac{\partial \xi_k}{\partial z_c}\frac{\partial^2 \xi_i}{\partial z_a\partial z_d}\\
=&\sum_{a,b,c,d}f_{ab}\frac{\partial f_{cd}}{\partial z_a}\frac{\partial \xi_k}{\partial z_b}\frac{\partial \xi_i}{\partial z_c}\frac{\partial \xi_j}{\partial z_d}+f_{ab}\frac{\partial f_{cd}}{\partial z_a}\frac{\partial \xi_i}{\partial z_b}\frac{\partial \xi_j}{\partial z_c}\frac{\partial \xi_k}{\partial z_d}+f_{ab}\frac{\partial f_{cd}}{\partial z_a}\frac{\partial \xi_j}{\partial z_b}\frac{\partial \xi_k}{\partial z_c}\frac{\partial \xi_i}{\partial z_d}\\
=&\sum_{a,b,c,d}\left(f_{ab}\frac{\partial f_{cd}}{\partial z_a}+f_{ac}\frac{\partial f_{db}}{\partial z_a}+f_{ad}\frac{\partial f_{bc}}{\partial z_a}\right)\frac{\partial \xi_i}{\partial z_c}\frac{\partial \xi_j}{\partial z_d}\frac{\partial \xi_k}{\partial z_b}
\end{align*}
Since $det\left(\frac{\partial \xi_j^{\alpha}(z,t)}{\partial z_{\lambda}}\right)_{\alpha,\lambda=1,...,n}\ne 0$, $\sum_a f_{ab}\frac{\partial f_{cd}}{\partial z_a}+f_{ac}\frac{\partial f_{db}}{\partial z_a}+f_{ad}\frac{\partial f_{bc}}{\partial z_a}=0$. So by Lemma \ref{formula}, $[\Lambda_j',\Lambda_j']=0$.
\end{proof}

\begin{remark}
In summary, for holomorphic Poisson manifold $(M_t,\Lambda_t)$ for each $t\in \Delta$ in the Poisson analytic family, there exists a bivector field $\Lambda'(t)$ on $M$ with $[\Lambda'(t),\Lambda'(t)]=0$ for $t\in \Delta$. Conversely, $\Lambda'(t)$ induces $\Lambda_t$ via diffeomorphism $\Psi$. More precisely, the $(2,0)$-part of  $\Psi_{*} \Lambda'(t)$ is $\Lambda_t$ for $t\in \Delta$.\footnote{For the type of a bivector field, we mean the decomposition $\wedge^2 T_{\mathbb{C}} M=\wedge^2 T^{1,0} \oplus T^{1,0}\otimes T^{0,1} \oplus \wedge^2 T^{0,1}$ with respect to the almost complex structure induced from the complex structure. If $\Lambda\in C^{\infty}(\wedge T_{\mathbb{C}} M)$, then we denote by $\Lambda^{2,0}$ the component of $C^{\infty}(\wedge^2 T^{1,0})$, by $\Lambda^{1,1}$ the component of $C^{\infty}(T^{0,1}\otimes T^{0,1})$, and by $\Lambda^{0,2}$ the component of $C^{\infty}(\wedge^2 T^{0,1})$. So we have $\Lambda=\Lambda^{2,0}+\Lambda^{1,1}+\Lambda^{0,2}$.}
\end{remark}

Now we discuss the condition when a $C^{\infty}$ bivector field $\Lambda$ on $M$ with $[\Lambda,\Lambda]=0$ gives a holomorphic bivector field $\Lambda^{2,0}$ with respect to the complex structure $M_t$ when we restrict $\Lambda$ to $(2,0)$ part. Before proceeding our discussion, we recall a bracket structure $[-,-]$ on $A=\bigoplus_{p+q=i,p\geq 0,q\geq 1} A^{0,p}(M,\wedge^q T_M)$ (See Appendix \ref{appendixc}), which we need for the computation of the integrability condition.

The bracket structure on $A$ is defined in the following way.
\begin{align*}
[-,-]:A^{0,p}(M,\wedge^q T_M)\times A^{0,p'}(M,\wedge^{q'} T_M)\to A^{p+p'}(M,\wedge^{q+q'-1} T_M)
\end{align*}
In local coordinates it is given by
\begin{align*}
[fdz_I\frac{\partial}{\partial z_J},gdz_K\frac{\partial}{\partial z_L}]=(-1)^{|K|(|J|+1)} dz_I\wedge dz_K [f\frac{\partial}{\partial z_J},g\frac{\partial}{\partial z_L}]\end{align*}
Then $(A[1],\bar{\partial},[-,-])$ is a differntial graded Lie algebra.
So we have the following. For $a\in A^{0,p}(M,\wedge^q T_M), b\in A^{0,p'}(M,\wedge^{q'} T_M)$,
\begin{enumerate}
\item $[a,b]=-(-1)^{(p+q+1)(p'+q'+1)}[b,a]$
\item $[a,[b,c]]=[[a,b],c]+(-1)^{(p+q+1)(p'+q'+1)}[b,[a,c]]$
\item $\bar{\partial}[a,b]=[\bar{\partial} a,b]+(-1)^{p+q+1}[a,\bar{\partial} b]$
\end{enumerate}

\begin{thm}\label{m}
If we take a sufficiently small polydisk $\Delta$, then for $t\in \Delta$, a $(2,0)$-part $\Lambda^{2,0}$ of a $C^{\infty}$ bivector field  $\Lambda=\sum_{\alpha,\beta=1}^n f_{\alpha \beta}(z)\frac{\partial}{\partial z_{\alpha}}\wedge \frac{\partial}{\partial z_{\beta}}$ on $M$ is holomorphic with respect to the complex structure $M_t$, if and only if it satisfies the equation
\begin{align*}
\bar{\partial}\Lambda-[\Lambda,\varphi(t)]=0
\end{align*}
Moreover, if $[\Lambda,\Lambda]=0$, then $[\Lambda^{2,0},\Lambda^{2,0}]=0$.

\end{thm}

\begin{proof}
Since $(2,0)$ part of $\Lambda=\sum_{\alpha,\beta=1}^n f_{\alpha \beta}(z)\frac{\partial}{\partial z_{\alpha}}\wedge \frac{\partial}{\partial z_{\beta}}$ is $\sum_{\alpha,\beta,i,j} f_{\alpha \beta}\frac{\partial \xi^i}{\partial z_{\alpha}}\frac{\partial \xi^j}{\partial z_{\beta}}\frac{\partial}{\partial \xi^i}\wedge \frac{\partial}{\partial \xi^j}$, we have to show that for each $i,j$,$\sum_{\alpha,\beta} f_{\alpha \beta}\frac{\partial \xi^i}{\partial z_{\alpha}}\frac{\partial \xi^j}{\partial z_{\beta}}$ is holomorphic with respect to the complex structure $M_t$, which is equivalent to $(\bar{\partial}-\varphi(t))(\sum_{\alpha,\beta} f_{\alpha \beta}\frac{\partial \xi^i}{\partial z_{\alpha}}\frac{\partial \xi^j}{\partial z_{\beta}})=0$ by Theorem \ref{text}, if and only if $\bar{\partial}\Lambda-[\Lambda,\varphi(t)]=0$. First we compute $\bar{\partial}\Lambda-[\Lambda,\varphi(t)]=\sum_{\alpha,\beta, v} \frac{\partial f_{\alpha\beta}}{\partial \bar{z}_v}d\bar{z}_v\frac{\partial}{\partial z_{\alpha}}\wedge \frac{\partial}{\partial z_{\beta}}-[\sum_{\alpha,\beta} f_{\alpha\beta}\frac{\partial}{\partial z_{\alpha}}\wedge \frac{\partial}{\partial z_{\beta}},\sum_{v,\lambda}  \varphi_v^{\lambda}(z,t) d\bar{z_v} \frac{\partial}{\partial z_{\lambda}}]=\sum_{\alpha,\beta, v} \frac{\partial f_{\alpha\beta}}{\partial \bar{z}_v}d\bar{z}_v\frac{\partial}{\partial z_{\alpha}}\wedge \frac{\partial}{\partial z_{\beta}}-\sum_{\alpha,\beta, v,\lambda} [f_{\alpha\beta}\frac{\partial}{\partial z_{\alpha}}\wedge \frac{\partial}{\partial z_{\beta}},\varphi_v^{\lambda} d\bar{z_v} \frac{\partial}{\partial z_{\lambda}}]=\sum_{\alpha,\beta, v} \frac{\partial f_{\alpha\beta}}{\partial \bar{z}_v}d\bar{z}_v\frac{\partial}{\partial z_{\alpha}}\wedge \frac{\partial}{\partial z_{\beta}}+\sum_{\alpha,\beta, v,\lambda} [f_{\alpha\beta}\frac{\partial}{\partial z_{\alpha}}\wedge \frac{\partial}{\partial z_{\beta}},\varphi_v^{\lambda} \frac{\partial}{\partial z_{\lambda}}]d\bar{z}_v=\sum_{\alpha,\beta, v} \frac{\partial f_{\alpha\beta}}{\partial \bar{z}_v}d\bar{z}_v\frac{\partial}{\partial z_{\alpha}}\wedge \frac{\partial}{\partial z_{\beta}}+\sum_{\alpha,\beta, v,\lambda} [f_{\alpha\beta}\frac{\partial}{\partial z_{\alpha}}\wedge \frac{\partial}{\partial z_{\beta}},\varphi_v^{\lambda} \frac{\partial}{\partial z_{\lambda}}]d\bar{z}_v=\sum_{\alpha,\beta, v} \frac{\partial f_{\alpha\beta}}{\partial \bar{z}_v}d\bar{z}_v\frac{\partial}{\partial z_{\alpha}}\wedge \frac{\partial}{\partial z_{\beta}}+\sum_{\alpha,\beta, v,\lambda}(f_{\alpha\beta}\frac{\partial \phi_v^{\lambda}}{\partial z_{\alpha}}\frac{\partial}{\partial z_{\lambda}}\wedge \frac{\partial}{\partial z_{\beta}}-\varphi_{v}^{\lambda} \frac{\partial f_{\alpha\beta}}{\partial z_{\lambda}}\frac{\partial}{\partial z_{\alpha}}\wedge \frac{\partial}{\partial z_{\beta}}+f_{\alpha\beta} \frac{\partial \varphi_v^{\lambda}}{\partial z_{\beta}}\frac{\partial}{\partial z_{\alpha}}\wedge \frac{\partial}{\partial z_{\lambda}})d\bar{z}_v$. By considering the coefficients of $d\bar{z}_v\frac{\partial}{\partial z_{\alpha}}\wedge \frac{\partial}{\partial z_{\beta}}$, $\bar{\partial}\Lambda-[\Lambda,\varphi(t)]=0$ is equivalent to $\sum_{\alpha,\beta,v} [\frac{\partial f_{\alpha\beta}}{\partial \bar{z}_v}+\sum_c (f_{c\beta}\frac{\partial \varphi^{\alpha}_v}{\partial z_c}-\varphi^c_v\frac{\partial f_{\alpha \beta}}{\partial z_c}+f_{\alpha c}\frac{\partial \varphi^{\beta}_v}{\partial z_c})] d\bar{z}_v\frac{\partial}{\partial z_{\alpha}}\wedge \frac{\partial}{\partial z_{\beta}}=0$ which is equivalent to 
\begin{center}
$(*)\frac{\partial f_{\alpha\beta}}{\partial \bar{z}_v}+\sum_c (f_{c\beta}\frac{\partial \varphi^{\alpha}_v}{\partial z_c}-\varphi^c_v\frac{\partial f_{\alpha \beta}}{\partial z_c}+f_{\alpha c}\frac{\partial \varphi^{\beta}_v}{\partial z_c})]=0$ for each $\alpha,\beta,v$.
\end{center}
On the other hand, we compute $(\bar{\partial}-\varphi(t))(\sum_{\alpha,\beta} f_{\alpha \beta}\frac{\partial \xi^i}{\partial z_{\alpha}}\frac{\partial \xi^j}{\partial z_{\beta}})=\sum_{\alpha,\beta}(\bar{\partial}-\varphi(t))(f_{\alpha \beta}\frac{\partial \xi^i}{\partial z_{\alpha}}\frac{\partial \xi^j}{\partial z_{\beta}})$ for each $i,j$. $\sum_{\alpha,\beta}(\bar{\partial}-\varphi(t))(f_{\alpha \beta}\frac{\partial \xi^i}{\partial z_{\alpha}}\frac{\partial \xi^j}{\partial z_{\beta}})=\sum_{\alpha ,\beta,v} (\frac{\partial f_{\alpha\beta}}{\partial \bar{z}_v}\frac{\partial \xi^i}{\partial z_{\alpha}}\frac{\partial \xi^j}{\partial z_{\beta}}+f_{\alpha\beta}\frac{\partial}{\partial \bar{z}_v}(\frac{\partial \xi^i}{\partial z_{\alpha}})\frac{\partial \xi^j}{\partial z_{\beta}}+f_{\alpha\beta}\frac{\partial \xi^i}{\partial z_{\alpha}}\frac{\partial}{\partial \bar{z}_v}(\frac{\partial \xi^j}{\partial z_{\beta}}))d\bar{z}_v-\sum_{\alpha,\beta,v,\lambda} \varphi^{\lambda}_v d\bar{z}_v(\frac{\partial f_{\alpha\beta}}{\partial z_{\lambda}}\frac{\partial \xi^i}{\partial z_{\alpha}}\frac{\partial \xi_j}{\partial z_{\beta}}+f_{\alpha\beta}\frac{\partial^2 \xi^i}{\partial z_{\alpha} \partial z_{\lambda}}\frac{\partial \xi^j}{\partial z_{\beta}}+f_{\alpha\beta}\frac{\partial \xi^i}{\partial z_{\alpha}}\frac{\partial \xi^j}{\partial z_{\beta}\partial z_{\lambda}})$. So $(\bar{\partial}-\varphi(t))(\sum_{\alpha,\beta} f_{\alpha \beta}\frac{\partial \xi^i}{\partial z_{\alpha}}\frac{\partial \xi^j}{\partial z_{\beta}})=0$ is equivalent to
\begin{center}
$(**)\sum_{\alpha ,\beta} (\frac{\partial f_{\alpha\beta}}{\partial \bar{z}_v}\frac{\partial \xi^i}{\partial z_{\alpha}}\frac{\partial \xi^j}{\partial z_{\beta}}+f_{\alpha\beta}\frac{\partial}{\partial z_{\alpha}}(\frac{\partial \xi^i}{\partial \bar{z}_v})\frac{\partial \xi^j}{\partial z_{\beta}}+f_{\alpha\beta}\frac{\partial \xi^i}{\partial z_{\alpha}}\frac{\partial}{\partial z_{\beta}}(\frac{\partial \xi^j}{\partial \bar{z}_v}))-\sum_{\alpha,\beta,c} \varphi^{c}_v(\frac{\partial f_{\alpha\beta}}{\partial z_{c}}\frac{\partial \xi^i}{\partial z_{\alpha}}\frac{\partial \xi_j}{\partial z_{\beta}}+f_{\alpha\beta}\frac{\partial^2 \xi^i}{\partial z_{\alpha} \partial z_{c}}\frac{\partial \xi^j}{\partial z_{\beta}}+f_{\alpha\beta}\frac{\partial \xi^i}{\partial z_{\alpha}}\frac{\partial \xi^j}{\partial z_{\beta}\partial z_{c}})=0$
\end{center}
for each $i,j,v$.
Since

\begin{equation*}
\left(
\begin{matrix}
\frac{\partial \xi^1}{\partial z_1} & \dots & \frac{\partial \xi^1}{\partial z_n}\\
\vdots & \vdots\\
\frac{\partial \xi^n}{\partial z_1} & \dots & \frac{\partial \xi^n}{\partial z_n}
\end{matrix}
\right)
\left(
\begin{matrix}
\varphi_{1}^1 & \dots & \varphi_{n}^1\\
\vdots & \vdots\\
\varphi_{1}^n & \dots & \varphi_{n}^n
\end{matrix}
\right)
=
\left(
\begin{matrix}
\frac{\partial \xi^1}{\partial \bar{z_1}} & \dots & \frac{\partial \xi^1}{\partial \bar{z_n}}\\
\vdots & \vdots\\
\frac{\partial \xi^n}{\partial \bar{z_1}} & \dots & \frac{\partial \xi^n}{\partial \bar{z_n}}
\end{matrix}
\right),
\end{equation*}
 we have $\frac{\partial \xi^i}{\partial \bar{z}_v}=\sum_c \frac{\partial \xi^i}{\partial z_c}\varphi^c_v$ and $\frac{\partial \xi^j}{\partial \bar{z}_v}=\sum_c \frac{\partial \xi^j}{\partial z_c}\varphi^c_v$.
So $(**)$ is equivalent to
\begin{center}
$\sum_{\alpha ,\beta} \frac{\partial f_{\alpha\beta}}{\partial \bar{z}_v}\frac{\partial \xi^i}{\partial z_{\alpha}}\frac{\partial \xi^j}{\partial z_{\beta}}+\sum_{\alpha,\beta,c}(f_{\alpha\beta}(\frac{\partial^2 \xi^i}{\partial z_{\alpha} \partial z_c}\varphi^c_v+\frac{\partial \xi_i}{\partial z_c}\frac{\partial \varphi^c_v}{\partial z_{\alpha}})\frac{\partial \xi^j}{\partial z_{\beta}}+f_{\alpha\beta}\frac{\partial \xi^i}{\partial z_{\alpha}}(\frac{\partial^2 \xi^j}{\partial z_{\beta} \partial z_c}\varphi^c_v+\frac{\partial \xi_j}{\partial z_c}\frac{\partial \varphi^c_v}{\partial z_{\beta}}))-\sum_{\alpha,\beta,c} \varphi^{c}_v(\frac{\partial f_{\alpha\beta}}{\partial z_{c}}\frac{\partial \xi^i}{\partial z_{\alpha}}\frac{\partial \xi_j}{\partial z_{\beta}}+f_{\alpha\beta}\frac{\partial^2 \xi^i}{\partial z_{\alpha} \partial z_{c}}\frac{\partial \xi^j}{\partial z_{\beta}}+f_{\alpha\beta}\frac{\partial \xi^i}{\partial z_{\alpha}}\frac{\partial \xi^j}{\partial z_{\beta}\partial z_{c}})=0$.
\end{center}
which is equivalent to
\begin{center}
$\sum_{\alpha ,\beta} \frac{\partial f_{\alpha\beta}}{\partial \bar{z}_v}\frac{\partial \xi^i}{\partial z_{\alpha}}\frac{\partial \xi^j}{\partial z_{\beta}}+\sum_{\alpha,\beta,c}(f_{\alpha\beta}(\frac{\partial \xi_i}{\partial z_c}\frac{\partial \varphi^c_v}{\partial z_{\alpha}})\frac{\partial \xi^j}{\partial z_{\beta}}+f_{\alpha\beta}\frac{\partial \xi^i}{\partial z_{\alpha}}(\frac{\partial \xi_j}{\partial z_c}\frac{\partial \varphi^c_v}{\partial z_{\beta}}))-\sum_{\alpha,\beta,c} \varphi^{c}_v(\frac{\partial f_{\alpha\beta}}{\partial z_{c}}\frac{\partial \xi^i}{\partial z_{\alpha}}\frac{\partial \xi_j}{\partial z_{\beta}})
=0$
\end{center}
which is equivalent to
\begin{center}
$\sum_{\alpha,\beta} [\frac{\partial f_{\alpha\beta}}{\partial \bar{z}_v}+\sum_c (f_{c\beta}\frac{\partial \varphi^{\alpha}_v}{\partial z_c}-\varphi^c_v\frac{\partial f_{\alpha \beta}}{\partial z_c}+f_{\alpha c}\frac{\partial \varphi^{\beta}_v}{\partial z_c})]\frac{\partial \xi^i}{\partial z_{\alpha}}\frac{\partial \xi^j}{\partial z_{\beta}}=0$ for each $i,j,v$.
\end{center}
Since $det\left(\frac{\partial \xi_j^{\alpha}(z,t)}{\partial z_{\lambda}}\right)_{\alpha,\lambda=1,...,n}\ne 0$, this is equivalent to
\begin{center}
$(***) \frac{\partial f_{\alpha\beta}}{\partial \bar{z}_v}+\sum_c (f_{c\beta}\frac{\partial \varphi^{\alpha}_v}{\partial z_c}-\varphi^c_v\frac{\partial f_{\alpha \beta}}{\partial z_c}+f_{\alpha c}\frac{\partial \varphi^{\beta}_v}{\partial z_c})=0$ for each $\alpha,\beta,v$.
\end{center}
Note that $(*)$ is same to $(***)$.

For the second statement, we can write $\Lambda=\sum_{a,b} f_{ab}\frac{\partial}{\partial z_a}\wedge \frac{\partial}{\partial z_b}=\sum_{a,b,i,j} f_{ab}\frac{\partial \xi_i}{\partial z_a}\frac{\partial \xi_j}{\partial z_b}\frac{\partial}{\partial \xi_i}\wedge \frac{\partial}{\partial \xi_j}+ 2f_{ab}\frac{\partial \xi_i}{\partial z_a}\frac{\partial \bar{\xi}_j}{\partial z_b}\frac{\partial}{\partial \xi_i}\wedge \frac{\partial}{\partial \bar{\xi}_j}+ f_{ab}\frac{\partial \xi_i}{\partial z_a}\frac{\partial \bar{\xi}_j}{\partial z_b}\frac{\partial}{\partial \bar{\xi}_i}\wedge \frac{\partial}{\partial \bar{\xi}_j}=\Lambda^{2,0}+\Lambda^{1,1}+\Lambda^{2,0}$. Since $[\Lambda,\Lambda]=0$, $(3,0)$ part of $[\Lambda,\Lambda]=0$. But $(3,0)$ part happens in $[\Lambda^{2,0},\Lambda^{2,0}]+[\Lambda^{2,0},\Lambda^{1,1}]^{3,0}$. Since $\Lambda^{2,0}$ is holomorphic with respect to the complex structure induced by $\varphi(t)$, $[\Lambda^{2,0},\Lambda^{1,1}]^{3,0}=0$.
\end{proof}

\begin{remark}
A $C^{\infty}$ complex bivector field $\Lambda$ on $M$ with $[\Lambda,\Lambda]=0$ gives a Poisson bracket on $C^{\infty}$ complex valued functions on $M$. We point out that when we restrict to holomorphic functions with respect to the complex structure $M_t$, this is exactly the Poisson bracket induced from $\Lambda^{2,0}$. 
\end{remark}

\section{Expression of  infinitesimal deformations in terms of $\varphi(t)$ and $\Lambda(t)$}\

In this section, we study how the infinitesimal deformation of $(M,\Lambda_0)$ in a family is represented in terms of $\varphi(t)$ and $\Lambda(t)$. Recall that $(\mathcal{M},\Lambda,B,\pi)$ is a Poisson analytic family of compact holomorphic Poisson manifolds with $(M,\Lambda_0)=\omega^{-1}(0)$ and for sufficiently small polydisk $\Delta\subset B$, $M_{\Delta}=\omega^{-1}(\Delta)$ is represented in the form $\mathcal{M}_{\Delta}=\bigcup_j U_j\times \Delta$ where $U_j=\{x_j\in \mathbb{C}^m||\xi_j|<1\}$ and the holomorphic Poisson structures on $U_j$ is $\sum_{\alpha\beta} g_{\alpha\beta}(\xi_j,t) \frac{\partial}{\partial \xi_j^{\alpha}}\wedge \frac{\partial}{\partial \xi_j^{\beta}}$ and $\xi_j^{\alpha}=f_{jk}^{\alpha}(\xi_k,t),\alpha=1,...,m$ on $U_k\times \Delta \cap U_j\times \Delta$. We showed that the infinitesimal deformation at $(M,\Lambda_0)$ is captured by the element $(\frac{(M_t,\Lambda_t)}{\partial t})_{t=0}\in HP^2(M,\Lambda_0)$ of the complex of sheaves of $0\to \Theta_{M} \to \wedge^2 \Theta_{M}\to \cdots \to \wedge^n \Theta_M\to 0$ by using the following \u{C}ech hypercohomology resolution associated with the open covering $\mathcal{U}^0=\{U_j^0:=U_j\times 0\}$. (See Proposition \ref{gg})

\begin{center}
$\begin{CD}
@A[\Lambda,-]AA\\
C^0(\mathcal{U}^0,\wedge^3 \Theta_M)@>-\delta>>\cdots\\
@A[\Lambda,-]AA @A[\Lambda,-]AA\\
C^0(\mathcal{U}^0,\wedge^2 \Theta_M)@>\delta>> C^1(\mathcal{U}^0,\wedge^2 \Theta_M)@>-\delta>>\cdots\\
@A[\Lambda,-]AA @A[\Lambda,-]AA @A[\Lambda,-]AA\\
C^0(\mathcal{U}^0,\Theta_M)@>-\delta>>C^1(\mathcal{U}^0,\Theta_M)@>\delta>>C^2(\mathcal{U}^0,\Theta_M)@>-\delta>>\cdots\\
@AAA @AAA @AAA @AAA \\
0@>>>0 @>>> 0 @>>> 0@>>> \cdots
\end{CD}$
\end{center}

And we can also compute the hypercohomology group of $0\to \Theta_M \to \wedge^2 \Theta_M\to \cdots \to \wedge^n \Theta_M\to 0$ by using the following Dolbeault type resolution.(See example \ref{rr})
\begin{center}
$\begin{CD}
@A[\Lambda,-]AA\\
A^{0,0}(M,\wedge^3 T_M)@>\bar{\partial}>>\cdots\\
@A[\Lambda,-]AA @A[\Lambda,-]AA\\
A^{0,0}(M,\wedge^2 T_M)@>\bar{\partial}>> A^{0,1}(M,\wedge^2 T_M)@>\bar{\partial}>>\cdots\\
@A[\Lambda,-]AA @A[\Lambda,-]AA @A[\Lambda,-]AA\\
A^{0,0}(M,T_M)@>\bar{\partial}>>A^{0,1}(M,T_M)@>\bar{\partial}>>A^{0,2}(M, T_M)@>\bar{\partial}>>\cdots\\
@AAA @AAA @AAA @AAA \\
0@>>>0 @>>> 0 @>>> 0@>>> \cdots
\end{CD}$
\end{center}
We describe how the element in the Cech hypercohomology look like in the Dolbeault hypercohomology.
In the picture below, we connect two resolutions. We only depict a part of resolutions that we need in the next page.

{\tiny
\[
\xymatrixrowsep{0.2in}
\xymatrixcolsep{0.1in}
\xymatrix{
& & \wedge^3 \Theta_M \ar[ld]  \ar[rr]  &  & C^0(\wedge^3 \Theta_M) \ar[ld]\\
&A^{0,0}(M,\wedge^3 T_M)  \ar[rr] & & C^0(\mathscr{A}^{0,0}(\wedge^3 T_M))\\
&& \wedge^2 \Theta_M \ar@{.>}[uu] \ar@{.>}[ld] \ar@{.>}[rr] & & C^0(\wedge^2 \Theta_M) \ar[uu] \ar[ld] \ar[rr]^{\delta} && C^1(\wedge^2 \Theta_X) \ar[ld]\\
& A^{0,0}(M,\wedge^2  T_M)  \ar[uu] \ar[ld] \ar[rr] && C^0(\mathscr{A}^{0,0}(\wedge^2 T_M)) \ar[uu] \ar[ld] \ar[rr]^{\delta} && C^1(\mathscr{A}^{0,0}(T_M))   \\
A^{0,1}(M,\wedge^2 T_M) \ar[rr]  && C^0(A^{0,1}(\wedge^2 T_M)) && C^0(\Theta_M) \ar@{.>}[ld]\ar@{.>}[uu] \ar@{.>}[rr]^{-\delta} & & C^1(\Theta_X) \ar[uu] \ar[ld]\\
& A^{0,0}(M,T_M) \ar@{.>}[uu] \ar@{.>}[rr]  \ar@{.>}[ld] && C^0(\mathscr{A}^{0,0}(T_M)) \ar[uu]^{[\Lambda,-]} \ar[rr]^{-\delta}  \ar[ld]^{\bar{\partial}} && C^1(\mathscr{A}^{0,0}(T_M)) \ar[ld] \ar[uu]\\
A^{0,1}(M,T_M) \ar[uu] \ar[rr] & & C^0(\mathscr{A}^{0,1}(T_M)) \ar[uu] \ar[rr] & & C^1(\mathscr{A}^{0,1}(T_M))
}\]}

Now we explicitly construct the isomorphism of second hypercomology groups from \u{C}ech hyperresolution and Dolbeault hyperresolution, namely
\begin{align*}
HP^2(M,\Lambda_0)\cong \frac{ker(A^{0,0}(M,\wedge^2 T_M)\oplus A^{1,0}(M, T_M)\to A^{0,0}(M,\wedge^3 T_M)\oplus A^{1,0}(M,\wedge^2 T_M)\oplus A^{2,0}(M, T_M))}{im(A^{0,0}(M,T_M)\to A^{0,0}(M,\wedge^2 T_M)\oplus A^{1,0}(M,\wedge T_M))}
\end{align*}
Note that each horizontal complex is exact except for edges of the ``real wall".

We define the map in the following way:
let $(b,a) \in \mathcal{C}^0(\mathcal{U}, \wedge^2 \Theta_M)\oplus \mathcal{C}^1(\mathcal{U},\Theta_M)$ be a cohomology class of $HP^2(M,\Lambda)$. Since $\delta a=0$, there exists a $c\in C^0(\mathcal{U},\mathscr{A}^{0,0}(T_M))$ such that $-\delta c=a$. Since $a$ is holomorphic $(\bar{\partial}a=0)$, by the commutativity $\bar{\partial} c\in A^{0,1}(M, T_M)$. And we claim that $[\Lambda,c]-b\in A^{0,0}(M,\wedge^2 T_M)$. Indeed $\delta([\Lambda,c]-b)=\delta([\Lambda,c])-\delta b=-[\Lambda,-\delta c]-\delta b=-[\Lambda,a]-\delta b=0$. Now we show that $(\bar{\partial} c, [\Lambda,c]-b)$ is a cohomology class of Dolbeault type resolution. Clearly $\bar{\partial}(\bar{\partial}c)=0.$ $[\Lambda,[\Lambda,c]-b]=0$. And $\bar{\partial} ([\Lambda,c]-b)+[\Lambda, \bar{\partial} c]=-[\Lambda,\bar{\partial} c]+[\Lambda,\bar{\partial c}]=0$. We define the map by $(b,a)\mapsto ([\Lambda,c]-b,\bar{\partial} c)$. 

Now we show that this map is well defined.
\begin{enumerate}
\item (independence of choice of $c$)
let $c'$ with $-\delta c'=a$. Then $-\delta(c-c')=0$. So $d=c-c'\in A^{0,0}(M,\Theta_M)$. Then $([\Lambda,c]-b,\bar{\partial c})-([\Lambda,c']-b,\bar{\partial} c')=([\Lambda,c-c'],\bar{\partial}(c-c'))=\bar{\partial}d+[\Lambda,d]$

\item (independence of choice of $(b,a)$)
Let $(b,a)$ and $(b',a')$ are in the same cohomology class.  We show that $(b-b',a-a')$ is mapped to $0$. Indeed, there exists $e\in C^0(\mathcal{U},\Theta_M)$ such that $-\delta e+[\Lambda,e]=(a-a')-(b-b')$. We can use $e$ as $c$. Then $(b-b',a-a')$ is mapped to $([\Lambda, e]-(b-b'),\bar{\partial} e)=(0,0)$.
\end{enumerate}

For the inverse map, let $(\beta,\alpha) \in A^{0,0}(M,\wedge^2  T_M)\oplus A^{0,1}(M,  T_M) $ be the cohomology class of Dolbeault type resolution. 
Then there exists $c\in C^0(\mathcal{U},\mathscr{A}^{0,0}(T_M))$ such that $\bar{\partial} c =\alpha$. We define the inverse map $(\beta,\alpha) \mapsto ([\Lambda,c]-\beta,-\delta c)$.

\begin{thm}\label{n}
$\left(\left(\frac{\partial \varphi(t)}{\partial t}\right)_{t=0},- \left(\frac{\partial \Lambda(t)}{\partial t}\right)_{t=0}\right)$ satisfies 

$[\Lambda_0, -\left(\frac{\partial \Lambda(t)}{\partial t}\right)_{t=0}]=0$,  $\bar{\partial} \left( -(\frac{\partial \Lambda(t)}{\partial t})_{t=0}\right) +[\Lambda_0,\left(\frac{\partial \varphi(t)}{\partial t}\right)_{t=0}]=0$, $\bar{\partial} \left(\frac{\partial \varphi(t)}{\partial t}\right)_{t=0}=0$, and under the isomorphism
\begin{align*}
HP^2(M)\cong \frac{ker(A^{0,0}(M,\wedge^2 T_M)\oplus A^{1,0}(M,\wedge T_M)\to A^{0,0}(M,\wedge^3 T_M)\oplus A^{1,0}(M,\wedge^2 T_M)\oplus A^{2,0}(M,T_M))}{im(A^{0,0}(M,T_M)\to A^{0,0}(M,\wedge^2 T_M)\oplus A^{1,0}(M,\wedge T_M))}
\end{align*}
$\left(\frac{(M_t,\Lambda_t)}{\partial t}\right)_{t=0}\in HP^2(M,\Lambda_0)$ corresponds to $\left(\left(\frac{\partial \varphi(t)}{\partial t}\right)_{t=0}, -\left(\frac{\partial \Lambda(t)}{\partial t}\right)_{t=0}\right)$

\end{thm}

\begin{proof}
By taking the derivative of our integrability condition with respect to $t$ and plugging $0$, we get the first claim. Now we construct the isomorphism between two second cohomology groups and their correspondence. Put

\begin{align*}
\theta_{jk}&=\sum_{\alpha=1}^n \left(\frac{\partial f_{jk}^{\alpha}(\xi_k,t)}{\partial t}\right)_{t=0}\frac{\partial}{\partial \xi_j^{\alpha}}\\
\sigma_j&=\sum_{r,s=1}^n \left(\frac{\partial g_{rs}(\xi,t)}{\partial t}\right)_{t=0} \frac{\partial}{\partial \xi_j^r}\wedge \frac{\partial}{\partial \xi_j^s}
\end{align*}
The infinitesimal deformation $\left(\frac{\partial (M_t,\Lambda_t)}{\partial t}\right)_{t=0}\in HP^2(M,\Lambda_0)$ is the cohomology class of the $(\{\theta_{jk}\},\{\sigma_j\})\in C^1(\mathcal{U}^0,\Theta)\oplus C^0(\mathcal{U}^0,\wedge^2 \Theta)$. We fix a tangent vector $\frac{\partial}{\partial t}\in T_0(\Delta)$, denote $\left(\frac{\partial f(t)}{\partial t}\right)_{t=0}$ by $\dot{f}$.  By differentiating
\begin{align*}
\xi_j^{\alpha}(z,t)=f_{jk}^{\alpha}(\xi_k(z,t),t)=f_{jk}(\xi_k^1(z,t),...,\xi_k^n(z,t),t)
\end{align*}
with respect to $t$ and putting $t=0$, we get
\begin{align*}
\dot{\xi_j}^{\alpha}:=\sum_{\beta=1}^n \frac{\partial \xi_j^{\alpha}}{\partial \xi_k^{\beta}}\dot{\xi_k^{\beta}}+\left( \frac{\partial f_{jk}^{\alpha}(\xi_k,t)}{\partial t}\right)_{t=0}
\end{align*}
where $\xi_j^{\alpha}=\xi_j^{\alpha}(z,0)$ and $\xi_k^{\beta}=\xi_k^{\beta}(z,0)$. Therefore putting
\begin{align*}
\xi_j=\sum_{\alpha=1}^n \dot{\xi_j}^{\alpha}\frac{\partial}{\partial \xi_j^{\alpha}}
\end{align*}
for each $j$, we have
\begin{align*}
\theta_{jk}=\xi_j-\xi_k
\end{align*}
Since $\dot{\xi_j}^{\alpha}=\left(\frac{\partial \xi_j^{\alpha}(z,t)}{\partial t}\right)_{t=0}$ is a $C^{\infty}$ function on $U_j$, $\xi_j$ is a $C^{\infty}$ vector field on $U_j$. So we have $\{\xi_j \} \in C^0(\mathcal{U}_0, \mathcal{A}^{0,0}(\Theta))$ and $\delta \{ \xi_j \}=-\{ \theta_{jk} \}$ where $\delta$ is the usual \u{C}ech map.

We need the following lemma.

\begin{lemma}\label{beta}
$\bar{\partial} \xi_j =\sum_{\lambda=1}^n \left( \frac{\partial \varphi^{\lambda}(z,t)}{\partial t}\right)_{t=0}\frac{\partial}{\partial z_{\lambda}}=\sum_{\lambda=1}^n \dot{\varphi}^{\lambda}\frac{\partial}{\partial z_{\lambda}}=\dot{\varphi}$ and $\dot{\Lambda}-{\sigma_j}+[\Lambda_0, \xi_j]=0$. More precisely,

\begin{align*}
\sum_{r,s=1}^n \left(\frac{\partial f_{rs} (z,t)}{\partial t}\right)_{t=0} \frac{\partial}{\partial z_r}\wedge \frac{\partial}{\partial z_s}-\sum_{\alpha,\beta=1}^n \left(\frac{\partial g_{\alpha\beta}^j(\xi_j,t)}{\partial t}\right)_{t=0} \frac{\partial}{\partial \xi_j^{\alpha}} \wedge \frac{\partial}{\partial \xi_j^{\beta}}+[\sum_{r,s=1}^n g_{rs}^j(\xi_j,0)\frac{\partial}{\partial \xi_j^r}\wedge \frac{\partial}{\partial \xi_j^s},\sum_{c=1}^n \dot{\xi}_j^c\frac{\partial}{\partial \xi_j^c}]=0
\end{align*}
equivalently,
\begin{align*}
(*)\sum_{r,s} \dot{f_{rs}} \frac{\partial}{\partial z_r}\wedge \frac{\partial}{\partial z_s}-\sum_{\alpha,\beta=1}^n \dot{g}_{\alpha\beta}^j \frac{\partial}{\partial \xi_j^{\alpha}} \wedge \frac{\partial}{\partial \xi_j^{\beta}}+[\sum_{r,s=1}^n g_{rs}^j(\xi_j,0)\frac{\partial}{\partial \xi_j^r}\wedge \frac{\partial}{\partial \xi_j^s},\sum_{c=1}^n \dot{\xi}_j^c\frac{\partial}{\partial \xi_j^c}]=0
\end{align*}
\end{lemma}

\begin{proof}
By differentiating
\begin{align*}
\bar{\partial} \xi_j^{\alpha}(z,t)=\sum_{\lambda=1}^n \varphi^{\lambda}(z,t)\frac{\partial \xi_j^{\alpha}(z,t)}{\partial z_{\lambda}}
\end{align*}
with respect to $t$, and putting $t=0$, we obtain
\begin{align*}
\bar{\partial}\dot{\xi}_j^{\alpha}=\sum_{\lambda=1}^{n} \dot{\varphi}^{\lambda}\frac{\partial \xi_j^{\alpha}(z,0)}{\partial z_{\lambda}}
\end{align*}
since $\varphi^{\lambda}(z,0)=0$. Hence
\begin{align*}
\bar{\partial} \xi_j=\sum_{\alpha=1}^n \bar{\partial} \dot{\xi}_j^{\alpha}\frac{\partial}{\partial \xi_j^{\alpha}}=\sum_{\lambda=1}^n\sum_{\alpha=1}^n \dot{\varphi}^{\lambda} \frac{\partial \xi_j^{\alpha}(z,0)}{\partial z_{\lambda}}\frac{\partial}{\partial \xi_j^{\alpha}}=\sum_{\lambda=1}^n \dot{\varphi}^{\lambda} \frac{\partial}{\partial z_{\lambda}}=\dot{\varphi}
\end{align*}
For $(*)$, we note that
\begin{align*}
\sum_{r,s=1} \dot{f_{rs}}\frac{\partial}{\partial z_r}\wedge \frac{\partial}{\partial z_s}=\sum_{r,s,a,b=1} \dot{f_{rs}} \frac{\partial \xi_j^a (z,0)}{\partial z_r}\frac{\partial \xi_j^b(z,0)}{\partial z_s}\frac{\partial}{\partial \xi_j^a}\wedge \frac{\partial}{\partial \xi_j^b}\\
\end{align*}
\begin{align*}
&\sum_{r,s,c=1}^n [g_{rs}^j(\xi_j,0) \frac{\partial}{\partial \xi_j^r}\wedge \frac{\partial}{\partial \xi_j^s}, \dot{\xi}_j^c\frac{\partial}{\partial \xi_j^c}]=\sum_{r,s,c=1} [g_{rs}^j(\xi_j,0)\frac{\partial}{\partial \xi_j^r},\dot{\xi}_j^c \frac{\partial}{\partial \xi_j^c}]\wedge \frac{\partial}{\partial \xi_j^s}-g_{rs}^j(\xi_j,0)[\frac{\partial}{\partial \xi_j^s},\dot{\xi}_j^c \frac{\partial}{\partial \xi_j^c}]\wedge\frac{\partial}{\partial \xi_j^r}\\
&=\sum_{r,s,c=1}^n g_{rs}^j(\xi_j,0)\frac{\partial \dot{\xi}_j^c}{\partial \xi_j^r}\frac{\partial}{\partial \xi_j^c}\wedge \frac{\partial}{\partial \xi_j^s}-\dot{\xi}_j^c\frac{\partial g_{rs}(\xi_j,0)}{\partial \xi_j^c}\frac{\partial}{\partial \xi_j^r}\wedge \frac{\partial}{\partial \xi_j^s}+g_{rs}^j(\xi_j,0)\frac{\partial \dot{\xi}_j^c}{\partial \xi_j^s}\frac{\partial}{\partial \xi_j^r}\wedge\frac{\partial}{\partial \xi_j^c}
\end{align*}
By considering the coefficients of $\frac{\partial}{\partial \xi_j^a}\wedge \frac{\partial}{\partial \xi_j^b}$, $(*)$ is equivalent to
\begin{align*}
(**) \sum_{r,s=1}^n \dot{f_{rs}} \frac{\partial \xi_j^a (z,0)}{\partial z_r}\frac{\partial \xi_j^b(z,0)}{\partial z_s}-\dot{g}_{ab}^j-\sum_{c=1}^n \dot{\xi}_j^c\frac{\partial g_{ab}(\xi_j,0)}{\partial \xi_j^c}+\sum_{c=1}^n g_{cb}^j(\xi_j,0)\frac{\partial \dot{\xi}_j^a}{\partial \xi_j^c}+g_{ac}^j(\xi_j,0)\frac{\partial \dot{\xi}_j^b}{\partial \xi_j^c}=0
\end{align*}
On the other hand, we have
\begin{align*}
g_{ab}^j(\xi_j^1(z,t),...,\xi_j^n(z,t),t_1,...,t_m)=\sum_{r,s=1}^n f_{rs}(z,t)\frac{\partial \xi_j^a(z,t)}{\partial z_r}\frac{\partial \xi_j^b(z,t)}{\partial z_s}
\end{align*}
By taking the derivative with respect to $t$ and putting $t=0$, we have
\begin{align*}
\frac{\partial g_{ab}^j(\xi_j,0)}{\partial \xi_j^1}\dot{\xi}_j^1+\cdots + \frac{\partial g_{ab}^j(\xi_j,0)}{\partial \xi_j^n}\dot{\xi}_j^n+\dot{g}_{ab}^j=\sum_{r,s=1}^n\dot{f}_{rs}\frac{\partial \xi_j^a(z,0)}{\partial z_r}\frac{\partial \xi_j^b(z,0)}{\partial z_s}+f_{rs}(z,0)( \frac{\partial \dot{\xi}_j^a}{\partial z_r}\frac{\partial \xi_j^b(z,0)}{\partial z_s}+\frac{\partial \xi_j^a(z,0)}{\partial z_r}\frac{\partial \dot{\xi}_j^b}{\partial z_s})
\end{align*}
Hence $(**)$ is equivalent to
\begin{align*}
\sum_{c=1}^n g_{cb}^j(\xi_j,0)\frac{\partial \dot{\xi}_j^a}{\partial \xi_j^c}+g_{ac}^j(\xi_j,0)\frac{\partial \dot{\xi}_j^b}{\partial \xi_j^c}=\sum_{r,s=1}^n f_{rs}(z,0)\frac{\partial \dot{\xi}_j^a}{\partial z_r}\frac{\partial \xi_j^b(z,0)}{\partial z_s}+ f_{rs}(z,0)\frac{\partial \xi_j^a(z,0)}{\partial z_r}\frac{\partial \dot{\xi}_j^b}{\partial z_s}
\end{align*}
Indeed,
\begin{align*}
\sum_{c=1}^n g_{cb}^j(\xi_j,0)\frac{\partial \dot{\xi}_j^a}{\partial \xi_j^c}+g_{ac}^j(\xi_j,0)\frac{\partial \dot{\xi}_j^b}{\partial \xi_j^c}&=\sum_{r,s,c=1}^n f_{rs}(z,0)\frac{\partial \xi_j^c(z,0)}{\partial z_r}\frac{\partial \xi_j^b(z,0)}{\partial z_s}\frac{\partial \dot{\xi}_j^a}{\partial \xi_j^c}+ f_{rs}(z,0)\frac{\partial \xi_j^a(z,0)}{\partial z_r}\frac{\partial \xi_j^c(z,0)}{\partial z_s}\frac{\partial \dot{\xi}_j^b}{\partial \xi_j^c}\\
&=\sum_{r,s=1}^n f_{rs}(z,0)\frac{\partial \dot{\xi}_j^a}{\partial z_r}\frac{\partial \xi_j^b(z,0)}{\partial z_s}+ f_{rs}(z,0)\frac{\partial \xi_j^a(z,0)}{\partial z_r}\frac{\partial \dot{\xi}_j^b}{\partial z_s}
\end{align*}
\end{proof}
Going back to our proof of Theorem \ref{n}, we defined the isomorphism between two hypercohomology groups $(b,a)\mapsto ([\Lambda,c]-b)$ where $-\delta c=a$ in the discussion above the Theorem \ref{n}. We take $(b,a)=(\{\sigma_j\},\{\theta_{jk}\})$ and $c=\{\xi_j\}$. Since $-\delta \{\xi_j\}=\{\theta_{jk}\}$, we have $-\delta c=a$. Then by the isomorphism $(\{\sigma_j\},\{\theta_{jk}\})$ is mapped to $([\Lambda,\{\xi_j\}]-\{\sigma_j\}, \bar{\partial}\{\xi_j\})$ which is $(-\dot{\Lambda}, \dot{\varphi})$ by Lemma \ref{beta}.
\end{proof}

\section{Integrability condition}\

We showed that given a Poisson analytic family $(\mathcal{M},\Lambda,B,\omega)$ deformation $(M_t,\Lambda_t)$ of $M$ near $(M_0,\Lambda_0)$ is represented by the vector $(0,1)$-form $\varphi(t)$ and the bivector field $\Lambda(t)$ on $M$ with $\varphi(0)=0$ and $\Lambda(0)=\Lambda_0$ satisfying the conditions: (1)$[\Lambda(t),\Lambda(t)]=0,(2)\bar{\partial} \Lambda(t)-[\Lambda(t),\varphi(t)]=0$ and (3)$\bar{\partial} \varphi(t)-\frac{1}{2}[\varphi(t),\varphi(t)]=0$ 

Conversely, we show that on the holomorphic Poisson manifold $(M,\Lambda_0)$, a vector $(0,1)$-form $\varphi$ and a bivector field $\Lambda$ on $M$ such that $\varphi$ and $\Lambda_0+\Lambda$ satisfying the interability condition define another holomorphic Poisson structure on $M$.

Let $\varphi=\sum_{\lambda=1}^n \varphi^{\lambda}_{\bar{v}}(z)d\bar{z}^v\frac{\partial}{\partial z_{\lambda}}$ be a $C^{\infty}$ vector $(0,1)$-form  and $\Lambda$ be a $C^{\infty}$ bivector field on a holomorphic Poisson manifold $(M,\Lambda_0)$ and suppose $det(\delta_v^{\lambda}-\sum_{\mu} \varphi_{v}^{\mu}\overline{\varphi_{\mu}^{\lambda}})\ne 0$. We assume that $\varphi$ and $\Lambda$ satisfies the integrability condition\footnote{If we replace $\varphi$ by $-\varphi$, then $(1),(2)$, and $(3)$ are equivalent to
\begin{align*}
L(\Lambda+\varphi)+\frac{1}{2}[\Lambda+\varphi,\Lambda+\varphi]=0 \,\,\,\,\text{where}\,\,L=\bar{\partial}+[\Lambda_0,-]
\end{align*}
which is a solution of the Maurer-Cartan equation of a differential graded Lie algebra $(\mathfrak{g}=\bigoplus_{i\geq 0} g^i=\bigoplus_{p+q-1=i, q \geq1} A^{0,p}(M,\wedge^q T_M),L,[-,-])$(See Appendix \ref{appendixc}). In the part \ref{part2} of the thesis, we prove that  this differential graded Lie algebra controls deformations of a holomorphic Poisson manifold $(M,\Lambda_0)$ in the language of functor of Artin rings.}

\begin{align*}
&(1)[\Lambda_0+\Lambda,\Lambda_0+\Lambda]=0\\
&(2)\bar{\partial} (\Lambda_0+\Lambda)-[\Lambda_0+\Lambda,\varphi]=0\\
&(3)\bar{\partial}\varphi-\frac{1}{2}[\varphi,\varphi]=0
\end{align*}

Then by the Newlander-Nirenberg theorem($\cite{New57}$,$\cite{Kod05}$), the condition (3) gives a finite open covering $\{U_j\}$ of $M$  and $C^{\infty}$-functions $\xi_j^{\alpha}=\xi_j^{\alpha}(z),\alpha=1,...,n$ on each $U_j$ such that $\xi_j:z\to \xi_j(z)=(\xi_j^1(z),...,\xi_j^n(z))$ gives complex coordinates on $U_j$ and $\{\xi_1,...,\xi_j,...\}$ defines another complex structure on $M$,which we denote by $M_{\varphi}$. And by Theorem \ref{m}, the condition $(1)$ and $(2)$ gives another holomorphic Poisson structure $(\Lambda_0+\Lambda)^{2,0}$ on $M$ with respect to the complex structure induced by $\varphi$ where $(\Lambda_0+\Lambda)^{2,0}$ means the $(2,0)$-part of $\Lambda_0+\Lambda$ with respect to the complex structure induced by $\varphi$.

\begin{example}[Hitchin-Goto family] 
Hitchin showed the following theorem.

\begin{thm}
Let $(M,\sigma)$ be a holomorphic Poisson manifold which satisfies the $\partial\bar{\partial}$-lemma. Then any class $\sigma([\omega])\in H^1(M,T)$ for $[\omega]\in H^1(M,T^*)$ is tangent to a deformation of complex structure induced by  $\phi(t)=\sigma(\alpha)$ where $\alpha=t\omega+\partial (t^2\beta_2+t^3\beta_3+\cdots)$ for $(0,1)$-forms $\beta_i$ with respect to the original complex structure.
\end{thm}
\begin{proof}
See $\cite{Hit12}$ theorem 1.
\end{proof}

And also Hitchin showed for each $\phi(t)$, there is a holomorphic Poisson structure $\sigma_t$ with respect to $M_{\phi(t)}$. We construct a Poisson analytic family $(\mathcal{M},\Lambda)$ such that $(M_t,\Lambda_t)=(M_{\phi(t)},\sigma_t)$ by showing that $(\phi(t),\sigma)$ satisfies the integrability condition.

\begin{lemma}\label{ab}
Let $\sigma \in C^{\infty}(\wedge^ 2T_M)$ be a bivector of a complex manifold $M$ such that $[\sigma,\sigma]=0$. Then we have $[\sigma,\sigma(\partial \beta)]=0$ where $\beta$ is an $(0,1)$-form.

\end{lemma}

\begin{proof}
If we write $\sigma=\sum_{l,k=1}^n \sigma^{lk}\frac{\partial}{\partial z_l}\wedge \frac{\partial}{\partial z_k}$ and $\beta= \sum_{i=1}^n f_id\bar{z_i}$. Then $\partial \beta=\sum_{i,j=1}^n \frac{\partial f_i}{\partial z_j} dz_j\wedge d\bar{z_i}$.  Then $\sigma(\partial \beta)=\sum_{i,l,k} \sigma^{lk}\frac{\partial f_i}{\partial z_l}\frac{\partial}{\partial z_k}d\bar{z_i}-\sigma^{lk}\frac{\partial f_i}{\partial z_k}\frac{\partial}{\partial z_l}d\bar{z_i}=\sum_{i,l,k} 2\sigma^{lk}\frac{\partial f_i}{\partial z_l}\frac{\partial}{\partial z_k}d\bar{z_i}$. So it is sufficient to show that 
\begin{align*}
\sum_{p,q,l,k} [\sigma^{pq}\frac{\partial}{\partial z_p}\wedge \frac{\partial}{\partial z_q},\sigma^{lk}\frac{\partial f_i}{\partial z_l}\frac{\partial}{\partial z_k}]=0
\end{align*}

$\sum_{p,q,l,k} [\sigma^{pq}\frac{\partial}{\partial z_p}\wedge \frac{\partial}{\partial z_q},\sigma^{lk}\frac{\partial f_i}{\partial z_l}\frac{\partial}{\partial z_k}]=\sum_{p,q,l,k} [\sigma^{pq}\frac{\partial}{\partial z_p},\sigma^{lk}\frac{\partial f_i}{\partial z_l}\frac{\partial}{\partial z_k}]\frac{\partial}{\partial z_q}-\sigma^{pq}[\frac{\partial}{\partial z_q},\sigma^{lk}\frac{\partial f_i}{\partial z_l}\frac{\partial}{\partial z_k}]\frac{\partial}{\partial z_p}\\=\sum_{p,q,l,k} \sigma^{pq}\frac{\partial \sigma^{lk}}{\partial z_p}\frac{\partial f_i}{\partial z_l}\frac{\partial}{\partial z_k}\wedge \frac{\partial}{\partial z_q}+\sigma^{pq}\sigma^{lk}\frac{\partial^2 f_{i}}{\partial z_p \partial z_l}\frac{\partial}{\partial z_k}\wedge \frac{\partial }{\partial z_q}-\sigma^{lk}\frac{\partial f_i}{\partial z_l}\frac{\partial \sigma^{pq}}{\partial z_k}\frac{\partial}{\partial z_p}\wedge \frac{\partial}{\partial z_q}-\sigma^{pq}\frac{\partial \sigma^{lk}}{\partial z_q}\frac{\partial f_i}{\partial z_l}\frac{\partial}{\partial z_k}\wedge \frac{\partial}{\partial z_p}-\sigma^{pq}\sigma^{lk}\frac{\partial^2 f_{i}}{\partial z_q\partial z_l}\frac{\partial}{\partial z_k}\wedge \frac{\partial}{\partial z_p}$.

Let's consider $\sum_{p,q,l,k} \sigma^{pq}\sigma^{lk}\frac{\partial^2 f_{i}}{\partial z_p \partial z_l}\frac{\partial}{\partial z_k}\wedge \frac{\partial }{\partial z_q}-\sigma^{pq}\sigma^{lk}\frac{\partial^2 f_{i}}{\partial z_q\partial z_l}\frac{\partial}{\partial z_k}\wedge \frac{\partial}{\partial z_p}=2\sum_{p,q,l,k} \sigma^{pq}\sigma^{lk}\frac{\partial^2 f_{i}}{\partial z_p \partial z_l}\frac{\partial}{\partial z_k}\wedge \frac{\partial }{\partial z_q}$. By considering the coefficient of $\frac{\partial}{\partial z_a}\wedge \frac{\partial}{\partial z_b}$ of $\sum_{p,q,l,k} \sigma^{pq}\sigma^{lk}\frac{\partial^2 f_{i}}{\partial z_p \partial z_l}\frac{\partial}{\partial z_k}\wedge \frac{\partial }{\partial z_q}$, we have $\sum_{p,l} \sigma^{pb}\sigma^{la}\frac{\partial^2 f_{i}}{\partial z_p \partial z_l}\frac{\partial}{\partial z_a}\wedge \frac{\partial }{\partial z_b}-\sum_{p,l} \sigma^{pa}\sigma^{lb}\frac{\partial^2 f_{i}}{\partial z_p \partial z_l}\frac{\partial}{\partial z_a}\wedge \frac{\partial }{\partial z_b}=\sum_{p,l} \sigma^{pb}\sigma^{la}\frac{\partial^2 f_{i}}{\partial z_p \partial z_l}\frac{\partial}{\partial z_a}\wedge \frac{\partial }{\partial z_b}-\sum_{p,l} \sigma^{la}\sigma^{pb}\frac{\partial^2 f_{i}}{\partial z_l \partial z_p}\frac{\partial}{\partial z_a}\wedge \frac{\partial }{\partial z_b}=0$

So we have $\sum_{p,q,l,k} [\sigma^{pq}\frac{\partial}{\partial z_p}\wedge \frac{\partial}{\partial z_q},\sigma^{lk}\frac{\partial f_i}{\partial z_l}\frac{\partial}{\partial z_k}]=\sum_{p,q,l,k} \sigma^{pq}\frac{\partial \sigma^{lk}}{\partial z_p}\frac{\partial f_i}{\partial z_l}\frac{\partial}{\partial z_k}\wedge \frac{\partial}{\partial z_q}-\sigma^{lk}\frac{\partial f_i}{\partial z_l}\frac{\partial \sigma^{pq}}{\partial z_k}\frac{\partial}{\partial z_p}\wedge \frac{\partial}{\partial z_q}-\sigma^{pq}\frac{\partial \sigma^{lk}}{\partial z_q}\frac{\partial f_i}{\partial z_l}\frac{\partial}{\partial z_k}\wedge \frac{\partial}{\partial z_p}=\sum_{q,k,l}\left( \sum_p \sigma^{pq}\frac{\partial \sigma^{lk}}{\partial z_p}\frac{\partial f_i}{\partial z_l}\frac{\partial}{\partial z_k}\wedge \frac{\partial}{\partial z_q}-\sigma^{lp}\frac{\partial f_i}{\partial z_l}\frac{\partial \sigma^{kq}}{\partial z_p}\frac{\partial}{\partial z_k}\wedge \frac{\partial}{\partial z_q}-\sigma^{qp}\frac{\partial \sigma^{lk}}{\partial z_p}\frac{\partial f_i}{\partial z_l}\frac{\partial}{\partial z_k}\wedge \frac{\partial}{\partial z_q}\right)=\sum_{q,k,l}\left( \sum_p \sigma^{pq}\frac{\partial \sigma^{lk}}{\partial z_p}\frac{\partial f_i}{\partial z_l}\frac{\partial}{\partial z_k}\wedge \frac{\partial}{\partial z_q}+\sigma^{pl}\frac{\partial f_i}{\partial z_l}\frac{\partial \sigma^{kq}}{\partial z_p}\frac{\partial}{\partial z_k}\wedge \frac{\partial}{\partial z_q}+\sigma^{kp}\frac{\partial \sigma^{lq}}{\partial z_p}\frac{\partial f_i}{\partial z_l}\frac{\partial}{\partial z_k}\wedge \frac{\partial}{\partial z_q}\right)=0 $ by Lemma \ref{formula}.
\end{proof}

Now we assume that $\phi(t)=\sigma(\alpha)$ constructed by Hitchin converges for $t\in \Delta \subset \mathbb{C}$. And we know that $\psi(t)=-\phi(t)$ satisfies $\bar{\partial}\psi(t)-\frac{1}{2}[\psi(t),\psi(t)]=0$. We can consider $\phi:=\phi(t)$ as a $C^{\infty} (0,1)$-vector on $M\times \Delta$. Then by Newlander-Nirenberg theorem$($$\cite{New57}$,$\cite{Kod05}$ p.268$)$, we can give a holomorphic coordinate on $M\times \Delta$ induced by $\phi$ (more precisely $-\phi$) . Let's denote the complex manifold induced by $\phi$ by $\mathcal{M}$. Then $\omega:\mathcal{M}\to \Delta$ is a family of compact complex manifolds. If we choosee a sufficiently fine locally finite open covering $\{U_j\}$ of $M$, we have $n+1$ holomorphic coordinates $\xi_j^{\beta}(z,t),\beta=1,...,n$ and $t$ on each $U_j\times \Delta$, and the map
\begin{equation*}
\xi_j:(z,t)\to(\xi_j^1(z,t),...,\xi_j^{n}(z,t),t)
\end{equation*}
gives local complex coordinates of $\mathcal{M}$ on $U_j\times \Delta$.

And we can think of $\sigma$  on $M$ as a $C^{\infty}$ bivector on $M\times \Delta$. $($more precisely, $\sigma\oplus 0$ $)$. We note that a bivector $(2,0)$ part $\Lambda$ of $\sigma\in C^{\infty}(\wedge^2 T_M)$ on $M\times \Delta$ is holomorphic with respect to the complex structure $\mathcal{M}$, in other words with respect to coordinates systems $\xi_j(z,t)$, if and only if it satisfies
\begin{align*}
\bar{\partial} \sigma +[\sigma,\phi(t)]=0
\end{align*}

Since $\sigma$ satisfies the equation by Lemma \ref{ab}, $\Lambda$ is a holomrphic bivector field on $\mathcal{M}$ and $\Lambda_t$ induces Poisson holomorphic structure on $M_{\phi(t)}$ for each $t$. If we write $\sigma=\sum_{\alpha,\beta=1}^n  \sigma^{\alpha\beta}(z_j)\frac{\partial}{\partial z_j^{\alpha}} \wedge \frac{\partial}{\partial z_j^{\beta}}$ on $U_j\times \Delta$, then in new complex coordinate systems, it becomes $\Lambda=\sum_{p,q,\alpha,\beta=1}^n \sigma^{\alpha\beta}(z_j)\frac{\partial \xi_j^p}{\partial z_j^{\alpha}}\frac{\partial \xi_j^q}{\partial z_j^q}\frac{\partial}{\partial \xi_j^p}\wedge \frac{\partial}{\partial \xi_j^q}$. So we have a Poisson analytic family $(\mathcal{M},\Lambda)$. For each $t$, we have 
\begin{align*}
\Lambda_t=\sum_{p,q,\alpha,\beta=1}^n \sigma^{\alpha\beta}(z_j)\frac{\partial \xi_j^p(z_j,t)}{\partial z_j^{\alpha}}\frac{\partial \xi_j^q(z_j,t)}{\partial z_j^q}\frac{\partial}{\partial \xi_j^p}\wedge \frac{\partial}{\partial \xi_j^q}
\end{align*}

On the other hand, Hitchin defines a holomorphic Poisson structure on $M_{\phi(t)}$ in the following way:
\begin{align*}
\sigma_t(f,g)=\sigma(\partial f,\partial g)
\end{align*}
for $f,g$ local holomorphic functions with respect to the complex structure at $t$. So we have $\sigma_t(\xi_j^p(z,t),\xi_j^q(z,t))=\sum_{\alpha,\beta=1}^n \sigma^{\alpha\beta}(z_j)\frac{\partial \xi_j^p(z_j,t)}{\partial z_j^{\alpha}}\frac{\partial \xi_j^q(z_j,t)}{\partial z_j^q}$. This implies that $(\mathcal{M}_t,\Lambda_t)=(M_{\phi(t)},\sigma_t)$. And since $\Lambda$ does not depend on $t$, we have 0 in the Poisson direction under the Poisson Kodaira Spencer map $\varphi_0: T_0\Delta\to HP^2(M,\sigma)$ by Theorem \ref{n}. More precisely $\varphi_0(\frac{\partial}{\partial t})=(\sigma([\omega]),0)$.

\end{example}

\chapter{Theorem of Existence for holomorphic Poisson structures}\label{chapter3}

\section{Theorem of existence}
In this chapter, we prove theorem of existence for holomorphic Poisson deformations under the assumption (\ref{assumption}) as an analogue of theorem of existence for deformations of complex structure.
\subsection{Statement of the theorem}\

Before we state the theorem of existence, we discuss the assumption (\ref{assumption}) \footnote{For my unfamiliarity of analysis, I do not know that we can relax the assumption} that we use in the proof of theorem of existence. Let $(M,\Lambda)$ be compact holomorphic Poisson manifold. First we note that the differential operator $L=\bar{\partial} +[\Lambda,-]$ is elliptic and so we have operators $L^*, H, G$ and $\Box=LL^*+L^*L$.(See Appendix \ref{appendixd})

we  introduce the H\"{o}lder norms in the spaces $A^p=A^{0,p-1}(M,T)\oplus \cdots \oplus A^{0,0}(M, \wedge^p T)$. To do this, we fix a finite open covering $\{U_j\}$ of $M$ such that $(z_j)$ are coordinates  on $U_j$. Let $\varphi\in A^p$,
\begin{align*}
\varphi=\sum_{r+s=p, s\geq 1} \varphi_{j \alpha_1\cdots\alpha_r\beta_1\cdots\beta_s}(z)d\bar{z}_j^{\alpha_1}\wedge \cdots \wedge d\bar{z}_j^{\alpha_r}\wedge \frac{\partial}{\partial z_j^{\beta_1}}\wedge\cdots \wedge\frac{\partial}{\partial z_j^{\beta_s}}
\end{align*}
Let $k\in \mathbb{Z},k\geq 0,\alpha\in \mathbb{R},0<\alpha<1$. Let $h=(h_1,...,h_{2n}),h_i\geq 0,\sum_{i=1}^{2n} h_i=|h|$ where $n=\dim\, M$. Then denote
\begin{align*}
D_j^h=\left(\frac{\partial}{\partial x_j^1}\right)^{h_1}\cdots \left(\frac{\partial}{\partial x_j^{2n}}\right)^{h_{2n}},\,\,\,\,\,z_j^{\alpha}=x_j^{2\alpha-1}+ix_j^{2\alpha}
\end{align*}

Then the H\"{o}lder norm $||\varphi||_{k+\alpha}$ is defined as follows:
\begin{align*}
||\varphi||_{k+\alpha}=\max_j \{ \sum_{h, |h|\leq k}\left( \sup_{z\in U_j}|D_j^h \varphi_{j \alpha_1\cdots\alpha_r\beta_1\cdots\beta_s}(z)|\right)+\sup_{y,z\in U_j,|h|=k} 
\frac{|D_j^h \varphi_{j \alpha_1\cdots\alpha_r\beta_1\cdots\beta_s}(y)-D_j^h \varphi_{j \alpha_1\cdots\alpha_r\beta_1\cdots\beta_s}(z)|}{|y-z|^{\alpha}} \},
\end{align*}
where the sup is over all $\alpha_1,...,\alpha_r,\beta_1,...,\beta_s$

Our assumption is the following. For any $\varphi\in A^2$,
\begin{equation}\label{assumption}
||\varphi||_{k+\alpha}\leq C(|| \Box \varphi||_{k-2+\alpha}+||\varphi||_0)
\end{equation}
where $k\geq 2$, $C$ is a constant which is independent of $\varphi$
\begin{center}
and $||\varphi||_0=\max_{j,\alpha_1,...,\beta_s} \sup_{z\in U_j} |\varphi_{j \alpha_1\cdots\alpha_r\beta_1\cdots\beta_s}(z)|$.
\end{center}
Now we state the theorem of existence of holomorphic Poisson deformations.
\begin{thm}[Theorem of Existence]\label{theorem of existence}
Let $(M,\Lambda_0)$ be a compact holomorphic Poisson manifold satisfying $($\ref{assumption}$)$ and suppose that $HP^3(M,\Lambda_0)=0$. Then there exists a Poisson analytic family $(
\mathcal{M},\Lambda,B,\omega)$ with $0\in B\subset \mathbb{C}^m$ satisfying the following conditions:
\begin{enumerate}
\item $\omega^{-1}(0)=(M,\Lambda_0)$
\item $\varphi_0:\frac{\partial}{\partial t}\to \left(\frac{\partial (M_t,\Lambda_t)}{\partial t}\right)_{t=0}$ with $(M_t,\Lambda_t)=\omega^{-1}(t)$ is an isomorphism of $T_0(B)$ onto $HP^2(M,\Lambda_0):T_0 B\xrightarrow{\rho_0} HP^2(M,\Lambda_0)$.
\end{enumerate}
\end{thm}

Assume that there is a Poisson analytic family $(\mathcal{M},\Lambda,B,\omega)$ satisfying $(1)$ and $(2)$. Take a sufficiently small $\Delta$ with $0\in \Delta \subset B$, and the $C^{\infty}$ vector $(0,1)$-form $\varphi(t)=\sum_{\lambda=1}^{n} \varphi^{\lambda}(z,t)\frac{\partial}{\partial z_{\lambda}}$ and $\Lambda(t)$ on $M$ defined by (\ref{b}) and (\ref{f}) defined by $(\mathcal{M}_{\Delta},\Lambda_{\Delta},\Delta, \omega)$. Then we have $[\Lambda(t),\Lambda(t)]=0,\bar{\partial}\Lambda(t)-[\Lambda(t),\varphi(t)]=0,\bar{\partial}\varphi(t)=\frac{1}{2}[\varphi(t),\varphi(t)]$ and we have $\varphi(0)=0,\Lambda(0)=\Lambda_0$. If we put
\begin{align*}
(\dot{\varphi}_{\lambda},-\dot{\Lambda}_{\lambda})=(\left(\frac{\partial \varphi(t)}{\partial t_{\lambda}}\right)_{t=0}, -\left(\frac{\partial \Lambda(t)}{\partial t_{\lambda}}\right)_{t=0}),\,\,\,\,\, \lambda=1,...,m,
\end{align*}
Then $\{(\dot{\varphi}_{1},-\dot{\Lambda}_{1}),...,(\dot{\varphi}_{m},-\dot{\Lambda}_m)\}$ forms a basis of $HP^2(M,\Lambda_0)$.

Conversely, given $(\beta_{\lambda},\pi_{\lambda})\in A^{0,1}(M,T)\oplus A^{2,0}(M,\wedge^2 T)$ for $\lambda=1,...,m$, such that $\{(\eta_{1},\pi_{1}),...,(\eta_{m},\pi_{m})\}$ forms a basis of $HP^2(M,\Lambda_0)$, assume that there is a family $\{(\varphi(t),\Lambda(t))|t\in \Delta \}$ of $C^{\infty}$ vector $(0,1)$-forms $\varphi(t)$ and $(2,0)$ vector $\Lambda(t)$ on $M$ with $0\in \Delta \subset \mathbb{C}^m$, which satisfy 
\begin{enumerate}
\item $[\Lambda(t),\Lambda(t)]=0$\\
\item $\bar{\partial} \Lambda(t)-[\Lambda(t),\varphi(t)]=0$\\
\item $\bar{\partial}\varphi(t)=\frac{1}{2}[\varphi(t),\varphi(t)]$
\end{enumerate}

and the initial conditions
\begin{align*}
\varphi(0)=0, \Lambda(0)=\Lambda_0,(\left(\frac{\partial \varphi(t)}{\partial t_{\lambda}}\right)_{t=0}, -\left(\frac{\partial \Lambda(t)}{\partial t_{\lambda}}\right)_{t=0})=(\eta_{\lambda},\pi_{\lambda}),\,\,\,\,\, \lambda=1,...,m,
\end{align*}

Since $\Delta$ is assumed to be sufficiently small, we may assume that $\varphi(t)=\sum_{\lambda}\sum_{v} \varphi^{\lambda}_v d\bar{z}_v\frac{\partial}{\partial z_{\lambda}}$ satisfies

\begin{align*}
det(\delta^{\lambda}_v-\sum_{\mu}\varphi^{\mu}_v(t)\overline{\varphi^{\lambda}_{\mu}(t)})\ne 0.
\end{align*}

Therefore, by the Newlander-Nirenberg theorem($\cite{New57}$,$\cite{Kod05}$ p.268), each $\varphi(t)$ determines a complex structure $M_{\varphi(t)}$ on $M$. And condition $(2)$,$(3)$ implies $(2,0)$-part $\Lambda(t)^{2,0}$ of $\Lambda(t)$ is a holomorphic Poisson structure on $M_{\varphi(t)}$. If the family $\{M_{\varphi(t)},\Lambda(t)^{2,0}\}$  is a Poisson analytic family, it satisfies the conditions (1) and (2) in Theorem \ref{theorem of existence} by our assumption and Theorem \ref{n}. we construct such a family $\{(\varphi(t),\Lambda(t))|t\in \Delta \}$ and then show that  $\{(M_{\varphi(t)},\Lambda(t)^{2,0})\}$ is a Poisson analytic family.(See \ref{subsection})

\begin{remark}
Constructing $\{\varphi(t),\Lambda(t)\}$ is equivalent to constructing $\{-\varphi(t),\Lambda(t)\}$. By replacing $\varphi(t)$ by $-\varphi(t)$, we construct $\varphi(t)$ satisfying
\begin{enumerate}
\item $[\Lambda(t),\Lambda(t)]=0$\\
\item $\bar{\partial} \Lambda(t)+[\Lambda(t),\varphi(t)]=0$\\
\item $\bar{\partial}\varphi(t)+\frac{1}{2}[\varphi(t),\varphi(t)]=0$
\end{enumerate}

and the initial conditions
\begin{align*}
\varphi(0)=0, \Lambda(0)=\Lambda_0,(-\left(\frac{\partial \varphi(t)}{\partial t_{\lambda}}\right)_{t=0}, -\left(\frac{\partial \Lambda(t)}{\partial t_{\lambda}}\right)_{t=0})=(-\eta_{\lambda},\pi_{\lambda}),\,\,\,\,\, \lambda=1,...,m,
\end{align*}
And we note that $(1),(2),(3)$ are equivalent to 

\begin{center}
$\bar{\partial} (\varphi(t)+\Lambda(t))+\frac{1}{2}[\varphi(t)+\Lambda(t),\varphi(t)+\Lambda(t)]=0$ 
\end{center}
\end{remark}

We construct $\alpha(t)=\varphi(t)+\Lambda(t)$ in the following section.

\subsection{Construction of $\alpha(t)=$$\varphi(t)+\Lambda(t)$}\
We use the Kuranishi methods. We need the following lemmas.

\begin{lemma}\label{lemma5.2.2}
For $\varphi,\psi\in A^2$, we have $||[\varphi,\psi]||_{k+\alpha}\leq C||\varphi||_{k+1+\alpha}||\psi||_{k+1+\alpha}$, where $C$ is independent of $\varphi$ and $\psi$.
\end{lemma}

\begin{lemma}\label{lemma5.2.3}
For $\varphi\in A^2$, we have $||G\varphi||_{k+\alpha}\leq C||\varphi||_{k-2+\alpha},k\geq 2$, where $C$ depends only on $k$ and $\alpha$, not on $\varphi$.
\end{lemma}

\begin{proof}
This follows from the assumption \ref{assumption}. See \cite{Mor71} p.160 Proposition 2.3.
\end{proof}

Let us now construct the $\varphi(t)$ and $\Lambda(t)$. We want to construct $\alpha(t):=\varphi(t)+\Lambda(t)=\Lambda_0+\sum_{\mu=1}^{\infty} \varphi_{\mu}(t)+\Lambda_{\mu}(t)$, where
\begin{align*}
\varphi_{\mu}(t)+\Lambda_{\mu}(t)=\sum_{v_1+\cdots+v_m=\mu} (\varphi_{v_1\cdots v_m}+\Lambda_{v_1\cdots v_m})t_1^{v_1}\cdots t_m^{v_m}
\end{align*}
where $\varphi_{v_1\cdots v_m}+\Lambda_{v_1\cdots v_m}\in A^{0,1}(M,T)\oplus A^{0,0}(M,\wedge^2  T)$ such that

\begin{align}
&\bar{\partial} \alpha(t)+\frac{1}{2}[\alpha(t),\alpha(t)]=0\\
&\alpha_1(t)=\varphi_1(t)+\Lambda_1(t)=\sum_{v=1}^m (\eta_v+\pi_v)t_v,
\end{align}
where $\{\eta_v+\pi_v\}$ is a basis for $\mathbb{H}^2\cong HP^2(M,\Lambda_0)$. 

Let $\beta(t)=\alpha(t)-\Lambda_0$. Then $(5.2.3)$ is equivalent to
\begin{align*}
L\beta(t)+\frac{1}{2}[\beta(t),\beta(t)]=0.
\end{align*}

Consider the equation
\begin{align*}
(*) \beta(t)=\beta_1(t)-\frac{1}{2}L^*G[\beta(t),\beta(t)],
\end{align*}
where $\beta_1(t)=\alpha_1(t)$.
$(*)$ has a unique formal power series solution $\beta(t)$. Indeed,
\begin{align*}
&\beta_2(t)=-\frac{1}{2}L^*G[\beta_1(t),\beta_1(t)]\\
&\beta_3(t)=-\frac{1}{2}L^*G([\beta_1(t),\beta_2(t)]+[\beta_2(t),\beta_1(t)])\\
&\beta_{\mu}(t)=-\frac{1}{2}L^* G\left( \sum_{\lambda=1}^{\mu-1}[\beta_{\lambda}(t),\beta_{\mu-\lambda}(t)]\right)
\end{align*}

\begin{proposition}
For small $|t|$, $\beta(t)=\sum_{\mu=1}^{\infty} \beta_{\mu}(t)$ converges in the norm $||\cdot||_{k+\alpha}$.
\end{proposition}

\begin{proof}
See \cite{Mor71} p.162 Proposition 2.4.
\end{proof}

\begin{proposition}
The $\beta(t)$ satisfies $L\beta(t)+\frac{1}{2}[\beta(t),\beta(t)]=0$ if and only if $H[\beta(t),\beta(t)]=0$, where $H:A^3=A^{0,2}(M,T)\oplus A^{0,1}(M,\wedge^2 T)\oplus A^{0,0}(M,\wedge^3 T)\to \mathbb{H}^3\cong HP^3(M,\Lambda_0)$ is the orthogonal projection to the harmonic subspace of $A^3$.
\end{proposition}

\begin{proof}
$(=>)$ $L\beta(t)=-\frac{1}{2}[\beta(t),\beta(t)]$. If we take $H$ on both sides, we have
\begin{align*}
0=HL\beta(t)=-\frac{1}{2}H[\beta(t),\beta(t)]
\end{align*}
$(<=)$ Let $H[\beta(t),\beta(t)]=0$. 

Set $\psi(t)=L\beta(t)+\frac{1}{2}[\beta(t),\beta(t)]=\bar{\partial} \beta(t)+[\Lambda_0,\beta(t)]+\frac{1}{2}[\beta(t),\beta(t)]$
\begin{align*}
2\psi(t)&=2L\beta(t)+[\beta(t),\beta(t)]\\
           &=-LL^*G[\beta(t),\beta(t)]+[\beta(t),\beta(t)]\\
           &=-LL^*G[\beta(t),\beta(t)]+\Box G[\beta(t),\beta(t)]\\
           &=-LL^*G[\beta(t),\beta(t)]+(LL^*+L^*L)G[\beta(t),\beta(t)]\\
           &=L^*LG[\beta(t),\beta(t)]=L^*GL[\beta(t),\beta(t)]\\
           &=2L^*G[L\beta(t),\beta(t)]
\end{align*}
since $(A[1],L,[-,-])$ is a differential graded Lie algebra (See Proposition \ref{d}).
So we have $\psi(t)=L^*G[\psi(t),\beta(t)]$. And by Lemma \ref{lemma5.2.2} and Lemma \ref{lemma5.2.3}, we have
\begin{align*}
||\psi(t)||_{k+\alpha}&=||L^*G[\psi(t),\beta(t)]||_{k+\alpha}\\
                               &\leq C_1||G[\psi(t),\beta(t)]||_{k+1+\alpha}\\
                               & \leq C_1 C_{k,\alpha} ||[\psi(t),\beta(t)]||_{k-1+\alpha}\\
                               & \leq C_1C_{k,\alpha}C||\psi(t)||_{k+\alpha}||\beta(t)||_{k+\alpha}
\end{align*}
Choose $|t|$ so small that $||\beta(t)||_{k+\alpha}C_1C_{k,\alpha}C<1$. Then we get the contradiction $||\psi(t)||_{k+\alpha}<||\psi(t)||_{k+\alpha}$ unless $\psi(t)=0$ for all small $t$ since $\beta(t)$ converges and $\beta(0)=0$.
\end{proof}

\begin{proposition}\label{pr}
$\alpha(t)=\varphi(t)+\Lambda(t)$ is $C^{\infty}$ in $(z,t)$ and holomorphic in $t$.
\end{proposition}

\begin{proof}
See \cite{Mor71} p.163 Proposition 2.6.
\end{proof}

With the assumption $HP^3(M,\Lambda_0)=0$, $\beta(t)$ solves the integrabiliy condition for small $|t|$, where $t\in \Delta_{\epsilon}$.

\subsection{Construction of a Poisson analytic family}\label{subsection}\

 We have constructed a family $\{(\varphi(t),\Lambda(t))|t\in \triangle_{\epsilon}\}$ of $C^{\infty}$ vector $(0,1)$-forms $\varphi(t)=\sum_{\lambda=1}^n\sum_{v=1}^n \varphi_v^{\lambda}(z,t)d\bar{z}^v\frac{\partial}{\partial z^{\lambda}}$ and $C^{\infty}$ $(2,0)$ bivector $\Lambda(t)=\sum_{\alpha,\beta=1}^n g_{\alpha\beta}(z,t)\frac{\partial}{\partial z_{\alpha}}\wedge \frac{\partial}{\partial z_{\beta}}$ satisfying the integrability condition $[\Lambda(t),\Lambda(t)]=0,\bar{\partial}\Lambda(t)=[\Lambda(t),\varphi(t)],\bar{\partial}\varphi(t)=\frac{1}{2}[\varphi(t),\varphi(t)]$ and the initial conditions $\varphi(0)=0, \Lambda(0)=\Lambda_0, (\left(\frac{\partial \varphi(t)}{\partial t_{\lambda}}\right)_{t=0}, -\left(\frac{\partial \Lambda(t)}{\partial t_{\lambda}}\right)_{t=0})=(\beta_{\lambda},\pi_{\lambda}),\lambda=1,...,m$, where $\varphi_v^{\lambda}(z,t)$ and $g_{\alpha\beta}(z,t)$ are $C^{\infty}$ functions of $z^1,...,z^n,t_1,...,t_m$ and holomorphic in $t_1,...,t_m$.

 Each $(\varphi(t),\Lambda(t))$ determines a holomorphic Poisson structure $(M_{\varphi(t)},\Lambda(t))$ on $M$. In order to show that $\{(M_{\varphi(t)},\Lambda(t))|t\in \Delta_{\epsilon}\}$ is a Poisson analytic family, we consider $\varphi=\varphi(t)$ as a vector $(0,1)$-form on the complex manifold $M\times \Delta_{\epsilon}$  and $\Lambda=\Lambda(t)$ as a $(2,0)$ bivector on $M\times \Delta_{\epsilon}$. Namely, we consider $\varphi(t)$ as
 
\begin{align*}
\varphi = \varphi(t)= \sum_{\lambda=1}^n \left(\sum_{v=1}^n \varphi_v^{\lambda} d\bar{z}^v+\sum_{\mu=1}^m \varphi_{n+\mu}^{\lambda} d\bar{t}_{\mu} \right)\frac{\partial}{\partial z^{\lambda}}+\sum_{\mu=1}^m\varphi^{n+\mu}\frac{\partial}{\partial t_{\mu}}
\end{align*}
with $\varphi^{t+\mu}=\varphi_{n+\mu}^{\lambda}=0,\mu=1,...,m$. Then since $\varphi_v^{\lambda}=\varphi_v^{\lambda}(z,t)$ are holomorphic in $t_1,...,t_m$ (Proposition \ref{pr}), we have $\frac{\partial \varphi_v^{\lambda}}{\partial \bar{t}_{\mu}}=0$ in
\begin{align*}
\bar{\partial}\varphi=\sum_{\lambda,v=1}^n \left(\sum_{\beta=1}^n\frac{\partial \varphi_v^{\lambda}}{\partial \bar{z}^{\beta}}d\bar{z}^{\beta}+\sum_{\mu=1}^m \frac{\partial \varphi_v^{\lambda}}{\partial \bar{t}_{\mu}}d\bar{t}_{\mu}\right) \wedge d\bar{z}^v\frac{\partial}{\partial z^{\lambda}}
\end{align*}

Similary since $g_{\alpha\beta}(z,t)$ is holomorphic in $t_1,...,t_m$ (Proposition \ref{pr}), we have $\frac{\partial g_{\alpha\beta}}{\partial \bar{t}_{\mu}}=0$ in

\begin{align*}
\bar{\partial} \Lambda=\sum_{\alpha,\beta} \left(\sum_{v=1}^n \frac{\partial g_{\alpha\beta}}{\partial \bar{z}^v}d\bar{z}^v+\sum_{\mu=1}^m \frac{\partial g_{\alpha\beta}}{\partial \bar{t}_{\mu}}d\bar{t}_{\mu}\right)\frac{\partial}{\partial z^{\alpha}}\wedge \frac{\partial}{\partial z^{\beta}}
\end{align*}

By $\bar{\partial}\varphi(t)$ we denote the exterior differential of $\varphi(t)$ as a vector $(0,1)$-form on $M$ with fixed $t$. Then $\bar{\partial}\varphi$ coincides with $\bar{\partial}\varphi(t)$ and we obtain $[\varphi,\varphi]=[\varphi(t),\varphi(t)]$. Similary $\bar{\partial}\Lambda(t)$ coincides with $\bar{\partial} \Lambda$ and we obtain $[\Lambda,\varphi]=[\Lambda(t),\varphi(t)]$ and $[\Lambda,\Lambda]=[\Lambda(t),\Lambda(t)]$. Therefore as a $C^{\infty}$ vector $(0,1)$-form on $M\times \Delta_{\epsilon}$ and $C^{\infty}$ $(2,0)$ bivector on $M\times \Delta_{\epsilon}$, $\varphi$ and $\Lambda$ satisfies $\bar{\partial}\varphi=\frac{1}{2}[\varphi,\varphi]$, $ \bar{\partial}\Lambda=[\Lambda,\varphi]$, and $[\Lambda,\Lambda]=0$.

Then by the Newlander-Nirenberg theorem($\cite{New57}$,$\cite{Kod05}$ p.268), $\varphi$ defines a complex structure $\mathcal{M}$ on $M\times \Delta_{\epsilon}$ and $ \bar{\partial}\Lambda=[\Lambda,\varphi]$, and $[\Lambda,\Lambda]=0$ imply that (2,0)-part $\Lambda^{2,0}$ of $\Lambda$ defines a holomorphic Poisson structure ($\mathcal{M},\Lambda^{2,0})$. If we choose a sufficiently fine locally finite open covering $\{U_j\}$ of $M$, and take a sufficiently small $\Delta_{\epsilon}$, we have  $C^{\infty}$ functions $\xi_j^{\beta}(z,t),\beta=1,...,m+n$ on each $U_j\times \Delta_{\epsilon}$, and the map
\begin{align*}
\xi_j:(z,t)\to (\xi_j^1(z,t),...,\xi_j^{n+m}(z,t))
\end{align*}
gives local complex coordinates of $\mathcal{M}$ on $U_j\times \Delta_{\epsilon}$, 
and $\xi_j^{n+\mu}(z,t)=t_{\mu}$ for $\mu=1,...,m$.

Then we have 
\begin{align*}
\xi_j:(z,t)\to (\xi_j^1(z,t),...,\xi_j^n(z,t),t_1,...,t_m).
\end{align*}

Therefore 
\begin{align*}
\omega:(\xi_j^1(z,t),...,\xi_j^n(z,t),t_1,...,t_m)\to (t_1,...,t_m)
\end{align*}
is a holomorphic map of $\mathcal{M}$ onto $\Delta_{\epsilon}$. For each $t\in \Delta_{\epsilon}$, $\omega^{-1}(t)$ is a holomorphic Poisson manifold whose system of local complex coordinates is given by $\{\xi_j^1(z,t),...,\xi_j^n(z,t)\}$ and a holomorphic Poisson structure is given by $(2,0)$-part $\Lambda(t)^{2,0}$ of $\Lambda(t)$. So we have $\omega^{-1}(t)=(M_{\varphi(t)},\Lambda(t)^{2,0})$. Thus $\{(M_{\varphi(t)},\Lambda(t)^{2,0})|t\in \Delta_{\epsilon}\}$ forms a Poisson analytic family $(\mathcal{M},\Lambda^{2,0},\Delta_{\epsilon},\omega)$.

\begin{example}
Let $U_i=\{[z_0,z_1,z_2]|z_i\ne0\}$ $i=0,1,2$ be an open cover of complex projective plane $\mathbb{P}_{\mathbb{C}}^2$. Let $x=\frac{z_1}{z_0}$ and $w=\frac{z_2}{z_0}$ be coordinates on $U_0$. Then the holomorphic Poisson structures on $U_0$ are parametrized by $t=(t_1,...,t_{10})\in \mathbb{C}^{10}$
\begin{align*}
(t_1+t_2x+t_3w+t_4x^2+t_5xw+t_6w^2+t_7x^3+t_8x^2w+t_9xw^2+t_{10}w^3)\frac{\partial}{\partial x}\wedge \frac{\partial}{\partial w}
\end{align*}
This parametrizes the whole holomorphic Poisson structures on $\mathbb{P}_{\mathbb{C}}^2$.$($See $\cite{Pin11}$ Proposition 2.2$)$. Let $\Lambda_0=x\frac{\partial}{\partial x}\wedge \frac{\partial}{\partial w}$ be the holomorphic Poisson structure on $\mathbb{P}_{\mathbb{C}}^2$.   Then $HP^2(\mathbb{P}_{\mathbb{C}}^2,\Lambda_0)=5$, $HP^3(\mathbb{P}_{\mathbb{C}}^2,\Lambda_0)=0$.$($See $\cite{Pin11}$ Example 3.5 $)$ and $w^2\frac{\partial}{\partial x}\wedge\frac{\partial}{\partial w}$, $x^3\frac{\partial}{\partial x}\wedge\frac{\partial}{\partial w}$, $x^2w\frac{\partial}{\partial x}\wedge\frac{\partial}{\partial w}$, $xw^2\frac{\partial}{\partial x}\wedge\frac{\partial}{\partial w}$  and $w^3\frac{\partial}{\partial x}\wedge\frac{\partial}{\partial w}$ are the representatives of the cohomology classes consisting of the basis of $HP^2(\mathbb{P}_{\mathbb{C}}^2,\Lambda_0)$. Let $t=(t_1,t_2,t_3,t_4,t_5)\in \mathbb{C}^5$. Let $\Lambda(t)=(t_1w^2+x+t_2x^3+t_3x^2w+t_5xw^2+t_5w^3)\frac{\partial}{\partial x}\wedge \frac{\partial}{\partial w}$ be the holomorphic Poisson structure on $\mathbb{P}_{\mathbb{C}}^2\times \mathbb{C}^5$. Then $(\mathbb{P}_{\mathbb{C}}^2\times \mathbb{C}^5,\Lambda(t),\mathbb{C}^5, \omega)$, where $\omega$ is the natural projection, is a Poisson analytic family with $\omega^{-1}(0)=(\mathbb{P}_{\mathbb{C}}^2,\Lambda_0)$. Since the complex structure does not change in the family, the Poisson Kodaira Spencer map is an isomorphism. Hence the family satisfies the theorem of existence for $(\mathbb{P}_{\mathbb{C}}^2,\Lambda_0)$. \end{example}

\section{A concept of  Kuranishi family in holomorphic Poisson category}\label{section6}
By following the lecture notes of Kuranishi \cite{Kur71}, we extend the definition of complex analytic family over a complex space to Poisson analytic family over a complex space and raise the question of existence of a complete family without assumption $HP^3(M,\Lambda)=0$ where $(M,\Lambda)$ is a holomorphic Poisson manifold.

In this section $\bold{M}$ is a real $C^{\infty}$ compact manifold. And an analytic set $S$ is by definition a subset of a domain $D$ (which is called the ambient space of $S$) defined by zero locus of finitely many holomorphic functions on $D$. A map $f:S\to S'$ between two analytic sets is called analytic if for each $s\in S$, $f$ can be extended to a complex analytic map from an open neighborhood of $s$ in $D$ into the ambient space of $S'$.

\begin{definition}
Let $S$ be an analytic set. By a $C^{\infty}$ family of holomorphic Poisson charts of $\bold{M}$ with a parameter in $S$  we mean 
\begin{enumerate}
\item a $C^{\infty}$ map
\begin{align*}
\tilde{z}:U\times S'\to \mathbb{C}^n
\end{align*}
where $U$ (resp $S'$) is an open subset of $\bold{M}$ (resp. of $S$), such that the map $z^t:U\to \mathbb{C}^n$ defined by $z^t(p)=\tilde{z}(p,t)$ is a holomorphic complex chart of $\bold{M}$ for each $t\in S'$.
\item A $C^{\infty}$-complex bivector field $\Lambda$ of the form $\sum_{i,j} g_{ij}(x,t)\frac{\partial}{\partial x_i}\wedge \frac{\partial}{\partial x_j}$ on $M\times S$ \footnote{Here we mean by a $C^{\infty}$ bivector field $\Lambda$ on $M\times S$ that for each $s\in S$, $g_{ij}(x,s)$ can be extended to be $C^{\infty}$ functions on $M\times U$ where $U$ is a neighorbood of $s$ in the ambient space of $S$  } such that $[\Lambda,\Lambda]=0$ for each $t\in S$ and $(2,0)$-part $\Lambda_t^{2,0}$ of the restriction $\Lambda_t$ of $\Lambda$ on $\bold{M} \times t$ is a holmorphic bivector field with respect to the complex structure defined by $z^t$.
\end{enumerate}
\end{definition}


Let $\tilde{z}^{\sharp}:U\times S'\to \mathbb{C}^n\times S'$ defined by $\tilde{z}^{\sharp}=(\tilde{z}(p,t),t)$. 

If $\tilde{w}$ is another such family with domain $V$ and over $S''$, we mean by change of charts from $\tilde{z}$ to $\tilde{w}$ the map $\tilde{w}^{\sharp}\circ \tilde{z}^{\sharp -1}$ of $\tilde{z}^{\sharp}((U\cap V)\times (S'\cap S''))$ to $\tilde{w}^{\sharp}((U\cap V)\times (S'\cap S''))$. The domain and the image of the change of charts are open subsets of $\mathbb{C}^n\times S$, and hence we may ask if the change is complex analytic.

Now we can define the notion of a Poisson analytic family of holomopric poisson structures on $\bold{M}$ 

\begin{definition}
Let $S$ be an analytic set. By a Poisson analytic family of holomorphic Poisson structures on $\bold{M}$ with a parameter in $S$ we mean
\begin{enumerate}
\item A $C^{\infty}$ complex bivector field $\Lambda$ on $M\times S$ of the form $\sum_{i,j} g_{ij}(x,s)\frac{\partial}{\partial x_i}\wedge \frac{\partial}{\partial x_j}$ with $[\Lambda,\Lambda]=0$ for each $s\in S$.
\item a collection $(\tilde{\Phi},\Lambda)$ of $C^{\infty}$ families of holomorphic Poisson charts of $\bold{M}$ with a parameter in $S$ satisfying 
\begin{enumerate}
\item If $\tilde{z}$,$\tilde{w}\in \tilde{\Phi}$, then the change of charts from $\tilde{z}$ to $\tilde{w}$ is complex analytic.
\item for any $p$ in $\bold{M}$ and any $t$ in $S$ there is a $\tilde{z}$ in $\tilde{\Phi}$ with domain $U$ and $S'$ such that $(p,t)\in U\times S'$.
\item If $\tilde{u}$ is a $C^{\infty}$ family of Poisson holomorphic charts of $\bold{M}$ with a parameter in $S$ and if the change of charts from $\tilde{z}$ to $\tilde{u}$ is complex analytic for any $\tilde{z}$ in $\tilde{\Phi}$, then $\tilde{u}$ is in $\tilde{\Phi}$.
\end{enumerate}
\end{enumerate}

\end{definition}

If this is the case, for each fixed $t$, $\Phi_t=\{z^t:\tilde{z}\in \tilde{\Phi}\}$ is a chart covering of a holomophic Poisson structure, say $(M_t,\Lambda_t^{2,0})$ on $\bold{M}$. $(M_t,\Lambda_t^{2,0})$ is called the holomoprhic Poisson structure in $\tilde{\Phi}$ over $t$. Thus we have a collection $\{(M_t,\Lambda_t):t\in S\}$ of holomoprhic Poisson structures on $\bold{M}$. 

Let $B$ be an analytic set. Denote by $\tau$ an analytic map $B\to S$ and let $\Lambda \circ {\tau}=\sum_{i,j} g_{ij}(x,h(s))\frac{\partial}{\partial x_i}\wedge \frac{\partial}{\partial x_j}$. If $\tilde{z}$ is a $C^{\infty}$ family of holomorphic Poisson chart of $\bold{M}$ with domain $U$ and over $S'\subset S$, then $\tilde{z}\circ(id\times \tau):U\times B'\to \mathbb{C}^n$ where $B'=\tau^{-1}(S')$ and $id$ is the identity map of $U$, is a $C^{\infty}$ family of holomorphic Poisson charts of $\bold{M}$ with a parameter in $B$ with respect to $\Lambda \circ {\tau}=\sum_{i,j} g_{ij}(x,h(s))\frac{\partial}{\partial x_i}\wedge \frac{\partial}{\partial x_j}$. It is clear that the collection $\{\tilde{z}\circ (id\times \tau):\tilde{z}\in \tilde{\Phi}\}$ can be enlarged to a unique Poisson analytic family of holomorphic Poisson structures $(\tilde{\Phi}\circ \tau,\Lambda\circ \tau)$ on $\bold{M}$. 
\begin{definition}
The above family $(\tilde{\Phi}\circ \tau,\Lambda\circ{\tau})$ is called the Poisson analytic family induced from $(\tilde{\Phi},\Lambda)$ by $\tau$.
\end{definition}

Let $(\tilde{\Phi},\Lambda)$ and $(\tilde{\Psi},\Lambda')$ be Poisson analytic family over an analytic set $B$. Denote by $\bold{f}$ a family of diffeomorphism of $\bold{M}$ parametrized by $B$, say $\{f^b:b\in B\}$

\begin{definition}
We say that $\bold{f}$ induces an isomorphism from $(\tilde{\Phi},\Lambda)$ to $(\tilde{\Psi},\Lambda')$ over the identity map of $B$ if the following conditions are satisfied:
\begin{enumerate}
\item for $\tilde{w}:V\times B'\to \mathbb{C}^n$ in $\tilde{\Psi}$, let $U, B''$ be open subsets such that $f^b(U)\subset V$ for all $b\in B''$. Then $(p,b)\in U\times B''\mapsto \tilde{w}(f^b(p),b)\in \mathbb{C}^n$ is an element of $\tilde{\Phi}$.
\item the map $(p,b)\in M\times B\mapsto f^b(p)\in \bold{M}$ is a $C^{\infty}$ map, i.e. on a neighborhood of each point it is $C^{\infty}$
\item let $F:\bold{M}\times B\to \bold{M}\times B$ defined by $(x,t)\to (f^t(x),t)$. Then $F_*\Lambda=\Lambda'$. In particular we have $f^b_* \Lambda_b^{2,0}=\Lambda_b^{'2,0}$ for each $b\in B$.
\end{enumerate}
\end{definition}


\begin{definition}
Let $(M,\Lambda_0)$ be a holomorphic Poisson structure on $\bold{M}$. We say that $(\tilde{\Phi},\Lambda)$ is a Poisson analytic family of deformations of $(M,\Lambda_0)$ over $(S,s_0)$ if the holomorphic Poisson structure in $\tilde{\Phi}$ over $s_0$ is $(M,\Lambda_0)$.
\end{definition}

\begin{definition}
A Poisson analytic family $(\tilde{\Phi},\Lambda)$ of deformations of $(M,\Lambda_0)$ over $(S,s_0)$ is called complete at $s_0$ if for any pointed analytic set $(B,b_0)$ and any Poisson analytic family $(\tilde{\Psi},\Lambda')$ of deformations of $(M,\Lambda)$ over $(B,b_0)$, we can find an open neighborhood $B'$ of $b_0$ and a complex analytic map $\tau:(B',b_0)\to (S,s_0)$ such that $(\tilde{\Psi}|_{B'},\Lambda'|_{B'})$ is isomorphic to the family $(\tilde{\Phi}\circ \tau,\Lambda\circ \tau)$.
\end{definition}
The following problem is an analogue of Kuranishi's completeness theorem.

\begin{problem}
Let $(M,\Lambda_0)$ be a compact holomorphic Poisson manifold. Then does there a complete Poisson analytic  family of deformations of $(M,\Lambda_0)$ exist?\footnote{For my unfamiliarity of analysis, I could not access to this problem.} \footnote{
Let $(M,\Lambda_0)$ be a compact holomorphic Poisson manifold. We fix a Hermitian metric on $M$ and define $L^*,\Box,G,...$, and so forth. Let $\{\eta_v+\pi_v|v=1,...,m\}$ be a base for $\mathbb{H}^2\cong HP^2(M,\Lambda_0)$. Assume that we have a unique convergent power series solution $\beta(t)$ of 
\begin{align*}
\beta(t)=\beta_1(t) +\frac{1}{2}L^*G[\beta(t),\beta(t)],
\end{align*}
where $\beta_1(t)=\sum_{v=1}^m (\eta_v+\pi_v)t_v$.  $\beta(t)=\alpha(t)-\Lambda_0$ satisfies
\begin{align*}
L\beta(t)+\frac{1}{2}[\beta(t),\beta(t)]=0
\end{align*}
if and only if $H[\beta(t),\beta(t)]=0$. Let $\{\beta_{\lambda}|\lambda=1,...,r\}$ be an orthonomal base of $\mathbb{H}^3$ and let $(-,-)$ be the inner product in $A^{0,0}(M,\wedge^3 T)\oplus A^{0,1}(M,\wedge^2 T)\oplus A^{0,2}(M,T)$. Then
\begin{align*}
H[\beta(t),\beta(t)]=\sum_{\lambda=1}^r ([\beta(t),\beta(t)],\beta_{\lambda})\beta_{\lambda}
\end{align*}
Hence $H[\beta(t),\beta(t)]=0$ if and only if $([\beta(t),\beta(t)],\beta_{\lambda})=0$ for $\lambda=1,...,r$. Since $\beta(t)$ is a power series in $t$ so is $([\beta(t),\beta(t)],\beta_{\lambda})=b_{\lambda}(t)$. Thus $b_{\lambda}(t)$ is holomorphic in $t$ for $\lambda=1,...,r$ and $|t|$ small $(|t|<\epsilon)$. Also $b_{\lambda}(0)=0$. Define an analytic set $S$ as follows:
\begin{align*}
S=\{t||t|<\epsilon,b_{\lambda}(t),\lambda=1,...,r\}
\end{align*}
Then $S$ is an analytic subset of $B_{\epsilon}$ containing the origin. I believe that this would be the base space of Kuranishi family for holomorphic Poisson deformations of $M$.}
\end{problem}

Lastly, we pose some natural questions I can not answer at this stage.

\begin{problem}
Can we establish the upper-semicontinuity theorem in a Poisson analytic family? $($See $\cite{Kod05}$ page 200$)$
\end{problem}

\begin{problem}
Let $(\mathcal{M},B,\omega)$ be a complex analytic family. Let $ b\in B$ and $\omega^{-1}(b)=M_b$ with a holomorphic Poisson structure $\Lambda_b$. What is the conditions or obstructions for the following? 
\begin{center}
``There exists an open neighborhood $U$ of $b$ in $B$ such that $(\mathcal{M}|_U,U,\omega)$ can be extended to be a Poisson analytic family $(\mathcal{M}|_U,\Lambda,U,\omega)$ such that $\omega^{-1}(b)=(M_b,\Lambda_b)$".\footnote{This problem is related to the operator $L=\bar{\partial}+[\Lambda,-]$}
\end{center}
\end{problem}

\begin{problem}
Can we establish the stability theorem for holomorphic Poisson submanifolds in the holomorphic Poisson category? $($$\cite{Kod63}$$)$ 
\end{problem}

\part{Infiniteismal Poisson deformations and Universal Poisson deformations of compact holomorphic Poisson manifolds}\label{part2}

In the part II of the thesis, we present infinitesimal deformations of a compact holomorphic Poisson manifold $(X,\Lambda_0)$ over an artinian local $\mathbb{C}$-algebra $(R,\mathfrak{m})$ with residue $\mathbb{C}$. We extend the method of \cite{Ran00} to show that given an infinitesimal Poisson deformation of $(X,\Lambda_0)$, as in the case of the part I of the thesis, we can canonically associate an element $\phi+\Lambda\in (A^{0,0}(X,T)\oplus A^{0,0}(X,\wedge^2 T))\otimes \mathfrak{m}$ satisfying the Maurer-Cartan equation $L(\phi+\Lambda)+\frac{1}{2}[\phi+\Lambda,\phi+\Lambda]=0$ of the differential graded Lie algebra $\mathfrak{g}=\bigoplus_i g_i=(\bigoplus_{p+q-1=i,p\geq 0, q\geq 1} A^{0,p}(X,\wedge^q T),L=\bar{\partial}+[\Lambda_0,-],[-,-])$. By using the language of functors of Artin rings, we show that the differential graded Lie algebra $\mathfrak{g}$ controls infinitesimal Poisson deformations of $(X,\Lambda_0)$. In other words, we establish the following theorem.

\begin{theorem}
Let $(X,\Lambda_0)$ be a compact holomorphic Poisson manifold. Then the Poisson deformation functor $PDef_{(X,\Lambda_0)}$ is controlled by the differential graded Lie algebra $\mathfrak{g}=(\bigoplus_{p+q-1=i,p\geq 0,q\geq 1}$
$A^{0,p}(X,\wedge^q T),L=\bar{\partial}+[\Lambda_0,-],[-,-])$. In other words, we have an isomorphism of two functors
\begin{align*}
Def_\mathfrak{g}\cong PDef_{(X,\Lambda_0)}
\end{align*}
\end{theorem}

 We  study universal Poisson deformation of a compact holomorphic Poisson manifold $(X,\Lambda_0)$ with $HP^1(X,\Lambda_0)=0$. Based on the method of \cite{Ran00}, we explicitly construct a $n$-th order universal Poisson deformation space $P_n^u$ over an artinian $\mathbb{C}$-algebra $(R_n^u,\mathfrak{m}^u_n)$ with exponent $n$ (i.e $\mathfrak{m}_n^{u\,n+1}=0$) such that any infinitesimal Poisson deformation of $(X,\Lambda_0)$ over an artinian local $\mathbb{C}$-algebra $(R,\mathfrak{m})$ with exponent $n$ (i.e $\mathfrak{m}^{n+1}=0$) can be induced from the $n$-th order universal Poisson deformation space via base change from a canonical ring homomorphism $r:R_n^u\to R$ up to equivalence. By taking the limit, we have a universal Poisson formal space. The main ingredient of universal Poisson deformation is the Jacobi complex or Quillen standard complex associated with the differential graded Lie algebra $\mathfrak{g}=\bigoplus_i g_i=(\bigoplus_{p+q-1=i,p\geq 0, q\geq 1} A^{0,p}(X,\wedge^q T), L:=\bar{\partial}+[\Lambda_0,-], [-,-])$. The base ring of an $n$-th universal Poisson deformation is given by $R_n^u:=\mathbb{C}\oplus \mathfrak{m}_n^u$, where $\mathfrak{m}_n^u=\mathbb{H}^0(J_n(\mathfrak{g}))^*$ which is the dual of $0$-th Jacobi cohomology group associated to $\mathfrak{g}$. For given an infinitesimal Poisson deformation, the associated element $\alpha:=\phi+\Lambda\in (A^{0,0}(X,T)\oplus A^{0,0}(X,\wedge^2 T))\otimes \mathfrak{m}$ gives an element $[\epsilon(\alpha)]$ in $\mathbb{H}^0(J_n(\mathfrak{g}))\otimes \mathfrak{m}$ which is the $0$-th Jacobi cohomology group associated with the differential graded Lie algebra $\mathfrak{g}\otimes \mathfrak{m}$. The element $[\epsilon(\alpha)]$ gives a ring homomorphism $r:R_n^u\to R$. We will present the full detail of the construction of Jacobi complex since we need actual computations for our main theorem\footnote{Similar result on universal Poisson deformation of a Poisson algebra is proved in \cite{Gin04} (Theorem 1.10). More precisely, for a Poisson algebra $A$ with $HP^1(A)$ and $HP^2(A)$ a finite-dimensional vector space over $\mathbb{C}$, there is an universal formal Poisson deformation of the algebra $A$. We also prove in the part \ref{part3} of the thesis that the Poisson deformation functor $PDef_{(X,\Lambda_0)}$ is prorepresentable when $(X,\Lambda_0)$ is a smooth projective Poisson scheme with $HP^1(X,\Lambda_0)=0$.} in the following:

\begin{theorem}\label{2theorem}
Let $(X,\Lambda_0)$ be a compact holomorphic Poisson manifold with $HP^1(X,\Lambda_0)= 0$ and $J$ be the Jacobi complex associated with the differential graded Lie algebra 
\begin{align*}
\mathfrak{g}=\bigoplus_i g_i=(\bigoplus_{p+q-1=i,p\geq 0,q\geq 1} A^{0,p}(X,\wedge^q T), L:=\bar{\partial}+[\Lambda_0,-], [-,-])
\end{align*} where $[-,-]$ is the Schouten bracket. Then
\begin{enumerate}
\item For each $n\geq 1$, $R_n^u=\mathbb{C}\oplus \mathbb{H}^0(J_n(\mathfrak{g}))^*$ is a local artinian $\mathbb{C}$-algebra with the residue $\mathbb{C}$ in a canonical way. The maximal ideal of $R_n^u$ is given by $\mathfrak{m}_n^u=\mathbb{H}^0(J_n(\mathfrak{g}))^*$ and have exponent $n$ (which means $\mathfrak{m}_n^{u\, n+1}=0$).

\item There is a $n$-th order universal Poisson deformation $P_n^u$ of $(X,\Lambda_0)$ over $R_n^u$ in the following sense:
for any artinian local $\mathbb{C}$-algebra $(R,\mathfrak{m})$ of exponent $n$ (which means $\mathfrak{m}^{n+1}=0$) and infinitesimal Poisson deformation $P$ of $(X,\Lambda_0)$ over $R$, there is a canonical homomorphism
$r:R_n^u\to R$
and an isomorphism of Poisson analytic spaces over $R$
\begin{align*}
P/R\xrightarrow{\sim} r^* P_n^u:=P_n^u\times_{Spec(R_n^u)} Spec(R);
\end{align*}

\item For each $n\geq 1$, $P_n^u/R_n^u$ fit together to form a direct system with limit, which give an universal  formal Poisson doformation $\hat{P}^u/ \hat{R}^u:=\varinjlim_n P_n^u/R_n^u$ of $(X,\Lambda_0)$ over $\hat{R}^u$ in the following sense: if $\hat{R}= \varprojlim_n R_n$ is a complete local noetheiran $\mathbb{C}$-algebra and $\hat{P}/\hat{R}= \varinjlim_n P_n/R_n$, then $\hat{r}=\varprojlim_n r_n:\hat{R}^u \to \hat{R}$ exists and $\hat{P}/\hat{R}\cong \hat{r}^*(\hat{P}^u/\hat{R}^u):=\hat{P}^u\times_{Spec(\hat{R}^u)} Spec(\hat{R})$.
\end{enumerate}
\end{theorem}

In chapter \ref{chapter4}, we present the construction of $n$-th Jacobi complex or Quillen standard complex associated to a differential graded Lie algebra $\mathfrak{g}$. We show that we can canonically define a local artinian $\mathbb{C}$-algebra structure on $\mathbb{C}\oplus \mathbb{H}^0(J_n(\mathfrak{g}))^*$ with residue $\mathbb{C}$ and exponent $n$, where $\mathbb{H}^0(J_n(\mathfrak{g}))^*$ is the dual space of $0$-th Jacobi cohomology group $\mathbb{H}^0(J_n(\mathfrak{g}))$ associated with the Jacobi complex $J_n(\mathfrak{g})$. We also describe a morphic element in $\mathbb{H}^0(J_n(\mathfrak{g}))\otimes \mathfrak{m}$ which gives an $\mathbb{C}$-algebra homomorphism $\mathbb{C}\oplus \mathbb{H}^0(J_n(\mathfrak{g}))^*\to R$, where $(R,\mathfrak{m})$ is a local artinian $\mathbb{C}$-algebra with residue $\mathbb{C}$ and exponent $n$.

In chapter \ref{chapter5}, we study infinitesimal Poisson deformations of compact holomorphic Poisson manifolds. We define infinitesimal deformations of compact holomorphis Poisson manifolds over local artinian $\mathbb{C}$-algebras with residue $\mathbb{C}$ which is an infinitesimal version of Poisson analytic families. We deduce the same integrability condition as in the Part \ref{part1} of the thesis. The idea is essentially same to the Part \ref{part1} of the thesis. However, we use the infinitesimal language.

In chapter \ref{chapter6}, we show that the differential graded Lie algebra $\mathfrak{g}$ controls infinitesimal Poisson deformations in the language of functor of Artin rings. We complete the proof of Theorem \ref{2theorem} on universal Poisson deformations.

\chapter{Jacobi complex}\label{chapter4}

We present the full details of the construction of the Jacobi complex or Quillen standard complex associated with a differential graded Lie algebra since we need actual computations for our infinitesimal Poisson deformations. See also \cite{Hin97} 2.2 Quillen standard complex.
\section{Preliminaries}\

Let $\mathfrak{g}=\bigoplus_{i\geq 0} g_i$ be a graded complex with a differential $d$ where $g_i$ is a vector space over a field $\mathbb{C}$. In other words, we have the following complex $\mathfrak{g}:g_0\xrightarrow{d} g_1\xrightarrow{d} g_2\xrightarrow{d} \cdots$.

\begin{definition}
The symmetric algebra of a graded complex $(\mathfrak{g},d)$ are defined as a graded complex $S(\mathfrak{g})=T(\mathfrak{g})/I$, where $T(\mathfrak{g})=\sum_{n\geq 0} \mathfrak{g}^{\otimes n}$ is the tensor algebra of $\mathfrak{g}$ and $I$ is the two sided ideal generated by elements of the form $a\otimes b-(-1)^{|a||b|}b\otimes a$ where $a,b$ are homogeneous elements of $\mathfrak{g}$. We denote $\overline{S(\mathfrak{g})}=\overline{T(\mathfrak{g})}/I$, where $\overline{T(\mathfrak{g})}=\sum_{n\geq 1} \mathfrak{g}^{\otimes n}$. We will denote by $x_1\odot \cdots \odot x_n$ the image of $x_1\otimes \cdots \otimes x_n$.
\end{definition}

\begin{definition}
The exterior algebra of a graded vector space $\mathfrak{g}$ are defined as $\bigwedge \mathfrak{g}=T(\mathfrak{g})/I$ where $T(\mathfrak{g})=\sum_{n\geq 0} \mathfrak{g}^{\otimes n}$ is the tensor algebra of $\mathfrak{g}$ and $I$ is the two sided ideal generated by elements of the form $a\otimes b +(-1)^{|a||b|}b\otimes a$ where $a,b$ are homogeneous elements of $\mathfrak{g}$. We denote $\overline{\bigwedge \mathfrak{g}}=\overline{T(\mathfrak{g})}/I$, where $\overline{T(\mathfrak{g})}=\sum_{n\geq 1} \mathfrak{g}^{\otimes n}$. We will denote by $x_1\wedge \cdots \wedge x_n$ the image of $x_1\otimes \cdots \otimes x_n$.

\end{definition}

\begin{remark}
we have
\begin{align*}
\wedge^n \mathfrak{g} &=\wedge ^n (g_0\oplus g_1\oplus g_2\oplus \cdots)\\
&=\bigoplus_{r_0+r_1+\cdots=n, r_i\geq 0} (\wedge^{r_0} g_0)\otimes (sym^{r_1} g_1)\otimes \cdots \otimes (\wedge^{r_{2k}} g_{2k})\otimes (sym^{r_{2k+1}} g_{2k+1}) \otimes \cdots
\end{align*}
where $\wedge^k V$ is the usual $k$-th anti-symmetric product of a vector space $V$ and $sym^k V$ is the usual $k$-th symmetric product of a vector space $V$ when we ignore the grading.
\end{remark}

\begin{definition}
We can define the coalgebra structure $\Delta'$ and $\Delta$ on $\overline{S(\mathfrak{g})}$ and $\overline{\bigwedge \mathfrak{g}}$ in the following way:
\begin{align*}
\Delta'(x_1\odot \cdots \odot x_n)=\sum_I (-1)^{s(I)} x_I\otimes x_{\bar{I}}\\
\Delta(v_1\wedge \cdots \wedge v_n)=\sum_I (-1)^{t(I)} v_I\otimes v_{\bar{I}}
\end{align*}
where the summation is over all subsets $I=\{r_1,...,r_p\}$, $r_1<\cdots <r_p$ and $\bar{I}=\{s_1,...,s_q\}$ such that $s_1<\cdots <s_q$ with $I\cup \bar{I}=\{1,...,n\}$, $x_I=x_{r_1}\odot \cdots \odot x_{r_p}$, $x_{\bar{I}}=x_{s_1}\odot \cdots\odot x_{s_q}$ and similary $v_I=v_{r_1}\wedge \cdots \wedge v_{r_p}, v_{\bar{I}}=v_{s_1}\wedge \cdots \wedge v_{s_q}$. Here $s(I)$ and $t(I)$ are determined in the following way:
\begin{align*}
x_1\odot\cdots \odot x_n&=(-1)^{s(I)}x_I\odot x_{\bar{I}}\\
v_1\wedge \cdots \wedge v_n&=(-1)^{t(I)}v_I\wedge v_{\bar{I}}
\end{align*}
\end{definition}

Let $(\mathfrak{g},d)$ be a graded complex and we denote $(\mathfrak{g}[n],d)$ be a graded complex by shifting the degree by $n$, i.e. $\mathfrak{g}[n]^i=g_{n+i}$.

\begin{remark}[d\'ecalage  isomorphism]
We have an isomorphism 
\begin{align*}
dec:S^n(\mathfrak{g}[1])&\cong (\bigwedge^n \mathfrak{g})[n]\\
\bar{x}_1\odot\cdots\odot\bar{x}_n&\mapsto (-1)^{\sum_{i=1}^{n-1} (n-i)p_i} x_1\wedge \cdots \wedge x_n
\end{align*}
where $x_i$ is an element of $\mathfrak{g}$, $\bar{x}_i$ is an element of $\mathfrak{g}[1]$ via the natural map $\mathfrak{g}\to \mathfrak{g}[1]$, and $p_i$ is the degree of $x_i$
\end{remark}

\begin{notation}
Let $I=\{p_1,..,p_r\}$. $A_I$ is defined by the following relation: $dec(\bar{x}_{p_1}\odot \cdots \odot \bar{x}_{p_r})=(-1)^{A_I}x_{p_1}\wedge \cdots \wedge x_{p_r}$. We will denote $\bar{x}_I=x_{p_1}\odot\cdots \odot x_{p_r}$ and $x_I=x_{p_1}\wedge \cdots \wedge x_{p_r}$. We denote $|I|=r$ the cadinality of $I$ and $|x_I|:=deg(x_{p_1})+\cdots +deg(x_{p_r})$
\end{notation}

Via this isomorphism $dec$, we have the following commutative diagram
\begin{center}
$\begin{CD}
\overline{S(\mathfrak{g}[1])}@>dec>> \overline{\bigwedge\mathfrak{g}}\\
@V\Delta' VV @VV\Delta V\\
\overline{S(\mathfrak{g}[1])}\otimes \overline{S(\mathfrak{g}[1])} @>\tilde{dec}>> \overline{\bigwedge\mathfrak{g}}\otimes \overline{\bigwedge\mathfrak{g}}
\end{CD}$
\end{center}
where $\tilde{dec}$ is defined in the following way: $\tilde{dec}(\bar{x}_I\otimes \bar{x}_J)=(-1)^{A_I+A_J+|x_I||J|}x_I\otimes x_J$. 

\subsection{Induced differential}\

Let $\mathfrak{g}$ be a differential graded Lie algebra.
We define the map $Q_n:\wedge^n \mathfrak{g}\to \wedge^{n-1} \mathfrak{g}$ by
\begin{align*}
Q_n(x_1\wedge \cdots \wedge x_n)=\sum_{i<j} (-1)^a[x_i,x_j]\wedge x_1\wedge \cdots \wedge \hat{x}_i\wedge \cdots \wedge \hat{x}_j\cdots x_n
\end{align*}
where $a$ is defined in the following way
\begin{align*}
x_1\wedge \cdots \wedge x_n=(-1)^ax_i\wedge x_j\wedge x_1\wedge \cdots \wedge \hat{x}_i\cdots \wedge \hat{x}_j\cdots\wedge x_n
\end{align*}
More precisely,
\begin{align*}
a=\underbrace{i-1+p_i(p_1+\cdots +p_{i-1})}_{\text{first moving $x_i$}}+\underbrace{j-2+p_j(p_1+\cdots +\hat{p}_i+\cdots +p_{j-1})}_{\text{second moving $x_j$}}
\end{align*}

And we define $d_n:\wedge^n \mathfrak{g}\to \wedge^n \mathfrak{g}$ inductively on $n$ by
\begin{align*}
d_n(x_1\wedge \cdots \wedge x_n)=dx_1\wedge x_2\wedge \cdots \wedge x_n+\sum_{i=2}^n(-1)^{p_1+\cdots+p_{i-1}}x_1\wedge\cdots \wedge x_{i-1}\wedge dx_i\wedge x_{i+1}\cdots \wedge x_n
\end{align*}

First we show that $d^2=0$ and $Q^2=0$. We show $d^2=0$ by induction on $k$, where $x_1\wedge \cdots \wedge x_k$. For $k=1$, the statement is true by the definition of $d$. Assume that the statement is true for $k=n-1$.
\begin{align*}
d\circ d(x_1\wedge\cdots \wedge  x_n)&=d(dx_1\wedge x_2\cdots\wedge x_n+(-1)^{p_1}x_1\wedge d(x_2\wedge \cdots \wedge x_n))\\
                                                             &=ddx_1\wedge x_2\cdots \wedge x_n +(-1)^{p_1+1}dx_1\wedge d(x_2\wedge \cdots \wedge x_n)\\
                                                             &+(-1)^{p_1}dx_1\wedge d(x_2\wedge \cdots \wedge x_n)+(-1)^{p_1+p_1}x_1\wedge dd(x_2\wedge \cdots \wedge x_n)=0
\end{align*}
by the induction hypothesis. For $Q^2=0$, we recall the definition of $Q$. Then $Q\circ Q(x_1\wedge \cdots \wedge x_n)$ is of the following form
\begin{align*}
\sum_{i<j<k} ((-1)^a [[x_i,x_j],x_k]+(-1)^b [[x_j,x_k],x_i]+(-1)^c[[x_i,x_k],x_j])x_1\wedge \cdots \wedge \hat{x}_i \cdots \wedge \hat{x}_j\cdots \wedge \hat{x}_j \cdots \wedge \hat{x}_k\wedge \cdots \wedge x_n
\end{align*}
where
\begin{align*}
a&=\underbrace{i-1+p_i(p_1+\cdots +p_{i-1})}_{\text{first moving $x_i$}}+\underbrace{j-2+p_j(p_1+\cdots +\hat{p}_i+\cdots +p_{j-1})}_{\text{second moving $x_j$}}+\\
      &\,\,\,\,\,\,\,\,\,\,\,\,\,\,\, \,\,\,\,\,\,\,\,\,\,\,\,\,\,\,\,\,\,\,\, \underbrace{k-3+p_k(p_1+\cdots +\hat{p}_i+\cdots+\hat{p}_j+\cdots +p_{k-1})}_{\text{third moving $x_k$}}\\
b&=\underbrace{j-1+p_j(p_1+\cdots+p_{j-1})}_{\text{first moving $x_j$}}+\underbrace{k-2+p_k(p_1+\cdots+\hat{p}_j+\cdots+p_k)}_{\text{second moving $x_k$}}+\underbrace{i-1+p_i(p_1+\cdots+p_{i-1})}_{\text{third moving $x_i$}}\\
c&=\underbrace{i-1+p_i(p_1+\cdots+p_{i-1})}_{\text{first moving $x_i$}}+\underbrace{k-2+p_k(p_1+\cdots+\hat{p}_i+\cdots +p_k)}_{\text{second moving $x_k$}}+\underbrace{j-2+p_j(p_1+\cdots +\hat{p}_i+\cdots +p_{j-1})}_{\text{third moving $x_j$}}   
\end{align*}

Set
\begin{align*}
d=i+j+k+p_i(p_1+\cdots +p_{i-1})+p_j(p_1+\cdots+p_{j-1})+p_k(p_1+\cdots+p_{k-1})
\end{align*}
Then
\begin{align*}
 (-1)^d&=(-1)^{a+p_jp_i+p_kp_i+p_kp_j}&(-1)^a&=(-1)^{d+p_jp_i+p_kp_i+p_kp_j}\\
 (-1)^d&=(-1)^{b+p_kp_j}&(-1)^b&=(-1)^{d+p_kp_j}\\
 (-1)^d&=(-1)^{c+1+p_kp_i+p_jp_i}&(-1)^c&=(-1)^{d+1+p_kp_i+p_jp_i}
\end{align*}
So we have
\begin{align*}
&(-1)^a [[x_i,x_j],x_k]+(-1)^b [[x_j,x_k],x_i]+(-1)^c[[x_i,x_k],x_j]\\
&=(-1)^d\left((-1)^{p_jp_i+p_kp_i+p_kp_j} [[x_i,x_j],x_k]+(-1)^{p_kp_j} [[x_j,x_k],x_i]+(-1)^{1+p_kp_i+p_jp_i}[[x_i,x_k],x_j]\right)=0
\end{align*}
by the following relations and graded Jocobi identity
\begin{align*}
(-1)^{p_jp_i+p_kp_i+p_kp_j} [[x_i,x_j],x_k]=(-1)^{p_jp_i+p_kp_j+p_kp_j} [x_k,[x_i,x_j]]\\
(-1)^{p_kp_j} [[x_j,x_k],x_i]=(-1)^{p_kp_j+p_ip_j+p_ip_k} [x_i,[x_j,x_k]]\\
(-1)^{1+p_kp_i+p_jp_i}[[x_i,x_k],x_j]=(-1)^{p_kp_j+p_jp_i+p_jp_i}[x_j,[x_k,x_i]]
\end{align*}
\begin{align*}
(-1)^{p_ip_k}[x_i,[x_j,x_k]]+(-1)^{p_jp_i}[x_j,[x_k,x_i]]+(-1)^{p_kp_j}[x_k,[x_i,x_j]]=0
\end{align*}

Hence we have $Q^2=0$.

Next we will show that $Qd=dQ$ by induction on $k$ where $x_1\wedge \cdots\wedge x_k$. For $k=2$, 
\begin{align*}
Q\circ d(x_1\wedge x_2)=Q(dx_1\wedge x_2+(-1)^{p_1}x_1\wedge dx_2)=[dx_1,x_2]+(-1)^{p_1}[x_1,dx_2]\\
d \circ Q(x_1\wedge x_2)=d[x_1,x_2]=[dx_1,x_2]+(-1)^{p_1}[x_1,dx_2]
\end{align*}
So induction holds for $k=2$. Now we assume that the statement holds for $k=n-1$. We will prove that it holds for $k=n$.

\begin{align*}
d(x_1\wedge\cdots\wedge x_n)&=dx_1\wedge(x_2\wedge\cdots\wedge x_n)+(-1)^{p_1} x_1\wedge d(x_2\wedge\cdots \wedge x_n)\\
dx_1\wedge(x_2\wedge\cdots\wedge x_n)&= (-1)^{n-1+(p_1+1)(p_2+\cdots+p_n)}(x_2\wedge\cdots \wedge x_n)\wedge dx_1\\
(-1)^{p_1}x_1\wedge d(x_2\wedge \cdots\wedge x_n)
                                                                                     &=(-1)^{p_1+ i-2+p_i(p_2+\cdots p_{i-1})} x_1\wedge d(x_i\wedge x_2\wedge \cdots \wedge \hat{x}_i\wedge \cdots\wedge x_n)\\
                                                                                     &=(-1)^{p_1+i-2+p_i(p_2+\cdots+p_{i-1})} x_1\wedge dx_i\wedge x_2\wedge \cdots\wedge \hat{x}_i\wedge \cdots\wedge x_n\\
                                                                                     &\,\,\,\,\,\,\,\,\,\,+(-1)^{p_1+i-2+p_i(p_2+\cdots+p_{i-1})+p_i} x_1\wedge x_i\wedge d(x_2\wedge \cdots\wedge \hat{x}_i\wedge \cdots \wedge x_n)\\
                                                                                     &=(-1)^{p_1+n-2+p_1(p_2+\cdots+p_n+1)}d(x_2\wedge \cdots \wedge x_n)\wedge x_1\\
\end{align*}
Then
\begin{align*}
Qd(x_1\wedge \cdots \wedge x_n)&=\sum_{i=2}^n (-1)^{i-2+p_i(p_2+\cdots +p_{i-1})}[dx_1,x_i]\wedge x_2\wedge \cdots \hat{x}_i\wedge \cdots \wedge x_n\\
                                                      &\,\,\,\,\,\,\,\,\,\,+(-1)^{n-1+(p_1+1)(p_2+\cdots +p_n)}Q(x_2\wedge \cdots\wedge x_n)\wedge dx_1\\
                                                      &+\sum_{i=2}^n(-1)^{p_1+i-2+p_i(p_2+\cdots+p_{i-1})}[x_1,dx_i]\wedge x_2\wedge \cdots \wedge \hat{x}_i\wedge \cdots \wedge x_n\\
                                                      &\,\,\,\,\,\,\,\,\,\,+\sum_{i=2}^n (-1)^{p_1+i-2+p_i(p_2+\cdots+p_{i-1})+p_i}[x_1,x_i]\wedge d(x_2 \wedge \cdots \wedge \hat{x}_i\wedge \cdots \wedge x_n)\\
                                                      &+(-1)^{p_1+n-2+p_1(p_2+\cdots+p_n+1)}Qd(x_2\wedge \cdots \wedge x_n)\wedge x_1
\end{align*}

Now we compute $dQ$. First we note the following relations
\begin{align*}
x_1\wedge(x_2\wedge \cdots x_n)&=(-1)^{i-2+p_i(p_1+\cdots +p_{i-1})} x_1\wedge x_i\wedge x_2 \wedge \cdots \wedge \hat{x}_i\wedge \cdots \wedge x_n\\
                                                      &=(-1)^{n-1+p_1(p_2+\cdots +p_n)}(x_2\wedge \cdots \wedge x_n)\wedge x_1
\end{align*}
Then
\begin{align*}
Q(x_1\wedge(x_2\wedge \cdots \wedge x_n))&=\sum_{i=2}^n (-1)^{i-2+p_i(p_1+\cdots +p_{i-1})} [x_1, x_i]\wedge x_2 \wedge \cdots \wedge \hat{x}_i\wedge \cdots \wedge x_n\\
                                                                        &\,\,\,\,\,\,\,\,\,\,+(-1)^{n-1+p_1(p_2+\cdots +p_n)}Q(x_2\wedge \cdots \wedge x_n)\wedge x_1
\end{align*}
Hence we have
\begin{align*}
dQ(x_1\wedge \cdots \wedge x_n)&=\sum_{i=2}^n (-1)^{i-2+p_i(p_1+\cdots +p_{i-1})} [dx_1, x_i]\wedge x_2 \wedge \cdots \wedge \hat{x}_i\wedge \cdots \wedge x_n\\
                                          &+\sum_{i=2}^n (-1)^{i-2+p_i(p_1+\cdots +p_{i-1})+p_1} [x_1, dx_i]\wedge x_2 \wedge \cdots \wedge \hat{x}_i\wedge \cdots \wedge x_n\\
                                          &+\sum_{i=1}^n (-1)^{i-2+p_i(p_1+\cdots +p_{i-1})+p_1+p_i} [x_1, x_i]\wedge d(x_2 \wedge \cdots \wedge \hat{x}_i\wedge \cdots \wedge x_n)\\
                                          &+(-1)^{n-1+p_1(p_2+\cdots +p_n)}dQ(x_2\wedge \cdots \wedge x_n)\wedge x_1\\
                                          &\,\,\,\,\,+(-1)^{n-1+p_1(p_2+\cdots +p_n)+(p_2+\cdots +p_n)}Q(x_2\wedge \cdots \wedge x_n)\wedge dx_1
\end{align*}

By induction hypothesis for $k=n-1$, we have $Qd(x_2\wedge \cdots \wedge x_n)=dQ(x_2\wedge \cdots \wedge x_n)$. Hence we have
\begin{align*}
Qd(x_1\wedge \cdots \wedge x_n)=dQ(x_1\wedge \cdots \wedge x_n)
\end{align*}
So the statement holds for $k=n$. Hence $Qd=dQ$. So if we we define $Q'$
\begin{align*}
Q'(x_1\wedge \cdots \wedge x_n)=((-1)^n d+Q)(x_1\wedge \cdots \wedge x_n)
\end{align*}

Then $Q'\circ Q'=0$.

\section{Jacobi complex}\

\begin{definition}[un-degree shifted complex]\label{2un}
Let $\mathfrak{g}$ be a differential graded Lie algebra. Let's consider the total complex of the following bicomplex with the differential $Q'$ defined as above,
\begin{center}
\tiny{$\begin{CD}
g_0 @>-d>> g_1@>-d>>g_2 @>-d>> g_3  @>>>\cdots\\
@AQAA @AQAA @AQAA @AQAA \\ 
\bigwedge^2 g_0 @>d>> (g_0\otimes g_1) @>d>> (g_0\otimes g_2) \oplus sym^2g_1@>d>> (g_0\otimes g_3)   \oplus (g_1\otimes g_2) @>>> \cdots \\
@AQAA @AQAA @AQAA @AQAA \\ 
\bigwedge^3 g_0 @>-d>> (\bigwedge^2 g_0)\otimes g_1 @>-d>> (\bigwedge^2 g_0 \otimes g_2)  \oplus(g_0\otimes sym^2 g_1) @>-d>> (\bigwedge^2 g_0\otimes g_3)\oplus (g_0\otimes g_1\otimes g_2) @>>> \cdots \\ 
@. @. @. \oplus sym^3 g_1 \\ 
@AQAA @AQAA @AQAA @AQAA \\
\bigwedge^4 g_0  @>d>> (\bigwedge^3 g_0) \otimes g_1@>d>> (\bigwedge^3 g_0\otimes g_2) @>d>> (\bigwedge^3 g_0\otimes g_3) @>>>\cdots\\
@.@. \oplus (\bigwedge^2 g_0 \otimes sym^2 g_1)@. \oplus (\bigwedge^2 g_0\otimes g_1\otimes g_2) \\
@. @. @. \oplus (g_0\otimes sym^3 g_1)  \\ 
@AQAA @AQAA @AQAA @AQAA \\
\cdots @>>> \cdots @>>> \cdots @>>> \cdots \\
@AQAA @AQAA @AQAA @AQAA\\
\wedge^n g_0 @> (-1)^nd>>(\wedge^{n-1}g_0) \otimes g_1 @>(-1)^nd >>\cdots @>(-1)^nd>> \cdots@>>> \cdots
\end{CD}$}
\end{center}
\end{definition}
Now we will define $n$-th  order Jacobi complex $J_n(\mathfrak{g})$. First we note that we can consider the direct sum of each components in the above whole complex as a subset $S_n$ of $S=\overline{S(\mathfrak{g}[1])}$ via $dec$. 

\begin{definition}
We define a differential $\bar{Q}$ on $\overline{S(\mathfrak{g}[1])}$ in a way that the following commutative diagram commutes.
\begin{center}
$\begin{CD}
\bar{x}_I@> dec >> (-1)^{A_I} x_I\\
@V \bar{Q} VV @V Q VV\\
\bar{Q}(\bar{x}_I) @>dec >> (-1)^{A_I} Q(x_I)
\end{CD}$
\end{center}
In other words, $\bar{Q}:=dec^{-1}\circ Q\circ dec$.
\end{definition}

\begin{definition}
 we define a differential $\bar{d}$ in a way that the following diagram commutes

\begin{center}
$\begin{CD}
\bar{x}_I @> dec >> (-1)^{A_I} x_I\\
@V \bar{d} VV @V (-1)^{|I|}d VV\\
\bar{d}(\bar{x}_I) @>dec >> (-1)^{A_I+|I|} d(x_I)
\end{CD}$
\end{center}
\end{definition}
We translate the above complex $(\bigwedge \mathfrak{g},d,Q)$ (un degree shifted complex in Definition \ref{2un}) in terms of $(\overline{S(\mathfrak{g}[1])},\bar{d},\bar{Q})$ via $dec$ to define Jacobi complex.
\begin{definition}[Jacobi complex]
Let $\mathfrak{g}$ be a differential graded Lie algebra. The $n$-th order Jacobi complex $J_n(\mathfrak{g})$ is the total complex of the following bicomplex $($here we have to be careful of the grading: we are working on $S_n\subset \overline{S(\mathfrak{g}[1])}$, hence here the grading of $g_i$ is actually $i-1)$
\begin{center}
\tiny{$\begin{CD}
g_0 @>\bar{d}>> g_1@>\bar{d}>>g_2 @>\bar{d}>> g_3  @>>>\cdots\\
@A\bar{Q}AA @A\bar{Q}AA @A\bar{Q}AA @A\bar{Q}AA \\ 
\bigwedge^2 g_0 @>\bar{d}>> (g_0\otimes g_1) @>\bar{d}>> (g_0\otimes g_2) \oplus sym^2g_1@>\bar{d}>> (g_0\otimes g_3)   \oplus (g_1\otimes g_2) @>>> \cdots \\
@A\bar{Q}AA @A\bar{Q}AA @A\bar{Q}AA @A\bar{Q}AA \\ 
\bigwedge^3 g_0 @>\bar{d}>> (\bigwedge^2 g_0)\otimes g_1 @>\bar{d}>> (\bigwedge^2 g_0 \otimes g_2)  \oplus(g_0\otimes sym^2 g_1) @>\bar{d}>> (\bigwedge^2 g_0\otimes g_3)\oplus (g_0\otimes g_1\otimes g_2) @>>> \cdots \\ 
@. @. @. \oplus sym^3 g_1 \\ 
@A\bar{Q}AA @A\bar{Q}AA @A\bar{Q}AA @A\bar{Q}AA \\
\bigwedge^4 g_0  @>\bar{d}>> (\bigwedge^3 g_0) \otimes g_1@>\bar{d}>> (\bigwedge^3 g_0\otimes g_2) @>\bar{d}>> (\bigwedge^3 g_0\otimes g_3) @>>>\cdots\\
@.@. \oplus (\bigwedge^2 g_0 \otimes sym^2 g_1)@. \oplus (\bigwedge^2 g_0\otimes g_1\otimes g_2) \\
@. @. @. \oplus (g_0\otimes sym^3 g_1)  \\ 
@A\bar{Q}AA @A\bar{Q}AA @A\bar{Q}AA @A\bar{Q}AA \\
\cdots @>>> \cdots @>>> \cdots @>>> \cdots \\
@A\bar{Q}AA @A\bar{Q}AA @A\bar{Q}AA @A\bar{Q}AA\\
\wedge^n g_0 @> \bar{d}>>(\wedge^{n-1}g_0) \otimes g_1 @>\bar{d} >>\cdots @>\bar{d}>> \cdots@>>> \cdots
\end{CD}$}
\end{center}

We give a bigrading on $S_n$ by $S_n^{-r,s}\cong_{dec} (\wedge^r \mathfrak{g})^s$. In other words, if $x\in S_n^{-r,s}$, then $x$ is of the form $x=\sum_i \bar{x}_{i_1}\odot \cdots \odot \bar{x}_{i_r}$ where $deg(x_{i_1})+\cdots +deg(x_{i_r})=s$ in $\mathfrak{g}$. Then we set $J_n(\mathfrak{g})^i=\bigoplus_{-r+s=i} S_n^{-r,s}$. For example, $J_n(\mathfrak{g})^0=g_1\oplus(g_0\otimes g_2)\oplus sym^2 g_1\oplus(\wedge^2 g_0\otimes g_3)\oplus (g_0\otimes g_1\otimes g_2)\oplus sym^3 g_1\oplus \cdots$. Then we define the Jacobi complex to be the following complex
\begin{align*}
\cdots \xrightarrow{\bar{d}+\bar{Q}} J_n(\mathfrak{g})^{i-1}\xrightarrow{\bar{d}+\bar{Q}} J_n(\mathfrak{g})^i\xrightarrow{\bar{d}+\bar{Q}} J_n(\mathfrak{g})^{i+1} \xrightarrow{\bar{d}+\bar{Q}}\cdots
\end{align*}

\end{definition}

 We also note natural inclusions $J_1(\mathfrak{g})\hookrightarrow J_2(\mathfrak{g})\hookrightarrow \cdots \hookrightarrow J_n(\mathfrak{g})\hookrightarrow \cdots$, which induces $\mathbb{H}^i(J_1(\mathfrak{g}))\to \mathbb{H}^i(J_2(\mathfrak{g}))\to \cdots \to \mathbb{H}^i(J_n(\mathfrak{g}))\to \cdots$.

\begin{definition}
We denote $i$-th cohomology group of the $n$-th order Jacobi complex associated with a differential graded Lie algebra $\mathfrak{g}$ by $\mathbb{H}^i(J_n(\mathfrak{g}))$. 
\end{definition}

\begin{remark}
We can modify the Jacobi complex $J_n(\mathfrak{g})$ by tensoring $\otimes \mathfrak{m}$ for some local artinian $\mathbb{C}$-algebra $(R,\mathfrak{m})$ with residue $\mathbb{C}$ to get $J_n(\mathfrak{g})\otimes \mathfrak{m}$. Then its cohomology groups coincide with $\mathbb{H}^i(J_n(\mathfrak{g}))\otimes \mathfrak{m}$.
\end{remark}

\begin{remark}
In practice, when we compute the Jacobi cohomology for our infinitesimal Poisson deformations, we use the first complex $($un-degree shifted complex in Definition $\ref{2un})$. In other words, we will work on $\overline{\bigwedge \mathfrak{g}}$ for actual computation of Jacobi cohomology groups. The reason why we pass from $\overline{\bigwedge\mathfrak{g}}$ to $\overline{S(\mathfrak{g}[1])}$ is that we want to give a commutative $\mathbb{C}$-algebra structure on $\mathbb{C}\oplus \mathbb{H}^0(J_n(\mathfrak{g}))^*$. We will explain this below.
\end{remark}

Let $\mathfrak{g}$ be a differential graded Lie algebra. Then $\overline{S(\mathfrak{g}[1])}$ has a symmetric coalgebra structure. Let $Q$ be the differential induced from the bracket $[-,-]$ as above. 

\begin{lemma}
We have $dec(\bar{Q}(\bar{x}_I)\odot \bar{x}_J)=(-1)^{A_I+A_J+|J||x_I|} Q(x_I)\wedge x_J$ and $dec(\bar{x}_I\odot \bar{Q}(\bar{x}_J))=(-1)^{A_I+A_J+|x_I|(|J|-1)}x_I\wedge Q(x_J).$
\end{lemma}
\begin{proof}
We simply note that $dec(\bar{x}_J)=(-1)^{A_J}x_J$, $dec\circ \bar{Q}(\bar{x}_I)=(-1)^{A_I}Q(x_I)$ and $dec\circ \bar{Q}(\bar{x}_J)=(-1)^{A_J}Q(x_J)$.
\end{proof}

By the definition of $\bar{Q}$ on $\overline{S(\mathfrak{g}[1])}$, we would like to show the following diagram commutes, (which means $\bar{Q}$ is a coderivation)

\begin{center}
$\begin{CD}
\overline{S(\mathfrak{g}[1])}@> \bar{Q} >> \overline{S(\mathfrak{g}[1])}\\
@V \triangle' VV @V \triangle' VV\\
\overline{S(\mathfrak{g}[1])}\otimes \overline{S(\mathfrak{g}[1])} @>\bar{Q}\otimes 1+1\otimes \bar{Q}>>\overline{S(\mathfrak{g}[1])}\otimes \overline{S(\mathfrak{g}[1])}
\end{CD}$
\end{center}

 Since $(\bar{Q}\otimes 1+1\otimes \bar{Q})\circ \triangle'(\bar{x}_1\odot\cdots \odot\bar{x}_n)=\sum_{I,J} (-1)^{S(I,J)}(\bar{Q}(\bar{x}_I)\otimes \bar{x}_J+(-1)^{|x_I|+|I|}\bar{x}_I\otimes \bar{Q}(\bar{x}_J))$, where $ \bar{x}_1\odot\cdots\odot\bar{x}_n=(-1)^{S(I,J)}\bar{x}_I\odot\bar{x}_J$.
First we note the following commutative diagram
\begin{center}
$\begin{CD}
\overline{S(g[1])} @>\bar{Q}>> \overline{S(g[1])} @>\triangle' >> \overline{S(g[1])}\otimes \overline{S(g[1])}\\
@V dec VV @V dec VV @V \tilde{dec} VV\\
\overline{\bigwedge g} @> Q >>  \overline{\bigwedge g} @> \triangle >> \overline{\bigwedge g} \otimes \overline{\bigwedge  g}
\end{CD}$
\end{center}

By this commutativity, we are working on $\overline{\bigwedge \mathfrak{g}}$ instead of $\overline{S(\mathfrak{g}[1])}$ to show that $\bar{Q}$ is a coderivation on $\overline{S(\mathfrak{g}[1])}$.

so our claim is equivalent to

\begin{equation}\label{2e}
\sum_{I,J} (-1)^{S(I,J)}\tilde{dec}(\bar{Q}(\bar{x}_I)\otimes \bar{x}_J+(-1)^{|x_I|+|I|}\bar{x}_I\otimes \bar{Q}(\bar{x}_J))=\Delta\circ Q\circ dec(\bar{x}_1\odot\cdots\odot \bar{x}_n)
\end{equation}

\begin{remark}
Via the above commutative diagram, for any $I$ and $J$, where $I \cup J=\{1,...,n\}$, $Q(x_I)\otimes x_J$ and $x_I\otimes Q(x_J)$ appear (up to sign) as terms in $\Delta\circ Q\circ dec(\bar{x}_1\odot\cdots\odot\bar{x}_n)$. Conversely, since
$Q(x_1\wedge \cdots \wedge x_n)=\sum (-1)^{P(i,j)}[x_i,x_j]\wedge x_1\wedge \cdots \wedge \hat{x}_i\wedge \cdots \wedge \hat{x}_j\cdots \wedge x_n$, $Q(x_I)\otimes x_J$  and $x_I\otimes Q(x_J)$ for all the pairs $I,J$ (up to sign) exhaust all the terms in $\Delta\circ Q\circ dec(\bar{x}_1\odot \cdots \odot \bar{x}_n)$. Hence in order to prove our claim $(\ref{2e})$, we only need to check that the signs for $Q(x_I)\otimes x_J$ and $x_I\otimes Q(x_J)$ in $\Delta\circ Q\circ dec(\bar{x}_1\odot\cdots \odot\bar{x}_n)$ equal to the signs for $Q(x_I)\otimes x_J$ and $x_I\otimes Q(x_J)$ in $\tilde{dec}(\bar{Q}(\bar{x}_I)\otimes \bar{x}_J)$ and $(-1)^{|x_I|+|I|}\tilde{dec}(\bar{x}_I\otimes\bar{Q}(\bar{x}_J))$.
\end{remark}


We now prove the claim (\ref{2e}). We note that $\bar{x}_I\odot \bar{x}_J=(-1)^{(|I|+|x_I|)(|J|+|x_J|)} \bar{x}_J\odot \bar{x}_I$

\begin{center}
$\begin{CD}
\bar{x}_I\odot \bar{x}_J@> dec >> (-1)^{A_I+A_J+|x_I||J|} x_I\wedge x_J\\
@V \bar{Q} VV @V Q VV\\
\bar{Q}(\bar{x}_I\odot \bar{x}_J) @>dec >> (-1)^{A_I+A_J+|x_I||J|} Q(x_I) \wedge x_J+\cdots
\end{CD}$
\end{center}

Hence $(-1)^{A_I+A_J+|x_I||J|} Q(x_I)\otimes x_J$ corresponds to $\bar{Q}(\bar{x}_I)\otimes \bar{x}_J$ via $\tilde{dec}$.

\begin{center}
$\begin{CD}
(-1)^{(|I|+|x_I|)(|J|+|x_J|)} \bar{x}_J\odot \bar{x}_I@> dec >> (-1)^{(|I|+|x_I|)(|J|+|x_J|)+A_I+A_J+|x_J||I|} x_J\wedge  x_I\\
@V \bar{Q} VV @V Q VV\\
\bar{Q}(\bar{x}_I\odot \bar{x}_J)=(-1)^{(|I|+|x_I|)(|J|+|x_J|)}\bar{Q}(\bar{x}_J\odot \bar{x}_I) @>dec >> (-1)^{(|I|+|x_I|)(|J|+|x_J|)+A_I+A_J+|x_J||I|} Q(x_J) \wedge  x_I +\cdots
\end{CD}$
\end{center}

We note that
\begin{align*}
(-1)^{(|I|+|x_I|)(|J|+|x_J|)+A_I+A_J+|x_J||I|} Q(x_J) \wedge  x_I &=(-1)^{(|I|+|x_I|)(|J|+|x_J|)+A_I+A_J+|x_J||I|+|I|(|J|-1)+|x_I||x_J|} x_I\wedge  Q(x_J)\\
                                                                                                    &=(-1)^{|x_I||J|+A_I+A_J+|I|} x_I\wedge Q(x_J)
\end{align*}

Hence $(-1)^{|x_I||J|+A_I+A_J+|I|} x_I\otimes Q(x_J)$  corresponds via $\tilde{dec}$ to 

\begin{align*}
(-1)^{|x_I||J|+A_I+A_J+|I|+ A_I+A_J+|x_I|(|J|-1)} \bar{x}_I\otimes \bar{Q}(\bar{x}_J)=(-1)^{|I|+|x_I|} \bar{x}_I\otimes \bar{Q}(\bar{x}_J)
\end{align*}
This completes the claim (\ref{2e}). Hence $\bar{Q}$ is a coderivation of $\overline{S(\mathfrak{g}[1])}$.

On the other hand, would like to show that

\begin{align*}
\bar{d}(\bar{x}_I\odot \bar{x}_J)=\bar{d}(\bar{x}_I)\odot \bar{x}_J +(-1)^{|I|+|x_I|}\bar{x}_I\odot \bar{d}(\bar{x}_J)
\end{align*} 
which implies that $\bar{d}$ defines a coderivation of $\overline{S(\mathfrak{g}[1])}$. In other words, the following commutative diagram holds

\begin{center}
$\begin{CD}
 \overline{S(\mathfrak{g}[1])}@> \bar{d} >> \overline{S(\mathfrak{g}[1])}\\
@V \triangle' VV @V \triangle' VV\\
\overline{S(\mathfrak{g}[1])}\otimes \overline{S(\mathfrak{g}[1])} @>\bar{d}\otimes 1+1\otimes \bar{d}>> \overline{S(\mathfrak{g}[1])}\otimes \overline{S(\mathfrak{g}[1])}
\end{CD}$
\end{center}

First we note the following relations

\begin{center}
$\begin{CD}
\bar{x}_I\odot \bar{x}_J@> dec >> (-1)^{A_I+A_J+|x_I||J|}x_I\wedge x_J\\
@V \bar{d} VV @V (-1)^{|I|+|J|}d VV\\
\bar{d}(\bar{x}_I\odot\bar{x}_J) @>dec >> (-1)^{A_I+A_J+|x_I||J|+|I|+|J|} d(x_I\wedge x_J)
\end{CD}$
\end{center}

\begin{align*}
\bar{d}(\bar{x}_I)\odot \bar{x}_J\xrightarrow{dec} (-1)^{A_I+|I|+A_J+|J|(|x_I|+1)}dx_I\wedge x_J
\end{align*}

\begin{align*}
\bar{x}_I\odot \bar{d}(\bar{x}_J) \xrightarrow{dec} (-1)^{A_I+A_J+|J|+|J||x_I|}x_I\wedge dx_J
\end{align*}

\begin{align*}
(-1)^{|I|+|x_I|}\bar{x}_I\odot \bar{d}(\bar{x}_J) \xrightarrow{dec} (-1)^{A_I+A_J+|J|+|J||x_I|+|I|+|x_I|}x_I\wedge dx_J
\end{align*}

\begin{align*}
(-1)^{A_I+A_J+|x_I||J|+|I|+|J|} d(x_I\wedge x_J)&=(-1)^{A_I+A_J+|x_I||J|+|I|+|J|} dx_I\wedge x_J+(-1)^{A_I+A_J+|x_I||J|+|I|+|J|+|x_I|} x_I\wedge dx_J\\
\end{align*}

Hence $\bar{d}$ defines a coderivation of $\overline{S(\mathfrak{g}[1])}$. 

In conclusion, $\bar{Q}'=\bar{d}+\bar{Q}$ is a coderivation of $\overline{S(\mathfrak{g}(1))}$. In other words, the following commutative diagram holds

\begin{center}
$\begin{CD}
 \overline{S(\mathfrak{g}[1])}@> \bar{Q}' >> \overline{S(\mathfrak{g}[1])}\\
@V \triangle' VV @V \triangle' VV\\
\overline{S(\mathfrak{g}[1])}\otimes \overline{S(\mathfrak{g}[1])} @>\bar{Q}'\otimes 1+1\otimes \bar{Q}'>> \overline{S(\mathfrak{g}[1])}\otimes \overline{S(\mathfrak{g}[1])}
\end{CD}$
\end{center}

\section{Morphic elements}\

We have a comultiplication map $\Delta'$ and coderivation $\bar{Q}'$ which is induced from differential and bracket from a differential graded Lie algebra $\mathfrak{g}$.
\begin{center}
$\begin{CD}
\overline{S(\mathfrak{g}[1])}@>\bar{Q}' >> \overline{S(\mathfrak{g}[1])}\\
@V\Delta' VV @VV\Delta' V \\
\overline{S(\mathfrak{g}[1])}\otimes \overline{S(\mathfrak{g}[1])}@>\bar{Q}'\otimes id+id \otimes \bar{Q}'>>\overline{S(\mathfrak{g}[1])}\otimes \overline{S(\mathfrak{g}[1])}
\end{CD}$
\end{center}

The following diagram commutes with the coderivation $\bar{Q}'$
\begin{center}
$\begin{CD}
\overline{S(\mathfrak{g}[1])}@>\Delta'>> \overline{S(\mathfrak{g}[1])}\otimes \overline{S(\mathfrak{g}[1])}\\
@V\Delta' VV @VV\Delta'\otimes id V \\
\overline{S(\mathfrak{g}[1])}\otimes \overline{S(\mathfrak{g}[1])}@>id\otimes \Delta'>>\overline{S(\mathfrak{g}[1])}\otimes \overline{S(\mathfrak{g}[1])} \otimes \overline{S(\mathfrak{g}[1])}
\end{CD}$
\end{center}
 Hence we have the following commutative diagram

\begin{center}
\[\begindc{\commdiag}[50]
\obj(0,2)[a]{$\mathbb{H}^0$}
\obj(3,2)[b]{$\bigoplus_{i+j=0} \mathbb{H}^i\otimes \mathbb{H}^j$}
\obj(0,1)[c]{$\bigoplus_{i+j=0} \mathbb{H}^i\otimes \mathbb{H}^j$}
\obj(0,0)[d]{$\mathbb{H}^0\otimes \mathbb{H}^0$}
\obj(3,1)[e]{$\bigoplus_{a+b+c=0}\mathbb{H}^a\otimes \mathbb{H}^b\otimes \mathbb{H}^c$}
\obj(6,2)[f]{$\mathbb{H}^0\otimes \mathbb{H}^0$}
\obj(6,0)[g]{$\mathbb{H}^0\otimes \mathbb{H}^0\otimes \mathbb{H}^0$}
\mor{a}{b}{}
\mor{a}{c}{}
\mor{c}{d}{}
\mor{c}{e}{}
\mor{b}{e}{}
\mor{b}{f}{}
\mor{f}{e}{}
\mor{d}{e}{}
\mor{e}{g}{}
\mor{f}{g}{}
\mor{d}{g}{}
\enddc\]
\end{center}

This induces a comultiplication map $\mathbb{H}^0(J_n(\mathfrak{g}))\to \sum_{i+j=0} \mathbb{H}^i(J_n(\mathfrak{g}))\otimes \mathbb{H}^j(J_n(\mathfrak{g}))\to \mathbb{H}^0(J_n(\mathfrak{g}))\otimes \mathbb{H}^0(J_n(\mathfrak{g}))$, where the last map is a projection.

And we have
\begin{center}
\[\begindc{\commdiag}[50]
\obj(0,1)[a]{$\overline{S(\mathfrak{g}[1])}$}
\obj(2,1)[b]{$\overline{S(\mathfrak{g}[1])}\otimes \overline{S(\mathfrak{g}[1])}$}
\obj(1,0)[d]{$\overline{S(\mathfrak{g}[1])}\otimes \overline{S(\mathfrak{g}[1])}$}
\mor{a}{b}{$\Delta'$}
\mor{a}{d}{$\Delta'$}
\mor{b}{d}{$\tau$}
\enddc\]
\end{center}
where $\tau(a\otimes b)=(-1)^{|a||b|}b\otimes a$.

This induces

\begin{center}
\[\begindc{\commdiag}[50]
\obj(0,2)[a]{$\mathbb{H}^0$}
\obj(0,1)[b]{$\bigoplus_{i+j=0} \mathbb{H}^i\otimes \mathbb{H}^j$}
\obj(2,2)[c]{$\bigoplus_{i+j=0} \mathbb{H}^i \otimes \mathbb{H}^j$}
\obj(0,0)[d]{$\mathbb{H}^0\otimes \mathbb{H}^0$}
\obj(4,2)[e]{$\mathbb{H}^0\otimes \mathbb{H}^0$}
\mor{a}{b}{}
\mor{a}{c}{}
\mor{c}{b}{}
\mor{b}{d}{}
\mor{c}{e}{}
\mor{e}{d}{}
\enddc\]
\end{center}

So the comultiplication map induces $\mathbb{H}^0(J_n(\mathfrak{g}))\to sym^2(\mathbb{H}^0(J_n(\mathfrak{g})))\cong sym^2(\mathbb{H}^0(J_n(\mathfrak{g}))^*)^*$, where $V^*$ is the dual space of a vector space $V$ over $\mathbb{C}$.

\begin{remark}
When we assume that $H^0(\mathfrak{g})=0$ which is the main assumption for universal Poisson deformation, we have $\mathbb{H}^i=0$ for $i<0$. In this case, we have simply $\bigoplus_{i+j=0} \mathbb{H}^i\otimes \mathbb{H}^j=\mathbb{H}^0\otimes \mathbb{H}^0$ and $\bigoplus_{a+b+c=0}\mathbb{H}^a\otimes \mathbb{H}^b\otimes \mathbb{H}^c=\mathbb{H}^0\otimes \mathbb{H}^0\otimes \mathbb{H}^0$.
\end{remark}

\begin{remark}
The commultiplication $\Delta'$ of $S_n\subset \overline{S(\mathfrak{g}[1])}$ induces a map 
\begin{align*}
*:\mathbb{H}^0(J_n(\mathfrak{g}))^*\times \mathbb{H}^0(J_n(\mathfrak{g}))^*\to \mathbb{H}^0(J_n(\mathfrak{g}))^*
\end{align*}
satisfying the associative and commutative laws. Hence $\mathbb{C}\oplus \mathbb{H}^0(J_n(\mathfrak{g}))^*$ is a commuative $\mathbb{C}$-algebra. Moreover, $\underbrace{\Delta'\circ \cdots \circ \Delta'}_{n+1}$ on $S_n$ is $0$. We have $\underbrace{\mathbb{H}^0(J_n(\mathfrak{g}))^**\cdots *\mathbb{H}^0(J_n(\mathfrak{g}))^*}_{n+1}=0$. 

In conclusion, set $\mathfrak{m}_n^u=\mathbb{H}^0(J_n(\mathfrak{g}))^*$. Then $\mathbb{C}\oplus \mathfrak{m}_n^u$ is a local commutative $\mathbb{C}$-algebra with the maximal ideal $\mathfrak{m}_n^u$ and residue $\mathbb{C}$ such that $\mathfrak{m}_n^{u\,n+1}=0$. We note that for our infinitesimal Poisson deformation, $\mathfrak{m}_n^u$ is finite dimensional. Hence $\mathbb{C}\oplus \mathfrak{m}_n^u$ is artinian. This proves our main Theorem $\ref{2theorem}$ $(1)$ in the Introduction of the part II of the thesis.
\end{remark}

\subsection{Morphic elements and ring homomorphism}


\begin{lemma}\label{2le}
Let $A\xrightarrow{v} B\xrightarrow{w} C$ be a complex of vector spaces over $\mathbb{C}$ and $C^*\xrightarrow{w^*} B^*\xrightarrow{v^*} A^*$ be induced complex of dual vector spaces. Then $[b_1]=[b_2]\in H(B)$ and $[f]=[g]\in H(B^*)$, where $[t]$ is a cohomology class for $t\in B$ or $t\in B^*$. Then $f(b_1)=g(b_2)$.
\end{lemma}

\begin{proof}
First we note that $w(b_1)=w(b_2)=0$ and $v^*(f)=f\circ v=0, v^*(g)= g\circ v=0$. Let $v(a)=b_1-b_2$ and $w^*(h)=h\circ w=f-g$. $f(b_1)-g(b_2)=f(b_1)-g(b_1)+g(b_1)-g(b_2)=(f-g)(b_1)+g(b_1-b_2)=h\circ w(b_1)+g\circ v(a)=0$
\end{proof}

Let $[\phi]\in \mathbb{H}^0(J_n(\mathfrak{g}))\otimes \mathfrak{m}$ where $\mathfrak{m}$ is the maximal ideal of some artinian local $\mathbb{C}$-algebra $(R,\mathfrak{m})$ with residue $\mathbb{C}$ and $\mathfrak{m}^{n+1}=0$. Then this $[\phi]$ defines a linear map $f:\mathbb{H}^0(J_n(\mathfrak{g}))^*\to \mathfrak{m}$ by $a\to a(\phi)$. By Lemma \ref{2le}, the linear map $f$ is independent of choices of $a$ and $\phi$. Let's consider a bilinear map $\mathbb{H}^0(J_n(\mathfrak{g}))^*\times \mathbb{H}^0(J_n(\mathfrak{g}))^*\to \mathfrak{m}^2$ defined by $(a,b)\mapsto f(a)f(b)=a(\phi)b(\phi)$. Since $\mathfrak{m}$ is commutative, this induces a map
\begin{align*}
f\times f:sym^2(\mathbb{H}^0(J_n(\mathfrak{g}))^*)\to \mathfrak{m}^2
\end{align*}

\begin{definition}[morphic elements]
We call $[\phi]\in\mathbb{H}^0(J_n(\mathfrak{g}))^*$ a morphic element if $f\times f$ defines a ring homormophism.
\end{definition}

In order for $f\times f$ to be a ring homomorphism (equivalently for $[\phi]$ to be a morphic element), we have to have $f(a) f(b)=f(a \cdot b)$ where $\cdot$ is induced from $\Delta':\mathbb{H}^0(J_n(\mathfrak{g}))\to sym^2 (\mathbb{H}^0(J_n(\mathfrak{g}))^*)^*$. More pricesly, by taking the dual of $\Delta'$, we have the map $\cdot:sym^2(\mathbb{H}^0(J_n(\mathfrak{g}))^*) \cong sym^2(\mathbb{H}^0(J_n(\mathfrak{g})))^* \to \mathbb{H}^0(J_n(\mathfrak{g}))^*$ by $a\cdot b=(a\otimes b)\circ \Delta'$. Hence $f(a)f(b)=f(a\cdot b)$ means that

\begin{align*}
(a\otimes b)(\phi\otimes \phi)=a(\phi)b(\phi)=(a \otimes b)(\Delta'(\phi))=(b\otimes a)(\Delta'(\phi))
\end{align*}

Hence $(a\otimes b)(\Delta'(\phi)-\phi\otimes \phi)=0$ for all $a,b\in \mathbb{H}^0(J_n(\mathfrak{g}))^*$. This implies that $\Delta'(\phi)=\phi \otimes \phi$. Hence for any morphic element $[\phi]\in \mathbb{H}^0(J_n(\mathfrak{g}))\otimes \mathfrak{m}$, we have $\Delta'(\phi)=\phi\otimes \phi$.

\subsection{Explicit description of a morphic element}\

We describe a morphic element $v$ in $\mathbb{H}^0(J_n(\mathfrak{g}))\otimes \mathfrak{m}$ for some local artinian $\mathbb{C}$-algebra $(R,\mathfrak{m})$ with residue $\mathbb{C}$ and $\mathfrak{m}^{n+1}=0$. When we consider the Jacobi bicomplex, a $0$-th cohomology class $v$ is of the form
\begin{align*}
v=v_1+\cdots +v_n
\end{align*}
where $v_i\in  S^i(\mathfrak{g}[1])\otimes \mathfrak{m}$. In particular $v_1\in g_1\otimes \mathfrak{m}$. For $v$ to be a morphic element, we have to have $\Delta'(v)=v\otimes v$ where $\Delta'$ is the comultiplication map. So $\sum \Delta'(v_i)=\sum v_i\otimes v_j$. So we have the following relations
\begin{align*}
\Delta'(v_1)&=0\\
\Delta'(v_2)&=v_1\otimes v_1\\
\Delta'(v_3)&=v_1\otimes v_2+v_2\otimes v_1\\
\Delta'(v_4)&=v_1\otimes v_3+v_2\otimes v_2+v_3\otimes v_1\\
\cdots\\
\Delta'(v_n)&=v_1\otimes v_{n-1}+\cdots +v_{n-1}\otimes v_1
\end{align*}
So $v_i$ is determined by $v_1,...,v_{i-1}$ inductively, hence completely determined by $v_1$. Since $\Delta'(v_2)=v_1\otimes v_1$, we see that $v_2=\frac{1}{2}v_1\odot v_1$. Since $\Delta'(v_3)=\frac{1}{2}v_1\otimes (v_1\odot v_1)+\frac{1}{2} (v_1\odot v_1)\otimes v_1$, we have $v_3=\frac{1}{3!}v_1\odot v_1\odot v_1$. Inductively, since $\Delta'(v_n)=\frac{1}{(n-1)!}v_1\otimes (v_1\odot \cdots \odot v_1)+\frac{1}{2!(n-2)!} (v_1\odot v_1)\otimes(v_1\odot \cdots \odot v_n)+\frac{1}{3!(n-3)!}(v_1\odot v_1\odot v_1)\otimes (v_1\odot \cdots \odot v_1)+\cdots +\frac{1}{(n-2)!2!}(v_1\odot\cdots \odot v_1)\otimes (v_1\odot v_1)+\frac{1}{(n-1)!}(v_1\odot \cdots \odot v_1)\otimes v_1$, and $\binom{n}{i}=\frac{n!}{i!(n-i)!}$, we have $v_n=\frac{1}{n!}v_1\odot \cdots \odot v_1$. Hence a morphic element is of the form
\begin{align*}
v_1+\frac{1}{2!}v_1\odot v_1+\cdots+\frac{1}{n!}v_1\odot \cdots\odot v_1
\end{align*}
where $\frac{1}{i!}\underbrace{v_1\odot \cdots \odot v_1}_i\in sym^i g_1\otimes \mathfrak{m}^i$.

\chapter{Infinitesimal Poisson deformations}\label{chapter5}

\section{Bracket calculus}\

Let $(\mathfrak{g}=\bigoplus_i g_i,\bar{\partial},[-,-])$ be a differential graded Lie algebra. Let $R$ is a local artinian $\mathbb{C}$-algebra with residue $\mathbb{C}$.
Then $\mathfrak{g}\otimes R$ is a differential graded Lie algebra with $(\mathfrak{g}\otimes R)_i=g_i\otimes R$, $\bar{\partial}(a\otimes r)=\bar{\partial}a\otimes r$ and $[a\otimes r_1,b\otimes r_2]=[a,b]\otimes r_1r_2$ for $a,b\in \mathfrak{g}$. Let $X$ be a complex manifold.  Since $ (\bigoplus_{p+q-1=i, p\geq 0, q\geq 1} A^{0,p}(X, \wedge^q T),\bar{\partial}, [-,-])$ is a differential graded Lie algebra (see Appendix \ref{appendixc}), this induces a differential graded Lie algebra structure on $\bigoplus_{p+q-1=i,p\geq1, q\geq 0} A^{0,p}(X,\wedge^q T)\otimes_{\mathbb{C}} R$.

\begin{definition}
Let $X$ be a compact complex manifold. We consider any element of the differential graded Lie algebra $A=\bigoplus_{p+q-1=i,p\geq 0,q\geq 1} A^{0,p}(X,\wedge^q T)\otimes_{\mathbb{C}} R$ as an operator acting on itself in the following way. Let $\psi\in A$. We define
\begin{align*}
\psi=[\psi,-]:A&\to A\\
a &\mapsto [\psi,a]
\end{align*}
\end{definition}

\begin{definition}
Let $\bar{\partial}$ be the differential as an operator acting on $A$ in $\mathfrak{g}$. We formally define that $[\bar{\partial},a]:=\bar{\partial}a$, i.e 
\begin{align*}
\bar{\partial}:=[\bar{\partial},-]:A&\to A\\
                    a&\mapsto [\bar{\partial},a]:=\bar{\partial}a
\end{align*}
\end{definition}

\begin{proposition}\label{2pro}
As an operator acting on $A$, we have
$\bar{\partial}a=\bar{\partial}\circ a-(-1)^{|a|} a\circ \bar{\partial}$
\end{proposition}

\begin{proof}
$[\bar{\partial},[a,b]]=\bar{\partial}[a,b]=[\bar{\partial}a,b]+(-1)^{|a|}[a,\bar{\partial}b]$
\end{proof}

\begin{proposition}\label{2pk}
As an operator acting on $A$, we have $[a,b]=a\circ b-(-1)^{|a||b|} b\circ a$
\end{proposition}

\begin{proof}
$[[a,b],c]=[a,[b,c]]-(-1)^{|a||b|}[b,[a,c]]$.
\end{proof}

\begin{definition}
We would like to define a differential $($which we will denote by still same $\bar{\partial}$$)$ which generalizes Proposition $\ref{2pro}$ on the space of operators on $L_A$ generated by $\{L_a=[a,-]|a\in A\}$, $identity$, and $\bar{\partial}:=[\bar{\partial},-]$ in compository manner. So $L_A$ has an identity element, which we denote by $id$. We set $\deg(\bar{\partial})=1$ and $deg(id)=0$. We define $\bar{\partial} X:= \bar{\partial}\circ X-(-1)^{|X|}X\circ \bar{\partial}$ for $X\in L_A$. Then $\bar{\partial}(id)=0$ and $\bar{\partial}(\bar{\partial})=\bar{\partial}([\bar{\partial},-])=0$$($First $\bar{\partial}$ is the differential on $L$ and second $\bar{\partial}$ is an operator in $L_A$$)$. We would like to define the grading on $L_A$ in the following way: $deg(a_1\circ\cdots \circ a_n)=|a_1|+\cdots +|a_n|$ for $a_i\in A$ as operators where $|a|$ is the degree of $a$ in $A$. Then $|\bar{\partial}X|=|X|+1$ since the degree of $|\bar{\partial}|=1$ and $\bar{\partial}(\bar{\partial}X)=0$ for $X\in L_A$. We define the bracket on $L_A$ in the following way: $[X,Y]=X\circ Y-(-1)^{|X||Y|}Y\circ X$.
\end{definition}

So we have $[\bar{\partial},a]=\bar{\partial} a=\bar{\partial}\circ a-(-1)^{|a|}a\circ \bar{\partial}$ and $[a,b]=a\circ b-(-1)^{|a||b|}b\circ a$ for $a,b\in A$ as operators. This coincides with Proposition \ref{2pro} and Proposition \ref{2pk}.

\begin{remark}
We have an embedding 
\begin{align*}
A&\hookrightarrow L_A\\
a&\mapsto a:=[a,-]
\end{align*}
which is bracket $([-,-])$ preserving, and differential $(\bar{\partial})$ preserving, in other words, 
\begin{align*}
[a,b]&\mapsto [[a,b],-]=[[a,[b,-]]-(-1)^{|a||b|}[b,[a,-]]=a\circ b-(-1)^{|a||b|}b\circ a\\
\bar{\partial}a &\mapsto [\bar{\partial}a,-]=\bar{\partial}\circ a-(-1)^{|a|}a\circ \bar{\partial}
\end{align*}
\end{remark}

\begin{proposition}[Product Rule]\label{2p}
For $X,Y\in L_A$, we have $\bar{\partial}(X\circ Y)=\bar{\partial}X\circ Y+(-1)^{|X|}X\circ \bar{\partial}Y.$
\end{proposition}

\begin{proof}
$\bar{\partial}(X\circ Y)=\bar{\partial}\circ X\circ Y-(-1)^{|X|+|Y|}X\circ Y\circ \bar{\partial}$. On the other hand $\bar{\partial}X\circ Y=\bar{\partial}\circ X\circ Y-(-1)^{|X|}X\circ \bar{\partial}\circ Y$, and $(-1)^{|X|}X\circ \bar{\partial}Y=(-1)^{|X|}X\circ \bar{\partial}\circ Y-(-1)^{|X|+|Y|}X\circ Y\circ \bar{\partial}$.
\end{proof}

\begin{example}
$exp(tu)\circ u=u\circ exp(tu), t\in \mathbb{R}$. Here $exp(x)=id+x+\frac{x\circ x}{2!}+\cdots$ for $x\in L_A$.
\end{example}

\begin{example}
$\frac{d}{dt} exp(tu)=exp(tu)u$ for $t\in \mathbb{R}$. Indeed, $\frac{d}{dt}exp(tu)=\frac{d}{dt}(id+tu+\frac{(tu)\circ (tu)}{2!}+\cdots)=\frac{d}{dt}(id+tu+\frac{1}{2}t^2 u\circ u+\cdots)=u+tu\circ u+\frac{1}{2}t^2u\circ u\circ u+\cdots=u(id+tu+\frac{1}{2}tu\circ tu+\cdots)=uexp(tu)$
\end{example}

\begin{example}
If $a\in A$ is holomorphic, then as a operator $\bar{\partial}(exp(a))=0$. Indeed, since $\bar{\partial}a=0$, we have $\bar{\partial}(exp(a))=\bar{\partial}(id+a+\frac{1}{2}a\circ a+\cdots)=0$ by applying the product rule.
\end{example}

\begin{example}
Let $deg(a)=0$. Then as an operator 
\begin{align*}
exp(a)\circ \bar{\partial}(exp(a))=[D(ad(a))(\bar{\partial}a)),-]=[\bar{\partial}a-\frac{[a,\bar{\partial}a]}{2!}+\frac{[a,[a,\bar{\partial}a]]}{3!}+\cdots,-]
\end{align*}
where
\begin{align*}
D(x)=\frac{exp(x)-1}{x}=\sum_{i=0}^{\infty} \frac{x^i}{(i+1)!}
\end{align*}

\end{example}

\section{Deformations of a compact holomorphic Poisson manifold}

Let $(X,\Lambda_0)$ be a compact holomorphic Poisson manifold. In other words, the structure sheaf $\mathcal{O}_X$ is a sheaf of Poisson algebras induced by $\Lambda_0$. We define an infinitesimal Poisson deformation of $(X,\Lambda_0)$ over $Spec\,R$, where $R$ is a local artinian $\mathbb{C}$-algebra with residue $\mathbb{C}$.
\begin{definition}
Let $(X,\Lambda_0)$ be a compact holomorphic Poisson manifold. An infinitesimal Poisson deformation of $(X,\Lambda_0)$ is a cartesian diagram

\begin{center}
$\begin{CD}
X@>i>> \mathcal{X}\\
@VVV @VV\pi V\\
Spec(\mathbb{C})@>>> Spec(R)
\end{CD}$
\end{center}
where $\pi$ is a proper flat morphism, $A$ is an local artinian $\mathbb{C}$-algebra with the residue $\mathbb{C}$, $X\cong \mathcal{X}\times_{Spec(A)} Spec(\mathbb{C})$, and $\mathcal{O}_{\mathcal{X}}$ is a sheaf of Poisson $R$-algebra, which induces the Poisson $\mathbb{C}$-algebra $\mathcal{O}_X$ given by $\Lambda_0$.
\end{definition}

\begin{remark}
When we ignore Poisson structures, an infinitesimal Poisson deformation is simply a infinitesimal deformation of a underlying compact complex manifold. As complex analytic spaces, $X$ is a closed subspace of $\mathcal{X}$. Since $Spec(R)$ is a fat point $($one point set with structure sheaf $R$ itself $)$ and $\mathcal{X}\times_{Spec(A)} Spec(\mathbb{C})$, we have $\mathcal{X}\cong X$ topologically.
\end{remark}

\begin{remark}
We can assume that $i:X\to \mathcal{X}$ is an identity map as topological spaces. In other words, $i$ is simply equivalent to that $\mathcal{O}_{\mathcal{X}}(U)\to \mathcal{O}_X(U)$ is surjective for any open set $U$ of $X$. In the sequel, I assume that $i$ is identity as topological spaces.
\end{remark}

\begin{remark}
The cartesian diagram can be seen in the sheaf theoretic setting which we mainly use in the part II of the thesis. The cartesian diagram is equivalent to the following cartesian diagram of sheaves
\begin{center}
$\begin{CD}
\mathcal{O}_{\mathcal{X}}@>>> \mathcal{O}_X\\
@AAA @AAA\\
R@>>> R/\mathfrak{m}\cong \mathbb{C}
\end{CD}$
\end{center}
 where a sheaf morphism $\mathcal{O}_{\mathcal{X}}\to \mathcal{O}_X$ over $\mathbb{C}$ on $X$ such that $\mathcal{O}_{\mathcal{X}}$ is a sheaf of flat Poisson $A$-algebra with $\mathcal{O}_{\mathcal{X}}\otimes_R  R/\mathfrak{m}\cong \mathcal{O}_X$ as sheaves of Poisson $\mathbb{C}$-algebras in the following sense:
for any open set $U$ of $X$, the Poisson bracket $\{-,-\}_{\mathcal{X}}:\mathcal{O}_{\mathcal{X}}(U)\times \mathcal{O}_{\mathcal{X}}(U)\to \mathcal{O}_{\mathcal{X}}(U)$ induces the Poisson bracket on $\mathcal{O}_{X}(U)$ in the following way.  $\{-,-\}_X:(\mathcal{O}_{\mathcal{X}}(U)\otimes_{\mathbb{C}} R/\mathfrak{m}) \times (\mathcal{O}_{\mathcal{X}}(U)\otimes_{\mathbb{C}} R/\mathfrak{m}) \to \mathcal{O}_{\mathcal{X}}(U)\otimes_{\mathbb{C}} R/\mathfrak{m}$ by $\{f\otimes r_1,g\otimes r_2 \}_X=\{f,g\}_{\mathcal{X}}\otimes r_1r_2$. 
\end{remark}

\begin{proposition}\label{2proq}
Let $U$ be an open ball containing $0\in \mathbb{C}^n$ with a coordinate $(z_1,...,z_n)$. An infinitesimal deformation $\mathcal{X}$ of $U$ over $Spec(R)$ is rigid. In other words, $\mathcal{X}\cong U \times Spec(R)$.
\end{proposition}

\begin{proof}
We refer to \cite{Ser06} Theorem 1.2.4.
\end{proof}

Let $\mathcal{X}$ be an infinitesimal Poisson deformation of a compact holomorphic Poisson manifold $(X,\Lambda_0)$ over a local artinian $\mathbb{C}$-algebra $(R,\mathfrak{m})$ with residue $\mathbb{C}$. Let $\{U_i\}$ be an open covering of $X$, where $U_i$ is isomorphic to an open ball. By Proposition \ref{2proq}, locally an infinitesimal deformation $\mathcal{X}$ of the underlying complex manifold $X$ is 
\begin{center}
$\begin{CD}
U_i@>>> \mathcal{X}|_{U_i}\cong U_i\times Spec(R)\\
@VVV @VV\pi V\\
Spec(\mathbb{C})@>>> Spec(R)
\end{CD}$
\end{center} 
We call such an open cover $\{U_i\}$  locally trivial  open covering of $\mathcal{X}$.
\begin{remark}
Let $\{U_i\}$ be a locally trivial open covering of $\mathcal{X}$. Since $\mathcal{O}_{\mathcal{X}}(U_i)\cong \mathcal{O}_X(U_i)\otimes_{\mathbb{C}} R$, the Poisson $R$-algebra structure on $\mathcal{O}_{\mathcal{X}}(U_i)$ is encoded in some holomorphic bivector field $\Lambda_i' \in \Gamma(U_i,\mathscr{A}^{0,0}(\wedge^2 T_X))\otimes_{\mathbb{C}} R$ with $\bar{\partial} \Lambda_i'=0$, $[\Lambda_i',\Lambda_i']=0$, where $\bar{\partial}$ on $\Gamma(U_i,\mathscr{A}^{0,0}(\wedge^2 T_X))\otimes_{\mathbb{C}} R$ is induced from $\bar{\partial}$ on $\Gamma(U_i,\mathscr{A}^{0,0}(\wedge^2 T_X))$ and $[-,-]$ is induced from the bracket  on $U_i$ by extending $R$-linearly. More precisely, let $U_i$ be an open ball with coordinate $z=(z_1,...,z_n)$. Then the Poisson structure on $\mathcal{O}_\mathcal{X}(U_i)$ over $R$, where $R$ is generated by $<1, m_1,...,m_r>$ over $\mathbb{C}$ is encoded in the bivector field 

\begin{align*}
\Lambda_i'=\sum_{\alpha,\beta} \Lambda_{\alpha\beta} \frac{\partial}{\partial z_{\alpha}}\wedge \frac{\partial}{\partial z_{\beta}}= \sum_{\alpha,\beta} (\Lambda_{\alpha\beta}^0+m_1\Lambda_{\alpha\beta}^1+\cdots m_r\Lambda_{\alpha\beta}^r)\frac{\partial}{\partial z_{\alpha}}\wedge \frac{\partial}{\partial z_{\beta}}
\end{align*}
with $[\Lambda_i',\Lambda_i']=0$, where $\Lambda_{\alpha\beta}^k(z)$ is holomorphic on $U_i$ for each $k$. Then

\begin{align*}
\{f,g\}=<\Lambda_i',df\wedge dg>=<\sum_{\alpha,\beta} \Lambda_{\alpha\beta}\frac{\partial}{\partial z_{\alpha}} \wedge \frac{\partial}{\partial z_{\beta}}, df\wedge dg>=\sum_{\alpha,\beta} \Lambda_{\alpha\beta}\left(\frac{\partial f}{\partial z_{\alpha}}\frac{\partial g}{\partial z_{\beta}}-\frac{\partial g}{\partial z_{\alpha}}\frac{\partial f}{\partial z_{\beta}}\right)
\end{align*}

On the other hand,

\begin{align*}
[[\Lambda_i',f],g]&=\sum_{\alpha,\beta} [[\Lambda_{\alpha\beta}\frac{\partial}{\partial z_{\alpha}}\wedge \frac{\partial}{\partial z_{\beta}},f],g]=\sum_{\alpha,\beta}[\Lambda_{\alpha\beta}\frac{\partial f}{\partial z_{\alpha}}\frac{\partial}{\partial z_{\beta}}-\Lambda_{\alpha\beta}\frac{\partial f}{\partial z_{\beta}}\frac{\partial}{\partial z_{\alpha}},g]\\
&=\sum_{\alpha,\beta}\Lambda_{\alpha\beta} \left(\frac{\partial f}{\partial z_{\alpha}}\frac{\partial g}{\partial z_{\beta}}-\frac{\partial g}{\partial z_{\alpha}}\frac{\partial f}{\partial z_{\beta}}\right)
\end{align*}

So we have the expression

\begin{align*}
\{f,g\}=[[\Lambda_i',f],g]
\end{align*}
Note that since $\Lambda_i'$ induces the Poisson structure of $(U_i,\Lambda_0)$, the Poisson structure of $(X,\Lambda_0)$ on $U_i$ is given by
\begin{align*}
\Lambda_0=\sum_{\alpha\beta}\Lambda_{\alpha\beta}^0\frac{\partial}{\partial z_{\alpha}}\wedge \frac{\partial}{\partial z_{\beta}}
\end{align*}

\end{remark}

\begin{definition}[equivalent Poisson deformations]
We say that an infinitesimal Poisson deformation $\mathcal{X}$ and an infinitesimal Poisson deformation $\mathcal{Y}$ of a holmorphic Poisson manifold $(X,\Lambda_0)$ over a local artinian $\mathbb{C}$-algebra $R$ are equivalent if the following diagram is commutative.
\begin{center}
\[\begindc{\commdiag}[50]
\obj(1,2)[aa]{$(X,\Lambda_0)$}
\obj(0,1)[bb]{$\mathcal{X}$}
\obj(2,1)[cc]{$\mathcal{Y}$}
\obj(1,0)[dd]{$Spec\,R$}
\mor{aa}{bb}{}
\mor{aa}{cc}{}
\mor{bb}{cc}{$f \cong $}
\mor{bb}{dd}{}
\mor{cc}{dd}{}
\enddc\]
\end{center}
where $f:\mathcal{X}\to \mathcal{Y}$ is a Poisson isomophism. In other words, $f^\sharp:\mathcal{O}_{\mathcal{Y}}\to f_*\mathcal{O}_{\mathcal{X}}$ is an isomorphism of Poisson sheaves.
\end{definition}

\section{Integrability condition}\
 
 Let $\mathcal{X}$ be a infinitesimal Poisson deformation of a compact holomorphic Poisson manifold $(X,\Lambda_0)$ over $R$, where $(R,\mathfrak{m})$ is an local artinian $\mathbb{C}$-algebra with residue $\mathbb{C}$.
 
Let $\mathcal{U}=\{U_{\alpha}\}$ be a locally trivial open covering of $\mathcal{X}$. Then we have a set of isomorphisms of Poisson $R$-algebra
\begin{align*}
\varphi_{\alpha}:\mathcal{O}_{\mathcal{X}}(U_{\alpha})\xrightarrow{\thicksim} \mathcal{O}_X(U_{\alpha})\otimes R
\end{align*}
where the Poisson $R$-structure on $\mathcal{O}_X(U_{\alpha})\otimes R$, which we denote by $\Lambda_{\alpha}'$, is induced from the Poisson structure of $\mathcal{O}_{\mathcal{X}}(U_{\alpha})$ via $\varphi_{\alpha}$. 
\begin{align*}
\theta_{\alpha\beta}:=\varphi_{\alpha}\varphi_{\beta}^{-1}:(\mathcal{O}_X(U_{\alpha}\cap U_{\beta})\otimes R,\Lambda_{\beta}')\to (\mathcal{O}_X(U_{\alpha}\cap U_{\beta})\otimes R,\Lambda_{\alpha}')
\end{align*}
are Poisson isomorphisms and satisfying cocycle condition $\theta_{\gamma\alpha}\circ \theta_{\alpha\beta}=\theta_{\gamma\beta}$.
\begin{lemma}\label{2len}
In the above, if $(R,\mathfrak{m})$ has the exponent $n$ (i.e $\mathfrak{m}^{n+1}=0)$, then $\theta_{\alpha\beta}=exp(t_{\alpha\beta})=I+t_{\alpha\beta}+\frac{1}{2}(t_{\alpha\beta})^2+\cdots +\frac{1}{n!}(t_{\alpha\beta})^n$ for some $t_{\alpha\beta}\in C^1(\mathcal{U},\Theta_X)\otimes \mathfrak{m}$.
\end{lemma}

\begin{proof}
We prove this by induction on the exponent $n$ of the maximal ideal $\mathfrak{m}$ of a local artinian $\mathbb{C}$-algebra $(R,\mathfrak{m})$. 

Let's prove that the proposition holds for $n=1$. We assume that $\mathfrak{m}^2=0$. Let $f:=f\otimes 1\in \mathcal{O}_X(U_{\alpha}\cap U_{\beta})\otimes R$, which we consider an element in $\mathcal{O}_X(U_{\alpha}\cap U_{\beta})$. Note that $\theta_{\alpha\beta}$ is completely determined by $\mathcal{O}_X(U_{\alpha}\cap U_{\beta})$ since $\theta$ is $R$-algebra map. Let $dim_{\mathbb{C}}\, \mathfrak{m}=r$ with $\mathfrak{m}=<e_1,...,e_r>$. Since $R=\mathbb{C}\oplus \mathfrak{m}$, $\theta_{\alpha\beta}(f)=f+e_1 h_1(f)+\cdots +e_r h_r(f)$. Then $t_{\alpha\beta}:=\sum_i e_ih_i$ is a derivation. Indeed, for $f:=f\otimes 1, g:=g\otimes 1\in \mathcal{O}_X(U_{\alpha}\cap U_{\beta})\otimes R$, we have $\theta_{\alpha\beta}(fg)=\theta_{\alpha\beta}(f)\theta_{\alpha\beta}(g)=fg+e_1h_1(fg)+\cdots+e_rh_r(fg)=(f+e_1h_1(f)+\cdots +e_rh_r(f))(g+e_1h_1(g)+ \cdots + e_rh_r(g))=fg+e_1(fh_1(g)+gh_1(f))+\cdots + e_r(fh_r(g)+gh_r(f))$ since $\mathfrak{m}^2=0$, so we have $\{t_{\alpha\beta}=\sum_i h_ie_i\}\in C^1(\mathcal{U},\Theta_X)\otimes \mathfrak{m}$ and $\theta_{\alpha\beta}=exp(t_{\alpha\beta})=exp(\sum h_ie_i)$. 

Now we assume that the proposition holds for up to $n-1$. Let $(R,\mathfrak{m})$ with $\mathfrak{m}^{n+1}=0$. Then $(R/\mathfrak{m}^n,\mathfrak{m}/\mathfrak{m}^{n})$ is a local artinian $\mathbb{C}$-algebra with residue $\mathbb{C}$ and exponent $n-1$.
We note that $\theta_{\alpha\beta}:\mathcal{O}_X(U_{\alpha}\cap U_{\beta})\otimes R \to \mathcal{O}_X(U_{\alpha}\cap U_{\beta})\otimes R$ induces $\bar{\theta}_{\alpha\beta}:\mathcal{O}(U_{\alpha}\cap U_{\beta})\otimes (R/\mathfrak{m}^n) \to \mathcal{O}(U_{\alpha}\cap U_{\beta})\otimes (R/\mathfrak{m}^n)$. Then by the induction hypothesis, $\bar{\theta}_{\alpha\beta}=exp(\bar{t}_{\alpha\beta})$ for some $\bar{t}_{\alpha\beta}\in C^1(\mathcal{U},\Theta_X)\otimes (\mathfrak{m}/\mathfrak{m}^n)$.
Let $t_{\alpha\beta}$ be an arbitrary lifting of $\bar{t}_{\alpha\beta}$ via $\mathfrak{m}\to \mathfrak{m}/\mathfrak{m}^n$. Let $\eta_{\alpha\beta}:=\theta_{\alpha\beta}-exp(t_{\alpha\beta})$. We claim that $\{\eta_{\alpha\beta}\} \in C^1(\mathcal{U},\Theta_X)\otimes \mathfrak{m}^n$. Indeed, let $f,g$ as above. Then $\eta_{\alpha\beta}(fg)=\theta_{\alpha\beta}(f)\theta_{\alpha\beta}(g)-(exp(t_{\alpha\beta})f)(exp(t_{\alpha\beta})g)=\theta_{\alpha\beta}(f)\theta_{\alpha\beta}(g)-\theta_{\alpha\beta}(f)(exp(t_{\alpha\beta})g)+\theta_{\alpha\beta}(f)(exp(t_{\alpha\beta})g)-(exp(t_{\alpha\beta})f)(exp(t_{\alpha\beta})g)=\theta_{\alpha\beta}(f)(\theta_{\alpha\beta}-exp(t_{\alpha\beta}))(g)+(exp(t_{\alpha\beta})g)(\theta_{\alpha\beta}-exp(t_{\alpha\beta}))(f)=f\eta_{\alpha\beta}(g)+g\eta_{\alpha\beta}(f)$ since $\eta_{\alpha\beta}(g),\eta_{\alpha\beta}(f)\in \mathcal{O}(U_{\alpha}\cap U_{\beta})\otimes \mathfrak{m}^n$ and $\mathfrak{m}^{n+1}=0$. So $\eta_{\alpha\beta}\in C^1(\mathcal{U},\Theta_X)\otimes \mathfrak{m}^n$. Hence we have 
\begin{align*}
\theta_{\alpha\beta}=exp(t_{\alpha\beta})+\eta_{\alpha\beta}=exp(t_{\alpha\beta}+\eta_{\alpha\beta})
\end{align*}
\end{proof}

\begin{proposition}[Compare \cite{Kod05} Theorem 2.4 page 64]\label{2prot}
Let $\mathcal{X}\to Spec\,R$ be an infinitesimal deformation of a compact complex manifold $X$. Then we have an isomorphism of sheaves $\mathscr{A}^{0,0}_{\mathcal{X}}\cong \mathscr{A}^{0,0}_X\otimes_{\mathbb{C}} R$. In other words, a locally trivial infinitesimal $C^{\infty}$-deformation of a compact complex manifold $X$ over a local artinian $\mathbb{C}$-algebra $R$ with residue $\mathbb{C}$ is $($globally$)$ trivial. \footnote{The proposition is the infinitesimal version of \cite{Kod05} Theorem 2.4 (page 64). It took a lot of time for me to prove this proposition. Eventually I got the idea of the proof of the proposition from \cite{Kod05} Theorem 2.4. } We call a map defining $\mathscr{A}^{0,0}_{\mathcal{X}}\cong \mathscr{A}_X^{0,0}\otimes R$ a $C^{\infty}$-trivialization.
\end{proposition}

\begin{remark}\label{2re}
Before the proof of Proposition $\ref{2prot}$, we will clarify the meaning of $\mathscr{A}_{\mathcal{X}}^{0,0}\cong \mathscr{A}_X^{0,0}\otimes_{\mathbb{C}} R$. We define a sheaf, denoted by $\mathscr{A}^{0,0}_{\mathcal{X}}$ on $X$ in the following way. Let $\mathcal{U}=\{U_{\alpha}\}$ be a locally trivial open covering of $\mathcal{X}$. So we have $\varphi_{\alpha}:\mathcal{O}_{\mathcal{X}}(U_{\alpha})\cong \mathcal{O}_X(U_{\alpha})\otimes R$ and $\varphi_{\alpha}\varphi_{\beta}^{-1}=exp(t_{\alpha\beta})$ on $\mathcal{O}_X(U_\alpha\cap U_\beta)$ satisfying cocycle conditions where $\{t_{\alpha\beta}\}\in C^1(\mathcal{U},\Theta_X)\otimes \mathfrak{m}$. Let $\mathscr{A}^{0,0}_X$ be a sheaf of complex valued $C^{\infty}$-functions on $X$. We will denote its section on $U$ by $\mathscr{A}^{0,0}_X(U):=\Gamma(U,\mathscr{A}^{\infty}_X)$ for an open set $U$ of $X$. Then $exp(t_{\alpha\beta}): \mathcal{O}_X(U_{\alpha}\cap U_{\beta})\otimes R\to \mathcal{O}_X(U_{\alpha}\cap U_{\beta})\otimes R$  induces $exp(t_{\alpha\beta}):\mathscr{A}^{0,0}_X(U_{\alpha}\cap U_{\beta})\otimes R\to \mathscr{A}^{0,0}_X(U_{\alpha}\cap U_{\beta})\otimes R$. Since $\{exp(t_{\alpha\beta})\}$ satisfies cocycle conditions, this defines a sheaf locally isomorphic to $\mathscr{A}^{0,0}_X|_{U_{\alpha}}\otimes R$ on $U_{\alpha}$, which is the sheaf $\mathscr{A}^{0,0}_{\mathcal{X}}$. We will denote its section on $U$ by $\mathscr{A}^{0,0}_\mathcal{X}(U):=\Gamma(U,\mathscr{A}^{0,0}_\mathcal{X})$. We will use the same $\varphi_{\alpha}$ for local trivialization of $\mathscr{A}^{0,0}_{\mathcal{X}} ($i.e $\varphi_{\alpha}:\mathscr{A}^{0,0}_\mathcal{X}(U_{\alpha})\to \mathscr{A}_X^{0,0}(U_{\alpha})\otimes R)$. We would like to construct an explicit morphism of sheaves $\mathscr{A}^{0,0}_{\mathcal{X}}\to \mathscr{A}^{0,0}_X\otimes R$.

\end{remark}
\begin{proof}[Proof of Proposition $\ref{2prot}$]

We will prove by induction on the exponent of the maximal ideal $\mathfrak{m}$ of $R$. First we show that the proposition holds for $(R,\mathfrak{m})$ with $\mathfrak{m}^2=0$. Let $\{U_{\alpha}\}$ be an locally trivial open covering of $\mathcal{X}$. Then locally we have $\mathscr{A}^{0,0}_{\mathcal{X}}(U_{\alpha})\cong \mathscr{A}^{0,0}_X(U_{\alpha})\otimes_\mathbb{C} R$. There exist $exp(t_{\alpha\beta}):\mathscr{A}_X^{0,0}(U_{\beta}\cap U_{\alpha})\otimes_\mathbb{C} R\to \mathscr{A}_X^{0,0}(U_{\alpha}\cap U_{\beta})\otimes_\mathbb{C} R$ and we have $exp(t_{\alpha\gamma})=exp(t_{\alpha\beta})\circ exp(t_{\beta\gamma})$ which is same to $1+t_{\alpha\gamma}=(1+t_{\alpha\beta})\circ (1+t_{\beta\gamma})=1+t_{\alpha\beta}+t_{\beta\gamma}$ on $U_{\alpha}\cap U_{\beta}\cap U_{\gamma}$. We also note that $t_{\alpha\alpha}=0$ and $t_{\alpha\beta}=-t_{\beta\alpha}$. We will show that $ \mathscr{A}^{0,0}_X\otimes_\mathbb{C} R\cong \mathscr{A}^{0,0}_{\mathcal{X}}$ by defining a map which is compatible with $exp(t_{\alpha\beta})$. Let $\{\rho_\alpha\}$ be a partition of unity subordinate to the open covering $\{U_{\alpha}\}$. Then for each $\alpha$, we set $s_\alpha:=\sum_\gamma \rho_{\gamma}t_{\gamma\alpha}\in \Gamma(U_{\alpha},\mathscr{A}^{0,0}(T_X))\otimes \mathfrak{m}$ $(U_\gamma\cap U_\alpha\ne \emptyset)$. Then we have 
\begin{align*}
exp(s_{\alpha})\circ exp(t_{\alpha\beta})&=1+t_{\alpha\beta}+\sum_\gamma \rho_\gamma t_{\gamma\alpha}=1+\sum_\gamma \rho_\gamma t_{\alpha\beta}+\sum_\gamma \rho_\gamma t_{\gamma\alpha}\\
&=1+\sum_\gamma \rho_\gamma t_{\gamma\beta}=exp(s_\beta)
\end{align*}
So we have the following commutative diagram
\begin{center}
\[\begindc{\commdiag}[100]
\obj(1,2)[aa]{$\mathscr{A}_X^{0,0}(U_{\alpha}\cap U_{\beta})\otimes_\mathbb{C} R$}
\obj(0,1)[bb]{$\mathscr{A}_X^{0,0}(U_{\beta}\cap U_{\alpha})\otimes_\mathbb{C} R$}
\obj(2,1)[cc]{$\mathscr{A}_X^{0,0}(U_{\alpha}\cap U_{\beta})\otimes_\mathbb{C} R$}
\mor{aa}{bb}{$exp(s_{\beta})$}[\atright,\solidarrow]
\mor{aa}{cc}{$exp(s_{\alpha})$}
\mor{bb}{cc}{$exp(t_{\alpha\beta})=1+t_{\alpha\beta} $}
\enddc\]
\end{center}
This compatibly shows that $\mathscr{A}^{0,0}_{\mathcal{X}}\cong \mathscr{A}^{0,0}_X\otimes_\mathbb{C} R$. So we proved the proposition for $R$ with the exponent $1$. 

Now we assume that the proposition holds for $R$ with exponent up to $n-1$. Let $(R,\mathfrak{m})$ be an artinian local $\mathbb{C}$-algebra with the residue $\mathbb{C}$ and exponent $n$ (i.e. $\mathfrak{m}^{n+1}$), and $\mathcal{X}$ be an infinitesimal deformation of $X$ over $R$. 
Then $(R/\mathfrak{m}^n,\mathfrak{m}/\mathfrak{m}^n)$ is a local artinian $\mathbb{C}$-algebra with exponent $n-1$.  Let $\{U_{\alpha}\}$ be a locally trivializing open covering of $\mathcal{X}$. There exist $exp(t_{\alpha\beta}):\mathscr{A}_X^{0,0}(U_{\beta}\cap U_{\alpha})\otimes_\mathbb{C} R_1\to \mathscr{A}_X^{0,0}(U_{\alpha}\cap U_{\beta})\otimes_\mathbb{C} R$ and we have $exp(t_{\alpha\gamma})=exp(t_{\alpha\beta})\circ exp(t_{\beta\gamma})$. Then by the induction hypothesis, we have $\mathscr{A}^{0,0}_{\mathcal{X}}\otimes_{R} R/\mathfrak{m}^n \cong \mathscr{A}^{0,0}_X \otimes_{\mathbb{C}} R/\mathfrak{m}^n$ and so there exist $\bar{s}_{\alpha}\in \Gamma(U_{\alpha}, \mathscr{A}^{0,0}(T_X))\otimes \mathfrak{m}/\mathfrak{m}^n$ such that $exp(\bar{s}_\alpha)=exp(\bar{t}_{\alpha\beta})\circ exp(\bar{s}_{\beta})$, where $\{\bar{t}_{\alpha\beta}\}$ is the image of the natural map $C^1(\mathcal{U},\Theta_X)\otimes \mathfrak{m}\to C^1(\mathcal{U},\Theta_X)\otimes R/\mathfrak{m}^n$.
 Let $s_\alpha$ be a lifting of $\bar{s}_\alpha$ by the natural map $\Gamma(U_{\alpha},\mathscr{A}^{0,0}(T_X))\otimes \mathfrak{m}\to \Gamma(U_{\alpha},\mathscr{A}^{0,0}(T_X))\otimes \mathfrak{m}/\mathfrak{m}^n$. Then $exp(t_{\alpha\beta})\circ exp(s_\beta)-exp(s_\alpha)$ is $0$ modulo $\mathfrak{m}^n$ and is a derivation. Indeed, set $A=exp(t_{\alpha\beta})\circ exp(s_\beta)$ and $B=exp(s_\alpha)$. Then $(A-B)(fg)=A(f)(A-B)(g)+B(g)(A-B)(f)=f(A-B)(g)+g(A-B)(f)$ since $A-B=0$ modulo $\mathfrak{m}^n$ and $\mathfrak{m}^{n+1}=0$. Set $ r_{\alpha\beta}:=exp(t_{\alpha\beta})\circ exp(s_\beta)-exp(s_\alpha)$. Since $exp(-t_{\alpha\beta})=exp(t_{\beta\alpha})$, we have $exp(t_{\beta\alpha})\circ r_{\alpha\beta}=exp(s_\beta)-exp(t_{n\beta\alpha})\circ exp(s_\alpha)$. So we get $exp(t_{\beta\alpha})\circ r_{\alpha\beta}=-r_{\beta\alpha}$. Since $r_{\alpha\beta}$ is $0$ modulo $\mathfrak{m}^n$, we have $r_{\alpha\beta}=-r_{\beta\alpha}$. Next we claim that $r_{\alpha\beta}=r_{\gamma\beta}+r_{\alpha\gamma}$. Indeed, since $r_{\alpha\beta}=exp(t_{\alpha\beta})\circ exp(s_{\beta})-exp(s_{\alpha})$ and $r_{\alpha\beta}=0$ modulo $\mathfrak{m}^n$, by applying $exp(t_{\gamma\alpha})$ on both sides,
\begin{align*}
 r_{\alpha\beta}&=exp(t_{\gamma\alpha})\circ r_{\alpha\beta}= exp(t_{\gamma\alpha})\circ exp(t_{\alpha\beta})\circ exp(s_{\beta})-exp(t_{\gamma\alpha})\circ exp(s_{\alpha})\\
 &=exp(t_{\gamma\beta})\circ exp(s_{\beta})-r_{\gamma\alpha}-exp(s_{\gamma})=r_{\gamma\beta}-r_{\gamma\alpha}=r_{\gamma\beta}+r_{\alpha\gamma}
\end{align*}

  Let $\{\rho_\alpha\}$ be a partition of unity subordinate to the open covering $\{U_{\alpha}\}$. Then for each $\alpha$, we set $e_\alpha:=\sum_\gamma \rho_{\gamma}r_{\alpha\gamma}\in \Gamma(U_{\alpha},\mathscr{A}^{0,0}(T_X))\otimes \mathfrak{m}^n$. Then we have
\begin{align*}
exp(t_{\alpha\beta})\circ exp(s_{\beta}+e_{\beta})&=exp(t_{\alpha\beta})\circ exp(s_\beta)+e_\beta=exp(s_\alpha)+r_{\alpha\beta}+e_\beta\\
&=exp(s_{\alpha}) +\sum_\gamma \rho_{\gamma} r_{\alpha\beta}+\sum_{\gamma} \rho_{\gamma} r_{\beta\gamma}=exp(s_{\alpha})+\sum_{\gamma}\rho_{\gamma}(r_{\alpha\beta}+r_{\beta\gamma})\\
&=exp(s_{\alpha})+\sum_{\gamma}\rho_{\gamma}r_{\alpha\gamma}=exp(s_{\alpha})+e_\alpha=exp(s_{\alpha}+e_\alpha)
\end{align*}
So we have the following commutative diagram
\begin{center}
\[\begindc{\commdiag}[100]
\obj(1,2)[aa]{$\mathscr{A}_X^{0,0}(U_{\alpha}\cap U_{\beta})\otimes_\mathbb{C} R$}
\obj(0,1)[bb]{$\mathscr{A}_X^{0,0}(U_{\beta}\cap U_{\alpha})\otimes_\mathbb{C} R$}
\obj(2,1)[cc]{$\mathscr{A}_X^{0,0}(U_{\alpha}\cap U_{\beta})\otimes_\mathbb{C} R$}
\mor{aa}{bb}{$exp(s_{\beta}+e_\beta)$}[\atright,\solidarrow]
\mor{aa}{cc}{$exp(s_{\alpha}+e_\alpha)$}
\mor{bb}{cc}{$exp(t_{\alpha\beta})$}
\enddc\]
\end{center}
This compatibly shows that $\mathscr{A}^{0,0}_{\mathcal{X}}\cong \mathscr{A}^{0,0}_X\otimes_\mathbb{C} R$. This completes Proposition $\ref{2prot}$.

\end{proof}

As in the proof of Proposition \ref{2prot}, there is a global $C^{\infty}$ trivialization $C:\mathscr{A}^{0,0}_X\otimes R\cong \mathscr{A}^{0,0}_{\mathcal{X}}$. This induces $C:\mathscr{A}^{0,0}(U_{\alpha})\otimes R \to \mathscr{A}^{0,0}_{\mathcal{X}}(U_{\alpha})$ such that $\varphi_{\alpha} \circ C$ is of the form $exp(s_{\alpha})$, where $s_{\alpha}\in \Gamma(U_{\alpha},\mathscr{A}^{0,0}(T_X))\otimes \mathfrak{m}$. Here $\varphi_{\alpha}:\mathscr{A}^{0,0}_\mathcal{X}(U_{\alpha})\cong \mathscr{A}^{0,0}_X(U_{\alpha}) \otimes R $ is a local trivialization of $\mathscr{A}^{0,0}_{\mathcal{X}}(U_{\alpha})$. (See Remark \ref{2re})

So we have 
\begin{align*}
exp(s_{\alpha})&=\varphi_{\alpha}\circ C:\mathscr{A}^{0,0}_X(U_{\alpha})\otimes R\to \mathscr{A}^{0,0}_X(U_{\alpha})\otimes R\\
exp(-s_{\beta})&=C^{-1}\circ \varphi_{\beta}^{-1}
\end{align*}
So we have $exp(s_{\alpha})exp(-s_{\beta})=\varphi_{\alpha} \circ C \circ C^{-1}\circ \varphi_{\beta}^{-1}=exp(t_{\alpha\beta})$.

So we have the following commutative diagram

\begin{center}
\[\begindc{\commdiag}[140]
\obj(0,2)[aa]{$\mathscr{A}^{0,0}_X(U_{\beta}\cap U_{\alpha})\otimes R$}
\obj(2,2)[bb]{$\mathscr{A}^{0,0}_X(U_{\alpha}\cap U_{\beta})\otimes R$}
\obj(1,1)[cc]{$\mathscr{A}^{0,0}_\mathcal{X}(U_{\alpha}\cap U_{\beta})$}
\obj(1,0)[dd]{$\mathscr{A}^{0,0}_X(U_{\alpha}\cap U_{\beta})\otimes R$}
\mor{aa}{bb}{$exp(t_{\alpha\beta})$}
\mor{cc}{aa}{$\varphi_{\beta}$}
\mor{cc}{bb}{$\varphi_{\alpha}$}[\atright,\solidarrow]
\mor{dd}{cc}{$C$}
\mor{dd}{aa}{$exp(s_{\beta})$}
\mor{dd}{bb}{$exp(s_{\alpha})$}[\atright,\solidarrow]
\enddc\]
\end{center}

\subsection{Integrability Condition}\

We set 
\begin{center}
$\phi_{\alpha}:=exp(-s_{\alpha})\bar{\partial}(exp(s_{\alpha}))=[D(ad(-s_{\alpha}))(\bar{\partial}s_{\alpha}),-]$ as an operator acting on $A$
\end{center}
where $D$ is the function
\begin{align*}
D(x)=\frac{exp(x)-1}{x}=\sum_{i=0}^{\infty} \frac{x^i}{(i+1)!}
\end{align*}
We have 
\begin{align*}
D(ad(s_{\alpha}))(\bar{\partial}s_{\alpha})=\bar{\partial}s_{\alpha}-\frac{[s_{\alpha},\bar{\partial} s_{\alpha}]}{2!}+\frac{[s_{\alpha},[s_{\alpha},\bar{\partial}s_{\alpha}]]}{3!}+\cdots \in \Gamma(U_{\alpha},\mathscr{A}^{0,1}(T))\otimes \mathfrak{m}
\end{align*}
Note that we have
\begin{align*}
0=\bar{\partial} exp(t_{\alpha\beta})=\bar{\partial}(exp(s_{\alpha})exp(-s_{\beta}))=\bar{\partial}(exp(s_{\alpha}))exp(-s_{\beta})+exp(s_{\alpha})\bar{\partial}(exp(-s_{\beta}))
\end{align*}
So we have
\begin{align*}
exp(-s_{\alpha})\bar{\partial}(exp(s_{\alpha}))=-\bar{\partial}(exp(-s_{\beta}))exp(s_{\beta})
\end{align*}
Since $0=\bar{\partial}(exp(-s_{\beta})exp(s_{\beta}))=\bar{\partial}(exp(-s_{\beta}))exp(s_{\beta})+exp(-s_{\beta})\bar{\partial}(exp(s_{\beta}))$, we have
\begin{align*}
-\bar{\partial}(exp(-s_{\beta}))exp(s_{\beta})=exp(-s_{\beta})\bar{\partial}(exp(s_{\beta}))
\end{align*}
Hence
\begin{align*}
exp(-s_{\alpha})\bar{\partial}(exp(s_{\alpha}))= exp(-s_{\beta})\bar{\partial}(exp(s_{\beta}))
\end{align*}

So $D(ad(s_{\alpha}))(\bar{\partial}s_{\alpha})$ glue together to give a global section
\begin{align*}
\phi \in A^{0,1}(X,T)\otimes \mathfrak{m}
\end{align*}

Since
\begin{align*}
\bar{\partial}\phi_{\alpha}=\bar{\partial}exp(-s_{\alpha})\bar{\partial}exp(-s_{\alpha})=\bar{\partial}exp(-s_{\alpha})exp(s_{\alpha})exp(-s_{\alpha})\bar{\partial}exp(s_{\alpha})=-\phi_{\alpha}\circ \phi_{\alpha}=-\frac{1}{2}[\phi_{\alpha},\phi_{\alpha}]
\end{align*}
We have
\begin{align*}
\bar{\partial}\phi=-\frac{1}{2}[\phi,\phi]
\end{align*}

Now we consider the Poisson structures. We need the following lemma.

\begin{lemma}\label{2lemp}
We have $[exp(a) f,exp(a) g]=exp(a)[f,g]$, where $deg(a)=0$.
\end{lemma}
\begin{proof}
We claim that 
\begin{align*}
\sum_{l+m=n} \frac{1}{l!m!}[[\underbrace{a,[a,[\cdots[a}_l,f]\cdots],[\underbrace{a,[a,[\cdots[a}_m,g]\cdots]]=\frac{1}{n!}[\underbrace{a,[a,[\cdots[a}_n,[f,g]]\cdots]
\end{align*}
Then this implies the lemma. We will prove this by induction. 
First we note that $[[a,f],g]=[a,[f,g]]-[f,[a,g]]$. So we have
\begin{align*}
[a,[f,g]]=[[a,f],g]+[f,[a,g]]
\end{align*}
Hence the claim holds for $n=1$. Now let's assume that the claim holds for $n-1$. We will prove that it holds for $n$.

\begin{align*}
&\frac{1}{n!}[\underbrace{a,[a,[\cdots[a}_n,[f,g]]\cdots]=\frac{1}{n}[a,\frac{1}{(n-1)!}[\underbrace{a,[a,[\cdots[a}_{n-1},[f,g]]\cdots]]\\
&=\frac{1}{n}[a,\sum_{p+q=n-1} \frac{1}{p!q!}[[\underbrace{a,[a,[\cdots[a}_p,f]\cdots],[\underbrace{a,[a,[\cdots[a}_q,g]\cdots]](\text{by induction hypothesis})\\
&=\frac{1}{n}\sum_{p+q=n-1}\frac{1}{p!q!}[[\underbrace{a,[a,[\cdots[a}_{p+1},f]\cdots],[\underbrace{a,[a,[\cdots[a}_q,g]\cdots]]\\
\end{align*}

\begin{align*}
&+\frac{1}{n}\sum_{p+q=n-1}\frac{1}{p!q!}[[\underbrace{a,[a,[\cdots[a}_p,f]\cdots],[\underbrace{a,[a,[\cdots[a}_{q+1},g]\cdots]]\\
&=\frac{1}{n}\sum_{l+m=n,l\geq 1}\frac{l}{l!m!}[[\underbrace{a,[a,[\cdots[a}_{l},f]\cdots],[\underbrace{a,[a,[\cdots[a}_m,g]\cdots]](\text{by taking $p+1=l,q=m$})\\
&+\frac{1}{n}\sum_{l+m=n,m\geq 1}\frac{m}{l!m!}[[\underbrace{a,[a,[\cdots[a}_l,f]\cdots],[\underbrace{a,[a,[\cdots[a}_{m},g]\cdots]](\text{by taking $p=l,q+1=m$})\\
&=\frac{1}{n}\sum_{l+m=n,l,m\geq 1} \frac{l+m}{l!m!}[[\underbrace{a,[a,[\cdots[a}_l,f]\cdots],[\underbrace{a,[a,[\cdots[a}_{m},g]\cdots]]\\
&+\frac{1}{n}\frac{n}{n!}[[\underbrace{a,[a,[\cdots[a}_{n},f]\cdots],g]+\frac{1}{n}\frac{n}{n!}[f,[\underbrace{a,[a,[\cdots[a}_{n},g]\cdots]]\\
&=\sum_{l+m=n} \frac{1}{l!m!}[[\underbrace{a,[a,[\cdots[a}_l,f]\cdots],[\underbrace{a,[a,[\cdots[a}_m,g]\cdots]]
\end{align*}
\end{proof}

\begin{definition}
Let $\mathcal{X}\to Spec\,R$ be an infinitesimal deformation of $X$, where $R$ is generated by $<1,m_1,...,m_r>$. We define a sheaf $\mathscr{A}^{0,p}(\wedge^q T_{\mathcal{X}/R})$ which is locally isomorphic to $\mathscr{A}^{0,p}(\wedge^q T_X)|_{U_{\alpha}}\otimes R$ on $U_{\alpha}$ in the following way: we define this sheaf by gluing sheaves $\mathscr{A}^{0,p}(\wedge^q T_X)|_{U_{\alpha}}\otimes R$ locally defined on each $U_{\alpha}$. On $U_{\alpha}\cap U_{\beta}$, we have isomorphisms $exp(t_{\alpha\beta}):(\mathscr{A}^{0,p}(\wedge^q T_X)|_{U_{\beta}})|_{U_{\alpha}\cap U_{\beta}}\to (\mathscr{A}^{0,p}(\wedge^q T_X)|_{U_{\alpha}})|_{U_{\alpha}\cap U_{\beta}}$. Since $exp(t_{\alpha\beta})=\exp(t_{\alpha \gamma})exp(t_{\gamma\beta})$ on $U_{\alpha}\cap U_{\beta}\cap U_{\gamma}$, we can glue these sheaves. By Lemma \ref{2lemp}, we can define a bracket $[-,-]$ on $\oplus_{p\geq 1,q\geq 0}\mathscr{A}^{0,p}(\wedge^q T_{\mathcal{X}/R})$ locally induced from the bracket on $[-,-]$ on $\oplus_{p\geq 1,q\geq 0}\Gamma(U_{\alpha},\mathscr{A}^{0,p}(\wedge^q T_X))\otimes R$.

Now we define the dolbeault differential  $\bar{\partial}_\mathcal{X}:\mathscr{A}^{0,p}(\wedge^q T_{\mathcal{X}/{R}})\to \mathscr{A}^{0,p+1}(\wedge^q T_{\mathcal{X}/R})$ of morphism of sheaves. Locally for $f \in \Gamma(U_{\alpha},\mathscr{A}^{0,p}(\wedge^q T_X))\otimes R$, which is of the form $f=f_0+m_1f_1+\cdots +m_rf_r$, where $f_i\in \Gamma(U_{\alpha},\mathscr{A}^{0,p}(\wedge^q T_X))$, we define $\bar{\partial}_\mathcal{X} f=\bar{\partial}f_0+m_1\bar{\partial}f_1+\cdots + m_r\bar{\partial}f_r$. We show that this is glued together. It is equivalent to show that $exp(t_{\alpha\beta})\circ\bar{\partial}_\mathcal{X}f=\bar{\partial}_\mathcal{X} \circ exp(t_{\alpha\beta})f$. It is enough to show that $t_{\alpha\beta}\circ \bar{\partial}_\mathcal{X}=\bar{\partial}_\mathcal{X}\circ t_{\alpha\beta}$, which is equivalent to $\bar{\partial}_\mathcal{X} t_{\alpha\beta}=0$ $($more precisely, let $t_{\alpha\beta}=t_0+m_1t_1+\cdots +m_rt_r$, then above condition is equivlent to $ \bar{\partial} t_i=0$, which is equivalent to $\{t_{\alpha\beta}\}\in C^1(\mathcal{U},\Theta_X)\otimes \mathfrak{m})$. So $\bar{\partial}_\mathcal{X}:\mathscr{A}^{0,p}(\wedge^q T_{\mathcal{X}/R})\to \mathscr{A}^{0,p+1}(\wedge^qT_{\mathcal{X}/R})$ is well defined with $\bar{\partial}_\mathcal{X}\circ \bar{\partial}_\mathcal{X}=0$. We will denote $\bar{\partial}_{\mathcal{X}}$ simply by $\bar{\partial}$.
\end{definition}

\begin{remark}\label{2remark}
Let $\mathcal{X}\to Spec\,R$ be an infinitesimal Poisson deformation of $(X,\Lambda_0)$. Let $\{U_{\alpha}\}$ be a locally trivial open covering of $\mathcal{X}$. Then $\mathcal{O}_{\mathcal{X}}(U_{\alpha})\cong \mathcal{O}_X(U_{\alpha})\otimes R$. The Poisson structure is encoded in $\Lambda_{\alpha}' \in \Gamma(U_{\alpha}, \mathscr{A}^{0,0}(\wedge^2 T))\otimes R$ with $[\Lambda_{\alpha}',\Lambda_{\alpha}']=0$ and $\bar{\partial}\Lambda_{\alpha}'=0$. On $U_{\alpha}\cap U_{\beta}$, $exp(t_{\alpha\beta})[[\Lambda_{\beta}',f],g]=[[\Lambda_{\alpha}',exp(t_{\alpha\beta})f],exp(t_{\alpha\beta})g]$. Since $exp(t_{\alpha\beta})[[\Lambda_{\beta}',f],g]=[[exp(t_{\alpha\beta})\Lambda_{\beta}',exp(t_{\alpha\beta}) f], exp(t_{\alpha\beta})g]$ by Lemma $\ref{2lemp}$,  we have $exp(t_{\alpha\beta})\Lambda_{\beta}'=\Lambda_{\alpha}'$. Hence the Poisson $R$-structure of $\mathcal{X}$ can be identified with the existence of a global section $\Lambda'$ of the sheaf $\mathscr{A}^{0,0}(\wedge^2 T_{\mathcal{X}/R})$ with $[\Lambda',\Lambda']=0$ and $\bar{\partial} \Lambda'=0$ such that $\Lambda'$ induces $\Lambda_0$.
\end{remark}

Going back to our discussion, since $exp(-s_{\alpha})=C^{-1}\circ (\varphi_{\alpha})^{-1}$ induces an isomorphism 
\begin{align*}
\Gamma(U_{\alpha},\mathscr{A}^{0,0}(\wedge^2T_X))\otimes R\cong \Gamma(U_{\alpha},\mathscr{A}^{0,0}(\wedge^2 T_X))\otimes R
\end{align*}
which is compatible with each $\alpha$, 
\begin{center}
\[\begindc{\commdiag}[130]
\obj(1,2)[aa]{$\Gamma(U_{\alpha}\cap U_{\beta},\mathscr{A}^{0,0}(\wedge^2 T_X))\otimes_\mathbb{C} R$}
\obj(0,1)[bb]{$\Gamma(U_{\alpha}\cap U_{\beta},\mathscr{A}^{0,0}(\wedge^2 T_X))\otimes_\mathbb{C} R$}
\obj(2,1)[cc]{$\Gamma(U_{\alpha}\cap U_{\beta},\mathscr{A}^{0,0}(\wedge^2 T_X))\otimes_\mathbb{C} R$}
\mor{bb}{aa}{$exp(-s_{\beta})$}[\atleft,\solidarrow]
\mor{cc}{aa}{$exp(-s_{\alpha})$}[\atright,\solidarrow]
\mor{bb}{cc}{$exp(t_{\alpha\beta})$}
\enddc\]
\end{center}
we have an isomorphism of sheaves $\mathscr{A}^{0,0}(\wedge^2 T_{\mathcal{X}/R})\cong \mathscr{A}^{0,0}(\wedge^2 T)\otimes R$. Since $\mathcal{O}_{\mathcal{X}}(U_{\alpha})\cong \mathcal{O}_X(U_{\alpha})\otimes R$ locally, the Poisson structure is locally encoded in $\Lambda'_{\alpha} \in A^{0,0}(\wedge^2 T)(U_{\alpha})\otimes R$  with $[\Lambda'_{\alpha},\Lambda'_{\alpha}]=0, \bar{\partial} \Lambda'_{\alpha}=0$. We define $\Lambda_{\alpha}'':=exp(-s_{\alpha})\Lambda'_{\alpha}$. Since $\Lambda_{\alpha}'=exp(s_{\alpha})exp(-s_{\beta})\Lambda_{\beta}'$, we have $\Lambda_{\alpha}''=exp(-s_{\alpha})\Lambda_{\alpha}'=exp(-s_{\beta})\Lambda_{\beta}'=\Lambda_{\beta}''$. So $\Lambda_{\alpha}''$ glue together to make a global section
\begin{align*}
\Lambda''\in A^{0,0}(X,\wedge^2 T)\otimes R
\end{align*}

By Lemma $\ref{2lemp}$, $[\Lambda_{\alpha}'',\Lambda_{\alpha}'']=[exp(-s_{\alpha})\Lambda_{\alpha}',exp(-s_{\alpha})\Lambda_{\alpha}']=exp(-s_{\alpha})[\Lambda_{\alpha}',\Lambda_{\alpha}']=0$.

\begin{align*}
&\bar{\partial}(\Lambda_{\alpha}'')+[D(ad(-s_{\alpha}))(\bar{\partial}s_{\alpha}),\Lambda_{\alpha}'']\\
&=\bar{\partial}(exp(-s_{\alpha})\Lambda_{\alpha}')+exp(-s_{\alpha})\circ \bar{\partial}(exp(s_{\alpha}))exp(-s_{\alpha})\Lambda_{\alpha}'\\
&=\bar{\partial}(exp(-s_{\alpha})\Lambda_{\alpha}')+(exp(-s_{\alpha})\circ \bar{\partial}\circ exp(s_{\alpha})-\bar{\partial})exp(-s_{\alpha})\Lambda_{\alpha}'\\
&=\bar{\partial}(exp(-s_{\alpha})\Lambda_{\alpha}')+exp(-s_{\alpha})\circ \bar{\partial}\Lambda_{\alpha}'-\bar{\partial}(exp(-s_{\alpha})\Lambda_{\alpha}')=0
\end{align*}
since $\bar{\partial}\Lambda_{\alpha}'=0$.

In conclusion, we have

\begin{enumerate}
\item $\bar{\partial}\phi+\frac{1}{2}[\phi,\phi]=0$
\item $[\Lambda'',\Lambda'']=0$
\item $\bar{\partial} \Lambda''+[\phi,\Lambda'']=0$
\end{enumerate}

By taking $\Lambda =\Lambda''-\Lambda_0\in A^{0,0}(X,\wedge^2 T)\otimes \mathfrak{m}$, the above equations are equivalent to 
\begin{align*}
L(\phi+\Lambda)+\frac{1}{2}[\phi+\Lambda,\phi+\Lambda]=0
\end{align*}
where $L=\bar{\partial} +[\Lambda_0,-]$, which is a solution of Maurer Cartan equation of the differential graded Lie algebra $\mathfrak{g}=(\bigoplus_{p+q-1=i,p\geq 0,q \geq 1} A^{0,p}(X,\wedge^q T)\otimes \mathfrak{m},L=\bar{\partial} +[\Lambda_0,-],[-,-])$

Now set $\alpha=\phi+\Lambda$. Then we have $L \alpha =-\frac{1}{2}[\alpha,\alpha]$ and $\alpha \wedge L\alpha=-L\alpha\wedge \alpha$.(note that $deg(\alpha)=1$.). We denote $\bar{\alpha}$ be the corresponding element in $\mathfrak{g}[1]$. Now we assume that $\mathfrak{m}^{n+1}=0$. We claim that
\begin{equation}\label{2eq}
\epsilon(\alpha):=(\bar{\alpha},\frac{1}{2}\bar{\alpha}\odot \bar{\alpha},\cdots, \frac{1}{n!}\underbrace{\bar{\alpha}\odot \cdots \odot \bar{\alpha}}_n)\in \bigoplus_{i=1}^n sym^i g_1\otimes \mathfrak{m}\subset J_n(\mathfrak{g})\otimes \mathfrak{m}
\end{equation}
is a hypercocycle in $S_n\subset \overline{S(\mathfrak{g}[1])}$, which corresponds to via $dec$
\begin{align*}
(\alpha,\cdots, (-1)^{\frac{i(i-1)}{2}}\frac{1}{2!}\underbrace{\alpha \wedge \cdots \wedge \alpha}_i,\cdots, (-1)^{\frac{n(n-1)}{2}}\frac{1}{n!}\underbrace{\alpha\wedge \cdots \wedge \alpha}_n)
\end{align*}
in $\overline{\bigwedge \mathfrak{g}}$.
For the claim (\ref{2eq}), we have to show that $(-1)^i L((-1)^{\frac{(i-1)i}{2}}\frac{1}{i!}\underbrace{\alpha \wedge \cdots \wedge \alpha}_i)+(-1)^{\frac{i(i+1)}{2}}\frac{1}{(i+1)!}\binom{i+1}{2}[\alpha,\alpha]\wedge \underbrace{\alpha\wedge \cdots \wedge \alpha}_{i-1}=0$. In other words, $L(\frac{1}{i!}\underbrace{\alpha \wedge \cdots \wedge \alpha}_i)+\frac{1}{(i+1)!}\binom{i+1}{2}[\alpha,\alpha]\wedge \underbrace{\alpha\wedge \cdots \wedge \alpha}_{i-1}=0$. Indeed,

\begin{align*}
&L(\frac{1}{i!}\underbrace{\alpha \wedge \cdots \wedge \alpha}_i)
=\frac{1}{i!}(L \alpha\wedge \alpha \wedge \cdots \wedge \alpha-\alpha\wedge L\alpha\wedge \alpha\wedge \cdots \wedge \alpha+\cdots +(-1)^{i-1}\alpha\wedge\cdots \wedge \alpha\wedge L \alpha)\\
&=\frac{1}{(i-1)!} L \alpha\wedge \alpha \wedge \cdots \wedge \alpha\\
&\frac{1}{(i+1)!}\frac{(i+1)i}{2}[\alpha,\alpha] \wedge \alpha \wedge \cdots \wedge \alpha=\frac{1}{(i-1)!}\frac{1}{2}[\alpha,\alpha]\wedge \alpha \wedge \cdots \wedge \alpha
\end{align*}
So we get the claim $(\ref{2eq})$ by the condition $L\alpha+\frac{1}{2}[\alpha,\alpha]=0$. So $\epsilon(\alpha)$ is a hypercocycle in $S_n\subset \overline{S(\mathfrak{g}[1])}$. So $[\epsilon(\alpha)]\in \mathbb{H}^0(J_n(\mathfrak{g}))\otimes \mathfrak{m}.$

\subsection{$[\epsilon(\alpha)]\in \mathbb{H}^0(J_n(\mathfrak{g}))$ as a canonical element associated with $\mathcal{X}$}\

For given an infinitesimal Poisson deformation $\mathcal{X}$ of $(\mathcal{X},\Lambda_0)$ over $(R,\mathfrak{m})$ with $\mathfrak{m}^{n+1}=0$ , we could find a cohomology class $[\epsilon(\alpha)]\in \mathbb{H}^0(J_n(\mathfrak{g}))\otimes \mathfrak{m}$, where $\alpha=\phi+\Lambda\in A^{0,1}(X,T)\oplus A^{0,0}(X,\wedge^2 T)$ for given an locally trivial open covering $\{U_{\alpha}\}$, $\varphi_{\alpha}$ and $C^{\infty}$-trivialization. In this subsection, we show that the cohomology class 
\begin{align*}
[\epsilon(\alpha)]\in \mathbb{H}^0(J_n(\mathfrak{g}))\otimes \mathfrak{m}
\end{align*}
is independent of choices of $C^{\infty}$-trivialization and local trivialization for fixed locally trivial open covering $\{U_{\alpha}\}$. Then for given two locally trivial open covering, by choosing refinement of these two open coverings, we conclude that the cohomology class $[\epsilon(\alpha)]$ is canonically associated to the infinitesimal Poisson deformation $\mathcal{X}$ of $(X,\Lambda_0)$.

\subsubsection{Independence of choices of local trivialization}\

For given an locally trivial open covering $\mathcal{U}=\{U_{\alpha}\}$, let's assume that we have two local trivialization $\varphi_{\alpha}:\mathcal{O}_\mathcal{X}(U_{\alpha})\to \mathcal{O}_X(U_{\alpha})\otimes R$ and $\varphi_{\alpha}':\mathcal{O}_n(U_{\alpha})\to \mathcal{O}_X(U_{\alpha})\otimes R $. We will show that $v=\phi+\Lambda \in A^{0,1}(X,T)\oplus A^{0,0}(X,\wedge^2 T)$ associated with $\varphi_{\alpha}$ and $C^{\infty}$-trivialization $C$ is same to $w=\psi+\Pi \in A^{0,1}(X,T)\oplus A^{0,0}(X,\wedge^2 T)$ associated with $\varphi_{\alpha}^{'}$ and same $C$. We have the following commutative diagram

\begin{center}
\[\begindc{\commdiag}[100]
\obj(0,1)[aa]{$\mathscr{A}^{0,0}_X(U_{\alpha})\otimes R$}
\obj(1,1)[bb]{$\mathscr{A}^{0,0}_\mathcal{X}(U_{\alpha})$}
\obj(2,1)[cc]{$\mathscr{A}_X^{0,0}(U_{\alpha}) \otimes R$}
\obj(2,0)[dd]{$\mathscr{A}^{0,0}_X(U_{\alpha})\otimes R$}
\mor{aa}{bb}{$C$}
\mor{bb}{cc}{$\varphi_{\alpha}$}
\mor{bb}{dd}{$\varphi_{\alpha}^{'}$}[\atright,\solidarrow]
\mor{cc}{dd}{$exp(t_{\alpha})$}
\enddc\]
\end{center}
for some $\{t_{\alpha}\}\in C^1(\mathcal{U},\Theta_X)\otimes \mathfrak{m}$ and $exp(s_{\alpha})=\varphi_{\alpha}\circ C$, $exp(s_{\alpha}')=\varphi_{\alpha}^{'}\circ C$. Then we have, by our construction as above, 
\begin{align*}
\psi_{\alpha}&=exp(-s_{\alpha}')\bar{\partial}(exp(s_{\alpha}'))=exp(-s_{\alpha})\circ exp(-t_{\alpha})\bar{\partial}(exp(t_{\alpha})\circ exp(s_{\alpha}))\\
                    &=exp(-s_{\alpha})\circ exp(-t_{\alpha})\circ \bar{\partial}(exp(t_{\alpha}))\circ exp(s_{\alpha})+exp(-s_{\alpha})\circ exp(-t_{\alpha})\circ exp(t_{\alpha})\circ \bar{\partial}exp(s_{\alpha})\\
                    &=exp(-s_{\alpha})\bar{\partial} (exp(s_{\alpha}))=\phi_{\alpha}
\end{align*}
since $\bar{\partial} t_{\alpha}=0$. On the other hand,
\begin{align*}
\Lambda_0+\Pi_{\alpha}=exp(-s_{\alpha}')\omega_{\alpha}'=exp(-s_{\alpha})\circ exp(t_{\alpha})\omega'_{\alpha}=exp(-s_{\alpha})\Lambda_{\alpha}'=\Lambda_0+\Lambda_{\alpha}.
\end{align*}
Hence we have $\Pi=\Lambda$. So we have  $v=w$. So $[\epsilon(v)]=[\epsilon(w)]$.
\subsubsection{Independence of choices of $C^{\infty}$-trivialization}\

Now, we show that the class $[\epsilon(v)]\in \mathbb{H}^0(J_n(\mathfrak{g}))\otimes \mathfrak{m}$ is independent of choice of $C^{\infty}$-trivialization. Let's assume that we have two $C^{\infty}$-trivialization $C:\mathscr{A}_X^{0,0}\otimes R\to \mathscr{A}_{\mathcal{X}}^{0,0}$ and $\tilde{C}:\mathscr{A}_X^{0,0}\otimes R\to \mathscr{A}_{\mathcal{X}}^{0,0}$. Then we have the following commutative diagram
\begin{center}
\[\begindc{\commdiag}[70]
\obj(0,1)[a]{$\mathscr{A}^{0,0}_X(X)\otimes R$}
\obj(2,1)[b]{$\mathscr{A}^{0,0}_\mathcal{X}(X)$}
\obj(1,0)[d]{$\mathscr{A}^{0,0}_X(X)\otimes R$}
\mor{a}{b}{$\tilde{C}$}
\mor{a}{d}{$exp(u)=C^{-1}\circ \tilde{C}$}[\atright,\solidarrow]
\mor{d}{b}{$C$}[\atright,\solidarrow]
\enddc\]
\end{center}
$C^{-1}\circ \tilde{C}$ is a homomorphism inducing identity up to $\mathfrak{m}$. As in the proof of Lemma \ref{2len}, we can show that there is an element $u\in A^{0,0}(T)\otimes \mathfrak{m}$ such that $C^{-1}\circ \tilde{C}=exp(u)$.

For given an locally trivial open covering $\{U_{\alpha}\}$ and local trivializations $\varphi_{\alpha}:\mathcal{O}_\mathcal{X}(U_{\alpha})\to \mathcal{O}_X(U_{\alpha})\otimes R$, let $\tilde{v}=\tilde{\phi}+\tilde{\Lambda}$ be the element of $(A^{0,1}(X,T)\oplus A^{0,0}(X,\wedge^2 T))\otimes \mathfrak{m}$ induced from $\tilde{C}$ and $v=\phi+\Lambda$ be the element of $(A^{0,1}(X,T)\oplus A^{0,0}(X,\wedge^2 T))\otimes \mathfrak{m}$ induced from $C$. We want to show that $[\epsilon(v_1)]=[\epsilon(v_0)]$.

Since $exp(s_{\alpha})=\varphi_{\alpha}\circ C$, we have $exp(s_{\alpha})\circ exp(u)=\varphi_{\alpha}\circ \tilde{C}$. Hence we have
\begin{align*}
\tilde{\phi}&=[exp(s_{\alpha})\circ exp(u)]^{-1}\bar{\partial}(exp(s_{\alpha})\circ exp(u))\\
          &=exp(-u)\circ exp(-s_{\alpha})\bar{\partial}exp(s_{\alpha})exp(u)+exp(-u)\circ exp(-s_{\alpha})exp(s_{\alpha})\bar{\partial}exp(u)\\
          &=exp(-u)\bar{\partial}exp(u)+exp(-u)\phi exp(u)\\
\Lambda_0+ \tilde{\Lambda}&=exp(-u)\circ exp(-s_{\alpha})(\Lambda_{\alpha}')\\
                                  &=exp(-u)(\Lambda_0+\Lambda)          
\end{align*}

We set for $t\in [0,1]\subset \mathbb{R}$,
\begin{align*}
\phi_t&=exp(-tu)\bar{\partial}exp(tu)+exp(-tu)\phi exp(tu)\\
\Lambda^t&=exp(-tu)(\Lambda_0+\Lambda)-\Lambda_0
\end{align*}
Note that $\phi_0=\phi, \Lambda^0=\Lambda$ and $\phi_1=\tilde{\phi},\Lambda^1=\tilde{\Lambda}$. We would like to show that $[\epsilon(v)]=[\epsilon(\phi+\Lambda)]=[\epsilon(\tilde{\phi}+\tilde{\Lambda})]=[\epsilon(\tilde{v})]$ by showing that $[\epsilon(v_t)]:=[\epsilon(\phi_t+\Lambda^t)]$ is constant independent of $t$. To this end, we note that since the operator $L=\bar{\partial}+[\Lambda_0,-]$ is elliptic, $\mathbb{H}^0(J_n(\mathfrak{g}))$ is finite dimensional vector space over $\mathbb{C}$. Since $\mathfrak{g}$ is a global section of a complex vector bundle over $X$, $\mathfrak{g}$ is endowed with a suitable metric, which induces a metric on $\mathbb{H}^0(J_n(\mathfrak{g}))$ by hodge theory which coincides with the standard Euclidean topology and likewise for $\mathbb{H}^0(J_n(\mathfrak{g}))\otimes \mathfrak{m}$. So differentiation of a function $\mathbb{R}\to \mathbb{H}^0(J_n(\mathfrak{g}))\otimes \mathfrak{m}$ makes sense if the derivative exists with respect to the metric on $\mathbb{H}^0(J_n(\mathfrak{g}))\otimes \mathfrak{m}$. Now we claim that $\frac{d}{dt}([\epsilon(v_t)])=\frac{d}{dt} [(\epsilon(\phi_t+\Lambda^t)]=[\frac{d}{dt} (\epsilon(\phi_t+\Lambda^t))]=0$, which implies that $[\epsilon(v_t)]$ is constant, and so $[\epsilon(\tilde{v})]=[\epsilon(v)]$. We note that

\begin{align*}
\phi_t':=\frac{d}{dt}(\phi_t)&=-uexp(-tu)\bar{\partial}(exp(tu))+exp(-tu)\bar{\partial}(exp(tu)u)\\
                                        &-uexp(-tu)\phi exp(tu)+exp(-tu)\phi exp(tu)u\\
                                        &=-uexp(-tu)\bar{\partial}(exp(tu))+exp(-tu)\bar{\partial}(exp(tu))u+\bar{\partial}u\\
                                        &-uexp(-tu)\phi exp(tu)+exp(-tu)\phi exp(tu)u\\
                                        &=-u\phi_t+\phi_t u+\bar{\partial} u=\bar{\partial}u-[u,\phi_t]\\
(\Lambda^{t})':=\frac{d}{dt}(\Lambda^t)&=-uexp(-tu)(\Lambda_0+\Lambda)=-u (\Lambda^t+\Lambda_0)\\
                                                             &=-[u,\Lambda^t]+[\Lambda_0,u]                                                                          
\end{align*}
So $v_t':=\frac{d}{dt} v_t=\frac{d}{dt}(\phi_t+\Lambda^t)=\phi_t'+(\Lambda^{t})'=L u-[u,\phi_t+\Lambda^t]=Lu-[u,v_t]$. And we have

\begin{align*}
\frac{d}{dt}(\epsilon(v_t))=(\bar{v}_t',\bar{v}_t'\odot \bar{v}_t,...,\frac{1}{(n-1)!}\bar{v}_t'\odot \underbrace{\bar{v}_t \odot \cdots \odot \bar{v}_t}_{n-1})\in \oplus_{i=1}^n sym^ig_1\otimes \mathfrak{m}
\end{align*}
which corresponds to
\begin{align*}
(v_t',\cdots, (-1)^{\frac{(i-1)i}{2}}\frac{1}{(i-1)!}v_t' \wedge \underbrace{v_t\wedge \cdots \wedge v_t}_{i-1}, \cdots, (-1)^{\frac{(n-1)n}{2}}\frac{1}{(n-1)!}v_t'\wedge\underbrace{v_t\wedge \cdots \wedge v_t}_{n-1})
\end{align*}
in $\overline{\bigwedge \mathfrak{g}}$.

 Note that $L v_t=-\frac{1}{2}[v_t,v_t].$\footnote{For given the locally trivial open covering $\{U_{\alpha}\}$ and local trivialization $\varphi_{\alpha}:\mathcal{O}_\mathcal{X}(U_{\alpha})\to \mathcal{O}_X(U_{\alpha})\otimes R$, let $C_t:\mathscr{A}^{0,0}_X \otimes R \to \mathscr{A}^{0,0}_{\mathcal{X}}$ be $C^{\infty}$-trivialization defined by $C_t=C\circ exp(tu)$. Then we have $\varphi_{\alpha}\circ C=exp(s_{\alpha})\circ exp(tu)$. Then we have
 \begin{align*}
 \phi_t&=(exp(s_{\alpha})\circ exp(tu))^{-1}\bar{\partial}(exp(s_{\alpha})\circ exp(tu))\\
 \Lambda_0+\Lambda_t&=(exp(s_{\alpha})\circ exp(tu))^{-1}\Lambda_{\alpha}'
 \end{align*}
 Hence by the construction, we have $Lv_t+\frac{1}{2}[v_t,v_t]=0$.} Let
\begin{align*}
(-\bar{u},-\bar{u}\odot \bar{\alpha}_t,...,\frac{1}{(n-1)!} (-\bar{u})\odot \underbrace{\bar{\alpha}_t\odot ...\odot\bar{\alpha}_t}_{n-1}) \in \bigoplus_{i=1}^{n-1} g_0\otimes sym^i g_1\otimes \mathfrak{m}
\end{align*}
which corresponds to
\begin{align*}
(-u,\cdots, (-1)^{\frac{(i-1)i}{2}} \frac{1}{i!}(-u)\wedge \underbrace{\alpha_t\wedge \cdots \wedge \alpha_t}_i, \cdots, (-1)^{\frac{(n-2)(n-1)}{2}} \frac{1}{(n-1)!} (-u) \wedge \underbrace{\alpha_t\wedge \cdots \wedge \alpha_t}_{n-1})
\end{align*}
in $\overline{\bigwedge \mathfrak{g}}$.

We claim that this is the coboundary of $\frac{d}{dt}(\epsilon(v_t))$. Indeed, (note that $deg(u)=0$ and $deg(v_t)=1$.)
\begin{align*}
&(-1)^{i+1}L((-1)^{\frac{(i-1)i}{2}}\frac{1}{i!}(-u)\wedge \underbrace{v_t\wedge \cdots \wedge v_t}_i)+(-1)^{\frac{i(i+1)}{2}}\frac{i+1}{(i+1)!}[-u,v_t]\wedge \underbrace{v_t\wedge \cdots \wedge v_t}_i
\\&+(-1)^{\frac{i(i+1)}{2}}\frac{1}{(i+1)!}\binom{i+1}{2}[v_t, v_t] \wedge (-u) \wedge \underbrace{v_t\wedge \cdots \wedge v_t}_{i-1}\\
&=(-1)^{\frac{i(i+1)}{2}}\frac{1}{i!}(Lu\wedge \underbrace{v_t\wedge \cdots \wedge v_t}_i +u\wedge Lv_t\wedge v_t\wedge \cdots\wedge v_t-u\wedge v_t\wedge Lv_t\wedge v_t\wedge \cdots \wedge v_t+\cdots\\
&+(-1)^{i-1}u\wedge \underbrace{v_t\wedge \cdots \wedge v_t}_{i-1}\wedge Lv_t)
-(-1)^{\frac{i(i+1)}{2}}\frac{1}{i!}[u,v_t]\wedge \underbrace{v_t\wedge \cdots \wedge v_t}_i\\
&+(-1)^{\frac{i(i+1)}{2}}\frac{1}{(i-1)!}\frac{1}{2}u\wedge [v_t,v_t]\wedge \underbrace{v_t\wedge \cdots \wedge v_t}_{i-1}\\
&=(-1)^{\frac{i(i+1)}{2}}(\frac{1}{i!}Lu\wedge \underbrace{v_t\wedge \cdots \wedge v_t}_i+\frac{1}{(i-1)!}u\wedge L v_t\wedge \underbrace{v_t \wedge \cdots \wedge v_t}_{i-1}-\frac{1}{i!}[u,v_t]\wedge \underbrace{v_t\wedge\cdots \wedge v_t}_i)\\
&+(-1)^{\frac{i(i+1)}{2}}\frac{1}{(i-1)!}\frac{1}{2}u\wedge [v_t,v_t]\wedge \underbrace{v_t\wedge \cdots \wedge v_t}_{i-1}\\
&=(-1)^{\frac{i(i+1)}{2}}\frac{1}{i!}(L u-[u,v_t])\wedge \underbrace{v_t\wedge \cdots \wedge v_t}_{i-1}\\
&=(-1)^{\frac{i(i+1)}{2}}\frac{1}{i!}v_t'\wedge \underbrace{v_t\wedge \cdots \wedge v_t}_i.
\end{align*}
So we have $[\frac{d}{dt}(\epsilon(v_t))]=0$.

 In conclusion, for given an infinitesimal Poisson deformation $\mathcal{X}$ of $(X,\Lambda_0)$, we can canonically associate the Poisson deformation $\mathcal{X}$ up to equivalence with the cohomology class $[\epsilon(v)]$. In the next chapter, we will show that under the assumption $HP^1(X,\Lambda_0)=0$, for given a choice $v,w\in (A^{0,1}(X,T)\oplus A^{0,0}(X,\wedge^2 T))\otimes \mathfrak{m}$ with $[\epsilon(v)]=[\epsilon(w)]$, the Poisson deformation associated with $v$ is equivalent to the Poisson deformation associated with $w$.
 
\chapter{Universal Poisson deformations}\label{chapter6}

\section{Isomorphism of two deformation functors $Def_\mathfrak{g}\cong PDef_{(X,\Lambda_0)}$}
In this section, we will show that two functors of Artin rings are isomorphic: namely the Poisson deformation functor $Def_{(X,\Lambda_0)}$ is isomorphic to the deformation functor $Def_{\mathfrak{g}}$ associated to the differential graded Lie algebra $\mathfrak{g}=(\bigoplus_{p+q-1=i,q\geq 1} A^{0,p}(X,\wedge^q T),L=\bar{\partial}+[\Lambda_0,-],[-,-])$. 
So this shows that deformations of a compact holomorphic Poisson manifold are controlled by the differential graded Lie algebra $\mathfrak{g}$. For deformation functors associated with a differential graded Lie algebra, we refer to \cite{Man04}.
\subsection{Deformation functors}

\begin{definition}
A functor of Artin rings is a covariant functor 
\begin{align*}
F:\bold{Art}\to \bold{Sets}
\end{align*}
such that $F(\mathbb{C})$ has only one element, where $\bold{Art}$ is the category of local artinian $\mathbb{C}$-algebra with residue $\mathbb{C}$, and $\bold{Sets}$ is the category of sets.
\end{definition}

\begin{definition}[functors associated with a DGLA $L$]
Let $L=(\bigoplus_{i\geq0}L_i,d,[-,-])$ be a differential graded Lie algebra and $(R,\mathfrak{m})\in \bold{Art}$. Let $MC(L)(R)$ be the set of all Maurer-Cartan elements of $L\otimes \mathfrak{m}$, i.e
\begin{align*}
MC_L(R)=\{x\in L_1\otimes \mathfrak{m} |dx+\frac{1}{2}[x,x]=0\}
\end{align*}
Then $MC_L:\bold{Art}\to \bold{Sets}$ is a functor.
\end{definition}

Let $g=L\otimes \mathfrak{m}$ be the induced differential graded Lie algebra from $L$. Then $g_0=L_0\otimes \mathfrak{m}$ is a nilpotent Lie algebra. Then the set $exp(g_0)=\{e^x|x\in g_0\}$ forms a group by the Campbell-Hausdorff formula. We have a group action of $exp(g_0)$ on $g_1=L_1\otimes \mathfrak{m}$ given by
\begin{align*}
e^a\cdot x:=x+\sum_{n\geq 1} \frac{(ad\,a)^{n-1}}{n!}([a,x]-da)\,\,\,\,\,\,\,\,\,\text{where}\,\,\,ad\,a:g_1\to g_1\,\,\,\text{defined by} \,\,b\mapsto [a,b]
\end{align*}
The action is known as the gauge action of $exp(g_0)$ on $g_1$ and the set of Maurer-Cartan elements is stable under the gauge action.

\begin{definition}
Let $x,y\in MC_L(R)$. We say that $x$ is gauge equivalent to $y$ if there exists $e^a\in exp(g_0)$ such that $e^a\cdot x=y$. Let $Def_L(R)$ be the set of all gauge equivalence classes of elements of $MC_L(R)$. Then the functor $Def_L:\bold{Art}\to \bold{Sets}$ is called the deformation functor associated to the differential graded Lie algebra $L$.
\end{definition}

\begin{definition}
We say that a functor of Artin rings $F$ is controlled by a differential graded Lie algebra $L$ if $F\cong Def_L$.
\end{definition}

\begin{definition}[Poisson deformation functor]
Let $(X,\Lambda_0)$ be a compact holomorphic Poisson manifold. The Poisson deformation functor $Def_{(X,\Lambda_0)}:\bold{Art}\to \bold{Sets}$ is defined by
\begin{align*}
PDef_{(X,\Lambda_0)}(R)=\{\text{equivalent classes of infinitesimal Poisson deformations of $(X,\Lambda_0)$ over $R$}\}
\end{align*}
\end{definition}

In the next subsection, we will prove that $Def_\mathfrak{g}\cong PDef_{(X,\Lambda_0)}$.
\subsection{The Poisson deformation functor $Def_{(X,\Lambda_0)}$ is controlled by the differential graded Lie algebra $\mathfrak{g}=(\bigoplus_{p+q-1=i,q\geq1} A^{0,p}(X,\wedge^q T),L=\bar{\partial}+[\Lambda_0,-],[-,-])$}\

Let $\mathcal{X}\to Spec\,R$ be an infinitesimal Poisson deformation of $(X,\Lambda_0)$, where $(R,\mathfrak{m})$ is generated by $<1,m_1,...,m_r>$ with exponent $n$. Let $\{U_{\alpha}\}$ be a locally trivial open cover of $\mathcal{X}$. We defined a sheaf $\mathscr{A}^{0,p}(\wedge^q T_{\mathcal{X}/R})$, a bracket $[-,-]$ and the deaulbault differential $\bar{\partial}_\mathcal{X}$. For $q=0$, we will denote this sheaf by $\mathscr{A}^{0,p}_\mathcal{X}$ and similarly for $\mathscr{A}^{0,p}_X$.

Given an element $v:=\phi+\Lambda\in (A^{0,1}(X,T)\oplus A^{0,0}(X,\wedge^2 T))\otimes \mathfrak{m}$ with $L(\phi+\Lambda)=-\frac{1}{2}[\phi+\Lambda,\phi+\Lambda]$. In particular, we have $\bar{\partial} \phi=-\frac{1}{2}[\phi,\phi]$.  We define an operator $\bar{\partial}+\phi:=\bar{\partial}\otimes 1+\phi$ on $\mathscr{A}^{0,p}\otimes R$ and a sequence
\begin{equation}\label{2se}
 0\to \mathscr{A}_X^{0,0}\otimes R\xrightarrow{\bar{\partial}+\phi} \mathscr{A}^{0,1}_X\otimes R \xrightarrow{\bar{\partial}+\phi}\cdots \xrightarrow{\bar{\partial}+\phi}  \mathscr{A}^{0,p}_X\otimes R\xrightarrow{\bar{\partial}+\phi}
\end{equation}

which is a complex by the condition $\bar{\partial}\phi=-\frac{1}{2}[\phi,\phi]$. By tensoring $\otimes_{R} R/\mathfrak{m}$, we have

\begin{align*}
0\to \mathscr{A}^{0,0}_X\otimes \mathbb{C}\xrightarrow{\bar{\partial}} \mathscr{A}^{0,1}_X\otimes \mathbb{C} \xrightarrow{\bar{\partial}}\cdots
\end{align*}
which is a acyclic resolution of $\mathcal{O}_X$. Hence the complex $(\ref{2se})$ is exact in positive degree.

We define
\begin{align*}
\mathcal{O}(v):=ker(\bar{\partial}+\phi:\mathscr{A}_X^{0,0}\otimes R\to \mathscr{A}_X^{0,1}\otimes R)
\end{align*}
which is a flat $R$-sheaf on $X$.\footnote{In general, let $R$ be an artinian local ring with residue field $k$ and $M$ be an $R$-module. Let
\begin{align*}
M\to N^0\to N^1\to \cdots
\end{align*} be a flat resolution such that this induces a resolution
\begin{align*}
M\otimes k\to N^0\otimes k\to N^1\otimes k\cdots
\end{align*} Then $M$ is a $R$-flat.} $\mathcal{O}(v)$ has a Poisson bracket induced by $\Lambda_0+\Lambda$. We define on $\{-,-\}$ on $\mathcal{O}(v)$ by
\begin{align*}
\{f,g\}:=[[\Lambda_0+\Lambda,f],g]
\end{align*}
for local sections $f,g\in \mathcal{O}(v)$. Then $\{-,-\}$ defines a biderivation since $d(gh)=gdh+hdg$ and $R$-bilinear since for $a=a_0+a_1m_1+\cdots a_rm_r\in R$ where $a_i\in \mathbb{C}$, we have $da=0$. $[\Lambda_0+\Lambda,\Lambda_0+\Lambda]=0$ shows that $\{-,-\}$ satisfies the Jacobi identity. So it remains to show that $\mathcal{O}(v)$ is closed under $\{-,-\}$. Note that for $f,g\in \mathcal{O}(v)$, we have $\bar{\partial}f+[\phi,f]=0$, $\bar{\partial}g+[\phi,g]=0$. We set $\omega=\Lambda_0+\Lambda$. Then we have $\bar{\partial} \omega+ [\phi,\omega]=0$
\begin{align*}
\bar{\partial}[[\omega,f],g]+[\phi,[[\omega,f],g]]&=[\bar{\partial}[\omega,f],g]+[[\omega,f],\bar{\partial}g]+[[\phi,[\omega,f]],g]+[[\omega,f],[\phi,g]]\\
                                                                        &=[[\bar{\partial}\omega,f],g]-[[\omega,\bar{\partial}f],g]+[[[\phi,\omega],f],g]-[[\omega,[\phi,f]],g]\\
                                                                        &=0
\end{align*}
So we have $\{f,g\}\in \mathcal{O}(v)$. Hence $\mathcal{O}(v)$ is a sheaf of Poisson $R$-algebras. We also have $\mathcal{O}(v)\otimes_{R} R/\mathfrak{m}\cong \mathcal{O}_X$ as Poisson sheaves over $\mathbb{C}$. So $\mathcal{O}(v)$ defines an infinitesimal Poisson deformation of $(X,\Lambda_0)$ over $R$.

Now we define the map 
\begin{align*}
\mathcal{O}:Def_{\mathfrak{g}}(R)&\to PDef_{(X,\Lambda_0)}(R)\\
v=\phi+\Lambda&\mapsto \mathcal{O}(v)
\end{align*}

We claim that $v$ is gauge equivalent to $w$ if and only if $\mathcal{O}(v)\cong\mathcal{O}(w)$ as Poisson $R$-sheaves. This claim shows that the map $\mathcal{O}$ is well-defined and injective. Let $v=\phi+\Lambda,w=\psi+\Pi\in (A^{0,1}(X, T)\oplus A^{0,0}(X,\wedge^2 T))\otimes \mathfrak{m}$. Since $v$ is gauge equivalent to $w$,  for some $a\in A^{0,0}(X,T)\otimes \mathfrak{m}$, we have
\begin{align*}
(1)\psi&=\phi+\sum_{n\geq 1} \frac{(ad\,a)^{n-1}}{n!}([a,\phi]-\bar{\partial}a)\\
(2)\Pi&=\Lambda+\sum_{n\geq 1} \frac{(ad\,a)^{n-1}}{n!}([a,\Lambda]-[\Lambda_0,a])=exp(a)(\Lambda_0+\Lambda)-\Lambda_0
\end{align*}

$(1)$ is equivalent that the following commutative diagram commutes:

\begin{center}
$\begin{CD}
\mathscr{A}^{0,0}_X\otimes R @>\bar{\partial}+\phi >> \mathscr{A}_X^{0,1}\otimes R\\
@V exp(a) VV @VV exp(a)V\\
\mathscr{A}^{0,0}_X\otimes R @>\bar{\partial}+\phi'>> \mathscr{A}^{0,1}_X\otimes R
\end{CD}$
\end{center}
which implies $\mathcal{O}(v)\cong \mathcal{O}(w)$ as sheaves of $R$-algebras.

$(2)$ means $\Lambda_0+\Pi=exp(a)(\Lambda_0+\Lambda)$ which means $\mathcal{O}(v)\cong \mathcal{O}(w)$ as sheaves of Poisson $R$-algebras. So we get the claim.

Now we show that $\mathcal{O}:Def_{\mathfrak{g}}(R) \to Def_{(X,\Lambda_0)}(R)$ is surjective. For given an infinitesimal Poisson deformation of $(X,\Lambda_0)$ over $(R,\mathfrak{m})$, we showed that there is a canonically associated element $v=\phi+\Lambda \in (A^{0,1}(X,T)\oplus A^{0,0}(X,\wedge^2 T))\otimes \mathfrak{m}$ with $L(\phi+\Lambda)+\frac{1}{2}[\phi+\Lambda,\phi+\Lambda]=0$.
We claim that for each $\alpha$, the following diagram is commutative.
\begin{center}
$\begin{CD}
\mathscr{A}^{0,0}_X(U_{\alpha})\otimes R @>\bar{\partial}+\phi >> \mathscr{A}^{0,1}_X(U_{\alpha})\otimes R\\
@V exp(s_{\alpha}) VV @VV exp(s_{\alpha})V\\
\mathscr{A}^{0,0}_X(U_{\alpha})\otimes R@>\bar{\partial}_{\mathcal{X}}=\bar{\partial}>> \mathscr{A}^{0,1}_X(U_{\alpha}) \otimes R
\end{CD}$
\end{center}
Note that $exp(s_{\alpha})=\varphi_{\alpha}\circ C$. Indeed, the commutativity means that 
\begin{align*}
\bar{\partial}f+\phi (f)&=\bar{\partial} f+exp(-s_{\alpha})\circ \bar{\partial}(exp(s_{\alpha}))f\\
                                 &=\bar{\partial}f+exp(-s_{\alpha})\circ (\bar{\partial}\circ exp(s_{\alpha})-exp(s_{\alpha})\circ \bar{\partial})f\\
                                 &=exp(-s_{\alpha})\circ \bar{\partial}\circ exp(s_{\alpha})f
\end{align*}

Since the diagram is compatible with each $\alpha$, we have the following commutative diagram of sheaves

\begin{center}
$\begin{CD}
\mathscr{A}^{0,0}_X\otimes R@>\bar{\partial}+\phi>> \mathscr{A}^{0,1}_X\otimes R\\
@V\cong VV @VV\cong V\\
\mathscr{A}^{0,0}_\mathcal{X}@>\bar{\partial}_\mathcal{X}>> \mathscr{A}^{0,1}_\mathcal{X}
\end{CD}$
\end{center}
So we have isomorphism of sheaves
\begin{align*}
\mathcal{O}(v):=ker(\bar{\partial}+\phi:\mathscr{A}_X^{0,0}\otimes R\to \mathscr{A}_X^{0,1}\otimes R)\cong \mathcal{O}_{\mathcal{X}}=ker(\bar{\partial}_\mathcal{X}:\mathcal{A}^{0,0}_\mathcal{X}\to \mathcal{A}^{0,1}_\mathcal{X})
\end{align*}

$\mathcal{O}(v)$ is a sheaf of Poisson $R$-algebras as above defined by
\begin{align*}
\{f,g\}:=[[\Lambda_0+\Lambda,f],g]\,\,\,\,\,\,\,\,\,\text{for local sections $f,g\in \mathcal{O}(v)$}
\end{align*}
Now we claim that as Poisson $R$-sheaves, we have
\begin{align*}
\mathcal{O}(v)\cong \mathcal{O}_{\mathcal{X}}
\end{align*}
We check this locally on $U_{\alpha}$: for $f,g\in \Gamma(U_{\alpha},\mathcal{O}(v))=ker(\bar{\partial}+\phi:\mathscr{A}^{0,0}_X(U_{\alpha})\otimes R\to \mathscr{A}_X^{0,1}(U_{\alpha})\otimes R)$,
\begin{align*}
exp(s_{\alpha})\{f,g\}&=exp(s_{\alpha})[[\Lambda_0+\Lambda,f],g]=exp(s_{\alpha})[[\Lambda_{\alpha}'',f],g]=[[exp(s_{\alpha})\Lambda_{\alpha}'',exp(s_{\alpha})f],exp(s_{\alpha})g]\\
&=[[exp(s_{\alpha})exp(-s_{\alpha})\Lambda_{\alpha}',exp(s_{\alpha})f],exp(s_{\alpha})g]=[[\Lambda_{\alpha}',exp(s_{\alpha})f],exp(s_{\alpha})g]\\
\end{align*}
where $\Lambda_{\alpha}'\in \Gamma(U_{\alpha}, \mathscr{A}^{0,0}(\wedge^2 T_X))\otimes R$ is the Poisson structure on $\mathcal{O}_X(U_{\alpha})\otimes R\cong \mathcal{O}_{\mathcal{X}}(U_{\alpha})$ and $\Lambda_{\alpha}''= exp(-s_{\alpha})\Lambda_{\alpha}'$. (See Remark \ref{2remark} for notations)

Hence the infinitesimal Poisson deformation $\mathcal{X}$ of $(X,\Lambda_0)$ over $R$ is equivalent to $\mathcal{O}(v):=ker(\bar{\partial}+\phi:\mathscr{A}_X^{0,0}\otimes R\to \mathscr{A}_X^{0,1}\otimes R)$ equipped with the Poisson structure $\Lambda_0+\Lambda$. This shows that the map $\mathcal{O}:Def_{\mathfrak{g}}(R)\to Def_{(X,\Lambda_0)}(R)$ is surjective. 

So we proved that for an artinian local $\mathbb{C}$-algebra $R$ with residue $\mathbb{C}$, we have an isomorphism $\mathcal{O}:MC_{\mathfrak{g}}(R)\to PDef_{(X,\Lambda_0)}(R)$. To show that $Def_{\mathfrak{g}}\cong PDef_{(x,\Lambda_0)}$, we have to show that $\mathcal{O}$ is a morphism of functors of Artin rings, in other words, $\mathcal{O}$ is compatible with any local homomorphism $R\to S$ in $\bold{Art}$.
 
 \begin{definition}[Base change]
Given an infinitesimal Poisson deformation $\mathcal{X}$ of $(X,\Lambda)$ over $R$, and a local $\mathbb{C}$-algebra homomorphism $(R,\mathfrak{m}_R)\to (S,\mathfrak{m}_S)$, we can define an infinitesimal Poisson deformation $\mathcal{X}\times_{Spec\,R} Spec\,S$ of $(X,\Lambda)$ over $S$ by base change.
\begin{center}
$\begin{CD}
X@>>>\mathcal{X}\times_{Spec\,R} Spec\,S @>>> \mathcal{X}\\
@VVV@VVV @VVV\\
Spec\,\mathbb{C}@>>>Spec\,S@>>> Spec\,R
\end{CD}$
\end{center}
We only need to explain the induced Poisson structure of $\mathcal{X}_S:=\mathcal{X}\times_{Spec\,R} Spec\,S$ over $S$. For any open set $U$ of $X$, $\mathcal{O}_{\mathcal{X}_S}(U)=\mathcal{O}_{\mathcal{X}}(U)\otimes_R S$. We define the Poisson bracket $\{-,-\}_S$ on $\mathcal{O}_{\mathcal{X}_S}(U)$ by
\begin{align*}
\{f\otimes s_1,g\otimes s_2\}_S=\{f,g\}_R\otimes s_1s_2
\end{align*}
where $\{-,-\}_R$ is the Poisson bracket on $\mathcal{O}_{\mathcal{X}}(U)$.
\end{definition}
 
We note that the induced infinitesimal Poisson deformation by the base change $(R,\mathfrak{m}_R)\to (S,\mathfrak{m}_S)$ can be interpreted in terms of a Mauer-Cartan element of $\mathcal{X}$. Let $\phi+\Lambda \in (A^{0,1}(X,T)\oplus A^{0,0}(X,\wedge^2 T))\otimes \mathfrak{m}_R$ be a Mauer Cartan element of $\mathcal{X}$. Hence $\mathcal{O}_{\mathcal{X}}$ is equivalent to $ker(\bar{\partial}+\phi:\mathscr{A}^{0,0}\otimes R\to \mathscr{A}^{0,1}\otimes R)$ with the Poisson structure $\Lambda_0+\Lambda$. The homomorphism $g:(R,m_R)\to (S, m_S)$ induces the homomorphisms $A^{0,p}(X,\wedge^q T)\otimes \mathfrak{m}_R\to A^{0,p}(X,\wedge^q T)\otimes \mathfrak{m}_S$. Let $\phi_S+\Lambda_S$ be the image of $\phi+\Lambda$, which also satisfy $L_S(\phi_S+\Lambda_S)+\frac{1}{2}[\phi_S+\Lambda_S,\phi_S+\Lambda_S]=0$. We have the following commutative diagram
\begin{center}
$\begin{CD}
(\mathscr{A}_X^{0,0}\otimes R,\Lambda_0+\Lambda_R)@>\bar{\partial}+\phi>> \mathscr{A}_X^{0,1}\otimes R\\
@VVV @VVV\\
(\mathscr{A}_X^{0,0}\otimes S, \Lambda_0+\Lambda_S) @>\bar{\partial}+\phi_S>> \mathscr{A}_X^{0,1}\otimes S
\end{CD}$
\end{center}
We claim that 
\begin{proposition}
$\mathcal{O}_{\mathcal{X}_S}$ is equivalent to $(\phi_S+\Lambda_S)$.
\end{proposition}
\begin{proof}
 We recall that for given locally trivial open covering $\{U_{\alpha}\}$, and local trivialization $\varphi_{\alpha}:\mathcal{O}_{\mathcal{X}}(U_{\alpha})\to \mathcal{O}_X(U_{\alpha})\otimes R$ and $C^{\infty}$-trivialization $C:\mathscr{A}^{0,0}_X\otimes R\to \mathscr{A}^{0,0}_{\mathcal{X}}$ for the family $\mathcal{X}$, we have $\phi=exp(-s_{\alpha})\bar{\partial}exp(s_{\alpha})\in A^{0,1}(X,T)\otimes \mathfrak{m}_R$ and $\Lambda=exp(-s_{\alpha})(\Lambda_0+\Lambda_{\alpha})-\Lambda_0\in A^{0,0}(X,\wedge^2 T)\otimes \mathfrak{m}_R$. Now we consider the family $\mathcal{X}_S$ over $S$. For the same open covering $\{U_{\alpha}\}$, the local trivialization is induced from $\psi_{\alpha}$ by tensoring $\otimes_R S$ and $C^{\infty}$ trivialization is also induced from $C$ by tensoring $\otimes_R S$. This observation gives the proposition.
\end{proof}

Hence we have the following commutative diagram
\begin{center}
$\begin{CD}
Def_{\mathfrak{g}}(R)@>\mathcal{O}>> PDef_{(X,\Lambda_0)}(R)\\
@VVV @VVV\\
Def_{\mathfrak{g}}(S)@>\mathcal{O}>> PDef_{(X,\Lambda_0)}(S)
\end{CD}$
\end{center}

Hence we proved the following theorem.

\begin{thm}
Let $(X,\Lambda_0)$ be a compact holomorphic Poisson manifold. Then the Poisson deformation functor $Def_{(X,\Lambda_0)}$ is controlled by the differential graded Lie algebra $\mathfrak{g}=(\bigoplus_{p+q-1=i,p\geq 1,q\geq 1}$
$A^{0,p}(X,\wedge^q T),L=\bar{\partial}+[\Lambda_0,-],[-,-])$. In other words, we have an isomorphism of two functors
\begin{align*}
Def_\mathfrak{g}\cong PDef_{(X,\Lambda_0)}
\end{align*}
\end{thm}

\section{Universal Poisson deformations}

Now we assume that for a holomorphic Poisson manifold $(X,\Lambda_0)$, $HP^1(X,\Lambda_0)=0$.

\subsection{Independence of choices of morphic elements giving the same cohomology class }\

Let $\mathfrak{m}$ be the maximal ideal of a local artinian $\mathbb{C}$-algebra $R$ with residue such that $\mathfrak{m}^{n+1}=0$. Our goal is that for given $v=\phi+\Lambda$, $w=\psi+\Pi \in g_1\otimes \mathfrak{m}$ such that $[\epsilon(v)]=[\epsilon(w)]$ in $\mathbb{H}^0(J_n(\mathfrak{g}))\otimes \mathfrak{m}$, where $\epsilon(v)=(\bar{v},\frac{1}{2}\bar{v}\odot \bar{v},...,\frac{1}{n!}\underbrace{\bar{v}\odot \cdots \odot \bar{v}}_n)\cong_{dec}(v,\cdots, (-1)^{\frac{(n-1)n}{2}} \frac{1}{n!}\underbrace{v\wedge \cdots \wedge v}_n)\in \bigoplus_{i=1}^n sym^ig_1\otimes \mathfrak{m}^i$ and $L v = -\frac{1}{2}[v,v]$, and same for $w$, we want to show that the Poisson deformation $\mathcal{O} (v)$ is equivalent to the Poisson deformation $\mathcal{O}(w)$. In other words, we want to show that 

\begin{enumerate}
\item there exists $u_0\in g_0\otimes \mathfrak{m}$ such that $exp(u_0)(\bar{\partial}+\phi)exp(-u_0)=\bar{\partial}+\psi$.
\begin{center}
$\begin{CD}
\mathscr{A}^{0,0}_X\otimes R @>\bar{\partial} +\phi >> \mathscr{A}^{0,1}_X\otimes R\\
@V exp(u_0) VV  @VV exp(u_0) V \\
\mathscr{A}^{0,0}_X\otimes R @>\bar{\partial}+\psi >> \mathscr{A}^{0,1}_X\otimes R
\end{CD}$
\end{center}
\item $exp(u_0)(\Lambda_0+\Lambda)=\Lambda_0+\Pi$
\end{enumerate}

We will prove the statement by induction on the exponent $k$ of maximal ideal of artinian local $\mathbb{C}$-algebra with residue $\mathbb{C}$. Let $k=1$. So we have $\mathfrak{m}^2=0$. Let $[\epsilon(v)]=[\epsilon(w)]\in \mathbb{H}^0(J_1(\mathfrak{g}))\otimes \mathfrak{m}$. Then there exists $u_0\in g_0\otimes \mathfrak{m}$ such that $(-1)^1Lu_0=w-v$. So we have $Lu_0=\bar{\partial}u_0+[\Lambda_0,u_0]=v-w$. In other words,
\begin{align*}
\bar{\partial}u_0=\phi-\psi\\
[\Lambda_0,u_0]=\Lambda-\Pi
\end{align*}

We note that $exp(u_0)=1+u_0$ and $exp(-u_0)=1-u_0$. Then since $u_0\in g_0\otimes \mathfrak{m}$, $\phi,\psi, \Lambda,\Pi \in g_1\otimes \mathfrak{m}$ and $\mathfrak{m}^2=0$, we have

\begin{enumerate}
\item \begin{align*}
\exp(u_0)(\bar{\partial}+\phi)exp(-u_0)&=(1+u_0)(\bar{\partial}+\phi)(1-u_0)=(\bar{\partial}+\phi+u_0\bar{\partial}+u_0\phi)(1-u_0)\\
                                                           &=\bar{\partial}+\phi+u_0\bar{\partial}+u_0\phi-\bar{\partial}\cdot u_0-\phi u_0-u_0\bar{\partial}\cdot u_0-u_0\phi u_0\\
                                                           &=\bar{\partial}+\phi+u_0\bar{\partial}-\bar{\partial}\cdot u_0=\bar{\partial}+\phi-\bar{\partial}u_0\\
                                                           &=\bar{\partial}+\psi
\end{align*}
\item \begin{align*}
exp(u_0)(\Lambda_0+\Lambda)&=(1+u_0)(\Lambda_0+\Lambda)=\Lambda_0+\Lambda+[u_0,\Lambda_0+\Lambda]\\
                                                  &=\Lambda_0+\Lambda-[\Lambda_0,u_0]=\Lambda_0+\Lambda-(\Lambda-\Pi)\\
                                                  &=\Lambda_0+\Pi
\end{align*}
\end{enumerate}

So the statement holds for $k=1$.

Now let's assume that the statement holds for $k=n-1$. Now let $(R,\mathfrak{m})$ be a local artinian $\mathbb{C}$-algebra with exponent $n$ and let $[\epsilon(v)]=[\epsilon(w)]\in \mathbb{H}^0(J_n(\mathfrak{g}))\otimes \mathfrak{m}$ where $v,w \in g_1\otimes \mathfrak{m}$. We have the following exact sequence of finite dimensional vector spaces
\begin{align*}
0\to \mathfrak{m}^{n}\to R \to  R/\mathfrak{m}^n \to 0
\end{align*}
So we have the following splitting as vector spaces 
\begin{align*}
R\cong (R/\mathfrak{m}^n) \oplus\mathfrak{m}^n
\end{align*}

Note that $(R/\mathfrak{m}^n, \mathfrak{m}/\mathfrak{m}^n)$ is a local artinian $\mathbb{C}$-algebra with exponent $n-1$. Now let $v=v_1+v_2=(\phi_1+\Lambda_1)+(\phi_2+\Lambda_2)$ and $w=w_1+w_2=(\psi_1+\Pi_1)+(\psi_2+\Pi_2)$, where $v_1,w_1\in g_1\otimes (R/\mathfrak{m}^n)$ and $v_2,w_2\in \mathfrak{m}^n$. Then we have $[\epsilon(v_1)]=[\epsilon(w_1)]\in \mathbb{H}^0(J_{n-1}(\mathfrak{g}))\otimes \mathfrak{m}/\mathfrak{m}^n$, where $v_1,w_1\in g_1\otimes \mathfrak{m}/\mathfrak{m}^n$. By the induction hypothesis, 
\begin{enumerate}
\item we have the following commutative diagram.

\begin{center}
$\begin{CD}
\mathscr{A}_X^{0,0}\otimes R/\mathfrak{m}^n@>\bar{\partial}+\phi_1 >> \mathscr{A}_X^{0,1}\otimes R/\mathfrak{m}^n\\
@V exp(u) VV @VV exp(u) V \\
\mathscr{A}_X^{0,0}\otimes R/\mathfrak{m}^n @>\bar{\partial}+\psi_1>> \mathscr{A}_X^{0,1}\otimes R/\mathfrak{m}^n
\end{CD}$
\end{center}
where some $u \in g_0\otimes \mathfrak{m}/\mathfrak{m}^n$. 
\item $exp(u)(\Lambda_0+\Lambda_1)=\Lambda_0+\Pi_1$.
\end{enumerate}

Let's consider the natural projection $g_0\otimes R \cong g_0\otimes (R/\mathfrak{m}^n\oplus \mathfrak{m}^n)\to g_0\otimes R/\mathfrak{m}^n$.
Choose the lifting of $u$ to be $u+0\in g_0 \otimes (R/\mathfrak{m}^n\oplus \mathfrak{m}^n)\cong g_0\otimes R$. Then
\begin{enumerate}
\item we have the following commutative diagram.
\begin{center}
$\begin{CD}
\mathscr{A}_X^{0,0}\otimes R@>\bar{\partial}+\phi_1+\phi_2 >> \mathscr{A}_X^{0,1}\otimes R\\
@V exp(u+0) VV @VV exp(u+0) V \\
\mathscr{A}_X^{0,0}\otimes R @>\bar{\partial}+\psi_1+\phi_2>> \mathscr{A}_X^{0,1}\otimes R
\end{CD}$
\end{center}

since $\phi_2=\phi_2\circ exp(u)=exp(u)\circ \phi_2$. (note that $\phi_2\in A^{0,1}(X,T)\otimes \mathfrak{m}^n$, $u\in g_0\otimes \mathfrak{m}$ and $\mathfrak{m}^{n+1}=0$.)
\item $exp(u)(\Lambda_0+\Lambda_1+\Lambda_2)=\Lambda_0+\Pi_1+\Lambda_2$ since $exp(u)(\Lambda_2)=\Lambda_2$ (note that $\Lambda_2\in A^{0,0}(X,\wedge^2 T)\otimes \mathfrak{m}^n$).
\end{enumerate}
Since $\mathcal{O}(v_1+v_2)$ and $\mathcal{O}(w_1+v_2)$ are equivalent Poisson deformations, we have $[\epsilon(v_1+v_2)]=[\epsilon(w_1+v_2)]$. If we show that $\mathcal{O}(w_1+v_2)$ is equivalent to $\mathcal{O}(w_1+w_2)$, then this means that $\mathcal{O}(v)$ is equivalent to $\mathcal{O}(w)$.

Since $[\epsilon(w_1+v_2)]=[\epsilon(w_1+w_2)]$, there exists $(u_0,...,u_{n-1})\in \bigoplus_{i=0}^{n-1}g_0\otimes sym^ig_1\otimes \mathfrak{m}^{i+1}$ such that 
\begin{center}
\tiny{$\begin{CD}
u_0 @>(-1)^1L>> w_2-v_2 \\
@. @A\delta AA\\
@. u_1 @>(-1)^2 L>> -\frac{1}{2}((w_1+w_2)^2-(w_1+v_2)^2)=0 \\
@. @. @A\delta AA\\
@. @. \cdots @>>> \cdots\\
@. @. @. @A\delta AA\\
@. @. @. u_{n-1} @>(-1)^n L>> (-1)^{\frac{(n-1)n}{2}}\frac{1}{n!} ((w_1+ w_2)^n - (w_1+v_2)^n) =0
\end{CD}$}
\end{center}

Since $v_2, w_2\in g_1\otimes \mathfrak{m}^n$ and $w_1\in g_1\otimes \mathfrak{m}$, we have 
\begin{align*}
(-1)^{\frac{(i-1)i}{2}}\frac{1}{i!}((w_1+ w_2)^i-(w_1+ v_2)^i)=0\,\,\,\text{ for } \,\,\, i>1.
\end{align*}
Let's consider $u_{n-1}\in  g_0\otimes sym^{n-1}g_1\otimes \mathfrak{m}^n$. Write $u_{n-1}=\sum_k a_k\otimes b_k$ where $a_k\in g_0$ and $b_k$ are linearly independent in $sym^{n-1}g_1\otimes \mathfrak{m}^n$. Then since $(-1)^nLu_{n-1}=0$, we have $\sum_k La_k\otimes b_k+a_k\otimes L b_k$=0. So $L a_k=0$. Since $b_k$ are linearly independent and, $L a_k\otimes b_k$ and $a_k\otimes Lb_k$ live in different spaces, we have $La_k=0$. Since $H^0(\mathfrak{g})=HP^1(X,\Lambda_0)=0$, we have $a_k=0$. $u_{n-1}=0$. In this way we can show that $u_1=...=u_{n-1}=0$. Hence we have $(-1)^1Lu_0=w_2-v_2\in g_1\otimes \mathfrak{m}^n$. So $Lu_0=v_2- w_2$. In other words, we have
\begin{align*}
\bar{\partial}u_0=\phi_2-\psi_2\\
[\Lambda_0,u_0]=\Lambda_2-\Pi_2
\end{align*}
Let $u_0:=x_1+x_2\in g_0\otimes R\cong g_0\otimes (R/\mathfrak{m}^n\oplus \mathfrak{m}^n)$. Since $Lx_1=0$ and $Lx_2=v_2-w_2$, we have $x_1=0$ by $H^0(\mathfrak{g})=0$. So $u_0\in g_0\otimes \mathfrak{m}^n$. Then we have the following commutative diagram.
\begin{center}
$\begin{CD}
\mathscr{A}_X^{0,0}\otimes R@>\bar{\partial}+\psi_1+\phi_2 >> \mathscr{A}_X^{0,1}\otimes R\\
@V exp(u_0) VV @VV exp(u_0) V \\
\mathscr{A}_X^{0,0}\otimes R@>\bar{\partial}+\psi_1+\psi_2>> \mathscr{A}_X^{0,1}\otimes R
\end{CD}$
\end{center}
Indeed,
\begin{align*}
exp(u_0)(\bar{\partial}+\psi_1+\phi_2)exp(-u_0)&=(1+u_0)(\bar{\partial}+\psi_1+\phi_2)(1-u_0)\\
                                                           &=\bar{\partial}+\psi_1+\phi_2+u_0\bar{\partial}-\bar{\partial}u_0\\
                                                           &=\bar{\partial}+\psi_1+\phi_2-\bar{\partial}u_0\\
                                                           &=\bar{\partial}+\psi_1+\psi_2
                                                           \end{align*}
And 
\begin{align*}
exp(u_0)(\Lambda_0+\Pi_1+\Lambda_2)&=(1+u_0)(\Lambda_0+\Pi_1+\Lambda_2)\\
                                                                         &=\Lambda_0+\Pi_1+\Lambda_2+[u_0,\Lambda_0]=\Lambda_0+\Pi_1+\Lambda_2-[\Lambda_0,u_0]\\
                                                                         &=\Lambda_0+\Pi_1+\Pi_2
\end{align*}                                                        
So the induction holds for $k=n$.

\subsection{$n$-th Universal Poisson deformations}\
Recall that $R_n^u=\mathbb{C}\oplus \mathfrak{m}_n^u:=\mathbb{C}\oplus \mathbb{H}^0(J_n(\mathfrak{g}))^*$ is a local artinian $\mathbb{C}$-algebra with residue $\mathbb{C}$ and expoent $n$ (i.e. $\mathfrak{m}_n^{u\,n+1}=0$).
\begin{definition}[$n$-th universal Poisson deformation]
Since the identity map $\mathbb{H}^*(J_n(\mathfrak{g}))\to\mathbb{H}^*(J_n(\mathfrak{g}))$ is a homomorphism,  it corresponds to a morphic element 
\begin{align*}
[\epsilon(v_u)]=[(\bar{v}_u,\frac{1}{2}\bar{v}_u\odot v_u,....,\frac{1}{n!}\underbrace{\bar{v}_u\odot\cdots \odot \bar{v}_u}_n)]\in \mathbb{H}^0(J_n(\mathfrak{g}))\otimes \mathfrak{m}_n^u
\end{align*}
where $v_u:=\phi_u+\Lambda_u\in g_1\otimes \mathbb{H}^0(J_n(\mathfrak{g}))^*$. Then $v_u$ defines an infinitesimal Poisson deformaiton $P_n^u:=\mathcal{O}(v_u)$ over a local artinian $\mathbb{C}$-algebra $R_n^u:=\mathbb{C}\oplus \mathbb{H}^0(J_n(\mathfrak{g}))^*$. We will call $P_n^u$ be a $n$-th order universial Poisson deformation of $(X,\Lambda_0)$ over $R_n^u$.
\end{definition}
Let $P$ be an infinitesimal Poisson deformation of $(X,\Lambda_0)$ over $(R,\mathfrak{m})$ with $\mathfrak{m}^{n+1}=0$. Assume that $HP^1(X,\Lambda_0)=0$.
Let $v=\phi+\Lambda$ be an Maurer Cartan element corresponding to the infinitesimal Poisson deformation $P$ of $(X,\Lambda_0)$ over $R$. Then $[\epsilon(v)]=[(\bar{v},\frac{1}{2}\bar{v}\odot \bar{v},...,\frac{1}{n!}\bar{v}\odot \cdots \odot \bar{v})]\in \mathbb{H}^0(J_n(\mathfrak{g}))\otimes \mathfrak{m}$, which induces a homomorphism $[\epsilon(v)]:\mathfrak{m}_n^u=\mathbb{H}^0(J_n(\mathfrak{g}))^* \to \mathfrak{m}$. Via the morphism $r:=[\epsilon(v)]$, $v_u\in g_1\otimes \mathfrak{m}_n^{u}$ is sent to $\tilde{v}_u\in g_1 \otimes \mathfrak{m}$. Then $\tilde{v}_u$ satisfies the Maurer Cartan equation since $v_u$ does. Hence $[\epsilon(\tilde{v}_u)]\in \mathbb{H}^0(J_n(\mathfrak{g}))\otimes \mathfrak{m}$ defines a morphic element and we have the corresponding homomorphism  $[\epsilon(\tilde{v}_u)]:\mathbb{H}^0(J_n(\mathfrak{g}))^*\xrightarrow{id=[\epsilon(v_u)]} \mathbb{H}^0(J_n(\mathfrak{g}))^*\xrightarrow{r} \mathfrak{m}$, which is exactly $[\epsilon(v)]$. Hence we have $[\epsilon(v)]=[\epsilon(\tilde{v}_u)]$. Hence by the assumption of $HP^1(X,\Lambda_0)=0$, the induced deformation $\tilde{v}_u=\tilde{\phi}_u+\tilde{\Lambda}_u$ from the deformation $v_u$ by the base change $[\epsilon(v)]:\mathbb{H}^0(J_n(\mathfrak{g}))^* \to \mathfrak{m}$

\begin{center}
$\begin{CD}
(\mathscr{A}_X^{0,0}\otimes R_n^u,\Lambda_0+\Lambda_u)@>\bar{\partial}+\phi_u>> \mathscr{A}_X^{0,1}\otimes R_n^u\\
@VVV @VVV\\
(\mathscr{A}_X^{0,0}\otimes R,\Lambda_0+\tilde{\Lambda}_u) @>\bar{\partial}+\tilde{\phi}_u>> \mathscr{A}_X^{0,1}\otimes R
\end{CD}$
\end{center}
 is equivalent to $v=\phi+\Lambda$, which represents the infinitesimal Poisson deformation $P$ of $(X,\Lambda_0)$ over $R$. 
 Then we have $P/R\cong r^* P_n^u=P_n^u\times_{Spec(R_n^u)} Spec(R)$. This proves our main Theorem \ref{2theorem} (2) in the Introduction of the part II of the thesis.
 
\subsection{Formal Completition}\

The natural map $\mathbb{H}^0(J_{n-1}(\mathfrak{g}))\to \mathbb{H}^0(J_n (\mathfrak{g}))$ gives dually the homomorphism $\mathbb{H}^0(J_n(\mathfrak{g}))^*\to \mathbb{H}^0(J_{n-1}(\mathfrak{g}))^*$. Set $\mathfrak{m}_n^u=\mathbb{H}^0(J_n(\mathfrak{g}))^*$ and $R_n^u:=\mathbb{C}\oplus \mathfrak{m}_n^u$. Take the inverse limit $\mathfrak{m}^u:=\varprojlim_n \mathfrak{m}_n^u$. Then we have 

\begin{align*}
\hat{R}^u:=\mathbb{C}\oplus \mathfrak{m}^u=\mathbb{C}\oplus \varprojlim_n \mathfrak{m}_n^u=\varprojlim_n (\mathbb{C}\oplus \mathfrak{m}_n^u)=\varprojlim_n R_n^u
\end{align*}

By our construction of $J_n(\mathfrak{g})$, we have $\mathbb{C}\oplus \mathfrak{m}^u/\mathfrak{m}^{u\,n+1}=\mathbb{C}\oplus \mathfrak{m}_n^u$. Hence $(\hat{R}^u,\mathfrak{m}^u)=\varprojlim(R_n^u,\mathfrak{m}_n^u)$ is a complete local noetherian $\mathbb{C}$-algebra with respect to the $\mathfrak{m}^u$-$adic$ topology.

From $\mathfrak{m}_n^u=\mathbb{H}^0(J_n(\mathfrak{g}))^*\to \mathfrak{m}_{n-1}^u=\mathbb{H}^0(J_{n-1}(\mathfrak{g}))^*$, the morphic element $[\epsilon(v_u=\phi_u+\Lambda_u)]=[(\bar{v}_u,\frac{1}{2}\bar{v}_u\odot \bar{v}_u,\cdots ,\frac{1}{n!}\bar{v}_u\odot \cdots \odot \bar{v}_u)]\in \mathbb{H}^0(J_n(\mathfrak{g}))\otimes \mathfrak{m}_n^u$ inducing the identity map on $\mathbb{H}^0(J_n(\mathfrak{g}))^*$,  gives a morphic element $[\epsilon(\tilde{v}_u=\tilde{\phi}_u+\tilde{\Lambda}_u)]=[(\bar{\tilde{v}}_u,\cdots, \frac{1}{(n-1)!}\bar{\tilde{v}}_u\odot \cdots \odot \bar{\tilde{v}}_u,0)]\in \mathbb{H}^0(J_n(\mathfrak{g}))\otimes \mathfrak{m}_{n-1}^u$ (via $\mathfrak{m}_n^u\to \mathfrak{m}_{n-1}^u$) which can be considered as an element in $\mathbb{H}^0(J_{n-1}(\mathfrak{g}))\otimes \mathfrak{m}_{n-1}^u$ inducing the identity map on $\mathbb{H}^0(J_{n-1}(\mathfrak{g}))^*$ and so we have the following commutative diagram:

\begin{center}
$\begin{CD}
\mathbb{H}^0(J_{n}(\mathfrak{g}))^* @>id=[\epsilon(v_u)] >> \mathbb{H}^0(J_{n}(\mathfrak{g}))^*\\
@VVV @VVV\\
\mathbb{H}^0(J_{n-1}(\mathfrak{g}))^*@>id=[\epsilon(\tilde{v}_u)]>> \mathbb{H}^0(J_{n-1}(\mathfrak{g}))^*
\end{CD}$
\end{center}
 Hence we have the following commutative diagram
 
 \begin{center}
$\begin{CD}
(\mathscr{A}_X^{0,0}\otimes R_n^u,\Lambda_0+\Lambda_u)@>\bar{\partial}+\phi_u>> \mathscr{A}_X^{0,1}\otimes R_n^u\\
@VVV @VVV\\
(\mathscr{A}_X^{0,0}\otimes R_{n-1}^u,\Lambda_0+\tilde{\Lambda}_u))@>\bar{\partial}+\tilde{\phi}_u>>\mathscr{A}_X^{0,1}\otimes R_{n-1}^u
\end{CD}$
\end{center}

So $n$-th universal Poisson  deformations $P_n^u/R_n^u$ fit together to form a direct system with limit

\begin{align*}
\hat{P}^u/\hat{R}^u=\varinjlim P_n^u/R_n^u
\end{align*}

Now we set $\hat{R}=\varprojlim (R_n,\mathfrak{m}_n)$ is a complete local noetherian $\mathbb{C}$-algebra where $(R_n,\mathfrak{m}_n)$ is a local artinian $\mathbb{C}$-algebra with residue $\mathbb{C}$ and $\hat{P}/\hat{R}= \varinjlim_n P_n/R_n$ is a formal Poisson analytic space over $\hat{R}$, where $P_n/R_n$ is an infinitesimal Poisson deformation of $(X,\Lambda_0)$, which can be interpreted as a sequence $\{r_n\}$ of morphic elements where $r_n \in \mathbb{H}^0(J_n(\mathfrak{g})) \otimes \mathfrak{m}_n$ such that $r_n$ induces $r_{n-1}\in \mathbb{H}^0(J_{n-1}(\mathfrak{g}))\otimes \mathfrak{m}_{n-1}$ by the natural map $\mathfrak{m}_n\to \mathfrak{m}_{n-1}$. Hence we have the following commutative diagram:
\begin{center}
$\begin{CD}
\mathbb{H}^0(J_n(\mathfrak{g}))^* @>r_n >> \mathfrak{m}_n\\
@VVV @VVV\\
\mathbb{H}^0(J_{n-1}(\mathfrak{g}))^*@>r_{n-1}>>  \mathfrak{m}_{n-1}
\end{CD}$
\end{center}
So we have the map $\hat{r}=lim_n r_n:\hat{R}^u\to \hat{R}$ which induces $\hat{P}/\hat{R}=\hat{r}^*(\hat{P}^u/\hat{R}^u)$. This proves our main Theorem \ref{2theorem} (3) in the Introduction of the part II of the thesis. So we complete the Theorem \ref{2theorem}.

\part{Deformations of algebraic Poisson schemes}\label{part3}

In the third part of the thesis, we study deformations of algebraic Poisson schemes over an algebraic closed field $k$, which is an algebraic version of the first part of the thesis. In chapter \ref{chapter7}, we discuss the definition of Poisson schemes, morphisms and cohomology. A Poisson scheme $X$ is a scheme whose structure sheaf $\mathcal{O}_X$ is a sheaf of Poisson $k$-algebras. Equivalently, a Poisson structure on a scheme $X$ is characterized by an element $\Lambda_0\in \Gamma(X,\mathscr{H}om_{\mathcal{O}_X}(\wedge^2\Omega_{\mathcal{O}_{X}/k}^1,\mathcal{O}_X))$ with $[\Lambda_0,\Lambda_0]=0$. By a deformation of a Poisson scheme $(X,\Lambda_0)$ we mean a commutative diagram
\begin{center}
$\xi:$
$\begin{CD}
(X,\Lambda_0) @>>> (\mathcal{X},\Lambda)\\
@VVV @VV{\pi}V\\
Spec(k) @>s>>S
\end{CD}$
\end{center}
where $\pi$ is flat and surjective, and $S$ is connected, $(\mathcal{X},\Lambda)$ is a Possoin scheme over $S$ defined by $\Lambda\in \Gamma(\mathcal{X}, \mathscr{H}om(\wedge^2 \Omega_{\mathcal{X}/S}^1,\mathcal{O}_{\mathcal{X}}))$ with $X \cong \mathcal{X} \times_S Spec(k)$ as a Poisson isomorphism. Note that when we ignore Poisson structures, a Poisson deformation is an usual flat deformation of an algebraic scheme $X$.  By following Sernesi's book \cite{Ser06}, we extend the formalism of ordinary flat deformations to Poisson deformations. We show that given a Poisson scheme $(X,\Lambda_0)$, first order Poisson deformation (i.e Poisson deformations over a dual number $k[\epsilon]$) whose underlying flat deformation (when we ignore Poisson structures) is locally trivial, is naturally in one to one correspondence with $HP^2(X,\Lambda_0)$ which is the second (truncated) Lichnerowicz-Poisson cohomology group, in other words $2$nd hypercohomology of the following complex of sheaves induced by $[\Lambda_0,-]$.
\begin{align*}
0\to \mathscr{H}om_{\mathcal{O}_X}(\Omega_{X/k}^1,\mathcal{O}_X)\xrightarrow{[\Lambda_0,-]} \mathscr{H}om_{\mathcal{O}_X}(\wedge^2 \Omega^1_{X/k},\mathcal{O}_X)\xrightarrow{[\Lambda_0,-]}\mathscr{H}om_{\mathcal{O}_X}(\wedge^3\Omega^1_{X/k},\mathcal{O}_X)\xrightarrow{[\Lambda_0,-]}\cdots
\end{align*}
We also show that for a smooth Poisson algebraic scheme over $k$, any small extension $e:0\to (t)\to \tilde{A}\to A\to 0$ (i.e $(A,\mathfrak{m}),(\tilde{A},\tilde{\mathfrak{m}})$ are local artinian $k$-algebras with residue $k$ and $t\cdot \tilde{\mathfrak{m}}=0$), and an infinitesimal Poisson deformation $\xi$ of $(X,\Lambda_0)$ over $Spec(A)$, we can associate 
an element $o_{\mathfrak{\xi}}(e)\in HP^3(X,\Lambda_0)$ such that $o_{\mathfrak{\xi}}(e)$ is $0$ if and only if a lifting of $\xi$ to $\tilde{A}$ exists. So $HP^3(X,\Lambda_0)$ is an obstruction space. We also show that if $HP^2(X,\Lambda_0)=0$, then $(X,\Lambda_0)$ is rigid, which means that any infinitesimal Poisson deformation of $(X,\Lambda_0)$ over $A$ is trivial for all local artinian $k$-agebra $A$. 

In chapter \ref{chapter8}, we discuss Poisson deformation functor $PDef_{(X,\Lambda_0)}$ which is a functor of Artin rings. For a local artinian $k$-algebra $A$ with residue $k$, $PDef_{(X,\Lambda_0)}(A)$ is the set of Poisson deformations over $Spec(A)$ up to Poisson equivalence. We show that for a smooth projective Poisson scheme $(X,\Lambda_0)$, $PDef_{(X,\Lambda_0)}$ satisfies Schlessinger's criterion $(H_0),(H_1),(H_2),(H_3)$ and so $PDef_{(X,\Lambda_0)}$ has a miniversal family. We also show that in addition if $HP^1(X,\Lambda_0)=0$, $PDef_{(X,\Lambda_0)}$ is pro-representable.

 In chapter \ref{chapter9}, we extend the construction of a cotangent complex (\cite{Sch67}) to Poisson cases. Let $A\to B$ be a Poisson homomorphism of Poisson $k$-algebras, and $M$ be a Poisson $B$-module. We construct $PT^i(B/A,M)$ in a similar way to construct $T^i(B/A,M)$ in \cite{Sch67}. As an application to Poisson deformation, we show that  for a Poisson algebra $B_0$, $PDef_{Spec(B_0)}(k[\epsilon])$ is a natural one to one correspondence with $PT^1(B_0/k,B_0)$. We also show that given a Poisson algebra $B_0$ and an Poisson ideal $I$ of $B_0$, deformations of a Poisson subscheme $Spec(C)$ of $Spec(B_0)$ over $Spec(k[\epsilon])$ is one to one correspondence with $PT^1(C/B_0,C)$ where $C=B_0/I$.

\chapter{Deformations of algebraic Poisson schemes}\label{chapter7} 

\section{Definitions of Poisson schemes, morphisms and cohomology}
In this section, every algebra is a commutative $k$-algebra, where $k$ is a field. Our reference is \cite{Lau13} Chapter 3. For algebraic geometry, we refer to \cite{Har77}, \cite{Liu02}.

\subsection{Characterization of a Poisson bracket $\{-,-\}$ of a Poisson algebra $A$ over $R$} 
In this subsection, we will characterize a Poisson structure of a commutative algebra $A$ over $R$ in terms of an element $\Lambda \in Hom_A(\Omega_{R/A}^1,A)$ with $[\Lambda,\Lambda]=0$ where $[-,-]$ is the Schouten bracket on $\bigoplus_{p\geq 1}Hom_A(\wedge^p \Omega_{A/R}^1,A)$. 
\begin{definition}
Let $A$ be a commutative $R$-algebra and let $p\geq 1$. A skew symmetric $p$-linear map $P\in Hom_R(\wedge^p A,A)$ is called a skew symmetric $p$-derivation of $A$ over $R$, if $P$ is a derivation in each of its components.

\end{definition}

Let $\Omega_{A/R}^1$ be the $A$-module of relative k\"{a}hler differential forms of $A$ over $R$. Then the $R$-module of all skew symmetric $p$-linear maps in $Hom_R(\wedge^p,A,A)$ is identified with $Hom_A(\wedge^p \Omega_{A/R}^1,A)$. Let $P$ be a skew symmetric $p$ linear map in $Hom_R(\wedge^p A,A)$. Then the associated $\tilde{P} \in Hom_A(\wedge^p \Omega_{A/R}^1,A)$ is defined in the following way: $\tilde{P}(da_1\wedge\cdots \wedge da_p):=P(a_1,\cdots a_p)$ where $d:A\to \Omega_{A/R}^1$ is the canonical map.

\subsubsection{The Shouten bracket on $\bigoplus_{p\geq 1} Hom_A(\wedge^p \Omega_{A/R}^1,A)$ and characterization of a Poisson bracket on $A$}

\begin{definition}
For $p,q \in \mathbb{N}$, a $(p,q)$-shuffle is a permutation $\sigma$ of the set $\{1,...,p+q\}$, such that $\sigma(1)< \cdots < \sigma(p)$ and $\sigma(p+1) < \cdots <\sigma(p+q)$. The set of all $(p,q)$-shuffles is denoted by $S_{p,q}$. For a shuffle $\sigma \in S_{p,q}$, we denote the signature of $\sigma$ by $sgn(\sigma)$. By convention, $S_{p,-1}:=\emptyset$ and $S_{-1,q}:= \emptyset $ for $p,q\in \mathbb{N}$.
\end{definition}

\begin{definition}
We define the Schouten bracket $[-,-]$ on $\bigoplus_{p\geq 1} Hom_A(\wedge^p \Omega_{A/R}^1,A)$, namely a family of maps
\begin{align*}
[-,-]:Hom_A(\wedge^p \Omega_{A/R}^1,A) \times Hom_A(\wedge^q \Omega_{A/R}^1,A) \to Hom_A(\wedge^{p+q-1}\Omega_{A/R}^1,A)
\end{align*}
for $p,q\in \mathbb{N}$ in the following way: let $P\in Hom_A(\wedge^p \Omega_{A/R}^1,A)$ and $Q\in Hom_A(\wedge^q \Omega_{A/R}^1,A)$, and for $F_1,...,F_{p+q-1}\in A$ by
\begin{align*}
[P,Q](dF_1\wedge \cdots\wedge dF_{p+q-1})=\sum_{\sigma\in S_{q,p-1}}sgn(\sigma)P(d(Q(dF_{\sigma(1)}\wedge...\wedge dF_{\sigma(q)}))\wedge dF_{\sigma(q+1)}\cdots\wedge dF_{\sigma(q+p-1)})\\
                                     -(-1)^{(p-1)(q-1)}\sum_{\sigma\in S_{p,q-1}}sgn(\sigma)Q(d(P(dF_{\sigma(1)}\wedge...\wedge dF_{\sigma(p)}))\wedge dF_{\sigma(p+1)}\wedge\cdots \wedge dF_{\sigma(p+q-1)})
\end{align*} 
\end{definition}

\begin{example}\label{3ex}
Let $P\in Hom_A(\wedge^2 \Omega_{A/R}^1,A)$ and $Q\in Hom_A(\Omega_{A/R}^1,A)$. Then
\begin{align*}
[P,Q](dF_1\wedge dF_2)=P(dQ(F_1)\wedge dF_2)-P(d(Q(F_2))\wedge dF_1)-Q(d(P(dF_1\wedge dF_2)))
\end{align*}
\end{example}

\begin{proposition}
Let $A$ be a commutative algebra over $R$. If $\Lambda$ is a skew symmetric biderivation of $A$ over $R$, i.e $\Lambda\in Hom_A(\wedge^2 \Omega_{A/R}^1,A)$, then $P$ defines a Poisson bracket $($i.e Jaocbi identity holds$)$ if and only if $[\Lambda,\Lambda]=0$.
\end{proposition}

\begin{proof}
See \cite{Lau13} Proposition 3.5  page 80.
\end{proof}

\begin{notation}
Let $A$ be a Poisson algebra over $R$ with a Poisson bracket $\{-,-\}$. Let $\Lambda$ be the associated biderivation with the Poisson bracket $\{-,-\}$. Then we will denote by $(A,\Lambda)$ the Poisson algebra $A$ over $R$ with the Poisson bracket $\{-,-\}$.
\end{notation}

\begin{remark}
Let $(A,\Lambda)$ be a Poisson algebra over $R$ with the Poisson structure $\Lambda \in Hom_A(\wedge^2, \Omega_{A/R}^1,A)$ with $[\Lambda,\Lambda]=0$. If we let $\mathfrak{g}=\oplus_{i\geq 0} g_i$, where $g_i=Hom_A(\wedge^{i+1}\Omega_{A/R}^1,A)$. Then $\mathfrak{g}=(\bigoplus_{i\geq 0} g_i,[-,-], [\Lambda,-])$ is a differential graded Lie algebra with the differential $[\Lambda,-]$. In other words, we have the following properties: for $P\in Hom_A(\wedge^p \Omega_{A/R}^1,A)$ and $Q\in Hom_A(\wedge^q \Omega_{A/R}^1,A)$ and $S\in Hom_A(\wedge^r \Omega_{A/R}^1,A) $,
\begin{enumerate}
\item $[\Lambda,[\Lambda,P]]]=0$ and $[\Lambda,P]\in Hom_A(\wedge^{p+1} \Omega_{A/S}^1,A)$
\item $[P,Q]=-(-1)^{(p-1)(q-1)}[Q,P]$
\item $[[P,Q],S]=[P,[Q,S]]-(-1)^{(p-1)(q-1)}[Q,[P,S]]$
\item $[\Lambda,[P,Q]]=[[\Lambda,P],Q]+(-1)^{p-1}[P,[\Lambda,Q]]$
\end{enumerate}

\begin{definition}
Let $(A,\Lambda)$ be a Poisson algebra over $R$. We define $i$-th truncated Lichnerowicz Poisson cohomology of $(A,\Lambda)$ to be the $i$-th cohomology group of the following complex
\begin{align*}
0\to  Hom_A(\Omega_{A/R}^1,A)\xrightarrow{[\Lambda,-]} Hom_A (\wedge^2 \Omega_{A/R}^1,A)\xrightarrow{[\Lambda,-]} Hom_A(\Omega_{A/R}^1,A)\xrightarrow{[\Lambda,-]} \cdots
\end{align*}
We will denote $i$-th Lichnerowicz Poisson cohomology group by $HP^i(A,\Lambda)$.
\end{definition}

\end{remark}

\subsubsection{Characterization of Poisson morphisms}\

Let $f:A\to B$ be a $R$-homomorphism. Then we have the following commutative diagram
\begin{center}
$\begin{CD}
 \Omega_{A/R}\otimes_A B@>>> \Omega_{B/R}\\
@Ad_{A} AA @Ad_{B} AA\\
A@>f>>B
\end{CD}$
\end{center}

So we have a canonical homomorphism $\wedge^2 \Omega_{A/R}\otimes_A B\to \wedge^2\Omega_{B/R}$.  This induces $f^*:Hom_B(\wedge^2 \Omega_{B/R},B)\to Hom_B(\wedge^2 \Omega_{A/R}\otimes_A B,B)\cong Hom_A(\wedge^2 \Omega_{A/R},B)$

\begin{proposition}
Let $(A,P)$ and $(B,Q)$ be two Poisson $R$-algebras. Then a homomorphism $A\to B$ of $R$-algebras is a Poisson homomorphism if and only if $f^*Q=f\circ P$.
\end{proposition}

\begin{proof}
Let $f$ be a Poisson $R$-homomorphism. In other words, $f(\{a,b\})=\{f(a),f(b)\}$. Then $f(P(d_A a,d_A b))=Q(d_B f(a),d_B f(b))=f^* Q(d_A a,d_B b)$. Hence we get $f^*Q=f\circ P$.
\end{proof}

\begin{example}[Poisson ideals]
Let $I$ be an ideal of a commutative $R$-algebra $A$ and set $B=A/I$. Let $\Lambda\in Hom_A(\wedge^2 \Omega_{A/R},A)$ be a Poisson structure on $A$ over $R$. The map $A\to B$ induces $Hom_{B}(\wedge^2 \Omega_{B/R},B)\to Hom_A(\wedge^2 \Omega_{A/R},B)$ which is injective since $\Omega^1_{A/R}\otimes_A B\to \Omega^1_{B/R}$ is surjective. Let $\bar{\Lambda}$ be the composition of $\Lambda$ followed by $A\to B$. If $\bar{\Lambda}$ has pre image $P$, then $P$ defines a Poisson structure on $B$, which makes $I$ to be a Poisson ideal of $A$. Indeed, we show that $[P,P]=0$. Since $Hom_{B}(\wedge^3\Omega_{B/R},B)\to Hom_A(\wedge^3 \Omega_{A/R}, B)$ is injective and $[P,P]$ is sent to $\overline{[\Lambda,\Lambda]}=0$, where $\overline{[\Lambda,\Lambda]}$ is the composition of $[\Lambda,\Lambda]$ followed by $A\to B$. We have $[P,P]=0$.

\end{example}

\subsection{Affine Poisson Schemes}

\subsubsection{Poisson $(k)$-sheaves on a topological space $X$}

\begin{definition}
Let $X$ be a topological space and let $k$ be a field. A Poisson presheaf $\mathcal{F}$ on $X$ consists of the following data:

\begin{enumerate}
\item An Poisson $k$-algebra $\mathcal{F}(U)$ for every open subset $U$ of $X$, and
\item a Poisson $k$-algebra homomorphism $\rho_{UV}:\mathcal{F}(U)\to \mathcal{F}(V)$ for every inclusion  wof open subset $V\subset U$. 
\end{enumerate}
which satisfy the following conditions:
\begin{enumerate}
\item $\mathcal{F}(\emptyset)=0$ for the empty set $\emptyset$.
\item $\rho_{UU}$ is the identity map $\mathcal{F}(U)\to \mathcal{F}(U)$
\item If we have three open subsets $W\subset V \subset U$, then $\rho_{UW}=\rho_{VW}\circ \rho_{UV}$.
\end{enumerate}
\end{definition}
We call $\rho_{UV}$ restriction maps, and we write $s|_V$ instead of $\rho_{UV}(s)$ for $s\in \mathcal{F}(U)$. We refer to $\mathcal{F}(U)$ as the sections of $\mathcal{F}$ over $U$.
\begin{definition}
 We say that a Poisson presheaf $\mathcal{F}$ is a Poisson sheaf if we have the following properties:
 \begin{enumerate}
 \item $($Uniqueness$)$ Let $s\in \mathcal{F}(U)$ for an open subset $U$ of $X$, and $\{ U_i \}$ be a open covering of $U$.  If $s|_{U_i}=0$ for every $i$, then $s=0$.
 \item $($Glueing local sections$)$ Let $U$  be an open subset of $X$ and $\{U_i\}$ be a open covering of $U$. Let $s_i\in \mathcal{F}(U_i)$ such that $s_i|_{U_i\cap U_j} =s_j|_{U_i\cap U_j}$. Then there exists $s\in \mathcal{F}(U)$ such that $s|_{U_i}=s_i$. 
 \end{enumerate}
 \end{definition}

\begin{definition}
Let $\mathcal{F}$ and $\mathcal{G}$ be Poisson presheaves on $X$. A morphism $f:\mathcal{F}\to \mathcal{G}$ is called Poisson morphism if $f(U):\mathcal{F}(U)\to \mathcal{G}(U)$ is a Poisson homomorphism for any open set $ U \subset X $.
\end{definition}

\begin{remark}\
\begin{enumerate}
\item Let $\mathcal{F}$ be Poisson presheaf on $X$, and let $x\in X$. The stalk $\mathcal{F}_x$ at $x$ is a Poisson $k$-algebra.
\item Let $\mathcal{F}$ be a Poisson presheaf on $X$. There exists a Poisson sheaf $\mathcal{F}^{+}$ associated to $\mathcal{F}$ and a morphism of Poisson presheaves $\theta:\mathcal{F}\to \mathcal{F}^+$ verifying the following universal property: for every Poisson morphism $\alpha:\mathcal{F}\to \mathcal{G}$, where $G$ is a Poisson sheaf, there exists a unique Poisson morphism $\tilde{\alpha}:\mathcal{F}^{+}\to \mathcal{G}$ such that $\alpha=\tilde{\alpha}\circ \theta$.
\end{enumerate}
\end{remark}

\begin{definition}[Poisson locally ringed spaces]
A Poisson ringed topological space consists of a topological space $X$ endowed with a Poisson sheaf $($a sheaf of Poisson $k$-algebras$)$ $\mathcal{O}_X$ on $X$ such that $\mathcal{O}_{X,x}$ is a local ring for every $x\in X$ which is a Poisson $k$-algebra. We denote it by $(X,\mathcal{O}_X)$.
\end{definition}

\begin{definition}
A Poisson morphism of Poisson ringed topological spaces
\begin{align*}
(f,f^{\sharp}):(X,\mathcal{O}_X)\to (Y,\mathcal{O}_Y)
\end{align*}
consists of a continuous map $f:X\to Y$ and a morphism of Poisson sheaves $f^\sharp:\mathcal{O}_Y\to f_* \mathcal{O}_X$ such that for every $x\in X$, the induced Poisson homomorphism $f^\sharp_x:\mathcal{O}_{Y,f(x)}\to \mathcal{O}_{X,x}$ is a local Poisson homomorphism. We define the compositions of two Poisson morphisms of Poisson ringed topological spaces in an obvious manner.
\end{definition}

\subsubsection{Affine Poisson schemes}\

We recall the following facts. 
\begin{lemma}
Let $A$ be an commutative $R$-algebra, $S\subset A$ a multiplicatively closed systems, and $A_S$ the corresponding localization of $A$. Then the module of relative differential forms $\Omega_{A_S/R}^1$ is given by the localization $(\Omega_{A/R}^1)_S$ and that the map
\begin{align*}
d:A_S\to (\Omega_{A/R}^1)_S,\,\,\,\,\, \frac{f}{s}\mapsto \frac{sd_{A/R}(f)-fd_{A/R}(s)}{s^2}
\end{align*}
serves as the exterior differential of $A_S$ over $R$ where $d_{A/R}:A\to \Omega_{A/R}^1$ denotes the exterior differential of $A$. And we have $\Omega_{A/R}^1\otimes_A A_S\xrightarrow{} \Omega_{A_S/R}^1$ and $\Omega_{A_S/A}^1=0$.
\end{lemma}

\begin{proof}
See \cite{Bos13} page 354 exercise 2.
\end{proof}

Let $A$ be a Poisson algebra over $k$. Now let $\Lambda\in Hom_A(\wedge^2 \Omega_{A/k}, A)$ be the Poisson $k$-structure on $A$, denoted by $\{-,-\}$. Let $S$ be a multiplicative system of $A$. Then $\Lambda$ induces a Poisson structure on $A_S$, denoted by $\{-,-\}_S$ from the natural map $Hom_A(\wedge^2 \Omega_{A/k}^1,A)\to Hom_A(\wedge^2 \Omega_{A/k}^1,A_S)\cong Hom_{A_S}(\wedge^2 A_S/k,A_S)$ . More precisely, we have
\begin{align*}
\{\frac{a_1}{s_1},\frac{a_2}{s_2}\}_S=\Lambda(d(\frac{a_1}{s_1}),d(\frac{a_2}{s_2}))=\Lambda(\frac{s_1da_1-a_1ds_1}{s_1^2},\frac{s_2da_2-a_2ds_2}{s_2^2})\\
=\frac{\{a_1,a_2\}}{s_1s_2}-\frac{a_2\{a_1,s_2\}}{s_1s_2^2}-\frac{a_1\{s_1,a_2\}}{s_1^2s_2}+\frac{a_1a_2\{s_1,s_2\}}{s_1^2s_2^2}
\end{align*}

\begin{proposition}
 $X=Spec (A)$ for a Poisson $k$-algebra $(A,\Lambda)$ is a Poisson ringed topological space.
 
\end{proposition}
\begin{proof}
Let $\mathfrak{p}$ be a prime ideal of $A$. Then $A_{\mathfrak{p}}$ has a natural Poisson structure induced from $(A,\Lambda)$ with the Poisson bracket $\{-,-\}_{\mathfrak{p}}$.  For any open set $U$ of $Spec(A)$ and $f,g\in \mathcal{O}_X(U)$. Then $a,b $ can be identified with $a,b:U\to \bigcup_{\mathfrak{p}\in U} A_{\mathfrak{p}}$ locally defined by an element of $A_f$ for $D(f)\subset U$. We define $\{a,b\}:U\to\bigcup_{\mathfrak{p}\in U} A_{\mathfrak{p}}$ by $\mathfrak{p}\to \{a_{\mathfrak{p}},b_{\mathfrak{p}}\}_{\mathfrak{p}}$. Since for each principle open set of the from $D(f)$, the Poisson structure on $D(f)$ are all induced from $\Lambda$, we have $\{a,b\}\in \mathcal{O}_X(U)$. Hence the structure sheaf $\mathcal{O}_X$ is a Poisson sheaf. Hence $X$ is a Poisson ringed topological space.

\end{proof}

\begin{definition}[affine Poisson schemes]
We define an affine Poisson scheme to be a Poisson ringed topological space isomorphic to some $(Spec\,A,\mathcal{O}_{Spec\,A})$ for a Poisson $k$-algebra $(A,\Lambda)$.
\end{definition}

\begin{definition}[Poisson schemes]
A Poisson scheme is a Poisson ringed topological space $(X,\mathcal{O}_X)$ admitting an open covering $\{U_i\}$ of $X$ such that $(U_i,\mathcal{O}_X|_{U_i})$ is an affine Poisson $k$-scheme for every $i$. 
\end{definition}

Note that any $k$-scheme can be considered to be a Poisson scheme since any $k$-algebra $A$ has trivial Poisson structure, i.e. $\{f,g\}=0$ for any $f,g \in A$. We will consider a scheme without Poisson structure to be a scheme with trivial Poisson structure.

\subsection{Poisson Schemes}

\begin{definition}
Let $X\to S$ be a morphism of schemes. There is an operation 
\begin{align*}
[-,-]:\mathscr{H}om_{\mathcal{O}_X}(\wedge^p \Omega_{X/S},\mathcal{O}_X )\times \mathscr{H}om_{\mathcal{O}_X}(\wedge^q \Omega_{X/S},\mathcal{O}_X )\to\mathscr{H}om_{\mathcal{O}_X}(\wedge^{p+q-1} \Omega_{X/S},\mathcal{O}_X )
\end{align*}
 which is called the Schouten bracket on a scheme $X$ over $S$.
\end{definition}
The bracket $[-,-]$ is defined in the following way: $\Gamma(U, \mathscr{H}om_{\mathcal{O}_X}(\wedge^p\Omega_{X/S},\mathcal{O}_X))$ is the set of elements of the form $\beta:U\to \bigcup_{x\in U} Hom_{\mathcal{O}_{X,x}}(\wedge^p\Omega_{\mathcal{O}_{X,x}/\mathcal{O}_{S,s}},\mathcal{O}_{X,x})$ such that for any $x$, there exists an affine open neighborhood $V$ of $s=f(x)$ and an affine open neighborhood $U\subset f^{-1}(V)$ of $x$ and $\alpha\in Hom_{\mathcal{O}_X(U_x)}(\wedge^p \Omega_{\mathcal{O}_X(U)/\mathcal{O}_S(V)}^1,\mathcal{O}_X(U))$ with $\beta(x)=\alpha_x$. So on $U$, we define $[\beta_1,\beta_2]:=U\to \bigcup_{x\in U} Hom_{\mathcal{O}_{X,x}}(\wedge^p\Omega_{\mathcal{O}_{X,x}/\mathcal{O}_{S,s}},\mathcal{O}_{X,x}), x\mapsto [\beta_1(x),\beta_2(x)]_x$, where $[-,-]_x$ is the Schouten bracket on $\oplus_p Hom_{\mathcal{O}_{X,x}}(\wedge^p\Omega_{\mathcal{O}_{X,x}/\mathcal{O}_{S,s}},\mathcal{O}_{X,x})$. Hence to show the existence of Schouten bracket on $X$ over $S$, we only need to check the following lemma.

\begin{lemma}
Let $A$ be a commutative $R$-algebra $(f:R\to A)$, and $\mathfrak{p}$ be a prime ideal of $A$. Let $\mathfrak{q}=f^{-1}(\mathfrak{p})$. The following diagram commutes
\begin{center}
$\begin{CD}
Hom_A(\wedge^p \Omega_{A/R}^1,A)\times Hom_A(\wedge^q \Omega_{A/R}^1,A)@>[-,-]>> Hom_A(\wedge^{p+q-1}\Omega_{A/R}^1,A)\\
@VVV @VVV\\
Hom_A(\wedge^p \Omega_{A/R}^1,A_{\mathfrak{p}})\times Hom_A(\wedge^q \Omega_{A/R}^1, A_{\mathfrak{p}})@>[-,-]>> Hom_A(\wedge^{p+q-1}\Omega_{A/R}^1,A_{\mathfrak{p}})\\
@V\cong VV @VV\cong V \\
Hom_{A_{\mathfrak{p}}}(\wedge^p \Omega_{A_{\mathfrak{p}}/R}^1,A_{\mathfrak{p}})\times Hom_{A_{\mathfrak{p}}}(\wedge^q \Omega_{A_{\mathfrak{p}}/R}^1,A_{\mathfrak{p}})@>[-,-]>> Hom_{A_{\mathfrak{p}}}(\wedge^{p+q-1}\Omega_{A_{\mathfrak{p}}/R}^1,A_{\mathfrak{p}})\\
@V\cong VV @VV\cong V\\
Hom_{A_{\mathfrak{p}}}(\wedge^p \Omega_{A_{\mathfrak{p}}/R_{\mathfrak{q}}}^1,A_{\mathfrak{p}})\times Hom_{A_{\mathfrak{p}}}(\wedge^q \Omega_{A_{\mathfrak{p}}/R_{\mathfrak{q}}}^1,A_{\mathfrak{p}})@>[-,-]>> Hom_{A_{\mathfrak{p}}}(\wedge^{p+q-1}\Omega_{A_{\mathfrak{p}}/R_{\mathfrak{q}}}^1,A_{\mathfrak{p}})
\end{CD}$
\end{center}
\end{lemma}

\begin{example}
Let $B\otimes_k A$ be a $A$-algebra where $k$ is a field and $B$ is a finitely generated $k$-algebra $($Hence $\Omega_{B/k}^1$ is finitely presented$)$. Then $Hom_{B\otimes_kA}(\wedge^p \Omega_{B\otimes_k A/A}^1,B\otimes_k A)\cong Hom_{B\otimes_k A}(\wedge^p \Omega_{B/k}^1 \otimes A,B\otimes_k A)\cong Hom_B(\wedge^p \Omega_{B/k},B\otimes_k A)\cong Hom_k(\wedge^p \Omega_{B/k}^1, B)\otimes _k A$. So the Schouten bracket $[-,-]_{B\otimes_k A}$ on $Hom_B(\wedge^p\Omega_{B/k}^1,B)\otimes A$ over $A$ can be seen as
\begin{align*}
[P\otimes a,Q\otimes b]_{B\otimes_k A}=[P,Q]_B\otimes ab
\end{align*}
\end{example}

Let $X$ be a scheme over $k$. We would like to characterize a Poisson structure on $X$ by an element $\Lambda\in \Gamma(X, \mathscr{H}om_{\mathcal{O}_X}(\wedge^2 \Omega_{\mathcal{O}_X/k}^1,\mathcal{O}_X))$ with $[\Lambda,\Lambda]=0$.

\begin{proposition}
Let $X$ be a scheme over $k$. The following are equivalent
\begin{enumerate}
\item $X$ is a Poisson scheme over $k$.
\item There exists a global section $\Lambda\in \Gamma(X, \mathscr{H}om_{\mathcal{O}_X}(\wedge^2 \Omega_{X/k}^1,\mathcal{O}_X))$ with $[\Lambda,\Lambda]=0$
\end{enumerate}
\end{proposition}

\begin{proof}
Let $X$ be a Poisson scheme over $k$. Then for each $x$,  $\mathcal{O}_{X,x}$ is a Poisson $k$-algebra. So we have $\Lambda_x\in Hom_{\mathcal{O}_{X,x}}(\wedge^2 \Omega_{\mathcal{O}_{X,x}/k}^1,\mathcal{O}_{X,x})$ with $[\Lambda_x,\Lambda_x]=0$. We define $\Lambda:X\to \bigcup_{x\in X} Hom_{\mathcal{O}_{X,x}}(\wedge^2 \Omega_{\mathcal{O}_{X,x}/k}^1,\mathcal{O}_{X,x})$. Since $X$ is locally defined by affine Poisson schemes, for each $x\in X$, there exists an affine neighborhood $Spec(A)$ of $x$ with $\Lambda_A \in Hom_A(\wedge^2 \Omega_{A/k}^1,A)$ with $[\Lambda_A,\Lambda_A]=0$ which induces $\Lambda_x$ for $x\in Spec(A)$. Hence $\Lambda\in \Gamma(X, \mathscr{H}om_{\mathcal{O}_X}(\wedge^2 \Omega_{X/k}^1,\mathcal{O}_X))$ with $[\Lambda,\Lambda]=0$.

Conversely, we assume that we have a global section $\Lambda\in \Gamma(X, \mathscr{H}om_{\mathcal{O}_X}(\wedge^2 \Omega_{X/k}^1,\mathcal{O}_X))$ with $[\Lambda,\Lambda]=0$. Then $\Lambda$ can be identified with $\Lambda:X\to \bigcup_{x\in X} Hom_{\mathcal{O}_{X,x}}(\wedge^2 \Omega_{\mathcal{O}_{X,x}/k}^1,\mathcal{O}_{X,x})$ as above. Hence $\mathcal{O}_{X,x}$ is a Poisson $k$-algebra induced from $\Lambda_x$ with the Poisson bracket $\{-,-\}_x$. We show that $\mathcal{O}_X$ is a sheaf of Poisson $k$-algebra. Let $U$ be open set of $X$. Let $f,g\in \mathcal{O}_X(U)$. Then $f,g$ can be identified with $f,g:U\to \bigcup_{x\in U} \mathcal{O}_{X,x}$ which are locally defined by elements of sections of affine open sets. We define $\{f,g\}$ by $U\to \bigcup_{x\in X}\mathcal{O}_{X,x}, x\mapsto \{f_x,g_x\}_x$. This makes $\mathcal{O}_X$ to be a sheaf of Poisson $k$-algebras.  For each $x$, there exists an affine open set $Spec(A)$ of $x$ such that $\Lambda$ on $Spec(A)$ is induced from a $\Lambda_A\in Hom_A(\wedge^2\Omega_{A/k}^1,A)$ with $[\Lambda_A,\Lambda_A]=0$. So $\Lambda_A$ defines a Poisson structure on $Spec(A)$. Hence $X$ is locally defined by affine Poisson schemes. Hence $X$ is a Poisson scheme over $k$.
\end{proof}

\begin{definition}
Let $X$ be a Poisson scheme over $k$. Let $f:X\to S$ be a morphism of schemes. We say that $X$ is Poisson over $S$ or a Poisson $S$-scheme if for any open set $U$ of $S$ and $\mathcal{O}_S(U)\to \mathcal{O}_X(f^{-1}(U))$, $\mathcal{O}_X(f^{-1}(U))$ is a Poisson $\mathcal{O}_S(U)$-algebra. In other words, $\{s,a\}=0 $ for any $s\in \mathcal{O}_S(U)$ and $a\in \mathcal{O}_X(f^{-1}(U))$
\end{definition}
\begin{proposition}
Let $X$ be a scheme over $k$ and $f:X\to S$ be a morphsim of schemes. The following are equivalent.
\begin{enumerate}
\item $X$ is a Poisson scheme over $S$.
\item There exists a global section $P\in \Gamma(X,\mathscr{H}om_{\mathcal{O}_X}(\wedge^2 \Omega_{X/S},\mathcal{O}_X))$ with $[P,P]=0$
\end{enumerate}
\end{proposition}

\begin{proof}
Let $\Lambda \in \Gamma(X,\mathscr{H}om_{\mathcal{O}_X}(\wedge^2 \Omega_{X/k}^1,\mathcal{O}_X))$ be the $(k)$-Poisson structure on $X$. Now we assume that $X$ is a Poisson scheme over $S$ via $f:X\to S$. We note that we have an exact sequence 
\begin{align*}
0\to \Gamma(X,\mathscr{H}om_{\mathcal{O}_X}(\wedge^2 \Omega_{X/S},\mathcal{O}_X))\to \Gamma(X,\mathscr{H}om_{\mathcal{O}_X}(\wedge^2 \Omega_{X/k}^1,\mathcal{O}_X)).
\end{align*}
We will show that $\Lambda$ is actually in $\Gamma(X,\mathscr{H}om_{\mathcal{O}_X}(\wedge^2 \Omega_{X/S},\mathcal{O}_X))$. For $x\in X$, via $f_*:\mathcal{O}_{S,f(x)}\to \mathcal{O}_{X,x}$, $\mathcal{O}_{X,x}$ is a Poisson $\mathcal{O}_{S,f(x)}$-algebra. Since $\Lambda_x$ is $\mathcal{O}_{S,f(x)}$-linear,  we have actually  
\begin{align*}
\Lambda_x\in Hom_{\mathcal{O}_{X,x}}(\Omega_{\mathcal{O}_{X,x}/\mathcal{O}_{S,f(x)}}^1,\mathcal{O}_{X,x})
\end{align*}
with $[\Lambda_x,\Lambda_x]=0$. Hence $P=:\Lambda\in \Gamma(X,\mathscr{H}om_{\mathcal{O}_X}(\wedge^2 \Omega_{X/S},\mathcal{O}_X))$ with $[P,P]=0$.

Conversely, assume there exists a global section $P\in \Gamma(X,\mathscr{H}om_{\mathcal{O}_X}(\wedge^2 \Omega_{X/S},\mathcal{O}_X))$ with $[P,P]=0$. Then $P$ defines a Poisson scheme over $k$ by the above exactness. Since for each $x\in X$, $\mathcal{O}_{X,x}$ is a Poisson $\mathcal{O}_{s,f(x)}$-algebra, $X$ is a Poisson scheme over $S$.

\end{proof}

\begin{definition}
Let $(X,P)$ and $(Y,Q)$ be Poisson schemes over $S$ with $g:X\to S$ and $h:Y\to S$. Then a morphism $f:X\to Y$ of schemes over $S$ is called a morphism of Poisson schemes over $S$ if for any open set  $U$ of $S$ and any open set $V$ of $h^{-1}(U)$ and any open set $W$ of $f^{-1}(V)$, and $f^\sharp:\mathcal{O}_Y(V)\to \mathcal{O}_X(W)$ is a Poisson $\mathcal{O}_S(U)$-homomorphism.
\end{definition}

Let $f:X\to Y$ be a morphism of schemes over $S$. Then we have $f^*\Omega_{Y/S}^1\to \Omega_{X/S}^1$. Since $\Omega_{Y/S}^1$ is quasi-coherent, we have $f^*(\wedge^2 \Omega_{Y/S}^1)\cong \wedge^2 f^* \Omega_{Y/S}^1$. So we have $\mathscr{H}om_{\mathcal{O}_X}(\wedge^2 \Omega_{X/S}^1,\mathcal{O}_X)\to \mathscr{H}om_{\mathcal{O}_X}(f^*(\wedge^2 \Omega_{Y/S}),f^*\mathcal{O}_Y)$. On the other hand, we have a natural sheaf morphism 
\begin{align*}
\mathscr{H}om_{\mathcal{O}_Y}(\wedge^2 \Omega^1_{Y/S},\mathcal{O}_Y)\to f_*f^*\mathscr{H}om_{\mathcal{O}_Y}(\wedge^2 \Omega_{Y/S}^1,\mathcal{O}_Y)\to f_*\mathscr{H}om_{\mathcal{O}_X}(f^*(\wedge^2 \Omega^1_{Y/S}),f^* \mathcal{O}_Y)
\end{align*}
 By taking the global sections, we have two morphsims
\begin{align*}
\alpha:&\Gamma(X, \mathscr{H}om_{\mathcal{O}_X}(\wedge^2\Omega^1_{X/S},\mathcal{O}_X))\to \Gamma(X,\mathscr{H}om_{\mathcal{O}_X}(f^*(\wedge^2 \Omega_{Y/S}),f^*\mathcal{O}_Y))\\
\beta:&\Gamma(Y,\mathscr{H}om_{\mathcal{O}_Y}(\wedge^2\Omega_{Y/S}^1,\mathcal{O}_Y))\to \Gamma(X,\mathscr{H}om_{\mathcal{O}_X}(f^*(\wedge^2 \Omega_{Y/S}),f^*\mathcal{O}_Y)
\end{align*}
If $f:(X,P)\to (Y,Q)$ is a morphism of Poisson schemes over $S$,  we have $\alpha(P)=\beta(Q)$.

\begin{proposition}[Glueing Poisson schemes]
Let $S$ be a $k$-scheme. Let us consider a family $\{X_i\}$ of Poisson schemes over $S$. We suppose given open subschemes $X_{ij}$ of $X_i$$($which is necessarily Poisson $S$-scheme$)$ and Poisson isomorphisms of $S$-schemes $f_{ij}:X_{ij}\to X_{ji}$ such that $f_{ii}=Id_{X_i}, f_{ij}(X_{ij}\cap X_{ji})=X_{ji}\cap X_{jk}$, and $f_{ik}=f_{jk}\circ f_{ij}$ on $X_{ij}\cap X_{ik}$. Then there exists an Poisson $S$-scheme $X$, unique up to isomorphism, with Poisson open immersion of $S$-schemes $g_i:X_i\to X$ such that $g_i:X_i\to X$ such that $g_i=g_j\circ f_{ij}$ on $X_{ij}$, and $X=\cup_i g_i(X_i)$.
\end{proposition}

\subsection{(truncated) Lichnerowicz-Poisson cohomology}

\begin{definition}\label{3def}
Let $(X,\Lambda)$ be a Poisson scheme over $S$. Then we define $i$-th (truncated) Lichnerowicz-Poisson cohomology is the $i$-th hypercohomology group of the following complex of sheaves
\begin{align*}
0\to \mathscr{H}om_{\mathcal{O}_X}(\Omega_{X/S}^1,\mathcal{O}_X)\xrightarrow{[\Lambda,-]} \mathscr{H}om_{\mathcal{O}_X}(\wedge^2 \Omega^1_{X/S},\mathcal{O}_X)\xrightarrow{[\Lambda,-]}\mathscr{H}om_{\mathcal{O}_X}(\wedge^3\Omega^1_{X/S},\mathcal{O}_X)\xrightarrow{[\Lambda,-]}\cdots
\end{align*}
We denote $i$-th cohomology group by $HP^i(X,\Lambda)$.
\end{definition}

\begin{remark}\label{3remark2}
Let $X=Spec(A)$ be an affine scheme with a Poisson structure $\Lambda\in Hom_A(\wedge^2 \Omega_{A/S},A)=\Gamma(X,\mathscr{H}om_{\mathcal{O}_X}(\wedge^2\Omega^1_{X/S},\mathcal{O}_X))$. Since $\mathscr{H}om_{\mathcal{O}_X}(\wedge^i\Omega_{X/S}^1,\mathcal{O}_X)$ are quasi coherent, and so its higher cohomology vanishes. Hence (truncated) Lichnerowicz-Poisson cohomology of $(X,\Lambda)$ is same to the (truncated) Lichnerowicz-Poisson cohomology of a Poisson algebra $(A,\Lambda)$.
\end{remark}

\section{Deformations of algebraic Poisson schemes}

\subsection{Basic materials on deformations of algebraic Poisson schemes}\

In this section, we discuss deformations of algebraic Poisson schemes by following \cite{Ser06}(see Chapter 1) in the Poisson context. 
We will always denote by $k$ a fixed algebraically closed field. All schemes will be assumed to be defined over $k$, locally noetherian and separated. If $S$ is a scheme and $s\in S$, we denote $k(s)=\mathcal{O}_{S,s}/\mathfrak{m}_s$ the residue field of $S$ at $s$. We denote by $\bold{Art}$ the category of local artinian $k$-algebras with residue field $k$.

\begin{definition}
Let $(X,\Lambda_0)$ be an algebraic Poisson scheme. A cartesian diagram of morphisms of schemes
\begin{center}
$\eta:$
$\begin{CD}
(X,\Lambda_0) @>i>> (\mathcal{X},\Lambda)\\
@VVV @VV{\pi}V\\
Spec(k) @>s>>S
\end{CD}$
\end{center}
is called a family of Poisson deformations or a Poisson deformation of $X$ parametrized by $S$ where $\pi$ is flat and surjective, and $S$ is connected, $(\mathcal{X},\Lambda)$ is a Possoin $S$-scheme with $\Lambda\in \Gamma(\mathcal{X}, \mathscr{H}om_{\mathcal{O}_{\mathcal{X}}}(\wedge^2 \Omega_{\mathcal{X}/S}^1, \mathcal{O}_{\mathcal{X}}))$ and $X \cong \mathcal{X} \times_S Spec(k)$ as Poisson isomorphism: in other words, $\Lambda_0$ is induced from $\Lambda$. We call $S$ and $(\mathcal{X},\Lambda)$ respectively the parameter scheme and the total Poisson $S$-scheme of the Poisson deformation $\eta$. If $S$ is algebraic, for each $k$-rational point $t\in S$ the scheme theoretic fiber $(\mathcal{X}(t),\Lambda(t))$ with the induced Poisson structure $\Lambda(t)$ from $\Lambda$ is also called a Poisson deformation of $(X,\Lambda_0)$. 
 a Poisson deformations  $\eta$ over $Spec(A)$ is called infinitesimal $($reps. first-order$)$ if $A\in \bold{Art}$ $($reps. if $A=k[\epsilon]$$)$. 
\end{definition} 

\begin{remark}
We will explain more in detail that $\Lambda_0$ is induced from $\Lambda$. Since $\Omega_{X/k}^1\cong i^*\Omega_{\mathcal{X}/S}^1$, canonical map $\mathscr{H}om_{\mathcal{O}_\mathcal{X}}(\wedge^2 \Omega_{\mathcal{X}/S}^1,\mathcal{O}_{\mathcal{X}})\to i_*i^*\mathscr{H}om_{\mathcal{O}_{\mathcal{X}}}(\wedge^2 \Omega_{\mathcal{X}/S}^1,\mathcal{O}_{\mathcal{X}})\to i_*\mathscr{H}om_{\mathcal{O}_X}(\wedge^2 \Omega_{X/k}^1,\mathcal{O}_X)$ induces $\Gamma(\mathcal{X},\mathscr{H}om_{\mathcal{O}_{\mathcal{X}}}(\wedge^2 \Omega_{\mathcal{X}/S}^1,\mathcal{O}_{\mathcal{X}}))\to \Gamma(X,\mathscr{H}om_{\mathcal{O}_X}(\wedge^2 \Omega_{X/k}^1,\mathcal{O}_X))$. Via this map $\Lambda$ is sent to $\Lambda_0$. So $(X,\Lambda_0)$ is a closed Poisson subscheme of $(\mathcal{X},\Lambda)$ since $i$ is a Poisson morphism.
\end{remark}
 
 Given another Poisson deformation 
\begin{center}
$\xi:$
$\begin{CD}
(X,\Lambda_0) @>>> (\mathcal{Y},\Lambda')\\
@VVV @VVV\\
Spec(k) @>>>S
\end{CD}$
\end{center}
of $(X,\Lambda_0)$ over $S$, an isomorphism of $\eta$ with $\xi$ is an Poisson $S$-isomorphism $\phi:(\mathcal{X},\Lambda)\to (\mathcal{Y},\Lambda')$ inducing the identity on $(X,\Lambda_0)$, i.e. such that the following diagram is commutative.

\begin{center}
\[\begindc{\commdiag}[50]
\obj(0,1)[a]{$(\mathcal{X},\Lambda)$}
\obj(1,2)[b]{$(X,\Lambda_0)$}
\obj(1,0)[c]{$S$}
\obj(2,1)[d]{$(\mathcal{Y},\Lambda')$}
\mor{b}{a}{}
\mor{b}{d}{}
\mor{a}{c}{}
\mor{d}{c}{}
\mor{a}{d}{$\phi$}
\enddc\]
\end{center}

By a pointed scheme, we will mean a pair $(S,s)$ where $S$ is a scheme and $s\in S$. If $K$ is a field, we call $(S,s)$ a $K$-pointed scheme if $K\cong k(s)$.

\begin{definition}[trivial Poisson deformation]
Let $(X,\Lambda_0)$ be a algebraic Poisson scheme, and $(S,s)$ be a $k$-pointed scheme $(S,s)$. We define a trivial Poisson family induced  by $(X,\Lambda_0)$ and $(S,s)$ to be the following Poisson deformation of $(X,\Lambda_0)$,
\begin{center}
$\begin{CD}
(X,\Lambda_0) @>>> (X\times_{Spec(k)} S,\Lambda_0\oplus 0)\\
@VVV @VV{\pi}V\\
Spec(k) @>s>>S
\end{CD}$
\end{center}
Here $(\Lambda_0\oplus 0)$ is the Poisson structure on $X\times_{Spec(k)} S$ over $S$ induced from the Poisson structure $\Lambda_0$ on $X$ via the following diagram.
\begin{center}
$\begin{CD}
X\times_{Spec(k)} S @>>> X\\
@VVV @VVV\\
S@>>> Spec(k)
\end{CD}$
\end{center}
 Poisson deformation of $(X,\Lambda_0)$ over $S$ is called trivial if it is isomorphic to the trivial Poisson family as above. 
\end{definition}

\begin{definition}[rigid Poisson deformations]
An algebraic Poisson scheme $(X,\Lambda_0)$ is called rigid if every infinitesimal Poisson deformations of $X$ over $A$ is trivial for every $Spec(A)$ in $\bold{Art}$.
\end{definition}

Given a Poisson deformation $\eta$ of $(X,\Lambda_0)$ over $S$ as above and a morphism $(S',s')\to (S,s)$ of $k$-pointed schemes there is induced a commutative diagram by base change
\begin{center}
$\eta':$
$\begin{CD}
(X,\Lambda_0) @>>> (\mathcal{X}\times_S S',\Lambda\oplus 0)\\
@VVV @VVV\\
Spec(k) @>>>S'
\end{CD}$
\end{center}
which is a Poisson deformation of $(X,\Lambda_0)$ over $S'$. This operation is functorial, in the sense that it commutes with composition of morphisms and the identity morphism does not change $\eta$. Moreover, it carries isomorphic Poisson deformations to isomorphic ones.

\begin{definition}[locally trivial Poisson deformations]
An infinitesimal Poisson deformation $\eta$ of $(X,\Lambda_0)$ is called locally trivial if for every point $x\in X$ has an open neighborhood $U_x\subset X$ such that 
\begin{center}
$\begin{CD}
(U_x,\Lambda_0|_{U_x}) @>>> (\mathcal{X}|_{U_x},\Lambda|_{U_x})\\
@VVV @VV{\pi}V\\
Spec(k) @>s>>S
\end{CD}$
\end{center}
is a trivial Poisson deformation of $U_x$. In other words, $(\mathcal{X}|_{U_x},\Lambda|_{U_x})\cong (U_x\times_{spec(k)} S,\Lambda_0\oplus 0)$ as Poisson schemes. 
\end{definition}

\subsection{Infinitesimal Poisson deformations}

\begin{definition}[small extension]
We say that for $(\tilde{A},\tilde{\mathfrak{m}}), (A,\mathfrak{m})\in \bold{Art}$, an exact sequence of the form $0\to (t)\to \tilde{A}\to A\to 0$ is a small extension if $t\in \tilde{\mathfrak{m}}$ is annihilated by $\tilde{\mathfrak{m}}. ($i.e $t\cdot \tilde{\mathfrak{m}}=0)$ so that $(t)$ is an one dimensional $k$-vector space.
\end{definition}

\begin{lemma}[compare $\cite{Ser06}$ Lemma 1.2.6 page 26]\label{3l}
Let $B_0$ be a Poisson $k$-algebra with the Poisson structure $\Lambda\in Hom_{B_0}(\wedge^2 \Omega_{B_0/k},B_0)$, and
\begin{align*}
e:0\to(t)\to \tilde{A}\to A\to 0
\end{align*}
a small extension in $\bold{Art}$. Let $\Lambda_0\in Hom_{B_0}(\wedge^2 \Omega_{B_0/k},B_0)\otimes_k A$ be a Poisson structure on $B_0\otimes_k A$ over $A$  inducing $\Lambda$. Let $\Lambda_1,\Lambda_2\in Hom_{B_0}(\wedge^2 \Omega_{B_0/k},B_0)\otimes_k \tilde{A}$ be Poisson structures on $B_0\otimes_k \tilde{A}$ over $\tilde{A}$ which induces $\Lambda_0$. This implies that there exists a $\Lambda'\in Hom_{B_0}(\wedge^2 \Omega_{B_0/k},B_0)$ such that $\Lambda_1-\Lambda_2=t\Lambda'$. Then there is one to one correspondence
\begin{align*}
\{\text{Poisson isomorphisms between} \,\,\,(B_0\otimes_k \tilde{A},\Lambda_1)\,\,\,\text{and}\,\,\,(B_0\otimes_k\tilde{A},\Lambda_2)\\\,\,\,\text{inducing the identity on}\,\,\,(B_0\otimes_k A,\Lambda_0)\}\\\to \{P\in Der_k(B_0,B_0)=Hom_{B_0}(\Omega_{B_0/k},B_0)| \Lambda'-[\Lambda, P]=\Lambda'+[P,\Lambda]=0\}
\end{align*}
In particular, when $\Lambda_1=\Lambda_2$, there is a canonical isomorphism of groups
\begin{align*}
\{\text{Poisson automorphisms between} \,\,\,(B_0\otimes_k \tilde{A},\Lambda_1)\,\,\,\text{and}\,\,\,(B_0\otimes_k\tilde{A},\Lambda_1)\\\,\,\,\text{inducing the identity on}\,\,\,(B_0\otimes_k A,\Lambda_0)\}\to PDer_k(B_0,B_0)=HP^1(B_0,\Lambda)
\end{align*}
\end{lemma}
\begin{proof}
Let $\theta:(B_0\otimes_k \tilde{A},\Lambda_1)\to (B_0\otimes_k \tilde{A}, \Lambda_2)$ be a Poisson isomorphism. Then $\theta$ is $\tilde{A}$-linear and induces the identity modulo by $t$. We have $\theta(x)=x+tPx$, where $P\in Der_{\tilde{A}}(B_0\otimes_k \tilde{A},B_0)=Der_k(B_0,B_0)=Hom_{B_0}(\Omega_{B_0/k}^1,B_0)$. When we think of $P$ as an element of $Hom_{B_0}(\Omega_{B_0/k}^1,B_0)$, we have $\theta(x)=x+tP(dx)$. We define the correspondence by $\theta \mapsto P$. Now we check that $\Lambda'-[\Lambda,P]=0$. Since $\theta$ is a Poisson map, for $x,y\in B_0$, we have by Example $(\ref{3ex})$,
\begin{align*}
&\theta(\Lambda_1(dx\wedge dy))=\Lambda_2(d(\theta x)\wedge d (\theta y))\\
&\Lambda_1(dx\wedge dy)+t P(d(\Lambda_1(dx\wedge dy)))=\Lambda_2((dx+td(P(dx)))\wedge (dy+td(P(dy))))\\
&\Lambda_1(dx\wedge dy)+t P(d(\Lambda(dx\wedge dy)))=\Lambda_2(dx\wedge dy)+t\Lambda(dx\wedge d(P(dy)))+t\Lambda(d(P(dx))\wedge dy)\\
&t[\Lambda'(dx\wedge dy)+P(d(\Lambda(dx\wedge dy)))-\Lambda(dx\wedge d(P(dy)))-\Lambda(d(P(dx))\wedge dy)]=0\\
&\Lambda'-[\Lambda,P]=0
\end{align*}
Since $\theta$ is determined by $P$, the correspondence is one to one.

Now we assume that $\Lambda_1=\Lambda_2$. So $\theta$ corresponds to $P$ with $[\Lambda,P]=0$. First we note that $P\in Hom_{B_0}(\Omega_{B_0/k}^1,B_0)$ with $[\Lambda,P]=0$ is a Poisson derivation. In other words, $P(\{x,y\})=\{Px,y\}+\{x,Py\}$. Indeed, $0=[\Lambda,P](dx\wedge dy)=\Lambda(d(Px)\wedge dy)-\Lambda(d(Py)\wedge dx)-P(d(\Lambda(dx\wedge dy))$.

Now we show that the correspondence is a group isomorphism. Indeed, let $\theta(x)=x+tPx$ and $\sigma(x)=x+tQx$ with $[\Lambda,P]=[\Lambda,Q]=0$. Then $\sigma(\theta(x))=\theta(x)+tQ(\theta(x))=x+tPx+tQ(x+tPx))=x+tPx+tQx=x+t(P+Q)x$. Hence $\theta+\sigma$ corresponds to $P+Q$. Since $[\Lambda,P+Q]=0$ and identity map corresponds to $0$, the correspondence is group isomorphism.
\end{proof}

\begin{proposition}[compare $\cite{Ser06}$ Proposition 1.2.9 page 29 and see also $\cite{Nam09}$ Proposition 8]\label{3q}
Let $(X,\Lambda_0)$ be an Poisson algebraic variety with $\Lambda_0\in \Gamma(X,\mathscr{H}om_{\mathcal{O}_X}(\wedge^2 \Omega_{X/k},\mathcal{O}_X))$. There is a $1-1$ correspondence:
\begin{align*}
&\kappa:\{\text{Poisson isomorphism classes of first order Poisson deformations of $(X,\Lambda_0)$}\\&\text {whose underlying flat deformation of $X$ is locally trivial}\}\to HP^2(X,\Lambda_0)
\end{align*}
such that $\kappa(\xi)=0$ if and only $\xi$ is the trivial Poisson deformation class. In particular, if $X$ is nonsingular, then we have $1-1$ correspondence 
\begin{align*}
\kappa:\{\text{Poisson isomorphism classes of first order Poisson deformations of $(X,\Lambda_0)$}\} \to HP^2(X,\Lambda_0)
\end{align*}
\end{proposition}

\begin{proof}
Given a first-order Poisson deformation of $(X,\Lambda_0)$ whose underlying flat deformation is locally trivial,
\begin{center}
$\begin{CD}
(X,\Lambda_0)@>>> (\mathcal{X},\Lambda)\\
@VVV @VVV\\
Spec(k)@>>> Spec(k[\epsilon])
\end{CD}$
\end{center}
we choose an affine open cover $\mathcal{U}=\{U_i=Spec(B_i)\}_{i\in I}$ of $X$ such that $\mathcal{X}|_{U_i}\cong U_i\times Spec(k[\epsilon])=Spec(B_i)\times Spec(k[\epsilon])$ is trivial for all $i$ with the induced Poisson structure $\Lambda_0+\epsilon\Lambda_i=\in Hom_{B_i}(\wedge^2 \Omega_{B_i/k}^1,B_i)\otimes_k k[\epsilon]$ on $U_i\times Spec(k[\epsilon])$ from $\Lambda$. For each $i$, we have a Poisson isomorphism
\begin{align*}
\theta_i:(U_i\times Spec(k[\epsilon]),\Lambda_i)\to (\mathcal{X}|_{U_i},\Lambda|_{U_i})
\end{align*}
Then for each $i,j\in I$, $\theta_{ij}:=\theta_i^{-1}\theta_j:(U_{ij}\times Spec(k[\epsilon]),\Lambda_j)\to (U_{ij}\times Spec(k[\epsilon]), \Lambda_i)$ is an Poisson isomorphism inducing the identity on $(U_{ij},\Lambda_0)$ by modulo $\epsilon$. Hence by Lemma \ref{3l}, $\theta_{ij}$ corresponds to a $p_{ij}\in \Gamma(U_{ij}, T_X)$ where $T_X=\mathscr{H}om(\Omega_X^1,\mathcal{O}_X)=Der_k(\mathcal{O}_X,\mathcal{O}_X)$ such that $\Lambda_i-\Lambda_j-[\Lambda_0,p_{ij}]=0$. We claim that $(\{p_{ij}\},\{\Lambda_i'\})\in C^1(\mathcal{U},\mathscr{H}om_{\mathcal{O}_X}(\Omega_{X/k}^1,\mathcal{O}_X)))\oplus C^0(\mathcal{U},\mathscr{H}om_{\mathcal{O}_X}(\wedge^2 \Omega_{X/k},\mathcal{O}_X))$ represents a cohomology class in the following diagram (see Appendix \ref{appendixb}):
\begin{center}
$\begin{CD}
@A[\Lambda_0,-]AA\\
C^0(\mathcal{U},\mathscr{H}om_{\mathcal{O}_X}(\wedge^3 \Omega_{X/k},\mathcal{O}_X)))@>-\delta>>\cdots\\
@A[\Lambda_0,-]AA @A[\Lambda_0,-]AA\\
C^0(\mathcal{U},\mathscr{H}om_{\mathcal{O}_X}(\wedge^2 \Omega_{X/k},\mathcal{O}_X)))@>\delta>> C^1(\mathcal{U},\mathscr{H}om_{\mathcal{O}_X}(\wedge^2 \Omega_{X/k},\mathcal{O}_X)))@>-\delta>>\cdots\\
@A[\Lambda_0,-]AA @A[\Lambda_0,-]AA @A[\Lambda_0,-]AA\\
C^0(\mathcal{U},\mathscr{H}om_{\mathcal{O}_X}(\Omega_{X/k}^1,\mathcal{O}_X)))@>-\delta>>C^1(\mathcal{U},\mathscr{H}om_{\mathcal{O}_X}(\Omega_{X/k}^1,\mathcal{O}_X)))@>\delta>>C^2(\mathcal{U},\mathscr{H}om_{\mathcal{O}_X}(\Omega_{X/k}^1,\mathcal{O}_X)))\\
@AAA @AAA @AAA \\
0@>>>0 @>>> 0 
\end{CD}$
\end{center}

Since $[\Lambda_0+\epsilon \Lambda_i,\Lambda_0+\epsilon\Lambda_i]=0$, we have $[\Lambda_0,\Lambda_i]=0$. Since on each $U_{ijk}$ we have $\theta_{ij}\theta_{jk}\theta_{ik}^{-1}=1_{U_{ijk}\times Spec(k[\epsilon])}$, we have $p_{ij}+p_{jk}-p_{ik}=0$, and so $\delta(\{p_{ij}\})=0$. Since $\Lambda_i-\Lambda_j-[\Lambda_0,p_{ij}]=0$, we have $\delta(\{\Lambda_i\})+[\Lambda_0,p_{ij}]=0$. Hence $(\{p_{ij}\},\{\Lambda_i\})$ defines a cohomology class.

Now we show that for two equivalent Poisson deformations of $(X,\Lambda_0)$, the cohomology class is same. If we have another Poisson deformation
\begin{center}
$\begin{CD}
(X,\Lambda_0)@>>> (\mathcal{X}',\Lambda')\\
@VVV @VVV\\
Spec(k)@>>> Spec(k[\epsilon])
\end{CD}$
\end{center}
and $\Phi:(\mathcal{X},\Lambda)\to (\mathcal{X'},\Lambda')$ is an Poisson isomorphism of deformations, then for each $i\in I$ there is an induced Poisson isomorphism:
\begin{align*}
\alpha_i:(U_i\times Spec(k[\epsilon]),\Lambda_0+\epsilon\Lambda_i) \xrightarrow{\theta_i} (\mathcal{X}|_{U_i},\Lambda|_{U_i} )\xrightarrow{\Phi|_{U_i}} (\mathcal{X}'|_{U_i},\Lambda'|_{U_i} )\xrightarrow{\theta_i^{'-1}} (U_i\times Spec(k[\epsilon]),\Lambda_0+\epsilon\Lambda_i')
\end{align*}
So $\alpha_i$ corresponds to $a_i\in \Gamma(U_i,T_X)$ such that $\Lambda_i'-\Lambda_i-[\Lambda_0,a_i]=0$ by Lemma \ref{3l}. Since $-\delta(\{a_i\})=a_i-a_j=p_{ij}'-p_{ij}$ and $\Lambda'_i - \Lambda_i=[\Lambda_0,a_i]$, $(\{p_{ij}\},\{\Lambda_i\})$ and $(\{p_{ij}'\},\{\Lambda_i'\})$ are cohomologous. 

Now we define an inverse map. Given an element in $HP^2(X,\Lambda_0)$, we represent it by a hyper Cech $1$-cocylce $(\{p_{ij}\}, \{\Lambda_i\})$ for an affine open cover $\mathcal{U}=\{U_i\}$ of $X$. So we have $[\Lambda_0,\Lambda_i]=0$, $p_{ij}+p_{jk}-p_{ik}=0$ and $\Lambda_j-\Lambda_i=[\Lambda_0,p_{ij}]=0$. By reversing the above process, the cohomology class gives a glueing condition to make a Poisson deformation of $(X,\Lambda_0)$ whose underlying deformation is a  locally trivial flat deformation.
\end{proof}

\begin{definition}[Poisson Kodaria-Spencer map in an algebraic Poisson family]
For every first-order Poisson deformation $\xi$ of a Poisson algebraic variety $(X,\Lambda_0)$ whose underlying flat deformation is locally trivial, the cohomology class $\kappa(\xi)\in HP^2(X,\Lambda_0)$ is called the Poisson Kodaira-Spencer class of $\xi$. Now we assume that $(X,\Lambda_0)$ is a nonsingular Poisson variety. Let's consider a Poisson deformation of $(X,\Lambda_0)$
\begin{center}
$\xi:$$\begin{CD}
(X,\Lambda_0)@>>> (\mathcal{X},\Lambda)\\
@VVV @VVfV\\
Spec(k)@>s>> S
\end{CD}$
\end{center}
where the base space $S$ is a connected algebraic $k$-scheme and $\mathcal{X}$ is a Poisson scheme over $S$ defined by $\Lambda\in \Gamma(\mathcal{X}, \mathscr{H}om_{\mathcal{O}_{\mathcal{X}}}(\wedge^2\Omega_{\mathcal{X}/S}^1,\mathcal{O}_{\mathcal{X}}))$ We define $k$-linear map, called Poisson Kodaira-Spencer map of the family $\xi$ at $s\in S$,
\begin{align*}
\kappa_{\xi,s}:T_{S,s}\to HP^2(X,\Lambda_0)
\end{align*}
in the following way: let $U$ be an affine open neighborhood of $s\in S$ and $d\in T_{S,s}=Der_k(\mathcal{O}_{S,s},\mathcal{O}_{S,s})$. Let $\bar{d}:\mathcal{O}_{S,s}\to \mathcal{O}_{S,s}/\mathfrak{m}_s$ induced by $d$ and the canonical surjection $\mathcal{O}_{S,s}\to \mathcal{O}_{S,s}/\mathfrak{m}_s,a\to \bar{a}$. Let's consider the following homomorphisms
\begin{align*}
\mathcal{O}_S(U)\to (\mathcal{O}_{S,s}/\mathfrak{m}_s)\oplus \epsilon(\mathcal{O}_{S,s}/\mathfrak{m}_s)\cong k[\epsilon]\to \mathcal{O}_{S,s}/\mathfrak{m}_s\cong k,(a\mapsto \bar{a}+\epsilon \bar{d}(a)\mapsto \bar{a})
\end{align*}
This defines a morphism $Spec(k)\to Spec(k[\epsilon])\to U\hookrightarrow S$. We pullback $(\mathcal{X},\Lambda)$ over $S$ to a first order Poisson deformation over $Spec(k[\epsilon])$ via the map $Spec(k[\epsilon])\to S$. Then by Proposition \ref{3q}, we can find a cohomology class in $HP^2(X,\Lambda_0)$.
\end{definition}

\subsection{Higher-order Poisson deformation-obstructions}\

Let $(X,\Lambda_0)$ be a nonsingular Poisson algebraic variety. Consider a small extension
\begin{align*}
e:0\to (t)\to \tilde{A}\to A\to 0
\end{align*}
in $\bold{Art}$. let
\begin{center}
$\xi:
\begin{CD}
(X,\Lambda_0)@>>> (\mathcal{X},\Lambda)\\
@VVV @VVV\\
Spec(k)@>>> Spec(A)
\end{CD}$
\end{center}
be an infinitesimal Poisson deformation of $(X,\Lambda_0)$ over $A$. A lifting of $\xi$ to $\tilde{A}$ is a infinitesimal Poisson deformation $\tilde{\xi}$ over $\tilde{A}$ inducing $\xi$. In other words,
\begin{center}
$\tilde{\xi}:
\begin{CD}
(X,\Lambda_0)@>>> (\tilde{\mathcal{X}},\tilde{\Lambda})\\
@VVV @VVV\\
Spec(k)@>>> Spec(\tilde{A})
\end{CD}$
\end{center}
and an Poisson isomorphism $\phi$ of Poisson deformations such that the following diagram commutes

\begin{center}
\[\begindc{\commdiag}[50]
\obj(0,1)[a]{$(\mathcal{X},\Lambda)$}
\obj(1,2)[b]{$(X,\Lambda_0)$}
\obj(1,0)[c]{$Spec(A)$}
\obj(2,1)[d]{$(\tilde{\mathcal{X}},\tilde{\Lambda})\times_{Spec(\tilde{\mathcal{A}})} Spec(A)$}
\mor{b}{a}{}
\mor{b}{d}{}
\mor{a}{c}{}
\mor{d}{c}{}
\mor{a}{d}{$\phi$}
\enddc\]
\end{center}

\begin{proposition}[compare $\cite{Ser06}$ Proposition 1.2.12]
Let $(X,\Lambda_0)$ be a nonsingular Poisson variety. Let $A\in \bold{Art}$ and an infinitesimal Poisson deformation $\xi$ of $(X,\Lambda_0)$ over $A$. To every small extension $e:0\to (t)\to \tilde{A}\to A\to 0$,  there is associated an element $o_{\xi}(e)\in HP^3(X,\Lambda_0)$ called the obstruction lifting $\xi$ to $\tilde{A}$, which is $0$ if and only if a lifting of $\xi$ to $\tilde{A}$ exists.
\end{proposition}

\begin{proof}
Let $\mathcal{U}=\{U_i=Spec(B_i)\}_{i\in I}$ be an affine open cover of $X$. We have Poisson isomorphisms $\theta_i:(U_i\times Spec(A),\Lambda_i)\to (\mathcal{X}|_{U_i},\Lambda|_{U_i})$ and $\theta_{ij}:=\theta_i^{-1}\theta_j$ is a Poisson isomorphism. We have $\theta_{ij}\theta_{jk}=\theta_{ik}$ on $U_{ijk}\times Spec(A)$. To give a lifting $\tilde{\xi}$ of $\xi$ to $\tilde{A}$ is equivalent to give a collection of $\{\tilde{\Lambda}_i\}$ where $\tilde{\Lambda}_i \in Hom_{B_i}(\Omega_{B_i /k}^1,B_i)\otimes_k \tilde{A}$ with $[\tilde{\Lambda}_i,\tilde{\Lambda}_i]=0$ is a Poisson structure on $U_i\times Spec(\tilde{A})$ and a collection of Poisson isomorphisms $\{\tilde{\theta}_{ij}\}$ where $\tilde{\theta}_{ij}:(U_{ij}\times Spec(\tilde{A}),\tilde{\Lambda}_j)\to (U_{ij}\times Spec(\tilde{A}),\tilde{\Lambda}_i)$ such that
\begin{enumerate}
\item $\tilde{\theta}_{ij}\tilde{\theta}_{jk}=\tilde{\theta}_{ik}$ as a Poisson isomorphism.
\item $\tilde{\theta}_{ij}$ restricts to $\theta_{ij}$ on $U_{ij}\times Spec(A)$.
\item $\tilde{\Lambda}_i$ restrits to $\Lambda_i$.
\end{enumerate}

From such data, we can glue together $(U_i\times Spec(\tilde{A}),\tilde{\Lambda}_i)$ to make a Poisson deformation $(\tilde{\mathcal{X}},\tilde{\Lambda})$. Now given a Poisson deformation $\xi=(\mathcal{X},\Lambda)$ over $A$ and a small extension $e:0\to (t)\to \tilde{A}\to A\to0$, we associate an element $o_{\xi}(e)\in HP^3(X,\Lambda_0)$. Choose arbitrary automorphisms $\{\tilde{\theta}_{ij}\}$ satisfying $(2)$ (for the existence of lifting, see \cite{Ser06} Lemma 1.2.8) and arbitrary $\tilde{\Lambda}_i\in Hom_{B_i}(\Omega_{B_i/k}^1,B_i)\otimes \tilde{A}$ satisfying $(3)$ (not necessarily $[\tilde{\Lambda}_i,\tilde{\Lambda}_i]=0$). The lifting exists since $Hom_{B_i}(\Omega_{B_i/k}^1,B_i)\otimes_k \tilde{A} \to Hom_{B_i}(\Omega_{B_i/k}^1,B_i)\otimes_k A$ is surjective. Let $\tilde{\theta}_{ijk}=\tilde{\theta}_{ij}\tilde{\theta}_{jk}\tilde{\theta}_{ik}^{-1}$. Since $\tilde{\theta}_{ijk}$ is an automorphism on $U_{ijk}\times Spec(\tilde{A})$ inducing the identity on $U_{ijk}\times Spec(A)$, $\tilde{\theta}_{ijk}$ corresponds to $\tilde{d}_{ijk}\in \Gamma(U_{ijk},T_X)$ and $d_{jkl}-d_{ikl}+d_{ijl}-d_{jkl}=0$. So we have $-\delta(\{d_{ijk}\})=0$. Since $[\tilde{\Lambda}_i,\tilde{\Lambda}_i]$ is zero  modulo $(t)$ by $[\Lambda_i,\Lambda_i]=0$, there exists $\Pi_i\in Hom_{B_i}(\wedge^3 \Omega_{B_i/k}^1, B_i)$ such that $[\tilde{\Lambda}_i,\tilde{\Lambda}_i]=t\Pi_i$. Since $0=[\tilde{\Lambda}_i,[\tilde{\Lambda}_i,\tilde{\Lambda}_i]]=[\tilde{\Lambda}_i,t\Pi_i]=t[\Lambda_0,\Pi_i]=0$, we have $[\Lambda_0,\Pi_i]=0$. 

 Let $\tilde{f}_{ij}: \mathcal{O}_X(U_{ij})\otimes_k \tilde{A} \to \mathcal{O}_X(U_{ji})\otimes_k \tilde{A}$ be the ring homomorphism corresponding to $\tilde{\theta}_{ij}$. We will denote by $\tilde{f}_{ij}\Lambda_i$ be the induced biderivation structure on $\mathcal{O}_X(U_{ji})\otimes_k \tilde{A}$ such that $\tilde{f}_{ij}:(\mathcal{O}_X(U_{ij})\otimes_k\tilde{A},\tilde{\Lambda}_i)\to (\mathcal{O}_X(U_{ji})\otimes_k\tilde{A}, \tilde{f}_{ij}\Lambda_i)$ is biderivation-preserving.  Since $\tilde{f}_{ij}\tilde{\Lambda}_i$ and $\tilde{\Lambda}_j$ are same modulo $(t)$ by $(3)$, there exists $\Lambda_{ij}'\in \Gamma(U_{ij}, \mathscr{H}om_{\mathcal{O}_X}(\wedge^2 \Omega_{X/k}^1, \mathcal{O}_X))$ such that $t\Lambda_{ij}'=\tilde{f}_{ij}\tilde{\Lambda}_i-\tilde{\Lambda}_j$. Then $t\Lambda_{ji}'=\tilde{f}_{ji}\Lambda_j-\Lambda_i$. By applying $\tilde{f}_{ij}$ on both sides, we have $t\Lambda_{ji}'=\tilde{\Lambda}_j-\tilde{f}_{ij}\tilde{\Lambda}_i=-t\Lambda_{ij}'$. Hence $\Lambda_{ji}'=-\Lambda_{ij}'$. Then $t\Pi_i-t\Pi_j=\tilde{f}_{ij}(t\Pi_i)-t\Pi_j=\tilde{f}_{ij}[\tilde{\Lambda}_i,\tilde{\Lambda}_i]-[\tilde{\Lambda}_j,\tilde{\Lambda}_j]=[\tilde{f}_{ij}\Lambda_i,\tilde{f}_{ij}\Lambda_i]-[\tilde{\Lambda}_j,\tilde{\Lambda}_j]=[\tilde{\Lambda}_j+t\Lambda_{ij}',\tilde{\Lambda}_j+t\Lambda_{ij}']-[\tilde{\Lambda}_j,\tilde{\Lambda}_j]=t[\Lambda_0,2\Lambda_{ij}']$. Hence we have $-\Pi_i-(-\Pi_j)+[\Lambda_0, 2\Lambda_{ij}']=0$. So we have $-\delta(\{-\Pi_i\})+[\Lambda_0,\{2\Lambda_{ij}'\}]=0$. In the following isomorphism
\begin{align*}
\tilde{\alpha}_{ijk}:U_{ijk}\times Spec(\tilde{A}) \xrightarrow{\tilde{\theta}_{ki}}U_{ijk}\times Spec(\tilde{A}) \xrightarrow{\tilde{\theta}_{jk}}U_{ijk}\times Spec(\tilde{A})\xrightarrow{\tilde{\theta}_{ij}}U_{ijk}\times Spec(\tilde{A})
\end{align*}
which corresponds to a $\tilde{d}_{ijk} \in \Gamma(U_{ijk},T_X)$. Then we have
\begin{align*}
Id+t\tilde{d}_{ijk}:\mathcal{O}_X(U_{ijk})\otimes_k \tilde{A} \xrightarrow{\tilde{f}_{ij}}\mathcal{O}_X(U_{ijk})\otimes_k \tilde{A} \xrightarrow{\tilde{f}_{jk}}\mathcal{O}_X(U_{ijk})\otimes_k \tilde{A} \xrightarrow{\tilde{f}_{ki}} \mathcal{O}_X(U_{ijk})\otimes_k \tilde{A}
\end{align*}
 $Id+t\tilde{d}_{ijk}: (\mathcal{O}_X(U_{ijk})\otimes \tilde{A} ,\tilde{\Lambda}_i) \to (\mathcal{O}_X(U_{ijk})\otimes_k \tilde{A} ,\tilde{f}_{ki}\tilde{f}_{jk}\tilde{f}_{ij}\tilde{\Lambda}_i)$ is an isomorphism compatible with bidervations. We note that $\tilde{\Lambda}_i-\tilde{f}_{ki}\tilde{f}_{jk}\tilde{f}_{ij}\tilde{\Lambda}_i=\tilde{\Lambda}_i-\tilde{f}_{ki}\tilde{f}_{jk}(\tilde{\Lambda}_j+t\Lambda_{ij}')=\tilde{\Lambda}_i-\tilde{f}_{ki}(\tilde{\Lambda}_k+t\Lambda'_{jk}+\Lambda_{ij}')=\tilde{\Lambda}_i-(\tilde{\Lambda}_i+t\Lambda_{ki}'+t\Lambda_{jk}'+t\Lambda_{ij}')=-t(\Lambda_{ki}'+\Lambda_{jk}'+\Lambda_{ij}')$.  Hence by Lemma \ref{3l}, we have $\Lambda_{ki}'+\Lambda_{jk}'+\Lambda_{ij}'-[\Lambda_0, \tilde{d}_{ijk}]=0$. So we have $-\delta(\{\Lambda_{ij}'\})+[\Lambda_0,\{\tilde{d}_{ijk}\}]=0$. Hence $\alpha=(\{-\Pi_i\}, \{2\Lambda_{ij}'\},\{2\tilde{d}_{ijk}\})\in C^0(\mathcal{U},\mathscr{H}om_{\mathcal{O}_X}(\wedge^3 \Omega_{X/k},\mathcal{O}_X)))\oplus C^1(\mathcal{U},\mathscr{H}om_{\mathcal{O}_X}(\wedge^2 \Omega_{X/k},\mathcal{O}_X))\oplus C^2(\mathcal{U},\mathscr{H}om_{\mathcal{O}_X}(\Omega_{X/k}^1,\mathcal{O}_X)))$ is a cocyle in the following diagram (see Appendix \ref{appendixb}).
\begin{center}
$\begin{CD}
@A[\Lambda_0,-]AA\\
C^0(\mathcal{U},\mathscr{H}om_{\mathcal{O}_X}(\wedge^3 \Omega_{X/k},\mathcal{O}_X)))@>-\delta>>\cdots\\
@A[\Lambda_0,-]AA @A[\Lambda_0,-]AA\\
C^0(\mathcal{U},\mathscr{H}om_{\mathcal{O}_X}(\wedge^2 \Omega_{X/k},\mathcal{O}_X)))@>\delta>> C^1(\mathcal{U},\mathscr{H}om_{\mathcal{O}_X}(\wedge^2 \Omega_{X/k},\mathcal{O}_X)))@>-\delta>>\cdots\\
@A[\Lambda_0,-]AA @A[\Lambda_0,-]AA @A[\Lambda_0,-]AA\\
C^0(\mathcal{U},\mathscr{H}om_{\mathcal{O}_X}(\Omega_{X/k}^1,\mathcal{O}_X)))@>-\delta>>C^1(\mathcal{U},\mathscr{H}om_{\mathcal{O}_X}(\Omega_{X/k}^1,\mathcal{O}_X)))@>\delta>>C^2(\mathcal{U},\mathscr{H}om_{\mathcal{O}_X}(\Omega_{X/k}^1,\mathcal{O}_X)))\\
@AAA @AAA @AAA \\
0@>>>0 @>>> 0 
\end{CD}$
\end{center}

We claim that given a different choice $\{\tilde{\theta}_{ij}' \}$ and $\{\tilde{\Lambda}_i'\}$ satisfying $(1),(2),(3)$, the associated cocyle $\beta=(\{-\Pi_i'\}, \{2\Lambda_{ij}''\},\{2\tilde{d}'_{ijk}\})\in C^0(\mathcal{U},\mathscr{H}om_{\mathcal{O}_X}(\wedge^3 \Omega_{X/k},\mathcal{O}_X)))\oplus C^1(\mathcal{U},\mathscr{H}om_{\mathcal{O}_X}(\wedge^2 \Omega_{X/k},\mathcal{O}_X))\oplus C^2(\mathcal{U},\mathscr{H}om_{\mathcal{O}_X}(\Omega_{X/k}^1,\mathcal{O}_X)))$ is cohomologous to the cocyle associated with $\{\tilde{\theta}_{ij}\}$ and $\{\tilde{\Lambda}_i\}$. Let $\tilde{f}_{ij}:\mathcal{O}_X(U_{ij})\otimes\tilde{A}\to \mathcal{O}_X(U_{ij})\otimes \tilde{A}$ corresponding to $\tilde{\theta}_{ij}$, and $\tilde{f}'_{ij}:\mathcal{O}_X(U_{ij})\otimes \tilde{A}\to \mathcal{O}_X(U_{ij})\otimes \tilde{A}$ corresdpong to $\theta_{ij}'$. Then $\tilde{f}_{ij}'=\tilde{f}_{ij}+tp_{ij}$ for some $p_{ij}\in \Gamma(U_{ij},T_X)$ \footnote{Since $\tilde{f}'_{ij}-\tilde{f}_{ij}$ is zero modulo $t$, we have $(\tilde{f}'_{ij}-\tilde{f}_{ij})(x)=0+tp_{ij}(x)$ for some map $p_{ij}$. We show that $p_{ij}$ is a derivation. Indeed $tp_{ij}(xy)=(\tilde{f}_{ij}-f_{ij})(xy)=\tilde{f}_{ij}(x)(\tilde{f}_{ij}-f_{ij})(y)+(\tilde{f}_{ij}-f_{ij})(x)f_{ij}(y)=\tilde{f}_{ij}(x)tp_{ij}(y)+tp_{ij}(y)f_{ij}(y)=t(xp_{ij}(y)+yp_{ij}(x))$. So $p_{ij}$ is a derivation and so an element in $\Gamma(U_{ij},T_X)$.} and $\tilde{\Lambda}'_i=\tilde{\Lambda}_i+t\Lambda_i'$ for some $\Lambda_i'\in Hom_{B_i}(\wedge^2 \Omega_{B_i/k}^1,B_i)$. For each $i,j,k$, $\tilde{\theta}_{ij}'\tilde{\theta}_{jk}'\tilde{\theta'}_{ik}^{-1}$ corresponds to the derivation $\tilde{d}'_{ijk}=\tilde{d}_{ijk}+(p_{ij}+p_{jk}-p_{ik})$. Hence $\delta(2\{p_{ij}\})=\{2\tilde{d}_{ijk}'-(2\tilde{d}_{ijk})\}$. We also note that  $t\Pi_i'=[\tilde{\Lambda}'_i,\tilde{\Lambda}'_i]=[\tilde{\Lambda}_i+t\Lambda_i',\tilde{\Lambda}_i+t\Lambda_i']=[\tilde{\Lambda}_i,\tilde{\Lambda}_i]+t[2\Lambda_i',\Lambda_0]=t\Pi_i+t[2\Lambda_i',\Lambda_0]$. Hence we have $[\Lambda_0,2\Lambda_i']=-\Pi_i-(-\Pi_i')$. Since $t\Lambda_{ij}'=\tilde{f}_{ij}\tilde{\Lambda}_i-\tilde{\Lambda}_j, t\Lambda_{ij}''=\tilde{f}_{ij}'\tilde{\Lambda}_i'-\tilde{\Lambda}_j'=\tilde{f}_{ij}\tilde{\Lambda}_i'+t[p_{ij},\tilde{\Lambda}_i']-\tilde{\Lambda}_j'=\tilde{f}_{ij}\tilde{\Lambda}_i+t\Lambda_i'+t[p_{ij},\Lambda_0]-\tilde{\Lambda}_j-t\Lambda_j'$, we have $\Lambda_{ij}'-\Lambda_{ij}''=-\Lambda_i'+[\Lambda_0,p_{ij}]+\Lambda_j'$. So $\delta(\{2\Lambda_i'\})+[\Lambda_0,\{2p_{ij}\}]=\Lambda_{ij}'-\Lambda_{ij}''.$ Hence $(\{2\Lambda_i'\}, \{-2p_{ij}\})$ is mapped to $\alpha-\beta$. Hence $\alpha$ and $\beta$ are cohomologous. So given a deformation $\xi$ and a small extension $e:0\to (t)\to \tilde{A}\to A\to 0$, we can associate an element $o_{\xi}(e):=$ the cohomology class of $\alpha \in HP^3(X,\Lambda_0)$. We also note that $o_{\xi}(e)=0$ if and only if there exists a collection of $\{\tilde{\theta}_{ij}\}$ and $\{\tilde{\Lambda}_i\}$ satisfying $(2),(3)$ with $[\tilde{\Lambda}_i,\tilde{\Lambda}_i]=0$ (which means $\tilde{\Lambda}_i$ defines a Poisson structure), $\Lambda_{ij}'=0$ (which implies $\tilde{f}_{ij}\tilde{\Lambda}_i=\tilde{\Lambda}_j$) and $\tilde{d}_{ijk}=0$ (which means $(1)$) if and only if there is a lifting $\tilde{\xi}$.

\end{proof}

\begin{definition}
The Poisson deformation $\xi$ is called unobstructed if $o_{\xi}$ is the zero map, otherwise $\xi$ is called obstructed. $(X,\Lambda_0)$ is unobstructed if every infinitesimal deformation of $(X,\Lambda_0)$ is unobstructed, otherwise $(X,\Lambda_0)$ is obstructed.
\end{definition}

\begin{corollary}
A nonsingular Poisson variety $(X,\Lambda_0)$ is unobstructed if $HP^3(X,\Lambda_0)=0$.
\end{corollary}

\begin{proposition}\label{3rigid}
A nonsingular Poisson variety $(X,\Lambda_0)$ is rigid if and only if $HP^2(X,\Lambda_0)=0$.
\end{proposition}

\begin{proof}
Assume that $(X,\Lambda_0)$ is rigid. Since any infinitesimal Poisson deformation  (in particular, any first order Poisson deformations) are trivial, $HP^2(X,\Lambda_0)=0$ by Proposition \ref{3q}. Assume that $HP^2(X,\Lambda_0)=0$. First we claim that given an infinitesimal Poisson deformation $\eta$ of $(X,\Lambda_0)$ over $A\in \bold{Art}$ and a small extension $e:0\to (t)\to \tilde{A}\to A\to 0$, any two liftings $ \xi, \tilde{\xi}$ to $\tilde{A}$ are equivalent. Let $\{U_i\}$ be an affine open covering of $\xi=(\mathcal{X},\Lambda)$ and $\tilde{\xi}=(\tilde{\mathcal{X}},\tilde{\Lambda})$. Let $\{\theta_i\}$ where $\theta_i :U_i\times Spec(\tilde{A})\to \mathcal{X}|_{U_i} $, $\{\Lambda_i\}$ where $\Lambda_i$ is the Poisson structure on $U_i\times Spec(\tilde{A})$ induced from from $\Lambda|_{U_i}$ and let $\theta_{ij}=\theta_i^{-1}\theta_j$ which corresponds to a $d_{ij}\in \Gamma(U_{ij},T_X)$. Let $\{\tilde{\theta}_i\}$  where $\tilde{\theta}_i :U_i\times Spec(\tilde{A})\to \mathcal{\tilde{X}}|_{U_i} $, $\{\tilde{\Lambda}_i\}$ where the induced Poisson structure from $\tilde{\Lambda}$ on $U_i\times Spec(\tilde{A})$ and let $\tilde{\theta}_{ij}=\tilde{\theta}_i^{-1}\tilde{\theta}_j$. Let $f_{ij}:(\mathcal{O}_X({U_{ij}})\otimes\tilde{A},{\Lambda}_i )\to (\mathcal{O}_X(U_{ij})\otimes \tilde{A},\Lambda_j)$ be the homomorphism corresponding to $\theta_{ij}$ and $\tilde{f}_{ij}:(\mathcal{O}_X(U_{ij})\otimes \tilde{A},\tilde{\Lambda}_i)\to (\mathcal{O}_X(U_{ij})\otimes \tilde{A},\tilde{\Lambda}_j)$ corresponding to $\tilde{\theta}_{ij}$. Since $\xi,\tilde{\xi}$ induce same Poisson deformation $\eta$ over $A$, we have
\begin{center}
$\tilde{f}_{ij}=f_{ij}+tp_{ij}$,$\,\,\,\,\,\,\,\,\,\,\, \tilde{\Lambda}_i=\Lambda_i+t\Lambda_i'$
\end{center}
for some $p_{ij}\in \Gamma(U_{ij},T_X)$.
Since for all $i,j,k$ we have $0=\tilde{d}_{ijk}=p_{ij}+p_{jk}-p_{ik}$. Since $0=[\tilde{\Lambda}_i,\tilde{\Lambda}_i]=[\Lambda_i+t\Lambda_i',\Lambda_i+t\Lambda_i']=2t[\Lambda'_i,\Lambda_0]$. Since $f_{ij}\Lambda_i=\Lambda_j$ and $\tilde{f}_{ij}\tilde{\Lambda}_i=\tilde{\Lambda}_j$,  we have ${\Lambda}_j+t\Lambda_j'=\tilde{\Lambda}_j=\tilde{f}_{ij}\tilde{\Lambda}_i=(f_{ij}+tp_{ij})(\Lambda_i+t\Lambda_i')=\Lambda_j-t[\Lambda_0,p_{ij}]+t\Lambda_i'$. Hence we have $\Lambda_j'-\Lambda_i'+[\Lambda_0,p_{ij}]=0$. Hence $(\{\Lambda_i'\},\{ p_{ij}\})$ defines a cocylce. Since $HP^2(X,\Lambda_0)=0$, there exists $\{a_i\}\in \mathcal{C}^0(\mathcal{U},T_X)$ such that $[\Lambda_0,a_i]=\Lambda_i'$ and $a_i-a_j=p_{ij}$. Now we explicitly construct a Poisson isomorphism $(\tilde{\mathcal{X}},\tilde{\Lambda}) \cong (\mathcal{X},\Lambda)$. We define a Poisson isomorphism locally on $U_i\times Spec(\tilde{A})$, and show that each map glue together to give an Poisson isomorphism $(\tilde{\mathcal{X}},\tilde{\Lambda}) \cong (\mathcal{X},\Lambda)$. We claim that $(U_i\times Spec(\tilde{A}),\Lambda_i)\to ( U_i\times Spec(\tilde{A}), \tilde{\Lambda}_i)$ is a Poisson isomorphism induced from $Id+ta_i:(\mathcal{O}_X(U_i)\otimes_k \tilde{A},\tilde{\Lambda}_i)\to (\mathcal{O}_X(U_i)\otimes_k \tilde{A},\Lambda_i)$. The inverse map is $Id-ta_{i}$. Since $\tilde{\Lambda}_i+t[a_i,\tilde{\Lambda}_i]=\tilde{\Lambda}_i+t[a_i,\Lambda_0]=\Lambda_i+t\Lambda_i'+t[a_i,\Lambda_0]=\Lambda_i$, $Id+ta_i$ is Poisson. We show that each Poisson isomorphism $\{Id+ta_i\}$ glues together to give a Poisson isomorphism $(\tilde{\mathcal{X}},\tilde{\Lambda}) \cong (\mathcal{X},\Lambda)$. Indeed, it is sufficient to show that the following diagram commutes.
\begin{center}
$\begin{CD}
(\mathcal{O}_X(U_{ij})\otimes_k Spec(\tilde{A}), \tilde{\Lambda}_j )@> Id+ta_j >> (\mathcal{O}_X(U_{ij})\otimes_k \tilde{A},\Lambda_j)\\
@A\tilde{f}_{ij}AA@AA f_{ij}A\\
(\mathcal{O}_X(U_{ij})\otimes Spec(\tilde{A}), \tilde{\Lambda}_i)@> Id+ta_i >> (\mathcal{O}_X(U_{ij})\otimes_k \tilde{A},\Lambda_i)
\end{CD}$
\end{center}
Indeed, the diagram commutes if and only if $(Id+ta_j)\circ \tilde{f}_{ij}=f_{ij}\circ (Id+ta_i)$ if and only if $\tilde{f}_{ij}+ta_j=f_{ij}+ta_i$ if and only if $p_{ij}=a_i-a_j$. Hence there is at most one lifting of $\xi$.

Now we prove that if $HP^2(X,\Lambda_0)=0$, then $(X,\Lambda_0)$ is rigid. We will prove by induction on the dimension on $(A,\mathfrak{m})\in \bold{Art}$. For $A$ with $dim_k A=2$, then any first order Poisson deformation is trivial. Let' assume that any infinitesimal Poisson deformation of $(X,\Lambda_0)$ over $A$ with $dim_k A\leq n-1$ is trivial.  Let $\xi$ be an infinitesimal Poisson deformation of $(X,\Lambda_0)$ over $A$ with  $dim_k A=n$ such that $\mathfrak{m}^{p-1}\ne 0$ and $\mathfrak{m}^p=0$. Choose an element $t\ne 0\in \mathfrak{m}^{p-1}$. Then $0\to (t)\to A\to A/(t)\to 0$ is a small extension and $dim_k A/(t)\leq n-1$. Hence induced Poisson deformation $\bar{\xi}$ over $A/(t)$ from $\xi$ is trivial by induction hypothesis. Since $\xi$ is a lifting of $\bar{\xi}$, and trivial Poisson deformation over $A$ is also a lifting of $\bar{\xi}$, $\xi$ is trivial since we have at most one lifting of $\bar{\xi}$.
\end{proof}

\chapter{Poisson deformation functors}\label{chapter8}

\section{Schlessigner's criterion}\footnote{For details, see $\cite{Har10}$ page 106-117.}
We discuss Functor of Artin rings in more detail and Schlessinger's criterions. We recall that $\bold{Art}$ is the category of local artinian $k$-algebras with residue field $k$. Before the discussion, we note that following: let $f:(A',\mathfrak{m}')\to( A,\mathfrak{m})$ and $g:(A'',\mathfrak{m}'')\to (A,\mathfrak{m})$ be two local homomorphisms of local artinian $k$-algebras with the residue $k$. So we have $f^{-1}(\mathfrak{m})=\mathfrak{m}'$ and $g^{-1}(\mathfrak{m})=\mathfrak{m}''$. Let's consider the fiber product $\bar{A}=A'\times_A A''=\{(a,b)|a\in A', b\in A'', f(a)=g(b)\}$ is also a local artinian $k$-algebra, which is defined by $(a_1,b_1)\cdot (a_2,b_2)=(a_1a_2,b_1b_2)$, $(a_1,b_1)+(a_2,b_2)=(a_1+b_1, a_2+b_2)$. We will show that $A'\times_A A''$ is a local Artininan $k$-algebra with the residue $k$. The maximal ideal is given by $\bar{\mathfrak{m}}=\{(m,n)|m\in \mathfrak{m}', n\in \mathfrak{m}'', f(n)=g(m)\}$. We will show that $\bar{m}$ is the unique maximal ideal. It is enough to show that $(a,b)\in \bar{A}-\bar{\mathfrak{m}}$ is an unit. Indeed, we have $a\in  A-\mathfrak{m}'$ and $b\in A'-\mathfrak{m}''$. So $a,b$ are units. Hence $(a^{-1},b^{-1})$ is a inverse of $(a,b)$. Now we show that $\bar{A}/\bar{\mathfrak{m}}\cong k$. We define a map $\varphi:\bar{A}\to A'/\mathfrak{m}'\times_{A/\mathfrak{m}} A''/\mathfrak{m}''$ by $(a,b)\mapsto (\bar{a},\bar{b})\cong k\times_k k\cong k$. Let $\varphi(a,b)=(\bar{a},\bar{b})=0$. Then $a\in \mathfrak{m}'$ and $b\in \mathfrak{m}''$. Hence $ker \varphi=\bar{\mathfrak{m}}$. So the natural maps $f':(\bar{A},\bar{\mathfrak{m}})\to (A',\mathfrak{m}')$ and $g':(\bar{A},\bar{\mathfrak{m}})\to (A'',\mathfrak{m}'')$  are local homomorphisms of local Artininan $k$-algebras and $f'^{-1}(\mathfrak{m'})=\bar{\mathfrak{m}}$, $g'^{-1}(\mathfrak{m}'')=\bar{\mathfrak{m}}$.

\begin{definition}
Let $R$ be a complete local $k$-algebra, and for each $A\in \bold{Art}$, we define $h_R$ be a functor of Artin rings in the following way
\begin{align*}
h_R:\bold{Art}&\to \bold{Sets}\\
         A&\mapsto h(R):=Hom_k(R,A)
\end{align*}
A covariant functor $F:\bold{Art}\to \bold{Set}$ that is isomorphic to a functor of the form $h_R$ for some complete local $k$-algebra $R$ is called pro-representable.
\end{definition}

Let $(R,\mathfrak{m})$ be a complete local $k$-algebra. Let $\varphi:h_R\to F$ be a morphism of functors of Artin rings, then for each $n$, we have the following commutative diagram from the canonical map $R/\mathfrak{m}^{n+1}\to R/\mathfrak{m}^n$.
\begin{center}
$\begin{CD}
Hom(R,R/\mathfrak{m}^{n+1})@>\varphi_{n+1}>> F(R/\mathfrak{m}^{n+1})\\
@VVV @VVV\\
Hom(R,R/\mathfrak{m}^n)@>\varphi_n>> F(R/\mathfrak{m}^n)
\end{CD}$
\end{center}
Let $\pi_n:R\to R/\mathfrak{m}^n$ be the canonical surjection. Then we have $\xi=\{\xi_n\}:=\{\varphi_n(\pi_n)\}\in \varprojlim F(R/\mathfrak{m}^n)$. We call $\xi=\{\xi_n\}$ a formal family of $F$ over $R$.

\begin{definition}
Let $\bold{C}$ be the category of complete local $k$-algebras with residue field $k$. Let $F$ be a functor of Artin rings. We define 
\begin{align*}
\hat{F}:\bold{C}&\to \bold{Sets}\\
        (R,\mathfrak{m}) &\mapsto \hat{F}(R)=\varprojlim F(R/\mathfrak{m}^n)
\end{align*}
\end{definition}

Let $\xi=\{\xi_n\}\in \hat{F}(R)$ be a formal family. Then from this, we can define a morphism of functors $h_R\to F$ of functor of Artin rings as follows. For $A\in \bold{Art}$, we define
\begin{align*}
h_R(A)=Hom_k(R,A)&\to F(A)\\
f&\mapsto F(g)(\xi_n)
\end{align*}
Here $g$ is defined in the following way. Since $A$ is artinian, $f:R\to A$ factor through $R/\mathfrak{m}^n$ by $g:R/\mathfrak{m}^n\to A$ for some $n$.

\begin{remark}
If $F$ is a functor of Artin rings, and $R$ is a complete local $k$-algebra with residue $k$, then there is a natural bijection between $\hat{F}(R)$ of formal families $\{\xi_n|\xi_n\in F(R/\mathfrak{m}^n)\}$ and the set of homomorphisms of functors $h_R\to F$. So if $F$ is pro-representable, there is an isomorphism $\xi:h_R\to F$ for some $R$, we can think of $\xi$ as an element of $\hat{F}(R)$. We say that the pair $(R,\xi)$ pro-represents the functor $F$.
\end{remark}

\begin{definition}
Let $F$ be a functor of Artin rings.
\begin{enumerate}
\item A pair $(R,\xi)$ with $R\in \bold{C}$ and $\xi\in \hat{F}(R)$ is called a versal family for $R$ if the associated map $h_R\to F$ is smooth. In other words, for every surjection $B\to A$ in $\bold{Art}$, the natural map $h_R(B)\to h_R(A)\times_{F(A)} F(B)$ is surjective. This implies that given a map $R\to A$ inducing an element $\eta\in F(A)$, given $\theta\in F(B)$ mapping to $\eta$, one can lift the map $R\to A$ to a map $R\to B$ inducing $\theta$.
\item A versal family $(R,\xi)$ with  $R\in \bold{C}$ and $\xi\in \hat{F}(R)$ is called a miniversal family or $F$ has a pro-representable hull $(R,\xi)$ if $h_R(k[\epsilon])\to F(k[\epsilon])$ is bijective.
\item A pair $(R,\xi)$ with  $R\in \bold{C}$ and $\xi\in \hat{F}(R)$ is called a universal family if it pro-represents the functor $F$.
\end{enumerate}
\end{definition}

\begin{thm}[Schlessinger's criterion]\label{sch1}
A functor of Artin rings has a miniversal family if and only if
\begin{itemize}
\item $(H_0)$ $F(k)$ has one element.
\item $(H_1)$ The natural map $F(A'\times_A A'')\to F(A')\times_{F(A)} F(A'')$ is surjective for every small extension $A''\to A$.
\item $(H_2)$ The natural map $F(A'\times_A A'')\to F(A')\times_{F(A)} F(A'')$ is bijective when $A''=k[\epsilon]$ and $A=k$.
\item $(H_3)$ $t_F:=F(k[\epsilon])$ is a finite-dimensional $k$-vector space.
\end{itemize}

\end{thm}

\begin{proof}
See \cite{Har10} Theorem 16.2.
\end{proof}

\begin{remark}
We explain in $(H_3)$ why $t_F:=F(k[\epsilon])$ is a $k$-vector space. Let $F$ be a functor of Artin rings satisfying $(H_0)$ and $(H_2)$. Then $F(k[\epsilon])$ can be considered to be $k$-vector space in the following way. Let's consider the following map 
\begin{align*}
\alpha:k[\epsilon]\times_k k[\epsilon]&\to k[\epsilon]\\
(a+b\epsilon,a+b'\epsilon)&\mapsto a+(b+b')\epsilon
\end{align*}
Then $\alpha((a+b\epsilon,a+b'\epsilon)\cdot (c+d\epsilon,c+d'\epsilon))=\alpha(ac+(ad+bc)\epsilon, ac+(ad'+b'c)\epsilon)=ac+(ad+bc+ad'+b'c)\epsilon$. On the other hand, $\alpha((a+b\epsilon,a+b'\epsilon))\cdot \alpha((c+d\epsilon,c+d'\epsilon))=(a+(b+b')\epsilon)(c+(d+d')\epsilon)=ac+(ad+ad'+bc+b'c)\epsilon$. Hence $\alpha$ is a homomorphism. So $\alpha$ induces $F(\alpha):F(k[\epsilon]\times_k k[\epsilon])\to F(k[\epsilon])$. Since $F$ satisfies $(H_3)$, we have $F(k[\epsilon])\times F(k[\epsilon])\cong F(k[\epsilon]\times_k k[\epsilon])$. So We have a map $F(k[\epsilon])\times F(k[\epsilon])\to F(k[\epsilon])$. This defines an addition. By the following commutativity of homomorphisms and the property $(H_3)$, the operation satisfies associativity$:$
\begin{center}
$\begin{CD}
k[\epsilon]\times_k k[\epsilon] \times_k k[\epsilon]((a+b\epsilon,a+b'\epsilon,a+b''\epsilon))@>>>  k[\epsilon]\times_k k[\epsilon]((a+(b+b')\epsilon,a+b''\epsilon))\\
@VVV @VVV\\
k[\epsilon]\times_k k[\epsilon]((a+b\epsilon,a+(b+b')\epsilon))@>>> k[\epsilon](a+(b+b'+b'')\epsilon)
\end{CD}$
\end{center}
The zero element is the image of $F(k)\to F(k[\epsilon])$ induced from $k\to k[\epsilon],k\to k+\epsilon\cdot 0$. The scalar multiplication by $c\in k$ is defined by $F(k[\epsilon])\to F(k[\epsilon])$ induced from $k[\epsilon]\to k[\epsilon], a+b\epsilon\mapsto a+(cb)\epsilon$. Then the inverse is defined by the map $F(k[\epsilon])\to F(k[\epsilon])$ induced by $k[\epsilon]\to k[\epsilon], a+b\epsilon\mapsto a-b\epsilon$.
\end{remark}

Let's assume that the functor $F$ satisfies $H_0$, $H_1$ and $H_2$. We claim that  for any small extension $0\to(t)\to A'\xrightarrow{p} A\to 0$ and any element $\eta\in F(A)$, there is a transitive group action of the vector space $t_F$ on the set $p^{-1}(\eta)$ if it is nonempty. Here $p:=F(p):F(A')\to F(A)$. Indeed,  we have an isomorphism 
\begin{align*}
\gamma:k[\epsilon]\times_k A'&\to A'\times_A A'\\
                 (x+y\epsilon,a') &\mapsto (a'+yt,a')\\
\gamma^{-1}:A'\times_A A' &\to k[\epsilon]\times_k A'\\
                    (b',a') &\mapsto (\overline{a'}+\overline{(b'-a')}\epsilon,a')
\end{align*}
where $\overline{a'}$ and $\overline{b'-a'}\in k$ are the residues of $a'$ and $b'-a'$ modulo by the maximal ideal of $A'$. Let $\beta:k[\epsilon]\times_k A'\to A, (x+y\epsilon,a')\mapsto a'+yt$. Then we have the following commutative diagram
\begin{center}
\[\begindc{\commdiag}[50]
\obj(0,1)[aa]{$k[\epsilon]\times_k A'$}
\obj(1,0)[bb]{$A'$}
\obj(2,1)[cc]{$A'\times_A A'$}
\mor{aa}{cc}{$\gamma$}
\mor{aa}{bb}{$\beta$}[\atright,\solidarrow]
\mor{cc}{bb}{$pr_1$}
\enddc\]
\end{center}
 Since $F$ satisfies $H_2$, we have a bijection $\alpha^{-1}:F(k[\epsilon])\times F(A')\to F(k[\epsilon]\times_k A')$, so a bijection $F(k[\epsilon]) \times F(A')\to F(A'\times_{A} A')$ induced from $F(\gamma)\circ \alpha^{-1}$.
Since we have the following commutative diagram
{\tiny{\begin{center}
\[\begindc{\commdiag}[70]
\obj(0,1)[aa]{$F(A')$}
\obj(1,0)[bb]{$F(A')\times_{F(A)} F(A')$}
\obj(2,1)[cc]{$F(A')$}
\obj(1,2)[dd]{$F(A'\times_A A')$}
\obj(1,3)[ee]{$F(k[\epsilon]\times_k A')$}
\obj(1,4)[ff]{$t_F\times F(A')$}
\mor{aa}{bb}{}
\mor{cc}{bb}{}
\mor{dd}{aa}{}
\mor{dd}{cc}{}
\mor{ee}{dd}{$F(\gamma)$}
\mor{ff}{ee}{$\alpha^{-1}$}
\mor{ff}{aa}{$F(\beta)\circ \alpha^{-1}$}[\atright,\solidarrow]
\mor{ff}{cc}{$pr$}
\mor{dd}{bb}{}
\enddc\]
\end{center}}}
The map $t_F\times F(A')\to F(A')\times_{F(A)} F(A')$ is surjective and an isomorphism on the second factor since $F$ satisfies $H_1$. If we take $\eta\in F(A)$ and fix $\eta'\in p^{-1}(\eta)$, then we get a surjective map $t_F\times \{\eta'\}\to p^{-1}(\eta) \times \{\eta'\}$, and hence there is a transitive group action of $t_F$ on $p^{-1}(\eta)$.
\begin{thm}[Schlessinger's criterion]\label{sch2}
Let $F$ be a functor of Artin rings. The functor $F$ is prorepresentable if and only if $F$ satisfies $(H_0),(H_1),(H_2),(H_3)$ and 
\begin{itemize}
\item $(H_4)$ For every small extension $p:A''\to A$ and every $\eta\in F(A)$ for which $p^{-1}(\eta)$ is nonempty, the group action of $t_F$ on $p^{-1}(\eta)$ is bijective.
\end{itemize}
\end{thm}

\begin{proof}
See \cite{Har10} Theorem 16.2.
\end{proof}

\section{Poisson Deformation functors}

\begin{definition}
Let $(X,\Lambda_0)$ be a Poisson algebraic scheme. For $A\in \bold{Art}$, we let
\begin{center}
$PDef_{(X,\Lambda_0)}=\{\text{infinitesimal Poisson deformations of $(X,\Lambda_0)$ over $A$}\}/\text{Poisson isomorphisms}$
\end{center}
Then $PDef_{(X,\Lambda_0)}$ is a functor of Artin rings.
\end{definition}

We will prove that $PDef_{(X,\Lambda_0)}$ has a miniversal family when $(X,\Lambda_0)$ is a smooth projective Poisson scheme. Before the proof, we note the following: let $B,B',B''$ are Poisson algebras over local Artininan $k$-algebras $A,A',A''$ with the residue $k$ as above. If $A'$-Poisson algebra homomoprhisms $p:B'\to B$, and $A''$-Poisson algebra homomoprhism $q:B''\to B$, then $\bar{B}=B'\times_B B''=\{(m,n)|m\in B',n\in B'', p(m)=q(n)\}$ is a Poisson $\bar{A}$-algebra defined in the following way: $\bar{A}\times \bar{B}\to \bar{B}, (a,b)\times (m,n)\mapsto (am,bn)$. Bracket is defined by $\{(m,n),(r,s)\}=(\{m,r\},\{n,s\})$. Then $\bar{B}\to B'$ is a Poisson homomorphism over $\bar{A}$ and $\bar{B}\to B''$ is a Poisson homomorphism over $\bar{A}$. The Poisson algebra satisfies an universal mapping property in the following sense: let $C$ be a Poisson algebra over $\bar{A}$ and assume that we have a Poisson $\bar{A}$-homomorphism $f:C\to B'$ and $g:C\to B''$ such that $p\circ f=q\circ g$. Then there is a unique Poisson homomorphism $h:C\to B'\times_B B''$ defined by $h(c)=(f(c),g(c))$ which is $\bar{A}$-Poisson homomorphism since $h(\{c_1,c_2\})=(f(\{c_1,c_2\}),g(\{c_1,c_2\})=(\{f(c_1),f(c_2)\},\{g(c_1),g(c_2)\})=\{(f(c_1),g(c_1)),(f(c_2),g(c_2))\}=\{h(c_1),h(c_2)\}$, and for $(a,b)\in \bar{A}$, we have $h((a,b)c)=(f((a,b)c),g((a,b)c))=(af(c),bg(c))=(a,b)(f(c),g(c))$

\begin{lemma}\label{3ll}
Let $A,A',A''$ be local artininan $k$-algebras, and let $\bar{A}=A'\times_A A''$. Let $B,B',B''$ be algebras over $A,A',A''$ with compatible maps $B'\to B$ and $B''\to B$, and assume that $B'\otimes_{A'} A\to B$ and $B''\otimes_{A''} A\to B$ are isomorphisms. Let $\bar{B}=B'\times_B B''$.
\begin{enumerate}
\item Assume $A''\to A$ is surjective. Then $\bar{B}\otimes_{\bar{A}} A'\to B'$ is an isomorphism and so $\bar{B}\otimes_{\bar{A}} A\to B$ is an isomorphism.
\item Now assume furthermore that $J=ker(A''\to A)$ is an ideal of square zeros and that $B',B''$ are flat over $A',A''$ respectively. Then $\bar{B}$ is flat over $\bar{A}$ and also $\bar{B}\otimes_{\bar{A}} A'' \to B''$ is an isomorphism.
\end{enumerate}
\end{lemma}
\begin{proof}
See \cite{Har10} Proposition 16.4.
\end{proof}

\begin{thm}\label{3thm1}
Let $(X,\Lambda_0)$ be a smooth projective Poisson scheme over $k$. Then the Poisson deformation functor $PDef_{(X,\Lambda_0)}$ of Poisson deformations of $(X,\Lambda_0)$ over local artinian rings has a miniversal family.

\end{thm}

\begin{proof}
We check Schlessinger's Criterion in Theorem \ref{sch1}. First $PDef_{(\mathcal{X},\Lambda_0)}(k)$ is one point set. So $(H_0)$ is satisfied. Now we prove $(H_2)$.  Let's consider the following commutative diagram of local artinian $k$-algebras with residue $k$. Assume that $A''\to A$ is a small extension.
\begin{center}
$\begin{CD}
\bar{A}=A'\times_A A''@>>> A'\\
@VVV @VVV\\
A''@>>> A
\end{CD}$
\end{center}
Let $(\mathcal{X},\Lambda)$ be an infinitesimal Poisson deformation of $(X,\Lambda_0)$ over $A$. Let $(\mathcal{X}',\Lambda')$ be an infinitesimal Poisson deformation of $(X,\Lambda_0)$ over $A'$ and $(\mathcal{X}'',\Lambda'')$ be an infinitesimal Poisson deformation of $(X,\Lambda_0)$ inducing $(\mathcal{X},\Lambda)$ via the above diagram. So we have the following fiber product of Poisson algebraic schemes.
\begin{center}
$\begin{CD}
(\mathcal{X},\Lambda)@>>> (\mathcal{X}',\Lambda')\\
@VVV @VVV\\
Spec(A)@>>> Spec(A')
\end{CD}$
\,\,\,\,\,\,\,\,\,
$\begin{CD}
(\mathcal{X},\Lambda)@>>> (\mathcal{X}'',\Lambda'')\\
@VVV @VVV\\
Spec(A)@>>> Spec(A')
\end{CD}$
\end{center}
Then we will define an infinitesimal  Poisson deformation $(\bar{\mathcal{X}},\bar{\Lambda})$ of $(X,\Lambda_0)$ inducing $(\mathcal{X},\Lambda)$ and $(\mathcal{X},\Lambda')$, which implies $(H_1)$. Since $Spec(A),Spec(A')$ and $Spec(A'')$ are one point sets, $\mathcal{X}$, $\mathcal{X}'$ and $\mathcal{X}''$ have the same topological space of $X$. For any open set $U\subset X$, $\mathcal{O}_\mathcal{X'}(U)\to \mathcal{O}_\mathcal{X}(U)$ is a Poisson $A'$-algebra homomoprhism and $\mathcal{O}_\mathcal{X''}(U)\to \mathcal{O}_\mathcal{X}(U)$ is a Poisson $A''$-algebra homomorphism. Choose an affine open cover $\{U_i=Spec(B_i)\}$ of $X$. Then $U_i$ are all affine open sets of $\mathcal{X},\mathcal{X}',\mathcal{X}''$.\footnote{Such open set exists: 
let $Z_0$ be a closed subscheme of a scheme $Z$, defined by a sheaf of nilpotent ideals $N\subset \mathcal{O}_Z$. If $Z_0$ is affine, then $Z$ is affine as well (See \cite{Ser06} Lemma 1.2.3 page 23)} For affine open set $U$ of $X$, let $\mathcal{O}_{\mathcal{X}}(U)=B, \mathcal{O}_{\mathcal{X}'}(U)=B'$ and $\mathcal{O}_{\mathcal{X}''}(U)=B''$. Then we have a Poisson homomorphism $p:B'\to B$ and $q:B''\to B$. Then $B'\otimes_{A'} A\to B$ is a Poisson isomorphism and $B''\otimes_{A''} A\to B$ is a Poisson isomorphism.  Let $B'\times_B B''$ which is a Poisson algebra over $\bar{A}$. Since $Spec(B),Spec(B')$ and $Spec(B'')$ has same topological spaces. $Spec(B)\to Spec(B')$ induced from $p:B'\to B$ is bijective and $Spec(B)\to Spec(B'')$ induced from $q:B''\to B$ is bijective. We would like to describe prime ideals of $\bar{B}=B'\times_B B''$. Since $A''\to A$ is surjective, by Lemma \ref{3ll}, $\varphi:\bar{B}\otimes_{\bar{A}} A'\to B'$ induced from $\bar{B}\to B'$ is an Poisson isomorphism. Then $Spec(B')$ is topologically homeomorphic to $Spec(\bar{B})$. Hence all prime ideal of $Spec(\bar{B})$ is induced from $Spec(B)$ from the map $\bar{B}\to B$. Let $\mathfrak{p}$ be the prime ideal of $\bar{B}$. Then $\mathfrak{p}=p'^{-1}q^{-1}(\mathfrak{q})$ for the unique prime ideal $\mathfrak{q}\in B$. Then we have $\bar{B}_{\mathfrak{p}}\cong B'_{p^{-1}(\mathfrak{q})}\times_{B_{\mathfrak{q}}} B''_{q^{-1}(\mathfrak{q})}$ which is a Poisson isomorphism induced from the natural fibered sum of Poisson algebras
\begin{center}
$\begin{CD}
\bar{B}_{\mathfrak{p}}@>>> B'_{p^{-1}(\mathfrak{q})}\\
@VVV @VVV\\
B''_{q^{-1}(\mathfrak{q})}@>>> B_{\mathfrak{q}}
\end{CD}$
\end{center}

Now we globalize our arguments. We would like to construct a Poisson scheme $\bar{\mathcal{X}}$ which is an infinitesimal Poisson deformation over $\bar{A}$. Our local model of $\mathcal{X}$ for affine open set $U$ of $X$ will be isomorphic to Poisson affine scheme $\mathcal{O}_{\mathcal{X}'}(U)\times_{\mathcal{O}_{\mathcal{X}}(U)} \mathcal{O}_{\mathcal{X}''}(U)$ with naturally induced Poisson structure. We assume that $\mathcal{X}\to \mathcal{X}'$ and $\mathcal{X}\to \mathcal{X}''$ are identity maps on $X$ when we ignore sheaf structures. Now we will define a Poisson sheaf $\mathcal{O}_{\bar{\mathcal{X}}}$ on $X$ which is a Poisson scheme $\bar{\mathcal{X}}$ over $\bar{A}$ in the following way: for open set $U$, $\mathcal{O}_{\bar{\mathcal{X}}}(U)$ is a set of elements $\phi:U\to \bigcup_{x\in U} \mathcal{O}_{\mathcal{X}',x}\times_{\mathcal{O}_{\mathcal{X},x}} \mathcal{O}_{\mathcal{X}'',x}$ such that for each $x\in X$, there exists an affine neighborhood $U_x$ of $x$ and an element $(a,b)\in \mathcal{O}_{\mathcal{X}''}(U_x)\times_{\mathcal{O}_{\mathcal{X}}(U_x)}\mathcal{O}_{\mathcal{X}''}(U_x)$ such that $\phi$ is canonically induced from $(a,b)$ from the natural Poisson map $\mathcal{O}_{\mathcal{X}'}(U_x)\times_{\mathcal{O}_{\mathcal{X}}(U_x)}\mathcal{O}_{\mathcal{X}''}(U_x)\to \mathcal{O}_{\mathcal{X}',x}\times_{\mathcal{O}_{\mathcal{X},x}} \mathcal{O}_{\mathcal{X}'',x}$. Then $\mathcal{O}_{\bar{\mathcal{X}}}$ is a Poisson sheaf where the Poisson structure is induced from the Poisson structure $\mathcal{O}_{\mathcal{X}',x}\times_{\mathcal{O}_{\mathcal{X},x}} \mathcal{O}_{\mathcal{X}'',x}$. Then $\mathcal{O}_{\bar{\mathcal{X}}}$ defines a Poisson scheme $\bar{\mathcal{X}}$. Indeed, for any affine open set $U$ of $X$, let $C=\mathcal{O}_{\mathcal{X}'}(U)\times_{\mathcal{O}_{\mathcal{X}'}(U)}\mathcal{O}_{\mathcal{X}''}(U)$. Then we have a natural sheaf homomorphism $\mathcal{O}_{Spec(C)}\to \mathcal{O}_{\bar{\mathcal{X}}}|_U$ which is a Poisson isomorphism since it is isomorphic at each stalk by the compatibility of localization as shown above.

We claim that $\bar{\mathcal{X}}\times_{Spec(\bar{A})} Spec(A')\cong \mathcal{X}'$ and $\bar{\mathcal{X}}\times_{Spec(\bar{A})}Spec(A'')\cong \mathcal{X}''$ as Poisson schemes. We simply note that $(\mathcal{O}_{\mathcal{X}',x}\times_{\mathcal{O}_{\mathcal{X},x}}\mathcal{O}_{\mathcal{X}'',x})\otimes_{\bar{A}} A'\cong \mathcal{O}_{\mathcal{X}',x}$ and $(\mathcal{O}_{\mathcal{X}',x}\times_{\mathcal{O}_{\mathcal{X},x}}\mathcal{O}_{\mathcal{X}'',x})\otimes_{\bar{A}} A''\cong \mathcal{O}_{\mathcal{X}'',x}$. $\bar{\mathcal{X}}$ is flat by Lemma \ref{3ll}.

Now we prove $(H_2)$. We assume that $A''=k[\epsilon]$ and $A=k$. Let $\mathcal{Y}$ be a infinitesimal Poisson deformation of $(X,\Lambda_0)$ over $\bar{A}$ inducing $(\mathcal{X},\Lambda')$ and $(\mathcal{X}'',\Lambda'')$, i.e $\mathcal{Y}\times_{Spec(\bar{A})}Spec(A')\cong \mathcal{X}'$ and $\mathcal{Y}\times_{Spec(\bar{A})} Spec(A'')\cong \mathcal{X}''$. We will construct an Poisson isomorphism $\bar{\mathcal{X}}\to \mathcal{Y}$. Equivalently, we will show an isomorphism of Poisson sheaves $\mathcal{O}_{\mathcal{Y}}\to \mathcal{O}_{\bar{\mathcal{X}}}$. Since $A'\to k$ and $k[\epsilon]\to k$ are surjective, $\mathcal{X}'\to \mathcal{Y}$ and $\mathcal{X}''\to \mathcal{Y}$ are closed immersions. Since $\mathcal{Y}$ is also an infinitesimal Poisson deformation of $\mathcal{X}=(X,\Lambda_0)$, we have the following commutative diagram.
\begin{center}
$\begin{CD}
\mathcal{Y}@<<< \mathcal{X}'\\
@AAA @AAA\\
\mathcal{X}'' @<<< \mathcal{X}= (X,\Lambda_0)
\end{CD}$
\end{center}
For each affine open set of $X$, we have the following commutative diagram of Poisson homomorphisms
\begin{center}
$\begin{CD}
\mathcal{O}_{\mathcal{Y}}(U)@>>> \mathcal{O}_{\mathcal{X}'}(U)\\
@VVV @VVV\\
\mathcal{O}_{\mathcal{X}''}(U) @>>> \mathcal{O}_{X}(U)
\end{CD}$
\end{center}
which is compatible with localization at each $x\in U$. So we have the following commutative diagram
\begin{center}
$\begin{CD}
\mathcal{O}_{\mathcal{Y}}(U)@>>> \mathcal{O}_{\mathcal{X}}(U)\times_{\mathcal{O}_X(U)}\mathcal{O}_{\mathcal{X}''}(U)\\
@VVV @VVV\\
\mathcal{O}_{\mathcal{Y},x} @>>> \mathcal{O}_{\mathcal{X}',x}\times_{\mathcal{O}_{\mathcal{X},x}}\mathcal{O}_{\mathcal{X}'',x}
\end{CD}$
\end{center}
This induces a natural Poisson homomrphism $\mathcal{O}_{\mathcal{Y},x}\to\mathcal{O}_{\mathcal{X}',x} \times_{\mathcal{O}_{\mathcal{X},x}}\mathcal{O}_{\mathcal{Y},x}$ which is necessarily isomorphism. Now we define a sheaf map. For each open set $U$, we can identity $\mathcal{O}_{\mathcal{Y}}(U)$ is the set of elements $\alpha:U\to \bigcup_{x\in U} \mathcal{O}_{\mathcal{Y},x}$, locally coming from an element of section of affine open set of $U$. Then the map $\mathcal{O}_{\mathcal{Y}}(U)\to \mathcal{O}_{\bar{\mathcal{X}}}(U)$ induced from $\mathcal{O}_{\mathcal{X}'',x}\to\mathcal{O}_{\mathcal{X}',x}\times_{\mathcal{O}_{\mathcal{X},x}}\mathcal{O}_{\mathcal{X}'',x}$ is well defined. $\mathcal{O}_{\mathcal{Y}}\to \mathcal{O}_{\mathcal{X}}$ is isomorphism since it is isomorphic at each stalk.

Lastly, since $(X,\Lambda_0)$ is a smooth projective Poisson scheme, $HP^2(X,\Lambda_0)$ is a finite dimensional $k$-vector space. We have $\kappa:t_{PDef_{(X,\Lambda_0)}}\cong HP^2(X,\Lambda_0)$ as a map in Proposition \ref{3q}. We show that $\kappa$ is isomorphism as $k$-vector spaces. We simply note that for given $\xi,\eta\in PDef(k[\epsilon])$ represented by $(Id+\epsilon p_{ij},\Lambda_0+t\Lambda_i')$ and $(Id+ \epsilon p_{ij}',\Lambda_0+t\Lambda_i'')$, $\xi+\eta$ is given by the data $(Id+\epsilon(p_{ij}+p_{ij}')\epsilon, \Lambda_0+\epsilon(\Lambda_i'+\Lambda_i''))$. So $t_{PDef_{(X,\Lambda_0)}}=PDef_{(X,\Lambda_0)}(k[\epsilon])$ is a finite dimensional $k$-vector space. Hence by Schlessinger's criterion (Theorem \ref{sch1}), $PDef_{(X,\Lambda_0)}$ has a miniversal family.

\end{proof}

\begin{lemma}\label{3lm}
Let $(X,\Lambda_0)$ be smooth projective Poisson scheme with $HP^1(X,\Lambda_0)=0$. Then for any infinitesimal Poisson deformation $(\mathcal{X},\Lambda)$ of $(X,\Lambda_0)$ over $A$ for any $A\in \bold{Art}$, we have $PAut((\mathcal{X},\Lambda)/(X,\Lambda_0))=Id$, where $PAut((\mathcal{X},\Lambda)/(X,\Lambda_0)):=$ the set of Poisson automorphisms of $(\mathcal{X},\Lambda)$ restricting to the identity Poisson automorphism of $(X,\Lambda_0)$.
\end{lemma}

\begin{proof}
We prove by induction on the dimension of $A$. Let $dim_k\,A=1$. Then $A=k$. So we have nothing to prove. Let's assume that the lemma holds for $A$ with $dim_k\,A\leq n-1$. Let $dim_k A=n$ and $(\mathcal{X},\Lambda)$ be an infinitesimal Poisson deformation of $(X,\Lambda_0)$ over $A$. Assume that the maximal ideal $\mathfrak{m}$ of $A$ satisfies $\mathfrak{m}^{p-1}\ne 0$ and $\mathfrak{m}^p=0$. Choose $t\ne 0\in \mathfrak{m}^{p-1}$. Then $A/(t) \in \bold{Art}$ with $dim_k\,A/(t)\leq n-1$ and $0\to (t)\to A\to A/(t)\to 0$ is a small extension. Now let $g:(\mathcal{X},\Lambda)\to (\mathcal{X},\Lambda)$ be a Poisson automorphism restricting to the identity Poisson automorphism of $(X,\Lambda_0)$. Let $\{U_i=Spec(B_i)\}$ be an affine open cover of $X$. Let $\{\theta_i\}$ where $\theta_i:U_i\times Spec({A})\to \mathcal{X}|_{U_i}$, $\{\Lambda_i\}$ where $\Lambda_i$ is the Poisson structure on $U_i\times Spec(A)$ induced from $\Lambda|_{U_i}$ via $\theta$ and let $\theta_{ij}={\theta}_i^{-1}{\theta}_j$ which corresponds to a Poisson homomorphism ${f}_{ij}:(\mathcal{O}_X(U_{ij})\otimes_k {A},{\Lambda}_i)\to (\mathcal{O}_X(U_{ij})\otimes_k {A}, {\Lambda}_j)$. Then $f$ can be described by the data $g_{i}: (\mathcal{O}_X(U_i)\otimes A,\Lambda_i)  \to (\mathcal{O}_X(U_i)\otimes A,\Lambda_i)$ which is a Poisson automorphism and the following commutative diagram
\begin{center}
$\begin{CD}
(\mathcal{O}_X(U_{ij}) \otimes A ,\Lambda_j)@>g_j>> (\mathcal{O}_X(U_{ij})\otimes A,\Lambda_j)\\
@Af_{ij}AA @AAf_{ij}A\\
(\mathcal{O}_X(U_{ij})\otimes A,\Lambda_i)@>g_i>> (\mathcal{O}_X(U_{ij})\otimes A,\Lambda_i)
\end{CD}$
\end{center}
Since by the induction hypothesis $g_i$ induce the identity on $\mathcal{O}_X(U_i)\otimes A/(t)$. $g_i$ is of the form $g_i=Id+td_ix$, where $d_i\in Der_k(\mathcal{O}_X(U_i),\mathcal{O}_X(U_i))$ with $[\Lambda_0,d_i]=0$ by Lemma \ref{3l}. Hence $\{d_i\}\in HP^1(X,\Lambda_0)$. Since $HP^1(X,\Lambda_0)=0$, we have $d_i=0$. Hence $g_i$ is the identity. So $g$ is the identity. This proves the lemma.
\end{proof}

\begin{proposition}
Let $(X,\Lambda_0)$ be a projective smooth Poisson scheme with $HP^1(X,\Lambda_0)=0$. Then the functor $PDef_{(X,\Lambda_0)}$ is pro-representable.
\end{proposition}


\begin{proof}
Since $PDef_{(X_0,\Lambda_0)}$ satisfies $(H_0),(H_1),(H_2),(H_3)$ by Theorem \ref{3thm1}, we will check $(H_4)$ to show that the Poisson deformation functor $PDef_{(X,\Lambda_0)}$ is pro-representable by Theorem \ref{sch2}. Let $0\to (t)\to \tilde{A}\xrightarrow{\mu} A\to 0$ be a small extension in $\bold{Art}$. Let $\xi=(\mathcal{X},\Lambda)\in PDef_{(X,\Lambda_0)}(A)$ be an infinitesimal Poisson deformation of $(X,\Lambda_0)$ over $A$. Let $p:=PDef_{(X,\Lambda_0)}(\mu):PDef_{(X,\Lambda_0)}(\tilde{A})\to PDef_{(X,\Lambda_0)}(A)$ induced from $\mu:\tilde{A}\to A$. We will define a map $G:HP^2(X,\Lambda_0)\times p^{-1}(\xi)\to p^{-1}(\xi)$ which is  a group action of $HP^2(X,\Lambda_0)$ acting on $p^{-1}(\xi)$. Let $\tilde{\xi}=(\tilde{\mathcal{X}},\tilde{\Lambda})\in p^{-1}(\xi)$ be a lifting of $\xi$ over $\tilde{A}$. Let $\{U_i=Spec(B_i)\}$ be an affine cover of ${\xi}=(\tilde{\mathcal{X}},\tilde{\Lambda})$. Let $\{\tilde{\theta}_i\}$ where $\tilde{\theta}_i:U_i\times Spec(\tilde{A})\to \tilde{\mathcal{X}}|_{U_i}$, $\{\tilde{\Lambda}_i\}$ where $\tilde{\Lambda}_i$ is the Poisson structure on $U_i\times Spec(\tilde{A})$ induced from $\tilde{\Lambda}|_{U_i}$ via $\tilde{\theta}_i$ and let $\tilde{\theta}_{ij}=\tilde{\theta}_i^{-1}\tilde{\theta}_j$ which corresponds to a Poisson homomorphism $\tilde{f}_{ij}:(\mathcal{O}_X(U_{ij})\otimes_k \tilde{A},\tilde{\Lambda}_i)\to (\mathcal{O}_X(U_{ij})\otimes_k \tilde{A},\tilde{\Lambda}_j)$. We note that $t\tilde{f}_{ij}:\mathcal{O}_X(U_{ij})\otimes\tilde{A}\to \mathcal{O}_X(U_{ij})\otimes \tilde{A}$ is same to $tId$ since $\tilde{f}_{ij}$ is identity map modulo by maximal ideal $\tilde{\mathfrak{m}}$ of $\tilde{A}$ (i.e $\tilde{f}_{ij}(x)-Id(x)\in \tilde{\mathfrak{m}}$), and the maximal ideal $\tilde{\mathfrak{m}}$ is killed by $(t)$, and so we have $t(f_{ij}-Id)(x)=0$.  Let $(\{\Lambda_i'\}, \{p_{ij}\})\in HP^2(X,\Lambda_0)$, where $\Lambda_i'\in Hom_{B_i}(\wedge^2 \Omega_{B_i/k},B_i)$ and $p_{ij}\in \Gamma(U_i,T_X)=Der_k(B_i,B_i)$. So we have $[\Lambda_0,\Lambda_i']=0$, $\Lambda_j'-\Lambda_i'+[\Lambda_0,p_{ij}]=0$ and $p_{ij}+p_{jk}-p_{ik}=0$. Then we will define another lifting $\bar{\xi}$ of $\xi$ from $(\{\Lambda_i'\}, \{ p_{ij}\})$ and $\tilde{\xi}$ by gluing $U_i\times Spec(\tilde{A})$ equipped with $\tilde{\Lambda}_i+t\Lambda_i'$ in the following way: $\tilde{f}_{ij}+tp_{ij}:\mathcal{O}_X(U_{ij})\otimes_k \tilde{A} \to \mathcal{O}_X(U_{ij})\otimes_k \tilde{A}$ which is an isomorphism with the inverse $\tilde{f}_{ij}-tp_{ij}$. Indeed, $(\tilde{f}_{ji}-tp_{ij})\circ (\tilde{f}_{ij}+tp_{ij}):\mathcal{O}_X(U_i)\otimes_k \tilde{A}\to \mathcal{O}_X(U_i)\otimes_k\tilde{A}$, $Id+t(\tilde{f}_{ji}p_{ij}-p_{ij}\tilde{f}_{ij})=Id+t(p_{ij}-p_{ij})=Id$. $\{\tilde{f}_{ji}+tp_{ij}\}$ satisfies cocylce condition:
\begin{align*}
(\tilde{f}_{ki}+tp_{ki})\circ (\tilde{f}_{jk}+tp_{jk})\circ (\tilde{f}_{ij}+tp_{ij})=\tilde{f}_{ki}\circ\tilde{f}_{jk}\circ\tilde{f}_{ij}+t(p_{ij}+p_{jk}+p_{ki})=Id\\
\end{align*}
Now let's consider the Poisson structures. Since $[\Lambda_0,\Lambda_i']=0$, we have $[\tilde{\Lambda}_i+t\Lambda_i',\tilde{\Lambda}_i+t\Lambda_i']=0$. Hence $\tilde{\Lambda}_i+t\Lambda_i'$ defines a Poisson structure on $\mathcal{O}_X(U_i)\otimes_k \tilde{A}$. We claim that $\tilde{f}_{ij}+tp_{ij}:(\mathcal{O}_X(U_{ij})\otimes_k \tilde{A},\tilde{\Lambda}_i+t\Lambda_i')\to \mathcal{O}_X(U_{ij})\otimes_k \tilde{A},\tilde{\Lambda}_j+t\Lambda_j')$ is a Poisson isomorphism. First we note that
\begin{align*}
&(\Lambda_j'-\Lambda_i'+[\Lambda_0,p_{ij}])(dx\wedge dy)\\
=&\Lambda_{j}'(dx\wedge dy)-\Lambda_{i}'(dx\wedge dy)+\Lambda_0(d(p_{ij}(x))\wedge dy)-\Lambda_0(d(p_{ij}(y)))\wedge dx)-p_{ij}(\Lambda_0(dx\wedge dy)))=0
\end{align*}
\begin{align*}
&(\tilde{f}_{ij}+tp_{ij})((\tilde{\Lambda}_i+t\Lambda_i')(dx\wedge dy))=(\tilde{f}_{ij}+tp_{ij})(\tilde{\Lambda}_i(dx\wedge dy)+t\Lambda_i'(dx\wedge dy))\\
&=\tilde{f}_{ij}(\tilde{\Lambda}_i(dx\wedge dy))+t\Lambda_i'(dx\wedge dy) +tp_{ij}(\Lambda_0(dx\wedge dy))=\tilde{\Lambda}_j(d\tilde{f}_{ij}(x)\wedge d\tilde{f}_{ij}(y))
+t\Lambda_{j}'(dx\wedge dy)\\&+t\Lambda_0(d(p_{ij}(x))\wedge dy)-t\Lambda_0(d(p_{ij}(y)))\wedge dx)-tp_{ij}(\Lambda_0(dx\wedge dy))+tp_{ij}(\Lambda_0(dx\wedge dy))\\
&=\tilde{\Lambda}_j(d\tilde{f}_{ij}(x)\wedge d\tilde{f}_{ij}(y))+t\Lambda_{j}'(dx\wedge dy)+t\Lambda_0(d(p_{ij}(x))\wedge dy)-t\Lambda_0(d(p_{ij}(y)))\wedge dx)
\end{align*}
On the other hand,
\begin{align*}
&(\tilde{\Lambda}_j+t\Lambda_j')(d(\tilde{f}_{ij}+tp_{ij})(x))\wedge d((\tilde{f}_{ij}+tp_{ij})(y))\\
=& \tilde{\Lambda}_j(d(\tilde{f}_{ij}+tp_{ij})(x))\wedge d((\tilde{f}_{ij}+tp_{ij})(y))
+t\Lambda_j'(d(\tilde{f}_{ij}(x))\wedge d(\tilde{f}_{ij}(y))\\=& \tilde{\Lambda}_j(d\tilde{f}_{ij}(x)\wedge d\tilde{f}_{ij}(y))+t\Lambda_0(d(p_{ij}(x))\wedge dy)+t\Lambda_0(dx\wedge d(p_{ij}(y)))+t\Lambda_j'(dx\wedge dy)
\end{align*}
Hence we can define another lifting $\bar{\xi}$ from $\tilde{\xi}$ and ($\{\Lambda_i'\},\{p_{ij}\}$)$\in HP^2(X,\Lambda_0)$. Now we show that the map 
\begin{align*}
G:HP^2(X,\Lambda_0)\times p^{-1}(\xi)&\to p^{-1}(\xi)\\
((\{\Lambda_i'\},\{p_{ij}\}),\tilde{\xi})&\mapsto \bar{\xi}
\end{align*}
 is well-defined. Let $(\{\Lambda_i''\},\{p_{ij}'\})$ represent the same cohomology class with $(\{\Lambda_i\},\{p_{ij}\})$ in $HP^2(X,\Lambda_0)$. Then we have to show that the lifting $\bar{\xi}$ defined by $(\{\Lambda_i\},\{p_{ij}\})$ is an equivalent infinitesimal Poisson deformation to the lifting $\bar{\xi}'$ defined by $(\{\Lambda_i\},\{p_{ij}'\})$. There exists $\{a_i\}$ where $a_i\in \Gamma(U_i, T_X)$ such that $[\Lambda_0,a_i]=\Lambda_i'-\Lambda_i''$ and $-\delta(\{a_i\})=a_i-a_j=p_{ij}-p_{ij}'$. To show $\bar{\xi}\cong \bar{\xi}'$, it is sufficient to show that the following diagram commutes and $Id+ta_i$ defines Poisson isomorphisms.
 
\begin{center}
$\begin{CD}
(\mathcal{O}_X(U_{ij})\otimes_k Spec(\tilde{A}), \tilde{\Lambda}_j+t\Lambda_j' )@> Id+ta_j >> (\mathcal{O}_X(U_{ij})\otimes_k \tilde{A},\tilde{\Lambda}_j+t\Lambda_j'')\\
@A \tilde{f}_{ij}+tp_{ij} AA@AA\tilde{f}_{ij}+tp_{ij}' A\\
(\mathcal{O}_X(U_{ij})\otimes Spec(\tilde{A}), \tilde{\Lambda}_i+t\Lambda_i')@> Id+ta_i >> (\mathcal{O}_X(U_{ij})\otimes_k \tilde{A},\tilde{\Lambda}_i+t\Lambda_i'')
\end{CD}$
\end{center}
Indeed, $(Id+ta_j)\circ (\tilde{f}_{ij}+tp_{ij})-(\tilde{f}_{ij}+tp_{ij}')\circ (Id+ta_i)=\tilde{f}_{ij}+tp_{ij}+ta_j-(\tilde{f}_{ij}+ta_i+tp_{ij}')=t(a_j-a_i+p_{ij}-p_{ij}')=0$. $\tilde{\Lambda}_i+t\Lambda_i'-(\tilde{\Lambda}_i+t\Lambda_i'')=t(\Lambda_i-\Lambda_i'')$ and $\Lambda_i-\Lambda_i''-[\Lambda_0,a_i]=0$. So by Lemma \ref{3l}, $\{Id+ta_i\}$ are Poisson isomorphisms.

Next,  we show that the group action $HP^2(X,\Lambda_0)\times p^{-1}(\xi)\to p^{-1}(\xi)$ is transitive. Let $\tilde{\xi}\in p^{-1}(\xi)$ be a lifting of $\xi$ as above. Choose arbitrary lifting $\eta\in p^{-1}(\xi)$ of $\xi$. We have to find $(\{\Lambda_i'\},\{p_{ij}\})\in HP^2(X,\Lambda_0)$ such that $((\{\Lambda_i'\},\{p_{ij}\}),\tilde{\xi})$ is mapped to $\eta$ under the action. Let $\eta$ consist of the data $f_{ij}':(\mathcal{O}_X(U_{ij}),\otimes \tilde{A},\tilde{\Lambda}_i')\to (\mathcal{O}_X(U_{ij})\otimes \tilde{A},\tilde{\Lambda}_j')$. Since $\tilde{\xi}$ and $\eta$ both induce $\xi$, we have $\tilde{f}_{ij}'=f_{ij}+tp_{ij}$ and $\tilde{\Lambda}'_i=\tilde{\Lambda}_i+t\Lambda_i'$ for some $p_{ij}\in \Gamma(U_{ij}, T_X)$ and $\Lambda_i'\in Hom_{B_0}(\wedge^2 \Omega_{B_i/k}^1,B_i)$. Then $(\{\Lambda_i'\},\{p_{ij}\})$ defines a cohomology class in $HP^2(X,\Lambda_0)$. So the group action is transitive.

Next, we show that the group action is free. Assume that for given $v=(\{\Lambda_i'\},\{p_{ij}\})\in HP^2(X,\Lambda_0)$, we have $G(v,\tilde{\xi})=\tilde{\xi}$. We have to show that $v=0 \in HP^2(X,\Lambda_0)$. Let $\tilde{\xi}$ and $G(v,\tilde{\xi})$ as above. Then $G(v,\tilde{\xi})=\tilde{\xi}$ implies that we have Poisson isomorphisms $g_i:(\mathcal{O}_X(U_i)\otimes_k \tilde{A},\tilde{\Lambda}_i) \to (\mathcal{O}_X(U_i)\otimes_k \tilde{A},\tilde{\Lambda}_i+t\Lambda_i')$ such that the following diagram is commutative.
\begin{center}
$\begin{CD}
(\mathcal{O}_X(U_{ij})\otimes_k \tilde{A},\tilde{\Lambda}_j)@>g_j>> (\mathcal{O}_X(U_{ij})\otimes_k \tilde{A},\tilde{\Lambda}_j+t\Lambda_j')\\
@A\tilde{f}_{ij}AA @AA\tilde{f}_{ij}+tp_{ij}A\\
(\mathcal{O}_X(U_{ij})\otimes_k \tilde{A},\tilde{\Lambda}_i)@>g_i>> (\mathcal{O}_X(U_{ij})\otimes_k \tilde{A},\tilde{\Lambda}_i+t\Lambda_i')
\end{CD}$
\end{center}
Since $g_i$ induces a Poisson automorphism on $\xi=(\mathcal{X},\Lambda)$ by modulo $(t)$, which is necessarily identity by Lemma \ref{3lm}.  Hence $g_i$ is the form of $g_i=Id+tq_i$, where $q_i\in Der_k(\mathcal{O}_X(U_i),\mathcal{O}_X(U_i))$ with $-\Lambda_i'-[\Lambda_0,q_i]=0$ by Lemma \ref{3l}. Since the diagram commutes, we have $0=g_j\tilde{f}_{ij}-(\tilde{f}_{ij}+tp_{ij})g_i=(Id+tq_j)\tilde{f}_{ij}-(\tilde{f}_{ij}+tp_{ij})(Id+tq_i)=t(q_j-p_{ij}-q_i)$. Hence $p_{ij}=q_j-q_i$. Hence $v=0 \in HP^2(X,\Lambda_0)$.

Lastly, we claim that $G$ is exactly same to $t_{PDef_{(\mathcal{X},\Lambda_0)}}\times p^{-1}(\xi)\to p^{-1}(\xi)$ that we defined in section 8.1. We set $F:=PDef_{(X,\Lambda_0)}$. Let's describe $t_F\times p^{-1}(\xi)\to p^{-1}(\xi)$ by following the definition in section 8.1. We note that here $A'=\tilde{A}$.
We now describe $\alpha^{-1}:t_F \times F(\tilde{A})\to F(k[\epsilon]\times_k \tilde{A})$. Given an element $v=(\{\Lambda_i'\},\{p_{ij}\})\in HP^2(X,\Lambda_0)$ gives a first order Poisson deformation $(\mathcal{X}_\epsilon,\Lambda_\epsilon)$ over $k[\epsilon]$ by Proposition \ref{3q}. Let $(\mathcal{X},\Lambda)\in F(\tilde{A})$ which is a lifting of $\xi \in F(A)$. Let $\{U_i\}$ be an affine open set of $X$. $(\mathcal{X}_{\epsilon},\Lambda_\epsilon)$ can be described by the data $(\{Id+\epsilon p_{ij}\},\{\Lambda_0+\epsilon\Lambda_i'\})$, where $\Lambda_0+\epsilon\Lambda_i'$ is a Poisson structure on $\mathcal{O}_X(U_i)\otimes_k k[\epsilon]$ and $Id+ \epsilon p_{ij}:\mathcal{O}_X(U_{ij})\otimes_k k[\epsilon]\to \mathcal{O}_X(U_{ij})\otimes_k k[\epsilon]$ an Poisson isomorphism. Similarly $(\mathcal{X},\Lambda)$ can be described by the data $(\{f_{ij}\},\{\Lambda_i\})$, where $\Lambda_i$ is a Poisson structure on $\mathcal{O}_X(U_i)\otimes_k \tilde{A}$ and $f_{ij}:\mathcal{O}_{X}(U_{ij})\otimes \tilde{A} \to \mathcal{O}_{X}(U_{ij})\otimes \tilde{A}$ an Poisson isomorphism. Then $\alpha^{-1}((\mathcal{X}_\epsilon,\Lambda_\epsilon), (\mathcal{X},\Lambda))$ is a fibered sum $\mathcal{X}_\epsilon\times_k \mathcal{X}$ which can be described as $(\{(Id+\epsilon p_{ij}, f_{ij})\},\{(\Lambda_0+\epsilon \Lambda_i',\Lambda_i)\})$.
We note that 
\begin{center}
$\begin{CD}
k[\epsilon]\otimes_k \tilde{A},(x+y\epsilon,a') @>>> \tilde{A}\times_A \tilde{A}, (a'+yt,a')@>>> \tilde{A},(a'+yt)\\
@.@VVV \\
@. \tilde{A},(a')
\end{CD}$
\end{center}
In the map $t_F\times F(\tilde{A})\to F(\tilde{A})\times_{F(A)} F(\tilde{A}), (v,\xi)\mapsto(\tau(v,\xi),\xi)$, $\tau(v, (\mathcal{X},\Lambda))=(\mathcal{X}_{\epsilon}\times_k \mathcal{X})\otimes_{k[\epsilon]\times_k \tilde{A}} \tilde{A}$ is induced from $k[\epsilon]\times_k \tilde{A}\to \tilde{A}, (x+y\epsilon,a')\mapsto (a'+yt)$. Hence $(\mathcal{X}_{\epsilon}\times_k \mathcal{X})\otimes_{k[\epsilon]\times_k \tilde{A}} \tilde{A}$ can be described as $(f_{ij}+tp_{ij}, \Lambda_i+t\Lambda_i')$ which is exactly same to $G(v,(\mathcal{X},\Lambda))$.

\end{proof}

\chapter{Poisson cotangent complex}\label{chapter9}

In this chapter, we extend the construction of Schlessinger and Lichtenbaum's cotangent complex (See \cite{Sch67}) in terms of Poisson algebras.\footnote{My original goal was to generalize Hartshorne's construction of `$T^i$ Functors' presented in \cite{Har10} page 18-25 and apply to deformation problems for not necessarily smooth Poisson schemes. Hartshorne's book \cite{Har10} led me to Lichtenbaum and Schlessinger's original paper \cite{Sch67}. I followed Shclessinger and Lichtenbaum \cite{Sch67} in the context of Poisson algebras in this chapter. However I could not succeed in globalization. See Remark \ref{3remarks}. There is also a general approach to cotangent complex for algebras over an operad (see \cite{Lod12}). If our construction turns out to be correct, I expect that our construction $PT^i$ for $i=0,1,2$ is equivalent to the general construction in the language of operads.} We follow their arguments in  Poisson context. In this chapter, every algebra is a $k$-algebra for a field $k$.

\section{Poisson modules and Poisson enveloping algebras}(See \cite{Fre06})
\begin{definition}
Let $A$ be a Poisson algebra. A Poisson module over $A$  is a $A$-module $M$  equipped with a bracket $\{-,-\}:A\otimes_k M\to M$ such that
\begin{enumerate}
\item $\{a,bm\}=\{a,b\}+b\{a,m\}$ 
\item $\{ab,m\}=a\{b,m\}+b\{a,m\}$
\item $\{a,\{b,m\}\}=\{\{a,b\},m\}+\{b,\{a,m\}\}$
\end{enumerate}
for $a,b\in A$ and $m\in M$.
\end{definition}

\begin{definition}
Let $M$ and $N$ be Poisson modules of a Poisson algebra $A$. A map $\phi:M\to N$ is a morphism of Poisson modules if
\begin{align*}
 \phi(am)&=a\phi(m)\\
 \phi(\{a,x\})&=\{a,\phi(m)\}
\end{align*}
for $a\in A$ and $m\in M$. We denote by $Hom_{\mathcal{U}_{Pois(A)}}(M,N)$ or $PHom_A(M,N)$ the $k$-module of morphisms of Poisson modules from $M$ to $N$.
\end{definition}

We will construct the Poisson enveloping algebra $\mathcal{U}_{Pois}(A)$ of a Poisson $k$-algebra $A$. This is a associative $k$-algebra which is characterized by the following property: ``The category of left $\mathcal{U}_{Pois(A)}$-modules is equivalent to the category of Poisson modules over $A$" (\cite{Fre06}).
\begin{definition}
The Poisson enveloping algebra $\mathcal{U}_{Pois(A)}$ is the associated $A$-algebra with unit generated by the symbols $X_a,a\in A$ by the quotient of the following relations
\begin{enumerate}
\item $X_a\cdot b=\{a,b\}+b\cdot X_a$
\item $X_{ab}=a\cdot X_b+b\cdot X_a$
\item $X_a\cdot X_b=X_{\{a,b\}}+X_b\cdot X_a$
\end{enumerate}
for $a,b\in A$
\end{definition}

Let $M$ be a Poisson module over $A$ (so $M$ is a left $\mathcal{U}_{Pois(A)}$-module), then $M$ is also a right $\mathcal{U}_{Pois(A)}$-module defined by
\begin{align*}
m\cdot a:=a\cdot m,\,\,\,\,\,\,\,\,\, m\cdot X_a=-\{a,m\}
\end{align*}

So given a Poisson module $M$ over $A$, by abuse of notation, we define $\{m,a\}:=m\cdot X_a$. Then in practice, we treat the bracket $\{-,-\}$ on $M$ like a bracket of a Poisson algebra.

\begin{definition}
Let $S\to A$ be a homomorphism of Poisson algebras. A Poisson $S$-derivation is a map $d:A\to M$ from a Poisson $k$-algebra $A$ to a Poisson $A$-module $M$ which is a $S$-linear derivation with respect to multiplication and with respect to the bracket $\{-,-\}$ of $A$. Explicitly, a Poisson $S$-derivation $d$ satisfies the following identites
\begin{align*}
d(sf)&=sd(f)\\
d(fg)&=df\cdot g+f\cdot dg\\
d\{f,g\}&=\{df,g\}+\{f,dg\}=-\{g,df\}+\{f,dg\}
\end{align*}
for $s\in S$ and $f,g\in A$.
We denote $PDer_S(A,M)$ the $A$-module of Poisson $S$-derivations from $A\to M$. The functor $PDer_S(A,-)$ is representable : there exists a Poisson representation denoted by $\Omega_{Pois(A)/S}^1$ such that
\begin{align*}
PDer_S(A,M)=Hom_{\mathcal{U}_{Pois(A)}}(\Omega_{\mathcal{U}_{Pois(A)/S}}^1,M)
\end{align*}
To construct $\Omega_{\mathcal{U}_{Pois(A)/S}}^1$, let $F$ be the free left  $\mathcal{U}_{Pois}(A)$-module generated by the symbols $df,f\in A$. Let $E$ be the $\mathcal{U}_{Pois(A)}$-submodule of $F$ generated by the elements of the form $ds,s\in S$, $d(f+g)-df-dg, d(fg)=df\cdot g+f\cdot dg$ and $d\{f,g\}=-X_g\cdot df +X_f\cdot dg$. $\Omega_{\mathcal{U}_{Pois(A)/S}}^1$ is defined to be $E/F$. We have a natural map $d:A\to \Omega_{\mathcal{U}_{Pois(A)/S}}^1$. Via this map, the functor $PDer_S(A,-)$ is represented by $\Omega_{Pois(A)/S}^1$.
\end{definition}

Let $\rho:B\to C$ be a Poisson homomorphism compatible with $A$. In other words, $A\to B\to C$ is a Poisson homomorphism.  Then there exist canonical homomorphism of $\mathcal{U}_{Pois(C)}$-modules
\begin{align*}
\alpha&:\mathcal{U}_{Pois(C)}\otimes_{\mathcal{U}_{Pois(B)}}\Omega_{Pois(B)/A}^1 \to \Omega_{Pois(C)/A}^1
\end{align*}

\begin{proposition}
Let $A\to B$ be a homomorphism of  Poisson $k$-algebras. Let $I$ be a Poisson ideal of $B$.

If $C=B/I$, we have an exact sequence of $\mathcal{U}_{Pois(C)}$-modules
\begin{equation}\label{3equ}
I/(I^2\oplus \{I,I\})\xrightarrow{\delta} \mathcal{U}_{Pois(C)}\otimes_{\mathcal{U}_{Pois(B)}} \Omega_{Pois(B)/A}^1\to \Omega_{{Pois(C)}/A}^1\to 0
\end{equation}
where for any $i\in I$, let $\bar{i}$ be the image of $i$ in $I/(I^2\oplus \{I,I\})$. We define $\delta(\bar{i})=1\otimes di$.
\end{proposition}

\begin{proof}

We show that $I/(I^2\oplus \{I,I\})$ is a Poisson $B/I$-module, defined by $\{\bar{b},\bar{i}\}=\overline{\{b,i\}}$. Let $b_1,b_2\in B$ with $b_1-b_2\in I$ and  $i_1,i_2\in I$ with $i_1-i_2\in I^2\oplus \{I,I\}$. Then $\{b_1,i_1\}-\{b_2,i_2\}=\{b_1,i_1-i_2\}+\{b_1-b_2,i_2\}\in I^2\oplus\{I,I\}$. Hence the bracket is well-defined. We claim that 
\begin{align*}
I/(I^2\oplus \{I,I\}) \cong\mathcal{U}_{Pois(C)}\otimes_{\mathcal{U}_{Pois(B)}} I
\end{align*}
as $\mathcal{U}_{Pois(C)}$-modules. We define a map $\varphi: I/I^2\oplus \{I,I\}\to \mathcal{U}_{Pois(C)}\otimes \mathcal{U}_{Pois(B)} I$ where $\varphi(\bar{i})=1\otimes i$. First we show that this map is well-defined. Let $i,j\in I$ with $i-j\in I^2\oplus \{I,I\}$. Then $i-j=\sum_n f_n\cdot g_n$, where $f_n\in I$ or $f_n$ is of the from $b\cdot X_a$ where $a\in I, b\in \mathcal{U}_{Pois(B)}$ ,and $g_n\in I$ as $\mathcal{U}_{Pois(B)}$-module action. Since $\bar{f_n}=0$ for $f_n\in I$ and $X_{\bar{a}}=0$ for $a\in I$, we have $1\otimes (i-j)=0$. Hence $\varphi$ is well-defined. Second, we show that $\varphi$ is a $\mathcal{U}_{Pois(C)}$-module homomorphism. Let $c\in \mathcal{U}_{Pois(C)}$. Then there exists  $b\in \mathcal{U}_{Pois(B)}$ with $\bar{b}=c$. Then $\varphi(c\cdot \bar{i})=\varphi(\bar{b}\cdot \bar{i})=\varphi(\overline{b\cdot i})=1\otimes b\cdot i=\bar{b}\otimes i=\bar{b}(1\otimes i)=\bar{b}\varphi(\bar{i})=c\varphi(\bar{i})$. Lastly we can define an inverse map $\varphi^{-1}:\mathcal{U}_{Pois(C)}\otimes_{\mathcal{U}_{Pois(B)}} I\to I/(I^2\oplus \{I,I\})$ by $\varphi^{-1}(c\otimes i)=c\cdot \bar{i}$. 

Now we prove the proposition. The sequence in (\ref{3equ})  is equivalent to 
\begin{align*}
\mathcal{U}_{Pois(C)}\otimes_{\mathcal{U}_{Pois(B)}} I\xrightarrow{\delta} \mathcal{U}_{Pois(C)}\otimes_{\mathcal{U}_{Pois(B)}} \Omega_{Pois(B)/A}^1\xrightarrow{\alpha} \Omega_{{Pois(C)}/A}^1\to 0
\end{align*}
where $\delta(\bar{b}\otimes i):=\bar{b}\otimes di$. Since $B\to B/I$ is surjective, $\alpha$ is surjective. Exactness of the sequence in (\ref{3equ}) is equivalent to the exactness of the following sequence
\begin{align*}
0\to Hom_{\mathcal{U}_{Pois(C)}}(\Omega_{Pois(C)/A}^1,N)\to Hom_{\mathcal{U}_{Pois(C)}}(\mathcal{U}_{Pois(C)}\otimes_{\mathcal{U}_{Pois(B)}} \Omega_{Pois(B)/A}^1,N)\\
\to Hom_{\mathcal{U}_{Pois(C)}}(I/(I^2\oplus \{I,I\}),N)
\end{align*}
for any left $\mathcal{U}_{Pois(C)}$-module $N$. Equivalently we have to show that the following sequence is exact.
\begin{align*}
0\to PDer_A(C,N)\to PDer_A(B,N) \xrightarrow{\phi} Hom_{\mathcal{U}_{Pois(B)}}(I,N)
\end{align*}
where for $D\in PDer_A(B,N)$, $\phi(D)$ is defined by restricting $D$ to $I$. Now we assume that $\phi(D)$ is zero, then $D(I)=0$ so $D$ factor through $B/I$, hence $D\in  PDer_A(C,N)$. This proves the proposition.

\end{proof}

\section{Poisson cotangent complex}

\begin{definition}[See $\cite{Umi12}$]
We construct a free Poisson algebra generated by $\{x_i\}$ over $k$. Let $g$ be a free Lie algebra generated by $\{x_i\}$ with the Lie bracket $[-,-]$. Free Poisson algebra generated by $\{x_i\}$ with the Poisson bracket $\{-,-\}$ is the polynomial algebra $k[g]$ with the Poisson bracket defined by $\{x_i,x_j\}:=[x_i,x_j]$. We will denote the free Poisson algebra by $k\{x_i\}$.
\end{definition}

\begin{definition}
Let $A$ be a Poisson algebra with the bracket $\{-,-\}_A$ over $k$. Let's consider the free Poisson algebra $P$ generated by $A$ and $\{x_i\}$. We denote by $A\{x_i\}$ the quotient of $P$ by  Poisson ideals generated by $\{a,b\}-\{a,b\}_A$ and $a\cdot b= a*b$ for $a,b\in A$ where $\cdot$ is the multiplication in $P$ and $*$ is the multiplication in $A$. We call $A\{x_i\}$ a free Poisson algebra over $A$ generated by $\{x_i\}$.  We also note that $\Omega_{\mathcal{U}_{Pois(A\{x_i\})/A}}^1$ is a free $\mathcal{U}_{Pois(A\{x_i\})}$-module generated by $dx_i$.
\end{definition}

\begin{remark}
We have the following universal mapping property: let $X=\{x_i\}$ and $A\to B$ be a Poisson homomorphism. let $j:X\to B$ be a map. Then there exists an unique Poisson homomorphism $A\{x_i\}\to B$ such that the following diagram commutes
\begin{center}
\[\begindc{\commdiag}[70]
\obj(0,1)[ee]{$A$}
\obj(1,1)[aa]{$A\{x_i\}$}
\obj(2,1)[bb]{$B$}
\obj(2,0)[cc]{$X$}
\mor{ee}{aa}{}
\mor{aa}{bb}{$u$}
\mor{cc}{aa}{}
\mor{cc}{bb}{$j$}
\enddc\]
\end{center}
\end{remark}

\begin{definition}
Let $A\to B$ be a Poisson homomorphism. By a Poisson extension of $B$ over $A$ we mean an exact sequence
\begin{align*}
(\mathscr{E}):0\to E_2\xrightarrow{e_2} E_1\xrightarrow{e_1}R\xrightarrow{e_0}B\to 0
\end{align*}
where $R$ is a Poisson algebra and $e_0:R\to B$ is a Poisson homomorphism such that $A\to B$ factor through $e_0:R\to B$, $E_1$ and $E_2$ are $\mathcal{U}_{Pois(R)}$-modules, $e_1$ and $e_2$ are $\mathcal{U}_{Pois(R)}$-module homomorphisms, and we have the following relation 
\begin{align*}
e_1(x)y=e_1(y)x,\,\,\,\,\,\,\,\,\,\ X_{e_1(x)}y+X_{e_1(y)}x=0
\end{align*}
for $x,y\in E_1$. We claim that $E_2$ is a Poisson $B$-module. Indeed, let $I$ be a kernel of $e_0$, which is a Poisson ideal of $R$. Now we give a Poisson $R/I$-module structure on $E_2$. First we note that $E_2$ is a $\mathcal{U}_{Pois(R)}$-module, and so Poisson $R$-module. Let $a\in B$ and $x\in E_2$. We define a Poisson $B$-module structure on $E_2$ by setting $b\cdot x:=r\cdot x$ and $\{b,x\}=:\{r,x\}$ where $r$ is a lifting of $b$ under $e_0:R\to R/I\cong B$. This is well-defined since given two lifting $r_1$ and $r_2$ of $b$ $(i.e \,\,\,\,r_1-r_2\in I)$, choose $y\in E_1$ with $e_1(y)=r_1-r_2$. Then  $e_2(r_1x-r_2x)=(r_1-r_2)e_2(x)=e_1(y)e_2(x)=e_1(e_2(x))y=0$ and so we have $r_1x=r_2x$. On the other hand, $e_2(\{r_1,x\}-\{r_2,x\})=X_{r_1-r_2}e_2(x)=X_{e_1(y)} e_2(x)=-X_{e_1(e_2(yx))}y=0$. So we have $\{r_1,x\}=\{r_2,x\}$. Then other property of a Poisson module trivially follows. 

Let $A\to A'\to B'$ be a homomorphism of Poisson algebras, and $\mathscr{E}'$ an Poisson extension of $B'$ over $A'$. By a homomorphism $\alpha:\mathscr{E}\to \mathscr{E}'$ of Poisson extensions we mean a collection $(b,\alpha_0,\alpha_1,\alpha_2)$ with the following commutative diagram
\begin{center}
$\begin{CD}
0@>>> E_2'@>e_2'>> E_1'@>e_1'>> R'@>e_0'>> B'@>>> 0\\
@. @AA\alpha_2A @AA\alpha_1A @AA\alpha_0A @AAbA @.\\
0@>>> E_2@>e_2>> E_1@>e_1>> R @>e_0>> B @>>>0
\end{CD}$
\end{center}
where $b$, $\alpha_0$ is a Poisson homomorphism such that the following diagram commutes as Poisson homomorphisms
\begin{center}
$\begin{CD}
A'@>>> R'@>e_0'>> B'\\
@AAA @AA\alpha_0A@AAbA\\
A@>>> R@>>e_0> B
\end{CD}$
\end{center}
and $\alpha_1$ and $\alpha_2$ are homomorphisms of $\mathcal{U}_{Pois(R)}$-modules. We consider $E_1'$ and $E_2'$ are $\mathcal{U}_{Pois(R)}$-modules via $\alpha_0$.
\end{definition}

\begin{definition}
Let $\mathscr{E}$ be an Poisson extension of $B$ over $A$, we define a complex $PL^\bullet(\mathscr{E})$ of $\mathcal{U}_{Pois(B)}$-modules or equivalently Poisson $B$-modules:
\begin{align*}
PL^\bullet(\mathscr{E}): 0\to E_2\xrightarrow{d_2}\mathcal{U}_{Pois(B)}\otimes_{\mathcal{U}_{Pois(R)}} E_1\xrightarrow{d_1} \mathcal{U}_{Pois(B)}\otimes_{\mathcal{U}_{Pois(R)}} \Omega_{\mathcal{U}_{Pois(R)/A}}^1\to 0
\end{align*}
where $d_2$ is induced from $e_2$, and $d_1$ is defined in the following way: let $I=Ker(e_0) ($a Poisson ideal of $R$ defining $B)$. We also note that we have the canonical map $d:R\to \Omega_{\mathcal{U}_{Pois(R)}/A}^1$, and by restricting $d$ on $I$, we have $d:I\to \Omega_{\mathcal{U}_{Pois(R)/A}}^1$ and then by tensoring $\mathcal{U}_{Pois(B)}$, we have $d:I/I^2\oplus \{I,I\}\cong \mathcal{U}_{Pois(B)}\otimes_{\mathcal{U}_{Pois(R)}} I\to \mathcal{U}_{Pois(B)}\otimes_{\mathcal{U}_{Pois(R)}} \Omega_{\mathcal{U}_{Pois(R)/A}}^1$. We define $d_2=d\circ(\mathcal{U}_{Pois(B)}\otimes_{\mathcal{U}_{Pois(R)}} e_1)$. This is well-defined since $im(E_1)=I$. $PL^\bullet(\mathscr{E})$ is a complex since $e_1\circ e_2=0$
\end{definition}

\begin{remark}
Any Poisson extension of $B$ over $A$ are all obtained in the following way. Choose a surjection $e:R\to B$ with $A\to R\to B$ Poisson homomorphisms. Let $I=Ker\,e_0$ be a Poisson ideal of $R$. Then choose an exact sequence of $\mathcal{U}_{Pois(R)}$-modules $0\to U\xrightarrow{i} F\xrightarrow{j} I\to 0$. Let $U_0$ be $\mathcal{U}_{Pois(R)}$-submodule of $F$ generated by $j(x)y-j(y)x$ and $X_{j(x)}y+X_{j(x)}y$, where $x,y\in F$. Then $j(U_0)=0$ since $j(x)j(y)-j(y)j(x)=0, X_{j(x)}j(y)+X_{j(y)}j(x)=\{j(x),j(y)\}+\{j(y),j(x)\}=0$. Hence $U_0$ is also a submodule of $U$. We take $e_2:U/U_0\to F/U_0$ and $e_1:F/U_0\to R$ which is well-defined since $j(U_0)=0$. Then $0\to U/U_0\to F/U_0\to R\to B\to 0$ is a Poisson extension of $B$ over $A$. Conversely, given a Poisson extension $(\mathscr{E}):0\to E_2\xrightarrow{e_2} E_1\xrightarrow{e_1}R\xrightarrow{e_0}B\to 0$, we have $U_0=0$.
\end{remark}

\begin{definition}[Free Poisson extension]
A Poisson extension of $B$ over $A$ is of the from $0\to U/U_0\to F/U_0\to R\to B\to 0$ where $R$ is a free Poisson algebra over $A$ and $F$ is a free $\mathcal{U}_{Pois(R)}$-module is called a free Poisson extension of $B$ over $A$.
\end{definition}

\begin{remark}
For a free Poisson extension of $B$ over $A$, $\mathscr{E}:0\to U/U_0\to F/U_0\xrightarrow{j} R=A\{x_i\}\xrightarrow{e_0} B\to 0$, we have $\mathcal{U}_{Pois(B)}\otimes_{\mathcal{U}_{Pois(R)}} (F/U_0) \cong \mathcal{U}_{Pois(B)}\otimes_{\mathcal{U}_{Pois(R)}}F$. Hence $PL^1(\mathscr{E})$ is free. Indeed, let's consider the natural map $\alpha:\mathcal{U}_{Pois(B)}\otimes_{\mathcal{U}_{Pois(R)}} F\to \mathcal{U}_{Pois(B)}\otimes_{\mathcal{U}_{Pois(R)}} F/U_0$, defined by $\sum b_i\otimes f_i\mapsto \sum b_i\otimes \bar{f}_i$, which is surjective. Let $\sum b_i\otimes \bar{f}_i=0$. Since $\mathcal{U}(e_0):\mathcal{U}_{Pois(R)}\to \mathcal{U}_{Pois(B)}$ is surjective, we have $\sum b_i\otimes \bar{f}=\sum 1\otimes a_i\bar{f}=0$, where $\mathcal{U}(e_0)(a_i)=b_i$. Hence $\sum a_i\bar{f}_i=0$. So $\sum a_if_i\in U_0$. Hence $\sum_i a_if_i= \sum a_j' g_j$, where $g_j$ is of the form $j(x)y-j(y)x$ or $X_{j(x)}y+X_{j(x)}x$ where $x,y \in F$. Let $t\in \mathcal{U}_{Pois(R)}$. Since $1\otimes t(j(x)y-j(y)x)=\mathcal{U}(e_0)(t)e_0(j(x))\otimes y -\mathcal{U}(e_0)(t)e_0(j(y))\otimes x=0$ and $1 \otimes t(X_{j(x)}y+X_{j(y)}x)= \mathcal{U}(e_0)(t)X_{e_0(j(x))}\otimes y+\mathcal{U}(e_0)(t)X_{e_0(j(y))} \otimes x=0$, we have $\sum 1\otimes a_if_i=\sum 1\otimes a_j'g_j=0$. Hence $\alpha$ is an isomorphism.
\end{remark}

Consider the following commutative diagram of Poisson homomoprhisms,
\begin{center}
$\begin{CD}
B@>b>> B'\\
@AAA @AAA\\
A@>a>> A'
\end{CD}$
\end{center}
Let $\mathscr{E}$ be a free Poisson extension of $B$ over $A$ and $\mathscr{E}'$ arbitrary Poisson extension of $B'$ over $A'$. Then there exists a homomorphism $\alpha: \mathscr{E}\to \mathscr{E}'$ extending $b$.
\begin{center}
$\begin{CD}
0@>>> E_2'@>e_2'>> E_1'@>e_1'>> R'@>e_0'>> B'@>>> 0\\
@. @AA\alpha_2A @AA\alpha_1A @AA\alpha_0A @AAbA @.\\
0@>>> U/U_0@>e_2>> F/U_0=(\oplus_k \mathcal{U}_{Pois(R)})/U_0@>e_1=j>> R=A\{x_i\} @>e_0>> B @>>>0
\end{CD}$
\end{center}
where $\alpha_0,\alpha_1$ and $\alpha_2$ are defined in the following way: for $\alpha_0$, we send $x_i$ to an arbitrary lifting of $b(e_0(x_i))$. Let $\{v_k\}$ be the canonical basis of $F$ (i.e $v_k$ has 1 in the $k$-th component, and $0$ in other components). For $\alpha_1$, first we define a map $\alpha_1':F\to E_1'$ by sending $v_k$ to an arbitrary lifting $w_k'$ of $\alpha_0(e_1(v_k))$. So we have $e_1'(w_k')=\alpha_0(e_1(v_k))=\alpha_0(j(v_k)).$ We show that $\alpha_2'(U_0)=0$, and so $\alpha_1$ is well-defined. Indeed, we claim that $\alpha_1'(j(x)y-j(y)x)=j(x)\alpha_1'(y)-j(y)\alpha_1'(x)=\alpha_0(j(x))\alpha_1'(y)-\alpha_0((j(y))\alpha_1'(x)= 0$. Let $x=\sum_i a_iv_i$ and $y=\sum_kb_k v_k$. Then $j(x)y-j(y)x=\sum_{i,k}a_ij(v_i)b_kv_k-b_kj(v_k)a_iv_i$. Hence $\alpha_1'(j(x)y-j(y)x)=\sum_{i,k}\alpha_0(a_i)\alpha_0(j(v_i))\alpha_0(b_k)\alpha_1'(v_k)-\alpha_0(b_k)\alpha_0(j(v_k))\alpha_0(a_i)\alpha_1'(v_i)$. It is sufficient to show that 
\begin{align*}
\alpha_0(a_i)\alpha_0(j(v_i))\alpha_0(b_k)\alpha_1'(v_k)-\alpha_0(b_k)\alpha_0(j(v_k))\alpha_0(a_i)\alpha_1'(v_i)=0.
\end{align*}
 Since $e_1'(\alpha_0(b_k)\alpha_1'(v_k))=\alpha_0(b_k)\alpha_0(j(v_k))$ and $e_1'(\alpha_0(a_i)\alpha_1'(v_i))=\alpha_0(a_i)\alpha_0(j(v_i))$, we get the claim. On the other hand, we claim that $\alpha_1'(X_{j(x)}y+X_{j(y)}x)=0$. Indeed, $X_{j(x)}y+X_{j(y)}x=\sum_{i,k} X_{a_ij(v_i)}b_kv_k+X_{b_kj(v_k)}a_iv_i$. So we have $\alpha_1'(X_{j(x)}y+X_{j(y)}x)=\sum_{i,k} X_{\alpha_0(a_i)\alpha_0(j(v_i))}\alpha_0(b_k)\alpha_1'(v_k)+X_{\alpha_0(b_k)\alpha_0(j(v_k))}\alpha_0(a_i)\alpha_1'(v_i)$. So it is sufficient to show that 
 \begin{align*}
 X_{\alpha_0(a_i)\alpha_0(j(v_i))}\alpha_0(b_k)\alpha_1'(v_k)+X_{\alpha_0(b_k)\alpha_0(j(v_k))}\alpha_0(a_i)\alpha_1'(v_i)=0.
 \end{align*}
  Since $e_1'(\alpha_0(b_k)\alpha_1'(v_k))=\alpha_0(b_k)\alpha_0(j(v_k))$ and $e_1'(\alpha_0(a_i)\alpha_1'(v_i))=\alpha_0(a_i)\alpha_0(j(v_i))$, we get the claim. For $\alpha_2$, we simply see that $U/U_0$ is sent to $E_2'$ via $\alpha_1$.

Next we claim that we have a homomorphism $\mathcal{U}_{Pois(B')}\otimes_{\mathcal{U}_{Pois(B)}}PL^\bullet(\mathscr{E})\to PL^\bullet(\mathscr{E}')$.
\begin{center}
\tiny{$\begin{CD}
0@>>>E_2'@>1\otimes d_2>>\mathcal{U}_{Pois(B')}\otimes_{\mathcal{U}_{Pois(R')}} E_1'@>1\otimes d_1>>\mathcal{U}_{Pois(B')}\otimes_{\mathcal{U}_{Pois(R')}} \Omega_{\mathcal{U}_{Pois(R')/A'}}^1@>>> 0\\
@. @AA\alpha_2'A @AA\alpha_1'A @AA\alpha_0'A \\
0@>>>\mathcal{U}_{Pois(B')}\otimes_{\mathcal{U}_{Pois(B)}} U/U_0@>1\otimes d_2>>\mathcal{U}_{Pois(B')}\otimes_{\mathcal{U}_{Pois(R)}} F/U_0@>1\otimes d_1>>\mathcal{U}_{Pois(B')}\otimes_{\mathcal{U}_{Pois(R)}} \Omega_{\mathcal{U}_{Pois(R)/A}}^1@>>> 0\end{CD}$}
\end{center}

 where $\alpha_2'$ and $\alpha_1'$ are canonical induced from $\alpha_2,\alpha_1$. For $\alpha_0'$, we note that we have a canonical commutative diagram
\begin{center}
$\begin{CD}
\Omega_{\mathcal{U}_{Pois(R)/A}}^1@>>> \Omega_{\mathcal{U}_{Pois(R')/A'}}^1\\
@AdAA @AAd'A\\
R@>\alpha_0>> R'
\end{CD}$
\end{center}
where $d$ and $d'$ are the canonical map. Let $I'=Ker(e_0')$ and $I=Ker(e_0)$. Then $\alpha_0(I)\subset I'$. Hence we have
\begin{center}
$\begin{CD}
\Omega_{\mathcal{U}_{Pois(R)/A}}^1@>>> \Omega_{\mathcal{U}_{Pois(R')/A'}}^1\\
@AdAA @AAd'A\\
I@>\alpha_0>> I'\\
@AAA @AAA\\
F/U_0@>\alpha_1>> E_1'
\end{CD}$
\end{center}
By tensoring $\mathcal{U}_{Pois(B')}$, we get $\alpha_0'$.

\begin{definition}
a complex of the form $PL^\bullet(\mathscr{E})$, where $\mathscr{E}$ is a free Poisson extension of $B$ over $A$ is called a Poisson cotangent complex of $B$ over $A$.

\end{definition}

\begin{definition}
We say that a Poisson homomorphism $A\to R$ has property $(L)$ if the following condition holds: let $A\to S$ be a Poisson homomorphism and $u:M\to S$ a homomorphism of $\mathcal{U}_{Pois(S)}$-modules such that $u(x)y=u(y)x$ and $X_{u(x)}y+X_{u(y)}x=0$ for $x,y\in M$. Then for any Poisson homomorphism $f,g:R\to S$ such that the following commutative diagram holds and $Im(f-g)\subset Im(u)$, there exists a Poisson biderivation $\lambda:R\to M$ such that $u\circ \lambda =f-g$.

\begin{center}
\[\begindc{\commdiag}[5]
\obj(0,10)[aa]{$M$}
\obj(10,10)[bb]{$S$}
\obj(10,0)[cc]{$R$}
\mor{aa}{bb}{$u$}
\mor{cc}{aa}{$\lambda$}
\mor(9,0)(9,10){$f$}
\mor(11,0)(11,10){$g$}[\atright,\solidarrow]
\enddc\]
\end{center}
Here a Poisson biderivation means $\lambda(xy)=\lambda(x)g(y)+f(g)\lambda(y)$ and $\lambda(\{x,y\})=\{\lambda(x),g(y)\}+\{f(x),\lambda(y)\}$. Recall that $M$ is a $(\mathcal{U}_{Pois(S)}-\mathcal{U}_{Pois(S)})$-bimodule, where right-module structure is defined by $m\cdot s:=s\cdot m$ and $m\cdot X_s:=-X_s\cdot m$ for $s\in S$.

\end{definition}

\begin{proposition}
Consider the following commutative diagram of Poisson homomorphisms,
\begin{center}
$\begin{CD}
B@>b>> B'\\
@AAA @AAA\\
A@>a>>A'
\end{CD}$
\end{center}
Let $\mathscr{E}$ be an Poisson extension of $B$ over $A$ and $\mathscr{E}$ be a Poisson extension of $B'$ over $A'$. Let  $\alpha,\beta:\mathscr{E}\to\mathscr{E}'$ be a homomorphism extending $b$:

\begin{center}
\[\begindc{\commdiag}[5]
\obj(0,10)[aa]{$0$}
\obj(10,10)[bb]{$E_2'$}
\obj(20,10)[cc]{$E_1'$}
\obj(30,10)[dd]{$R'$}
\obj(40,10)[ee]{$B'$}
\obj(50,10)[ff]{$0$}
\obj(0,0)[gg]{$0$}
\obj(10,0)[hh]{$E_2$}
\obj(20,0)[ii]{$E_1$}
\obj(30,0)[jj]{$R$}
\obj(40,0)[kk]{$B$}
\obj(50,0)[ll]{$0$}
\mor{aa}{bb}{}
\mor{bb}{cc}{$e_2'$}
\mor{cc}{dd}{$e_1'$}
\mor{dd}{ee}{$e_0'$}
\mor{ee}{ff}{}
\mor{gg}{hh}{}
\mor{hh}{ii}{$e_2$}
\mor{ii}{jj}{$e_1$}
\mor{jj}{kk}{$e_0$}
\mor{kk}{ll}{}
\mor(9,0)(9,10){$\beta_2$}
\mor(11,0)(11,10){$\alpha_2$}[\atright,\solidarrow]
\mor(19,0)(19,10){$\beta_1$}
\mor(21,0)(21,10){$\alpha_1$}[\atright,\solidarrow]
\mor(29,0)(29,10){$\beta_0$}
\mor(31,0)(31,10){$\alpha_0$}[\atright,\solidarrow]
\mor(40,0)(40,10){$b$}
\enddc\]
\end{center}

If $R$ has property $(L)$, then $\bar{\alpha}$ and $\bar{\beta}$ are homotopic maps of  $\mathcal{U}_{Pois(B')} \otimes_{\mathcal{U}_{Pois(B)}} PL^\bullet(\mathscr{E})\to PL^{\bullet}(\mathscr{E}')$.
\end{proposition}

\begin{proof}
There exists a Poisson biderivation $\lambda:R\to E_1'$ such that $e_1'\circ \lambda=\beta_0-\alpha_0$. Let $\theta:E_1\to E_1'$ be the map $\theta=(\beta_1-\alpha_1)-\lambda\circ e_1$. We note that $e_1'\circ \theta=e_1'\circ (\beta_1-\alpha_1-\lambda\circ e_1)=(\beta_0-\alpha_0)\circ e_1-(\beta_0-\alpha_0)\circ e_1=0$. Thus $Im(\theta)$ is in the $\mathcal{U}_{Pois(B')}$-module $Im(e_2')\cong E_2'$. Then $(e_0'\circ \alpha_0(r)-e_0'\circ \beta_0(r))x=(e_0\circ b(r)-e_0\circ b(r))x=0 $ and $\{e_0'\circ\alpha_0(r)-e_0'\circ \beta_0(r),x\}=0$ for $x\in E_2'$. Hence on $E_2'$ (so $Im(\theta)$) as an $\mathcal{U}_{Pois(R)}$-module, the action of $R$ via $\alpha_0$ and $\beta_0$ coincides. We claim that $\theta$ is Poisson $R$-linear or equivalently $\mathcal{U}_{Pois(R)}$-linear. In other words, $\theta(rx)=\beta_0(r)\theta(x)=\alpha_0(r)\theta(x)$ and $\theta(\{r,x\})=\{\beta_0(r),\theta(x)\}=\{\alpha_0(r),\theta(x)\}$ for $r\in R$ and $x\in E_1$. Indeed,
\begin{align*}
\theta(rx)&=\beta_0(r)\beta_1(x)-\alpha_0(r)\alpha_1(x)-\lambda(re_1(x))\\
             &=\beta_0(r)\beta_1(x)-\alpha_0(r)\alpha_1(x)-\lambda(r)\alpha_0(e_1(x))-\beta_0(r)\lambda(e_1(x))\\
\beta_0(r)\theta(x)&=\beta_0(r)\beta_1(x)-\beta_0(r)\alpha_1(x)-\beta_0(r)\lambda(e_1(x)) \\
\theta(rx)-\beta_0(r)\theta(x)&= (\beta_0(r)-\alpha_0(r))\alpha_1(x) -\lambda(r)\alpha_0(e_1(x))=e_1'(\lambda(r))\alpha_1(x)-\lambda(r)e_1'(\alpha_1(x))=0         
\end{align*}
On the other hand,
\begin{align*}
\theta(\{r,x\})&=\{\beta_0(r),\beta_1(x)\}-\{\alpha_0(r),\alpha_1(x)\}-\lambda(\{r,e_1(x)\})\\
                   &=\{\beta_0(r),\beta_1(x)\}-\{\alpha_0(r),\alpha_1(x)\}-\{\lambda r,\alpha_0(e_1(x))\}+\{\beta_0(r),\lambda(e_1(x))\}\\
\{\beta_0(r), \theta(x)\}&=\{\beta_0(r),\beta_1(x)\}-\{\beta_0(r),\alpha_1(x)\} -\{\beta_0(r), \lambda(e_1(x))\}\\
\theta(\{r,x\})-\{\beta_0(r),\theta(x)\}&=\{\beta_0(r)-\alpha_0(r),\alpha_1(x)\}-\{\lambda r,\alpha_0(e_1(x))\}=X_{e_1'(\lambda r)}\alpha_1(x)+X_{e_1'(\alpha_1(x))}\lambda r =0               
\end{align*}
Note that the Poisson biderivation $\lambda:R\to E_1'$ induces a Poisson derivation 
\begin{align*}
1\otimes \lambda:R\to \mathcal{U}_{Pois(B')}\otimes_{\mathcal{U}_{Pois(R')}} E_1'
\end{align*}
since the action induced by $\alpha_0$ and $\beta_0$ coincides. Then by the universal mapping property of $\Omega_{Pois(R)/A}^1$, there is a $\bar{\lambda}:\mathcal{U}_{Poiss(B')}\otimes_{\mathcal{U}_{Pois(R)}}\Omega_{\mathcal{U}_{Pois(R)/A}}^1\to \mathcal{U}_{Pois(B')}\otimes_{\mathcal{U}_{Pois(R')}} E_1'$ such that $\bar{\lambda}(b\otimes dx)=b\otimes \lambda x$. On the other hand, the Poisson $R$-module map $\theta:E_1\to Im (e_2')\cong E_2'$ induces a $\bar{\theta}:\mathcal{U}_{Pois(B')}\otimes_{\mathcal{U}_{Pois(R)}} E_1\to E_2'$

\begin{center}
\tiny{
\[\begindc{\commdiag}[170]
\obj(0,1)[aa]{$E_2'$}
\obj(1,1)[bb]{$\mathcal{U}_{Pois(B')}\otimes_{\mathcal{U}_{Pois(R')}} E_1'$}
\obj(2,1)[cc]{$\mathcal{U}_{Pois(B')}\otimes_{\mathcal{U}_{Pois(R')}} \Omega_{\mathcal{U}_{Pois(R')/A}}^1$}
\obj(0,0)[dd]{$\mathcal{U}_{Pois(B')}\otimes_{\mathcal{U}_{Pois(R)}} E_2$}
\obj(1,0)[ee]{$\mathcal{U}_{Pois(B')}\otimes_{\mathcal{U}_{Pois(R)}} E_1$}
\obj(2,0)[ff]{$\mathcal{U}_{Pois(B')}\otimes_{\mathcal{U}_{Pois(R)}} \Omega_{\mathcal{U}_{Pois(R)/A}}^1$}
\mor{aa}{bb}{$e_2'$}
\mor{ff}{bb}{$\bar{\lambda}$}
\mor{ee}{aa}{$\bar{\theta}$}
\mor{dd}{aa}{$\bar{\beta}_2-\bar{\alpha}_2$}
\mor{ee}{bb}{$\bar{\beta}_1-\bar{\alpha}_1$}
\mor{ff}{cc}{$\bar{\beta}_0-\bar{\alpha}_0$}
\mor{dd}{ee}{$1\otimes e_2$}
\mor{bb}{cc}{$f':=1\otimes d'\circ e_1'$}
\mor{ee}{ff}{$f:=1\otimes d\circ e_1$}
\enddc\]}
\end{center}
where $d:R\to \mathcal{U}_{Pois(R)/A}$ and $d':R'\to \mathcal{U}_{Pois(R')/A}$ are the canonical maps. Now we claim that $\bar{\beta}_2-\bar{\alpha}_2=\bar{\theta}\circ(1\otimes e_2)$, $\bar{\beta}_1-\bar{\alpha}_1=e_2'\circ \bar{\theta}+\bar{\lambda}\circ f'$ and $\bar{\beta}_0-\bar{\alpha}_0=f'\circ \bar{\lambda}$. For $\bar{\beta}_2-\bar{\alpha}_2=\bar{\theta} \circ (1\otimes e_2)$, we note that for $x\in E_2$, $\theta(e_2(x))=(\beta_1-\alpha_1)(e_2(x))-\lambda(e_1(e_2(x)))=e_2'\circ (\beta_2-\alpha_2)(x)$. For $\bar{\beta}_0-\bar{\alpha}_0=f'\circ \bar{\lambda}$, we note that since $e_1'\circ \lambda=\beta_0-\alpha_0$, we have $f'\circ \bar{\lambda}(b\otimes dx)=f'(b\otimes \lambda x)=b\otimes (d'\circ e'(\lambda x))=b\otimes d'((\beta_0-\alpha_0)(x))=(\bar{\beta}_0-\bar{\alpha}_0)(b\otimes dx)$. For $\bar{\beta}_1-\bar{\alpha}_1=e_2'\circ \bar{\theta}+\bar{\lambda}\circ f$, we note that for $x\in E_1$, $\theta(x)=(\beta_1-\alpha_1)(x)- \lambda(e_1(x))$.
\end{proof}

\begin{lemma}
A free Poisson algebra $A\{x_i\}$ over a Poisson algebra $A$ generated by $\{x_i\}$ satisfies the property $(L)$.
\end{lemma}

\begin{proof}
Let $A\to S$ be a Poisson homomorphism of Poisson algebras and $u:M\to S$ be a homomorphism of $\mathcal{U}_{Pois(S)}$ such that $u(x)y=u(y)x$ and $X_{u(x)}y+X_{u(y)}x=0$, for all $x,y\in M$. Let $f,g:R\to S$ be Poisson homomorphism compatible with $A$ such that $Im(f-g)\subset Im(u)$. We would like to define a Poisson biderivation $d:R\to M$ such that $u\circ d=f-g$. Let $R=A\{x_i\}$. Since $Im(f-g)\subset Im(u)$, we define $d(x_i)$ in $M$ satisfying $u(d(x_i))=f(x_i)-g(x_i)$ and $d(a)=0$. In general we define $d$ in the following way: for example
\begin{align*}
d(x_1x_2[x_3, ax_4x_5]):=d(x_2)g(x_2)[g(x_3),g(a)g(x_4)g(x_5)]+f(x_1)d(x_2)[g(x_3),g(a)g(x_4)g(x_5)]\\
+f(x_1)f(x_2)[d(x_3),g(a)g(x_4)g(x_5)]+f(x_1)f(x_2)[f(x_3),d(a),g(x_4)g(x_5)]\\
+f(x_1)f(x_2)[f(x_3),f(a)d(x_2)g(x_5)]+f(x_1)f(x_2)[f(x_3),f(a)f(x_4)d(x_5)]
\end{align*}
where $[-,-]$ is the Poisson bracket on $A\{x_i\}$. We will show that this is well-defined and so by definition, $d$ is a Poisson biderivation. Let $L$ to be the free Lie algebra generated by $A$ and $x_i$. First we show that $d$ is well-defined on $L$. Simply we define $d$ on free algebra generated by $A$ and $x_i$ by the above relation, where the operation is $[-,-]$. We show that $d([x,x])=0$ and $d([x,[y,z]]+[y,[z,x]]+[z,[x,y]])=0$ where $x,y,z\in A\cup \{x_i\}$. Then $d$ is well-defined on $L$. Indeed, for $[x,x]$, we have to show $[dx,g(x)]+[f(x),dx]=0$. Since $X_{u(dx)}dx+X_{ud(x)}dx=0$, we have $[f(x)-g(x),dx]+[f(x)-g(x),dx]=0$. Hence we have $[f(x)-g(x),dx]=0$. So $[dx,g(x)]+[f(x),dx]=0$. 

For $[x,[y,z]]+[y,[z,x]]+[z,[x,y]]$, we want to show that $d([x,[y,z]]+[y,[z,x]]+[z,[x,y]])=0$, equivalently,
\begin{align*}
&[dx,[g(y),g(z)]]+[f(x),[dy,g(z)]]+[f(x),[f(y),dz]]\\
+&[dy,[g(z),g(x)]]+[f(y),[dz,g(x)]]+[f(y),[f(z),dx]]\\
+&[dz,[g(x),g(y)]]+[f(z),[dx,g(y)]]+[f(z),[f(x),dy]]=0
\end{align*}

For $[x,[y,z]]$, we note that $X_{u(d(x))}d([y,z])+X_{u(d([y,z]))}d(x)=[f(x)-g(x),[dy,g(z)]+[f(y),dz]]+[[f(y)-g(y),g(z)],dx]+[[f(y),f(z)-g(z)],dx]=[f(x),[dy,g(z)]]-[g(x),[dy,g(z)]]+[f(x),[f(y),dz]]-[g(x),[f(y),dz]]+[[f(y),g(z)],dx]-[[g(y),g(z)],dx]+[[f(y),f(z)],dx]-[[f(y),g(z)]dx]=[f(x),[dy,g(z)]]-[g(x),[dy,g(z)]]+[f(x),[f(y),dz]]-[g(x),[f(y),dz]]-[[g(y),g(z)],dx]+[[f(y),f(z)],dx]$. Hence we get the following. (however, we do not use this relationship for $[x,[y,z]]$. We do this for the symmetric arguments $[y,[z,x]]$ and $[z,[x,y]]$ in the below.)

\begin{align*}
X_{u(d(x))}d([y,z])+X_{u(d([y,z]))}d(x)&=[f(x),[dy,g(z)]]-[g(x),[dy,g(z)]]\\
                                                                     &+[f(x),[f(y),dz]]-[g(x),[f(y),dz]]\\
                                                                     &-[[g(y),g(z)],dx]+[[f(y),f(z)],dx]=0
\end{align*}
For $[y,[z,x]]$, by symmetry we have
\begin{align*}
X_{u(d(y))}d([z,x])+X_{u(d([z,x]))}d(y)&=[f(y),[dz,g(x)]]-[g(y),[dz,g(x)]]\\
                                                                     &+[f(y),[f(z),dx]]-[g(y),[f(z),dx]]\\
                                                                     &-[[g(z),g(x)],dy]+[[f(z),f(x)],dy]=0
\end{align*}
For $[z,[x,y]]$, by symmetry we have
\begin{align*}
X_{u(d(z))}d([z,y])+X_{u(d([x,y]))}d(x)&=[f(z),[dx,g(y)]]-[g(z),[dx,g(y)]]\\
                                                                     &+[f(z),[f(x),dy]]-[g(z),[f(x),dy]]\\
                                                                     &-[[g(x),g(y)],dz]+[[f(x),f(y)],dz]=0
\end{align*}

Then we have

\begin{align*}
&[dx,[g(y),g(z)]]+[f(x),[dy,g(z)]]+[f(x),[f(y),dz]]\\
+&[dy,[g(z),g(x)]]+[f(y),[dz,g(x)]]+[f(y),[f(z),dx]]\\
+&[dz,[g(x),g(y)]]+[f(z),[dx,g(y)]]+[f(z),[f(x),dy]]\\
=&[dx,[g(y),g(z)]]+[f(x),[dy,g(z)]]+[f(x),[f(y),dz]]\\
 +&[g(y),[dz,g(x)]]+[g(y),[f(z),dx]]-[[f(z),f(x)],dy]\,\,\,\text{from}\,\,\, X_{u(d(z))}d([z,y])+X_{u(d([x,y]))}d(x)=0\\
 +&[g(z),[dx,g(y)]]+[g(z),[f(x),dy]]-[[f(x),f(y)],dz]\,\,\,\text{from}\,\,\,X_{u(d(z))}d([z,y])+X_{u(d([x,y]))}d(x)=0\\
 =&-[g(y),[g(z),dx]]+[g(z),[g(y),dx]]-[dy,[g(z),f(x)]]-[g(z),[f(x),dy]]-[f(y),[dz,f(x)]]-[dz,[f(x),f(y)]]\\
 +&[g(y),[dz,g(x)]]+[g(y),[f(z),dx]]-[[f(z),f(x)],dy]\\
  +&[g(z),[dx,g(y)]]+[g(z),[f(x),dy]]-[[f(x),f(y)],dz]\\
  =&-[g(y),[g(z),dx]]-[dy,[g(z),f(x)]]-[f(y),[dz,f(x)]]\\
+&[g(y),[dz,g(x)]]+[g(y),[f(z),dx]]-[[f(z),f(x)],dy]\\
  =&[g(y),[f(z)-g(z),dx]]+[dy,[f(z)-g(z),f(x)]]-[[f(y),[dz,f(x)]]+[g(y),[dz,g(x)]]\\
  =&[g(y),[f(z)-g(z),dx]]-[f(z)-g(z),[f(x),dy]]-[f(x),[dy,f(z)-g(z)]]-[[f(y),[dz,f(x)]]+[g(y),[dz,g(x)]]
\end{align*}

On the other hand, from $[g(y), X_{u(dz)}dx+X_{u(dx)}dz]=0$, we have
\begin{align*}
[g(y),[f(z)-g(z),dx]]+[g(y),[f(x)-g(x),dz]]=0
\end{align*}
From $[f(x),X_{u(dz)}dy+X_{u(dy)}dz]=0$, we have
\begin{align*}
[f(x),[f(z)-g(z),dy]]+[f(x),[f(y)-g(y),dz]]=0
\end{align*}
Hence we have
\begin{align*}
&[g(y),[f(z)-g(z),dx]]-[f(z)-g(z),[f(x),dy]]-[f(x),[dy,f(z)-g(z)]]-[[f(y),[dz,f(x)]]+[g(y),[dz,g(x)]]\\
=&-[g(y),[f(x)-g(x),dz]]-[f(z)-g(z),[f(x),dy]]-[f(x),[f(y)-g(y),dz]]\\
&-[[f(y),[dz,f(x)]]+[g(y),[dz,g(x)]]\\
=&-[g(y),[f(x),dz]]-[f(z)-g(z),[f(x),dy]]-[f(x),[f(y)-g(y),dz]]-[[f(y),[dz,f(x)]]\\
=&-[g(y),[f(x),dz]]-[f(z)-g(z),[f(x),dy]]-[f(x),[f(y),dz]]+[f(x),[g(y),dz]]-[[f(y),[dz,f(x)]]\\
=&[[f(x),g(y)],dz]-[f(z)-g(z),[f(x),dy]]-[[f(x),f(y)],dz]
\end{align*}

From $X_{u(dz)}[f(x),dy]+X_{u([f(x),dy])}dz=0$, we have
\begin{align*}
&[f(z)-g(z),[f(x),dy]]+[[f(x),f(y)-g(y)],dz]\\
=&[f(z)-g(z),[f(x),dy]]+[[f(x),f(y)],dz]-[[f(x),g(y)],dz]=0
\end{align*}

Hence we have
\begin{align*}
[[f(x),g(y)],dz]-[f(z)-g(z),[f(x),dy]]-[[f(x),f(y)],dz]=0
\end{align*}

Hence $d$ is well-defined on $L$. Since $(f-g)([x,y])=[f(x),(f-g)(x)]+[(f-g)(x),g(x)]$, we have $u\circ d=f-g$ on $L$. Now we show that $d$ is well-defined on the free commutative algebra $S$ generated by $L$. Let $V=\{T_j\}$ be a basis of $L$ over $k$. Then we define $d$ on $S$ by the following formula, for a monomial $T_{i_1}\cdots T_{i_n}$,
\begin{align*}
d(T_{i_1}\cdots T_{i_n})=\sum_{k=1}^n f(T_{i_1}\cdots T_{i_{k-1}})d(T_{i_k})g(T_{i_{k+1}}\cdots T_{i_n}).
\end{align*}
Then $d$ is well-defined on $S$ and we have $u\circ d=f-g$ on $S$ (See \cite{Sch67} Lemma 2.1.6). We define a bracket $[-,-]_S$ on $S$ such that $[T_i,T_j]_S:=[T_i,T_j]$ and we use the relation $[x,yz]=y[x,z]+z[x,y]$ for $x,y,z\in L$. Then from $[x,yz]-y[x,z]-z[x,y]$ for $x,y,z\in S$, we want to show that $d([x,yz])=d(y[x,z]+z[x,y])$, equivalently,




\begin{align*}
[dx,g(y)g(z)]+[f(x),dyg(z)]+[f(x),f(y)dz]=dy[g(x),g(z)]+f(y)[dx,g(z)]+f(y)[f(x),dz]\\
+dz[g(x),g(y)]+f(z)[dx,g(y)]+f(z)[f(x),dy]
\end{align*}
Equivalently,
\begin{align*}
g(y)[dx,g(z)]+g(z)[dx,g(y)]+dy[f(x),g(z)]+g(z)[f(x),dy]+dz[f(x),f(y)]+f(y)[f(x),dz]\\
=dy[g(x),g(z)]+f(y)[dx,g(z)]+f(y)[f(x),dz]
+dz[g(x),g(y)]+f(z)[dx,g(y)]+f(z)[f(x),dy]
\end{align*}


We note that
\begin{align*}
u(dy)d[x,z]-u(d[z,x])dy=(f(y)-g(y))([dx,g(z)]+[f(x),dz])-([f(x)-g(x),g(z)]+[f(x),f(z)-g(z)])dy\\
=f(y)[dx,g(z)]-g(y)[dx,g(z)]+f(y)[f(x),dz]-g(y)[f(x),dz]+[g(x),g(z)]dy-[f(x),f(z)]dy=0
\end{align*}

We also note that
\begin{align*}
u(dz)d[x,y]-u(d[y,x])dz=(f(z)-g(z))([dx,g(y)]+[f(x),dy])-([f(x)-g(x),g(y)]+[f(x),f(y)-g(y)])dz\\
=f(z)[dx,g(y)]-g(z)[dx,g(y)]+f(z)[f(x),dy]-g(z)[f(x),dy]+[g(x),g(y)]dz-[f(x),f(y)]dz=0
\end{align*}
So we have
\begin{align*}
dy[g(x),g(z)]+f(y)[dx,g(z)]+f(y)[f(x),dz]+dz[g(x),g(y)]+f(z)[dx,g(y)]+f(z)[f(x),dy]\\
=g(y)[dx,g(z)]+g(y)[f(x),dz]+[f(x),f(z)]dy+g(z)[dx,g(y)]+g(z)[f(x),dy]+[f(x),f(y)]dz
\end{align*}

\begin{align*}
g(y)[dx,g(z)]+g(z)[dx,g(y)]+dy[f(x),g(z)]+g(z)[f(x),dy]+dz[f(x),f(y)]+f(y)[f(x),dz]\\
-g(y)[dx,g(z)]-g(y)[f(x),dz]-[f(x),f(z)]dy-g(z)[dx,g(y)]-g(z)[f(x),dy]-[f(x),f(y)]dz\\
=dy[f(x),g(z)]+f(y)[f(x),dz]-g(y)[f(x),dz]-[f(x),f(z)]dy\\
=(f(y)-g(y))[f(x),dz]-[(f(x),(f(z)-g(z))]dy\\
=u(dy)[f(x),dz]-u([f(x),dz])dy=0
\end{align*}

Lastly since $da=0$ for all $a\in A$, $d$ is well-defined on $A\{x_i\}$, and $d$ is a Poisson biderivation and $u\circ d=f-g$.

\end{proof}

\begin{corollary}
Consider the following commutative diagram of Poisson homomorphisms
\begin{center}
$\begin{CD}
B@>b>> B'\\
@AAA @AAA\\
A@>a>>A'
\end{CD}$
\end{center}
Let $\mathscr{E}$ an free Poisson extension of $B$ over $A$ and $\mathscr{E}'$ be a Poisson extension of $B'$ over $A'$. Then there exists a homomrphim $\alpha:\mathscr{E}\to \mathscr{E}'$ extending $b$. If $\beta:\mathscr{E}\to \mathscr{E}'$ is any other homomorphism extending $b$, then $\bar{\alpha}$ and $\bar{\beta}$ are homotopic maps of $\mathcal{U}_{Pois(B')}\otimes_{\mathcal{U}_{Pois(B)}} PL^\bullet(\mathscr{E})\to PL^\bullet (\mathscr{E}')$.
\end{corollary}

\begin{definition}
Let $A\to B$ a Poisson homomoprhism, $\mathscr{E}$ be a free extension of $B$ over $A$, and $M$ be a Poisson $B$-module. We define $PT^i(B/A,M):=H^i(Hom_{\mathcal{U}_{Pois(B)}}(PL^{\bullet}(\mathscr{E}),M)), i=0,1,2$. Since $Hom_{\mathcal{U}_{Pois(B)}}(-,M)$ is an contravariant additive functor and any two Poisson cotangent complexes $PL^\bullet(\mathscr{E})$ and $PL^\bullet(\mathscr{F})$ induced from two free Poisson extensions $\mathscr{E}$ and $\mathscr{F}$ of $B$ over $A$ are homotopically equivalent, and so $Hom_{\mathcal{U}_{Pois(B)}}(PL^\bullet(\mathscr{E}),M)$ is homotopically equivalent to $Hom_{\mathcal{U}_{Pois(B)}}(PL^\bullet(\mathscr{F}),M)$. Hence $PT^i(B/A,M)$ is well-defined.
\end{definition}

\begin{proposition}
Let $A\to B$ be Poisson homomorphism of Poisson algebras over $k$. Then for $i=0,1,2$, $PT^i(B/A,\cdot)$ is a covariant, additive functor from the category of Poisson $B$-modules to category of $B$-modules. If $0\to M'\to M\to M''\to 0$ is a short exact sequence of Poisson $B$-modules, then there is a long exact sequence
\begin{align*}
0&\to PT^0(B/A,M')\to PT^0(B/A,M)\to PT^0(B/A,M'')\to\\
&\to PT^1(B/A,M')\to PT^1(B/A,M)\to PT^1(B/A,M'')\to\\
&\to PT^2(B/A,M')\to PT^2(B/A,M)\to PT^2(B/A,M'')
\end{align*}
\end{proposition}

\begin{proof}
By construction, $PT^i(B/A,\cdot)$ is a covariant additive functor. Let $\mathcal{E}:0\to U/U_0\to F/U_0\to R=A\{x_i\}\to B\to 0$ be a free Poisson extension of $B$ over $A$. Then $PL^0=\mathcal{U}_{Pois(B)}\otimes_{\mathcal{U}_{Pois(R)}}\Omega_{\mathcal{U}_{Pois(R)/A}}^1$ which is free $\mathcal{U}_{Pois(B)}$-module and $PL^1=\mathcal{U}_{Pois(B)}\otimes_{\mathcal{U}_{Pois(R)}} F$ which is a free $\mathcal{U}_{Pois(B)}$-module and $PL^2=U/U_0$. Let's consider the induced diagram.
\begin{center}
\tiny{$\begin{CD}
0@>>> Hom_{\mathcal{U}_{Pois(B)}}(PL^0,M')@>>>  Hom_{\mathcal{U}_{Pois(B)}}(PL^0,M')@>>> Hom_{\mathcal{U}_{Pois(B)}}(PL^0,M'')@>>> 0\\
@.@VVV @VVV @VVV\\
0@>>>Hom_{\mathcal{U}_{Pois(B)}}(PL^1,M')@>>> Hom_{\mathcal{U}_{Pois(B)}}(PL^1,M')@>>>Hom_{\mathcal{U}_{Pois(B)}}(PL^1,M')@>>>0\\\
@.@VVV@VVV @VVV\\
0@>>> Hom_{\mathcal{U}_{Pois(B)}}(PL^2,M')@>>> Hom_{\mathcal{U}_{Pois(B)}}(PL^2,M')@>>>  Hom_{\mathcal{U}_{Pois(B)}}(PL^2,M')
\end{CD}$}
\end{center}
First row and second row are exact since $PL^0$ and $PL^1$ are free, and third row is exact for the first two terms by the right exactness of $Hom_{\mathcal{U}_{Pois(B)}}(PL^2,\cdot)$. Hence we get the proposition.
\end{proof}

 \begin{proposition}\label{3pro}
For any Poisson homomorphism $A\to B$ and and Poisson $B$-module $M$, $PT^0(B/A,M)=Hom_{\mathcal{U}_{Pois(B)}}(\Omega_{\mathcal{U}_{Pois(B)/A}}^1,M)=PDer_A(B,M)$. In particular, $PT^0(B/A,B)=Hom_{\mathcal{U}_{Pois(B)}}(\Omega_{\mathcal{U}_{Pois(B)/A}}^1,B)=PDer_A(B,B)$
 \end{proposition}
\begin{proof}
Let's consider a exact sequence $0\to I\to A\{x_i\}\xrightarrow{\phi} B\to 0$ for some free Poisson algebra $A\{x_i\}$ over $A$ and $\phi$ is a Poisson homomorphism compatible with $A$. Then we have an exact sequence
\begin{align*}
\mathcal{U}_{Pois(B)}\otimes_{\mathcal{U}_{Pois(A\{x_i\})/A}} I\cong I/(I^2\oplus \{I,I\}) \to \mathcal{U}_{Pois(B)}\otimes_{\mathcal{U}_{Pois(A\{x_i\})}}\Omega_{\mathcal{U}_{Pois(A\{x_i\}/A)}}^1\to \Omega_{\mathcal{U}_{Pois(B)/A}}^1
\end{align*}
Let's choose a free Poisson $\mathcal{U}_{Pois(A\{x_i\})}$-module $F$ such that $F\to I\to 0$. Then we have an exact sequence
\begin{align*}
PL^1\to PL^0\to \Omega_{\mathcal{U}_{Pois(B)/A}}^1
\end{align*}
By taking $Hom_{\mathcal{U}_{Pois(B)}}(\cdot, M)$ which is a right exact functor, we have 
\begin{align*}
PT^0(B/A,M)=Hom_{\mathcal{U}_{Pois(B)}}(\Omega_{\mathcal{U}_{Pois(B)/A}}^1,M)=PDer_A(B,M)
\end{align*}
\end{proof}
 
\begin{proposition}\label{3prot}
If $A\to B$ is a surjective Poisson algebra homomorphism with kernel $I$, then $PT^0(B/A,M)=0$ for all $M$, and $PT^1(B/A,M)=Hom_{\mathcal{U}_{Pois(B)}}(I/I^2\oplus \{I,I\},M)$. In particular $PT^1(B/A,B)=Hom_{\mathcal{U}_{Pois(B)}}(I/I^2\oplus\{I,I\},B)$ 

\end{proposition}

\begin{proof}
We take $R=A$ as a free Poisson extension with no generating set over $A$. Since $\Omega_{\mathcal{U}_{Pois(A)/A}}^1=0$, we have $PT^0(B/A,M)=0$. Let's consider the exact sequence $0\to U\to F\xrightarrow{j} I\to 0$. Let $U_0$ be $\mathcal{U}_{Pois(R)}$-submodule of $F$ generated by $j(x)y-j(y)x$ and $X_{j(x)}y+X_{j(y)}x$, where $x,y \in F$. Then $0\to U/U_0\to F/U_0\to I\to 0$ is exact. So we have an exact sequence.
\begin{align*}
\mathcal{U}_{Pois(B)}\otimes_{\mathcal{U}_{Pois(A)}}U/U_0\to \mathcal{U}_{Pois(B)}\otimes_{\mathcal{U}_{Pois(A)}} F/U_0\to I/I^2\oplus \{I,I\}\to 0
\end{align*}
Hence we have $PT^1(B/A,M)=Hom_{\mathcal{U}_{Pois(B)}}(I/I^2\oplus \{I,I\}, M)$.
\end{proof}
 
\begin{proposition}\label{3pror}
Given a commutative diagram
\begin{center}
$\begin{CD}
B @>b>> B'\\
@AAA @AAA\\
A @>a>> A'
\end{CD}$
\end{center}
where every morphisms are Poisson homomorphisms. Let $M'$ be a Poisson $B'$-module. Then there are natural homomorphism
\begin{align*}
PT^i(B'/A',M')\to PT^i(B/A,M')
\end{align*}
\end{proposition}
\begin{proof}
Let $\mathscr{E}:0\to F_2\to F_1\to P\to B\to 0$ be a free Poisson extension of $B$ over $A$ and $\mathscr{E}:0\to F_2'\to F_1'\to P'\to B'\to 0$ be a free Poisson extension of $B'$ over $A'$. Then we have a homomorphism $\alpha:\mathscr{E}\to \mathscr{E}'$ extending $b:B\to B'$ and this induces $\bar{\alpha}:PL^\bullet (\mathscr{E})\to PL^\bullet(\mathscr{E}')$ (hence $\mathcal{U}_{Pois(B')}\otimes_{\mathcal{U}_{Pois(B)}} PL^\bullet (\mathscr{E})\to PL^\bullet(\mathscr{E}'))$ and so we have
\begin{center}
\small{$\begin{CD}
Hom_{\mathcal{U}_{Pois(B')}}(PL^{0'},M')@>>> Hom_{\mathcal{U}_{Pois(B')}}(PL^{1'},M')@>>> Hom_{\mathcal{U}_{Pois(B')}}(PL^{2'},M')\\
@VVV @VVV @VVV\\
Hom_{\mathcal{U}_{Pois(B)}}(PL^{0},M')@>>> Hom_{\mathcal{U}_{Pois(B)}}(PL^{1},M')@>>> Hom_{\mathcal{U}_{Pois(B)}}(PL^{2},M')
\end{CD}$}
\end{center}
So this induces $PT^i(B'/A',M')\to PT^i(B/A,M')$. The construction is independent of the choices of $\mathscr{E}$ and $\mathscr{E'}$. Indeed, Let $\mathscr{F}$ be another free Poisson extension of $B$ over $A$ and $\mathscr{F'}$ be another free Poisson extension of $B'$ over $A'$. Let $\bar{\beta}:PL^\bullet(\mathscr{F})\to PL^\bullet(\mathscr{F}')$ a homomorphism extending $B\to B'$ similarly to $PL^\bullet(\mathscr{E})\to PL^\bullet(\mathscr{E}')$ as above. Let $\bar{\gamma}:PL^\bullet (\mathscr{F})\to PL^\bullet(\mathscr{E})$ and $\bar{\gamma}':PL^\bullet(\mathscr{E}')\to PL^\bullet(\mathscr{F}')$ be homotopically equivalent maps. Then in the diagram
\begin{center}
$\begin{CD}
PL^\bullet(\mathscr{E})@>\bar{\alpha}>> PL^\bullet(\mathscr{E}')\\
@A\bar{\gamma}AA @VV\bar{\gamma}'V\\
PL^\bullet(\mathscr{F})@>\bar{\beta}>\bar{\gamma}'\circ \bar{\alpha}\circ \bar{\gamma}> PL^\bullet(\mathscr{F}')
\end{CD}$
\end{center}
$\bar{\beta}$ and $\bar{\gamma}'\circ \bar{\alpha}\circ \bar{\gamma}$ are hotomopic maps since $\mathscr{F}$ is a free Poisson extension. Hence the induced maps $PT^i(B'/A',M')\to PT^i(B/A,M)$ are equal.

\end{proof}

\begin{definition}[Short Poisson extension]
Let $A\to B$ a Poisson homomorphism of Poisson $k$-algebras, $k$ a field, and $M$ be a Poisson $B$-module. By a short Poisson extension of $B$ over $A$ by $M$, we mean an exact sequence:
\begin{align*}
0\to M\xrightarrow{i} E\xrightarrow{k} B\to 0
\end{align*}
where $A\to E$ is an Poisson homomorphism of Poisson algebras and $k:E\to B$ a Poisson homomorphism compatible with $A$ $($i.e $A\to B$ factor through $E$$)$ and $M$ is regarded as a square zero Poisson ideal in $E$, i.e $i(M)\cdot i(M)=0$ and $\{i(M),i(M)\}=0$. We note that a short Poisson extension is a Poisson extension since $M$ is a $\mathcal{U}_{Pois(E)}$-module $($the $\mathcal{U}_{Pois(E)}$-module structure is induced from $\mathcal{U}_{Pois(B)}$-module structure via $k$$)$, and so $i(x)y-i(y)x$ and $X_{i(x)}y+X_{i(y)}x$ are trivially $0$ because $ki(x)=0$ for all $x\in M$. Hence $0\to 0\to M\to E\to B\to 0$ is a Poisson extension. If $E'$ is another short Poisson extension of $B$ over $A$ by $M$, we say that $E$ and $E'$ are equivalent if there exists a Poisson homomorphism $\theta:E\to E'$ compatible with $A$ inducing the following commutative diagram
\begin{center}
$\begin{CD}
0@>>>  M@>i>>E@>k>>B @>>> 0\\
@. @|@V\theta VV @| @.\\
0@>>> M@>i'>> E' @>k'>>B@>>>0
\end{CD}$
\end{center}
\end{definition}

\begin{definition}
We define $PEx^1(B/A,M)$ be the set of equivalence classes of short Poisson extensions of $B$ over $A$ by $M$.
\end{definition}

\begin{lemma}
Let $0\to M\xrightarrow{i}E\xrightarrow{k}B\to0$ be a short Poisson extension of $B$ over $A$ by $M$ as above. Given an Poisson homomorphism $A\to C$ and two Poisson homomorphisms $f_1,f_2:C\to E$ compatible with $A$ such that $kf_1=k f_2$, the induced map $f_2-f_1:C\to M$ is an Poisson $A$-derivation. 
\end{lemma}

\begin{proof}
We assume that $i$ is an inclusion. First we note that $M$ is an Poisson $B$-module. We define $C$-module structure on $M$ by setting $c\cdot m:=f_1(c) m=f_2(c)m$ which is well-defined since $M^2=0$. We define an Poisson $C$-module structure on $M$ by setting $\{c,m\}:=\{f_1(c),m\}=\{f_2(c),m\}$ which is well-defined since $\{M,M\}=0$.
Cleary $f_2-f_1$ is $A$-linear. We note that
\begin{align*}
(f_2-f_1)(c_1c_2)&=f_2(c_1)f_2(c_2)-f_1(c_1)f_1(c_2)
\\&=f_2(c_1)f_2(c_2)-f_2(c_1)f_1(c_2)+f_2(c_1)f_1(c_2)-f_1(c_1)f_1(c_2)\\
&=f_2(c_1)(f_2(c_2)-f_1(c_2))+f_1(c_2)(f_2(c_1)-f_1(c_1))
\\&=c_1\cdot(f_2-f_1)(c_1)+c_2\cdot (f_2-f_1)(c_1)\\
(f_2-f_1)(\{c_1,c_2\})&=\{f_2(c_1),f_2(c_2)\}-\{f_1(c_1),f_1(c_2)\}\\
&=\{f_2(c_1),f_2(c_2)\}-\{f_2(c_1),f_1(c_2)\}+\{f_2(c_1),f_1(c_2)\}-\{f_1(c_1),f_1(c_2)\}\\
&=\{f_2(c_1),(f_2-f_1)(c_2)\}+\{(f_2-f_1)(c_1),f_1(c_2)\}\\
&=\{c_1,(f_2-f_1)(c_2)\}-\{c_2,(f_2-f_1)(c_1)\}
\end{align*}

\end{proof}

\begin{definition}
The short Poisson extension of $B$ over $A$ by $M:0\to M\xrightarrow{i}E\xrightarrow{k}B\to0$ is called trivial if it has a section, that is if there exists a Poisson homomorphism $\sigma:B\to E$ such that $k\sigma=1_B$ and $\sigma$ is compatible with $A$.
\end{definition}

Given an Poisson $B$-module $M$, a trivial short Poisson extension of $B$ over $A$ by $M$ can be constructed by considering the Poisson algebra $B\tilde{\oplus} M$ whose underlying $A$-module is $B\oplus M$ with multiplication and bracket defined by
\begin{align*}
(b_1,m_1)(b_2,m_2)&=(b_1b_2,b_1m_2+b_2m_1)\\
\{(b_1,m_1),(b_2,m_2)\}&=(\{b_1,b_2\},-\{b_2,m_1\}+\{b_1,m_2\})
\end{align*}

The first projection
\begin{align*}
p:B\tilde{\oplus} M\to B
\end{align*}
is a Poisson homomorphism compatible with $A$ and defines an Poisson extension of $B$ over $A$ by $M$.

A section of $p$ can be identified with a Poisson $A$-derivations $d:B\to M$. Indeed, if we have a section $\sigma:B\to B\tilde{\oplus} M$ with $\sigma(b)=(b,d(b))$, then for all $b,b'\in B$
\begin{align*}
\sigma(bb')&=(bb',d(bb'))=\sigma(b)\sigma(b')=(b,d(b))(b',d(b'))=(bb',bd(b')+b'd(b))\\
\sigma(\{b,b'\})&=(\{b,b'\},d(\{b,b'\})=\{\sigma(b),\sigma(b')\}=\{(b,d(b)),(b',d(b'))\}=(\{b,b'\},-\{b',d(b)\}+\{b,d(b')\})
\end{align*}

and if $a\in A$ then $\sigma(ab)=(ab,d(ab))=a\sigma(b)=a(b,d(b))=(ab,ad(b))$. Hence $d:B\to M$ is a Poisson $A$-derivation. Conversely, every Poisson $A$-derivation $d:B\to M$ defines a section $\sigma_d:B\to B\tilde{\oplus} M$ by $\sigma_d(b)=(b,d(b))$.

\begin{proposition}
Every trivial short Poisson extension $E$ of $B$ by $M$ is isomorphic to $(B\tilde{\oplus} M,p)$. 
\end{proposition}
\begin{proof}
If $\sigma:B\to E$ is a section, an isomorphism $\xi:B\tilde{\oplus} M\to E$ is given by $\xi((b,m))=\sigma(b)+m$. We check only Poisson compatibility. $\xi(\{(b_1,m_1),(b_2,m_2)\})=\sigma(\{b_1,b_2\})-\{b_2,m_1\}+\{b_1,m_2\}=\{\sigma(b_1),\sigma(b_2)\}-\{b_2,m_1\}+\{b_1,m_2\}$. On the other hand, $\{(\xi((b_1,m_1)),\xi((b_2,m_2))\}=\{\sigma(b_1)+m_1,\sigma(b_2)+m_2\}=\{\sigma(b_1),\sigma(b_2)\}-\{b_2,m_1\}+\{b_1,m_2\}$. We define an inverse map $\xi^{-1}(e')=(k(e'),e'-\sigma k(e'))$.
\end{proof}

Let $\mathscr{E}:0\to M\to E\to B\to 0$ be a short Poisson extension of $B$ over $A$ by $M$ as above. Let $h:A\to E$ be the Poisson homomorphism from $\mathscr{E}$. Then we have by Proposition \ref{3prot},
\begin{align*}
PT^1(B/E,M)=Hom_{\mathcal{U}_{Pois(B)}}(M/M^2\oplus \{M,M\},M)=Hom_{\mathcal{U}_{Pois(B)}}(M,M).
\end{align*} 
Then by Proposition \ref{3pror}, $h$ induces 
\begin{align*}
h^*:=Hom_{\mathcal{U}_{Pois(B)}}(M,M)=PT^1(B/E,M)\to PT^1(B/A,M)
\end{align*} 

\begin{theorem}\label{3thm}
The assignment $\mathscr{E} \to h^*(id)$ induces a bijection
\begin{align*}
\rho:PEx^1(B/A,M)\to PT^1(B/A,M)
\end{align*}
in which the class of the trivial short Poisson extension corresponds to $0$.
\end{theorem}

\begin{proof}
Let
\begin{align*}
\mathscr{F}:0\to F_2\to F_1\to P\to B\to 0
\end{align*}
be a fixed free Poisson extension of $B$ over $A$. Given an short extension $\mathscr{E}:0\to M\to E\to B\to 0$, which is also an Poisson extension, we have a homomorphism $\alpha:\mathscr{F}\to \mathscr{E}$ extending the identity $B\to B$, unique up to homotopy,
\begin{center}
$\begin{CD}
0@>>> 0@>>> M@>j>> E@>k>> B@>>> 0\\
@. @AA\alpha_2A @AA\alpha_1A @AA\alpha_0A @AAidA @.\\
0@>>> F_2@>\bar{i}>> F_1@>\bar{j}>> P@>\bar{k}>> B @>>>0
\end{CD}$
\end{center}
 Let's consider the cotangent complex of $B$ over $A$
\begin{align*}
0\to F_2\xrightarrow{\bar{i}}\mathcal{U}_{Pois(B)}\otimes_{\mathcal{U}_{Pois(R)}} F_1\xrightarrow{d\circ\bar{j}} \mathcal{U}_{Pois(B)}\otimes_{\mathcal{U}_{Pois(P)}} \Omega_{\mathcal{U}_{Pois(P)/A}}^1\to 0
\end{align*}
Then we have
\begin{align*}
Hom_{\mathcal{U}_{Pois(P)}}( \Omega_{\mathcal{U}_{Pois(P)/A}}^1,M)\to Hom_{\mathcal{U}_{Pois(P)}}(F_1,M)\to Hom_{\mathcal{U}_{Pois(B)}}(F_2,M)
\end{align*}
The class of $\alpha_1:F_1\to M$ is $h^*(id)$. Let $\mathscr{E}':0\to M\to E'\to B\to 0$ be an equivalent short Poisson extension with $\mathscr{E}$ given by $\theta:E\to E'$. The we have a homomorphism $\mathscr{F}\to \mathscr{E}'$ extending the identity $B\to B$,
\begin{center}
$\begin{CD}
0@>>> 0@>>> M@>j'>> E@>k'>> B@>>> 0\\
@. @AA\alpha_2A @AA\alpha_1A @AA\theta \circ\alpha_0A @AAidA @.\\
0@>>> F_2@>\bar{i}>> F_1@>\bar{j}>> P@>\bar{k}>> B @>>>0
\end{CD}$
\end{center}
We also have the same $\alpha_1:F_1\to M$. Hence $\rho$ is well-defined.

Now we define the inverse map $\rho^{-1}$ of $\rho$ in the following way: given $e\in PT^1(B/A,M)$, choose $f:F_1\to M$ inducing $e$. We put $J=(P\oplus M)/K$ where $K$ is the Poisson ideal of $(P\oplus M)$ generated by the elements of the form $(\bar{j}(x),-f(x))$ for $x\in F_1$. We note that we have actually $K=\{(\bar{j}(x),-f(x))|x\in F_1\}$. Indeed, let $(p,m)\in  P\oplus M$. Then $(p,m)\cdot (\bar{j}(x),-f(x))=(p\cdot\bar{j}(x),-pf(x)+\bar{j}(x)m)=(\bar{j}(px),-f(px))$ since $M$ is a Poisson $P$-module via $P\xrightarrow{\bar{k}} B$, so $\bar{j}(x)m$ means $\bar{k}(\bar{j}(x))m=0$. On the other hand $\{(p,m),(\bar{j}(x),-f(x))\}=(\{p,\bar{j}(x)\}_P,-\{p,f(x)\}_M-\{\bar{j}(x), m\}_M)=(\bar{j}(\{p,x\}_{F_1}),f(\{p,x\}_{F_1})$ since $\{\bar{j}(x),m\}_M$ means $\{\bar{k}\bar{j}(x),m\}_M=0$. Hence $K=\{(\bar{j}(x),-f(x))|x\in F_1\}$. 

Now we claim that the following sequence is a short Poisson extension of $B$ over $A$ by $M$,
\begin{align*}
\mathscr{E}:0\to M\xrightarrow{j} (P\oplus M)/K\xrightarrow{k} B\to 0
\end{align*}
where $j(m):=$ the class of $(0,m)=\overline{(0,m)}$, and $k(\overline{(p,m)}):=\bar{k}(p)\in B$. $k$ is well-defined and a surjective Poisson map since for $(p,m)=(\bar{x},-f(x))\in K=\{(\bar{j}(x),-f(x))|x\in F_1\}$, $\bar{k}(p)=\bar{k}(\bar{j}(x))=0$. Now let $k(\overline{(p,m)})=\bar{k}(p)=0$. Then there exists $x\in F_1$ such that $\bar{j}(x)=p$. Then $(p,m)-(0,m+f(x))=(p,-f(x))=(\bar{j}(x),-f(x))\in K$. Hence $j(m+f(x))=\overline{(0,m+f(x))}=\overline{(p,m)}$. Hence $ker(k)\subset im(j)$ and clearly $im(j)\subset ker(k)$. Hence $im(j)=ker(k)$. Now we show that $j$ is injective. Let $\overline{(0,m)}=0$. Then $(0,m)\in K$. So we have $0=\bar{j}(x)$ and $m=-f(x)$ for some $x\in F_1$. Hence $x=\bar{i}(y)$ for some $y\in F_2$. Note that under the map $Hom_{\mathcal{U}_{Pois(P)}}( F_1,M)\to Hom_{\mathcal{U}_{Pois(B)}}(F_2,M)$, $f$ goes to $0$ since $f$ defines the cohomology class $e$, and so $f\circ \bar{i}=0$. Hence $m=-f(x)=-f(\bar{i}(y))=0$. So $j$ is injective. We have the following commutative diagram.
\begin{center}
$\begin{CD}
0@>>> 0@>>> M@>j>> (P\oplus M) /K@>k>> B@>>> 0\\
@. @AAA @AAfA @AA\alpha_0'A @AAidA @.\\
0@>>> F_2@>\bar{i}>> F_1@>\bar{j}>> P@>\bar{k}>> B @>>>0
\end{CD}$
\end{center}
where $\alpha_0'(p)=\overline{(p,0)}$. Indeed, let $x\in F_1$. $\alpha_0'(\bar{j}(x))=\overline{(\bar{j}(x),0)}$ and $j(f(x))=\overline{(0,f(x))}$. $(\bar{j}(x),0)-(0,f(x))=(\bar{j}(x),-f(x))\in K$. Thus $e=$ the class of $f=h^*(id)$.

Now we show that $\rho^{-1}$ is well-defined. In other words, $\rho^{-1}(e)$ is independent of choices of $f$ inducing $e\in T^1(B/A,M)$ up to equivalence of short Poisson extensions. Let $f,f': F_1\to M$ inducing $e$. Then there exists $v:\Omega_{\mathcal{U}_{Pois(P)/A}}^1\to M$ such that $f'-f=v\circ d\circ \bar{j}$ where $F_1\xrightarrow{\bar{j}}P\xrightarrow{d}\Omega_{\mathcal{U}_{Pois(P)/A}}^1\xrightarrow{v} M$. Let $\mathscr{E}'$ be the Poisson extension constructed from $f':F_1\to M$ as above,
\begin{align*}
0\to M\xrightarrow{j'} (P\oplus M)/K'\xrightarrow{k'} B\to 0
\end{align*}
where $K'$ is the Poisson ideal $\{(\bar{j}(x),-f'(x))|x\in F_1\}$. Consider an endomorphism $P\oplus M\to P\oplus M$ defined by $\varphi:(p,m)\mapsto (p,-v( d(p))+m)$, which is one to one and onto with the inverse $\varphi^{-1}(p,m)=(p,v(d(p))+m)$. We show that $\varphi$ is a Poisson homomorphism. Indeed,
\begin{align*}
\varphi((p_1,m_1)(p_2,m_2))=&\varphi(p_1p_2,p_1m_2+p_2m_1)=(p_1p_2,-v(p_1dp_2+p_2dp_1)+p_1m_2+p_2m_1)\\
&=(p_1p_2,p_1(-v(dp_2)+m_2)+p_2(-v(dp_1)+m_1))\\
&=(p_1,-v(dp_1)+m_1)(p_2,-v(dp_2)+m_2)\\
&=\varphi(p_1,m_1)\varphi(p_2,m_2)\\
\varphi(\{(p_1,m_1),(p_2,m_2)\})&=\varphi((\{p_1,p_2\}_P,\{p_1,m_2\}_M-\{p_2,m_1\}_M)\\
&=(\{p_1,p_2\}_P, -v(-\{p_2,dp_1\}_M+\{p_1,dp_2\}_M)+\{p_1,m_2\}_M-\{p_2,m_1\}_M)\\
&=(\{p_1,p_2\}_P,-\{p_2,m_1-v(dp_1)\}_M+\{p_1,m_2-v(dp_2)\}_M)\\
&=(\{(p_1,m_1-v(dp_1)),(p_2,m_2-v(dp_2))\})\\
&=\{\varphi(p_1,m_1),\varphi(p_2,m_2)\}
\end{align*}
On the other hand, 
\begin{align*}
\varphi((\bar{j}(x),-f(x))&=(\bar{j}(x), -v(d(\bar{j}(x)))-f(x))=(\bar{j}(x),-f'(x))\\
\varphi^{-1}(-\bar{j}(x),f'(x))&=(-\bar{j}(x), -v(d(\bar{j}(x)))+f'(x))=(-\bar{j}(x),f(x))
\end{align*} for $x\in F_1$. Hence $\varphi$ maps $K$ to $K'$. Hence $\varphi$ induces an isomorphism $(P\oplus M)/K\to (P\oplus M)/K'$. Hence $\mathscr{E}$ is equivalent to $\mathscr{E}'$. Hence $\rho^{-1}$ is well-defined.

Lastly we show that the class of trivial short Poisson extension $0\to M\to B\tilde{\oplus} M\to B\to 0$ corresponds to $0\in PT^1(B/A,M)$ via $\rho$. In the following diagram
\begin{center}
$\begin{CD}
0@>>> 0@>>> M@>>> B\tilde{\oplus} M @>p>>B@>>> 0\\
@. @AAA @AA\alpha_1A @AA\alpha_0A @AAidA @.\\
0@>>> F_2@>\bar{i}>> F_1@>\bar{j}>> P@>\bar{k}>> B @>>>0
\end{CD}$
\end{center}
Let $q: B\tilde{\oplus} M\to M$ be the projection. Then $q\circ \alpha_0:P\to M$ be a Poisson $A$-derivation with $\alpha_1=q\circ \alpha_0\circ \bar{j}$. Hence there exists a map $v:\Omega_{\mathcal{U}_{Pois(P)/A}}^1\to M$ such that $q\circ \alpha_0=v\circ d$. So we have $\alpha_1=v\circ d\circ \bar{j}$. Hence the class of $\alpha_1$ is $0$.
\end{proof}

\section{First order Poisson deformations of affine Poisson schemes}

Let's explain in more detail a short Poisson extension of $R$ over $A$ by $I$ : $0\to I\xrightarrow{j} R'\xrightarrow{\phi} R\to0$, where $I$ is a Poisson $R$-module with $j(I)\cdot j(I)=0$ and $\{j(I),j(I)\}=0$. Here $\phi$ is a Poisson homomorphism of Poisson $k$-algebras compatible with a Poisson algebra $A$. Then we assume that $I$ is a Poisson $R'$-module via $\phi$, and so $j$ is a Poisson $R'$-module homomorphism.

When we say that a short Poisson extension of $R$ over $A$ by $R$ means that we have an exact sequence $0\to R\xrightarrow{j} R'\xrightarrow{\phi} R\to0$, where $R$ is a natural Poisson $R$-module, the image of $j$ satisfies $j(R)^2=0$ and $\{j(R),j(R)\}=0$, $R$ also have a $R'$-module structure via $\phi$. By these $R'$-module structure, and $j$ is a Poisson $R'$- module homomorphism. Now we show that this extension gives a $k[\epsilon]$-Poisson algebra structure on $R'$. Note that $j$ is completely determined by $1$ since $j(f)=j(f\cdot 1)=j(f'\cdot 1)=f'\cdot j(1)$ where $f'$ is a lift of $f$. We will give a $k[\epsilon]$-algebra structure on $R'$ by $\epsilon\to j(1)$ since $j(1)^2=0$. To show $k[\epsilon]$-Poisson algebra structure on $R'$, we have to show that $\{R', \epsilon\}=0$, equivalently $\{r,j(1)\}=0$ for all $r\in R'$. Since $j$ is a $R'$-module homomorphism, $\{r,j(1)\}_{R'}=j(X_r\cdot 1)=j(\{\phi(r),1\}_R)=0$.

\begin{proposition}[compare $\cite{Har10}$ Theorem 5.1]\label{3prr}
Let $B_0$ be a Poisson $k$-algebra, and let $X_0=Spec(B_0)$. Then there is a natural isomorphism 
\begin{align*}
PDef_{B_0}(k[\epsilon])\cong PEx^1(B_0/k,B_0)
\end{align*}
where the class of trivial Poisson deformation corresponds to $0\in PT_{B_0}^1$.
\end{proposition}

\begin{proof}
A first order Poisson deformation of $B_0$ consists of a flat Poisson $k[\epsilon]$-algebra $B$ with Poisson $k$-isomorphism $B\otimes_{k[\epsilon]} k \cong B_0$ with the following commutative diagram
\begin{center}
$\begin{CD}
B@>\phi>> B_0\\
@AAA @AAA\\
k[\epsilon]@>>> k
\end{CD}$
\end{center}
where $\phi$ is a Poisson homomorphism over $k$. We note that a algebra $B$ is flat over $k[\epsilon]$ if and only if $0\to B_0\otimes_{k} (\epsilon)\cong B_0 \to B$ is exact. (see \cite{Har10} Proposition 2.2). So given a first order Poisson deformation of $B_0$, we have an exact sequence $ 0\to B_0\xrightarrow{j=\epsilon} B\xrightarrow{\phi} B/\epsilon B \cong B_0\to 0$, where the first map $\epsilon(b_0):=\epsilon \cdot b$, where $b$ is a lifting of $b_0$ via $\phi$. This is a short Poisson extension of $B_0$ over $k$ by $B_0$ since $B_0$ is a Poisson $B$-module and $\epsilon(B_0)^2=\{\epsilon(B_0),\epsilon(B_0)\}=0$ and the induced $B_0$-module structure on $B_0$ via $\phi$ is given by the multiplication on $B_0$. Let $B'$ be an equivalent Poisson deformation of $B_0$ over $k[\epsilon]$ with the first order Poisson deformation $B$. Then we would like to show that  the following diagram commutes
\begin{center}
$\begin{CD}
0@>>> B_0@>j>> B@>\phi>> B_0@>>> 0\\
@. @| @V\cong \Phi VV @| \\
0  @>>> B_0  @>j'>> B'@>\phi'>> B_0 @>>>0
\end{CD}$
\end{center}
where $\Phi$ is a Poisson homomorphism over $k[\epsilon]$ defining equivalent first order Poisson deformations of $B_0$.
Since right diagram commutes by definition of equivalence, we only check that the left diagram commutes. $j(b_0)=\epsilon b$ where $\phi(b)=b_0$. Since $\phi'(\Phi(b))=b_0$, $j'(b_0)=\epsilon \Phi(b)$. Hence the diagram commutes. So equivalent first order Poisson deformations corresponds to equivalent short Poisson extensions $B_0$ over $k$ by $B_0$.

Conversely, let $0\to B_0\xrightarrow{j} B\xrightarrow{\phi} B_0 \to 0$ be a Poisson extension. Then $B$ is a Poisson $k[\epsilon]$-algebra by the above discussion. In this case, we can identify $B_0$ with $j(B_0)=\epsilon B=B\otimes_{k[\epsilon]} (\epsilon)$. Hence $B$ is flat over $k[\epsilon]$ and $B/\epsilon B=B\otimes_{k[\epsilon]} k\cong B_0$. Since $\epsilon B$ is a Poisson ideal ($\phi$ is a Poisson map), $\phi$ induces a Poisson isomorphism $B\otimes_{k[\epsilon]} k\cong B_0$. Given a equivalent Poisson extension, since the right diagram in the above commutes, it gives an equivalent Poisson deformations of $B_0$.

Let $B_0\otimes_k k[\epsilon]=B_0\oplus \epsilon B_0$ be the trivial Poisson deformation of $B_0$. Then the associated Poisson extension is trivial $0\to B_0\to B_0\oplus \epsilon B_0 \to B_0$:
\end{proof}

\begin{corollary}
Let $B_0$ be a Poisson $k$-algebra. Then the set of first order Poisson deformations of $B_0$ is in natural one to one correspondence $PT^1(B_0/k,B_0)$.
\begin{align*}
PDef_{B_0}(k[\epsilon])\cong PT^1(B_0/k,B_0)
\end{align*}
\end{corollary}

\begin{proof}
By Theorem \ref{3thm} and Proposition \ref{3prr}, we have $PDef_{B_0}(k[\epsilon])\cong PEx^1(B_0/k,B_0)\cong PT^1(B_0/k,B_0)$.
\end{proof}

\section{First order deformations of a Poisson closed subscheme of an affine Poisson scheme}

Let $X$ be a Poisson scheme over $k$ and let $Y$ be a closed Poisson subscheme of $(X,\Lambda_0)$. We will define a Poisson deformation of $Y$ over $Spec\, k[\epsilon]$ in $X$ to be a Poisson subscheme $Y'\subset (X\times Spec(k[\epsilon]),\Lambda_0\oplus 0)$ where $(X\times Spec(k[\epsilon]),\Lambda_0\oplus 0)$  is the trivial Poisson deformation of $X$ over $Spec(k[\epsilon])$, $Y'\times_{Spec \,k[\epsilon]} k=Y$ and $Y'$ is flat over $Spec(k[\epsilon])$. 

We discuss deformations of Poisson subschemes when $X$ is an affine Poisson scheme. Then $X$ corresponds to a Poisson $k$-algebra $(B,\Lambda_0)$, and $Y$ is defined by an Poisson ideal $I\subset B$. We would like to find Poisson ideals of $k[\epsilon]$-Poisson algebra $(B'=B\otimes_k k[\epsilon],\Lambda_0\oplus 0)$ such that the image of $I'$ in $B=B'/\epsilon B'$ is $I$ and $B'/I'$ is flat over $k[\epsilon]$. We note that the flatness of $B'/I'$ over $k[\epsilon]$ is equivalent to the exactness of $0\to B/I\xrightarrow{\epsilon} B'/I'\to B/I\to 0$ where $\epsilon$ is the multiplication by $\epsilon$ if and only if $0\to I\xrightarrow{\epsilon} I'\to I\to 0$ is exact. (see \cite{Har10} page 11)

\begin{proposition}[compare $\cite{Har10}$ Proposition 2.3]\label{3p}
To give a Poisson ideal $I'\subset B'=B\oplus \epsilon B$ such that $B'/I'$ is flat over $k[\epsilon]$ and the image of $I'$ in $B$ is $I$ is equivalent to an element 
\begin{align*}
\varphi\in Hom_{\mathcal{U}_{Pois(B)}}(I,B/I)=Hom_{\mathcal{U}_{Pois(B/I)}}(I/I^2\oplus \{I,I\},B/I). 
\end{align*}
In particular, $\varphi=0$ corresponds to the trivial deformation given by $I'=I\oplus \epsilon I$ inside $B'=B\oplus \epsilon B$. \end{proposition}

\begin{proof}
$B'=B\oplus \epsilon B$ is naturally a Poisson $B$-algebra in the following way: $a\cdot (b+\epsilon c)=ab+\epsilon ac$ and $\{a,b+\epsilon c\}=\{a,b\}+\epsilon\{a,c\}$. Now let $I'$ be a Poisson ideal of $B'=B\oplus \epsilon B$ such that $B'/I'$ is flat over $k[\epsilon]$ and $\pi(I')=I$, where $\pi:B\oplus \epsilon B\to B$ be the projection. Let $x\in I$ and choose a lifting $x'$ of $x$ via $\pi$. Then $x'=x+\epsilon y\in I'$ for some $y\in B$. Let $x''=x+\epsilon y'\in I'$ be an another lifting of $x$. Then $y-y'\in I$ by the flatness of $B'/I'$ over $k[\epsilon]$. So the image $\bar{y}$ in $B/I$ is uniquely determined. So $\varphi:I\to B/I$, $x\to \bar{y}$ is well-defined. We claim that this is a Poisson $B$-module homomorphism. Indeed, let $\varphi(x)=\bar{y}$. Since $x+\epsilon y\in I'$, we have $bx+\epsilon by\in I'$ for $b\in B$. So  $\varphi(b\cdot x)=\overline{by}=b\bar{y}=b\cdot\varphi(y)$. On the other hand, since $\{b,x\}+\epsilon\{b,y\}\in I'$, we have $\varphi(\{b,x\})=\overline{\{b,x\}}=\{b,\bar{x}\}=\{b,\varphi(x)\}$.

Conversely, let $\varphi\in Hom_{\mathcal{U}_{Pois(B)}}(I,B/I)$. Define
\begin{align*}
I'=\{x+\epsilon y|x\in I,y\in B, \,\,\text{the image of $\bar{y}$ of $y$ in $B/I$ is equal to $\varphi(x)$}\}
\end{align*}
We claim that $I'$ is a Poisson ideal of $B'$. Let $x+\epsilon y\in I'$ and $a+\epsilon b\in B'$. Then $x\in I$ and $\varphi(x)=\bar{y}$. Since $(a+\epsilon b)(x+\epsilon y)=ax+\epsilon(bx+ay)$ and $\overline{bx+ay}=\overline{ay}$, we have $\varphi(ax)=a\varphi(x)=\overline{ay}$. On the other hand, since $\{a+\epsilon b, x+\epsilon y\}=\{a,x\}+\epsilon(\{b,x\}+\{a,y\})$ and $\overline{\{b,x\}+\{a,y\}}=\overline{\{a,y\}}$, we have  $\varphi(\{a,x\})=\overline{\{a,y\}}=\{a,\varphi(x)\}$. We have a natural exact sequence $0\to I \xrightarrow{\epsilon} I'\to I\to 0$, where $\epsilon$ means multiplication by $\epsilon$. Since the exactness means the exactness of $0\to B/I\xrightarrow{\epsilon} B'/I'\to B/I\to 0$, $B'/I'$ is a flat over $k[\epsilon]$.

These two construction is one to one correspondence. When $\varphi=0\in Hom_{\mathcal{U}_{Pois(B)}}(B,B/I)$, we have $I'=I\oplus \epsilon I$.
\end{proof}

\begin{corollary}
Let $B_0$ be a Poisson $k$-algebra and $I$ be a Poisson ideal of $B_0$. Let $C=B_0/I$. Then the set of first order deformations of Poisson closed subscheme $Spec(B_0/I)$ of an affine Poisson scheme $Spec(B_0)$ is in natural one to one correspondence with $PT^1(C/B_0,C)$.
\end{corollary}

\begin{proof}
This follows from Proposition \ref{3prot} and Proposition \ref{3p}
\end{proof}

\begin{remark}\label{3remarks}
If our construction of Poisson cotangent complex may turn out to be correct and use right languages, there is a ``globalization'' problem, which I cannot solve at this point. In $\cite{Sch67}$, Lichtenbaum and Schlessinger actually constructed  a quasi-coherent sheaf $\mathcal{T}^i (X/Y,\mathcal{F})$ for a morphism of schemes $f:X\to Y$ and a quasi-coherent sheaf $\mathcal{F}$ of $\mathcal{O}_X$-module where $X$ is separated, $Y$ is noetherian and $f$ is locally of finite type. I could not show that $PT^i(B/A,M)$ commutes with localization and so that we can define a sheaf $\mathcal{PT}^i(X/Y,\mathcal{F})$ where $f:X\to Y$ is a morphism of Poisson schemes satisfying suitable finiteness conditions and $\mathcal{F}$ is a quasi-cohernt Poisson $\mathcal{O}_X$-module.
\end{remark}

\appendix
\chapter{Basic materials on Poisson algebras and holomorphic Poisson manifolds}\label{appendixa}

In this section, we present basic facts about holomorphic Poisson manifolds relevant to our discussions.  Our reference is \cite{Lau13}.

\begin{definition}
A commutative $\mathbb{C}$-algebra $A$ with identity is called a Poisson algebra if there is a Poisson bracket $\{-.-\}$ such that  
\begin{enumerate}
\item $(A,\{-,-\})$ is a Lie algebra over $\mathbb{C}$
\item The multiplications are compatible in the sense that
\begin{align*}
\{a\cdot b,c\}=a\cdot \{b,c\}+ b\cdot \{a,c\}
\end{align*}
for any $a,b,c\in A$
\end{enumerate}
\end{definition}

\begin{definition}
Let $(A,\{-,-\}_1)$ and $(B,\{-,-\}_2)$ be Poisson algebras. A $\mathbb{C}$-algebra homomorphism $\phi:A\to B$ is called a Poisson map if it is compatible with Poisson brackets
\begin{align*}
 \phi(\{a,b\})=\{\phi(a),\phi(b)\}
\end{align*}
for any $a,b\in A$.
\end{definition}

\begin{definition}
A holomorphic Poisson manifold $M$ is a complex manifold such that the structure sheaf is a sheaf of Poisson algebra. In other words, for any open set $U\in M$, $\mathcal{O}_M(U)$ is a Poisson algebra and for any open set $V\subset U$, the restriction map $\mathcal{O}_M(U)\to \mathcal{O}_M(V)$ is a Poisson map.
\end{definition}

We recall a Schouten bracket, denoted by $[-,-]_{Sch}$, on holomorphic polyvector fields $\bigoplus_{i\geq 0} H^0(M,\wedge^i T_M)$ on $M$.

\begin{proposition}
Let $M$ be a $n$-dimensional holomorphic Poisson manifold. Then there exist a holomorphic bivector field $\Lambda \in H^0(M,\wedge^2 T_M)$ with $[\Lambda,\Lambda]_{Sch}=0$ and $\{-,-\}$ is defined in the following way: let $U=(z_1,...,z_n)$ be a coordinate neighborhood of $M$ and $\Lambda=\sum_{i,j=1}^n \Lambda_{ij}(z)\frac{\partial}{\partial z_i}\wedge \frac{\partial}{\partial z_j}$ on $U$ , then for 
\begin{align*}
(*)\{f,g\}=\Lambda(df,dg)=\Lambda(\partial f ,\partial g )=\sum_{i,j=1}^n \Lambda_{ij}(z)\left(\frac{\partial f}{\partial z_i}\frac{\partial g}{\partial z_j}-\frac{\partial g}{\partial z_i}\frac{\partial f}{\partial z_j}\right)
\end{align*}
for any holomorphic functions $f,g\in \mathcal{O}_M(U)$.

Conversley, a holomorphic bivector field $\Lambda$ in $H^0(M,\wedge^2 T_M)$ with $[\Lambda,\Lambda]_{Sch}=0$ makes $M$ a holomorphic Poisson manifold by $(*)$.
\end{proposition}

We denote a holomorphic Poisson manifold by $(M,\Lambda)$.

\begin{definition}
A holomorphic map $f:(M,\Lambda)\to (M,\Lambda')$ between holomorphic Poisson manifolds is called a Poisson map if for all open sets $U\subset M$ and $V\subset N$ with $f(U)\subset V$, the induced map (defined by pullback) $f^*:\mathcal{O}_N(V)\to \mathcal{O}_M(U)$ is a Poisson map.
\end{definition}

We recall that for each $p\in M$, we have a linear $f_{*}:T_p M\to T_{f(p)} N$. This extends to $f_{*}:\wedge^2 T_p M\to \wedge^2 T_{f(p)} N$.
\begin{proposition}
Let $f:(M,\Lambda) \to (N,\Lambda')$ be a holomorphic map. Then $f$ is a Poisson map if and only if 
\begin{align*}
f_{*} \Lambda_p=\Lambda'_{f(p)}
\end{align*}
for all $p \in M$.
\end{proposition}

\begin{definition}
Let $(M,\Lambda)$ be a holomorphic Poisson manifold and $N$ be a complex submanifold of $M$. Assume that $N$ is a holomorphic Poisson manifold with a holomorphic bivector field $\Lambda'$. Then $(N,\Lambda')$ is called a Poisson submanifold if the inclusion map $i:(N,\Lambda')\to (M,\Lambda)$ is a Poisson map. Hence $\Lambda'$ is uniquely determined by $\Lambda$ by restricting $\Lambda$ to $N$.
\end{definition}

\begin{proposition}\label{box}
Let $(M,\Lambda)$ be a holomorphic Poisson manifold and $N$ be a complex submanifold of $M$. Then the following are equivalent.
\begin{enumerate}
\item $N$ is a holomorphic Poisson submanifold of $(M,\Lambda)$.
\item The ideal sheaf $\mathcal{I}_N\subset \mathcal{O}_M$ defined by $\mathcal{I}_N(U)=\{f \in \mathcal{O}_M(U)|f|_{N\cap U}=0\}$ for open sets $U\subset M$ is a Poisson ideal of $\mathcal{O}_M(U)$. That is, $\mathcal{I}_N(U)$ is an ideal under the Poisson bracket$:$ if $f\in \mathcal{I}_M(U)$ and $g\in \mathcal{O}_N(U)$, then $\{f,g\}\in \mathcal{I}_M(U)$.
\item For every $p\in N$, the bivector $\Lambda_p$ belongs to $\wedge^2 T_p N$.
\end{enumerate}
\end{proposition}

\begin{definition}
If $(M,\Lambda)$ and $(N,\Lambda)$ are holomorphic Poisson manifolds, then we can make $M\times N$ a holomorphic Poisson manifold $(M\times N,\Lambda\oplus \Lambda')$ induced from $\Lambda$ and $\Lambda'$ in the following way. Let $U=(z_1,...,z_n)$ and $V=(w_1,...,w_m)$ be coordinate neighborhoods of $M$ and $N$ respectively. Then $U\times V=(z_1,..,z_n,w_1,...,w_m)$ is a coordinate neighborhood of $M\times N$. Let $\Lambda=\sum_{i,j} g_{ij}(z)\frac{\partial}{\partial z_i}\wedge \frac{\partial}{\partial z_j}$ and $\Lambda'=\sum_{r,s} h_{rs}(w)\frac{\partial}{\partial w_r}\wedge \frac{\partial}{\partial w_s}$. We define 
\begin{align*}
\Lambda\oplus\Lambda'=\sum_{i,j=1}^n g_{ij}(z)\frac{\partial}{\partial z_i}\wedge \frac{\partial}{\partial z_j}+\sum_{r,s=1}^m h_{rs}(w)\frac{\partial}{\partial w_r}\wedge \frac{\partial}{\partial w_s}
\end{align*}
Then $\Lambda\oplus \Lambda'$ is a holomorphic bivector field on $M$ and $[\Lambda\oplus\Lambda',\Lambda\oplus \Lambda']_{Sch}=0$.
\end{definition}

\chapter{Hypercohomology}\label{appendixb}

In this appendix, we present the basic materials on hypercohomology relevant to our discussion. Our reference is \cite{EV92} Appendix. In this section we fix $M$ as a complex manifold. We denote $T_M=T$ by the holomorphic tangent bundle of $M$ and $\Theta_M=\Theta$ be the sheaves of germs of holomorphic section of $T$.

\begin{definition}
The map $\sigma:\mathcal{F}^{\bullet}\to \mathcal{I}^{\bullet}$ between two complexes of sheaves of $\mathbb{C}$-module on $M$ is called an injective resolution of $\mathcal{F}^{\bullet}$ if $\mathcal{I}^{\bullet}$ is a complex of $\mathbb{C}$-module bounded below, $\sigma$ is a quasi isomorphism, and the sheaves $\mathcal{I}^i$ are injective for all $i$.
\end{definition}

\begin{remark}
Every complex of $\mathbb{C}$-modules on $M$ which is bounded below admits an injective resolution.
\end{remark}

\begin{definition}
Let $\mathcal{F}^{\bullet}$ be a complex of $\mathbb{C}$-modules on $M$ which is bounded below. Then the hypercohomology group $\mathbb{H}^a(M,\mathcal{F})$ is defined to be the $\mathbb{C}$-module
\begin{align*}
\mathbb{H}^a(M,\mathcal{F}^{\bullet}):=\frac{ker\,\,\,\Gamma(M,\mathcal{I}^a)\to \Gamma(M,\mathcal{I}^{a+1})}{im\,\,\,\Gamma(M,\mathcal{I}^{a-1})\to\Gamma(M,\mathcal{I}^a)}
\end{align*}
\end{definition}

\begin{remark}
This definition does not depend on the choice of injective resolution.
\end{remark}

\begin{proposition} \label{alpha}
If $\sigma:\mathcal{F}^{\bullet}\to \mathcal{G}^{\bullet}$ is a quasi-isomorphism and if $H^a(M,\mathcal{G}^i)=0$ for all $a>0$ and all $i$, then 
\begin{align*}
\mathbb{H}^a(M,\mathcal{F}^{\bullet})=\frac{ker\,\,\, \Gamma(M,\mathcal{G}^a)\to \Gamma(X,\mathcal{G}^{a+1})}{im\,\,\,\Gamma(M,\mathcal{G}^{a-1})\to \Gamma(M,\mathcal{G}^a)}
\end{align*}
We call $\mathcal{G}^{\bullet}$ an acyclic resolution of $\mathcal{F}^{\bullet}$.
\end{proposition}

\begin{example}\label{rr}
Let $(M,\Lambda)$ be a $n$-dimensional compact holomorphic Poisson manifold. Let $\mathcal{F}^{\bullet}$ be the complex of sheaves $0\to \Theta\to \wedge^2 \Theta\to\cdots\to \wedge^n \Theta\to 0$ induced by $[\Lambda,-]$. We denote by $\mathscr{A}^{0,p}(\wedge^q T)$ the sheaf of germs of $C^{\infty}$-section of $\wedge^p \bar{T}^*\otimes \wedge^q T$ and by $A^{0,p}(M,\wedge^q T)$ the global section of $\mathscr{A}^{0,p}(\wedge^q T)$. Let's consider the following bicomplex of sheaves. 
\begin{center}
$\begin{CD}
@A[\Lambda,-]AA\\
\mathscr{A}^{0,0}(\wedge^3 T)@>\bar{\partial}>>\cdots\\
@A[\Lambda,-]AA @A[\Lambda,-]AA\\
\mathscr{A}^{0,0}(\wedge^2 T)@>\bar{\partial}>> \mathscr{A}^{0,1}(\wedge^2 T)@>\bar{\partial}>>\cdots\\
@A[\Lambda,-]AA @A[\Lambda,-]AA @A[\Lambda,-]AA\\
\mathscr{A}^{0,0}(T)@>\bar{\partial}>>\mathscr{A}^{0,1}(T)@>\bar{\partial}>>\mathscr{A}^{0,2}(T)@>\bar{\partial}>>\cdots\\
@AAA @AAA @AAA @AAA \\
0@>>>0 @>>> 0 @>>> 0@>>> \cdots
\end{CD}$
\end{center}
Each rows is a resolution of $\wedge^i \Theta$. Hence $\mathcal{F}^{\bullet}$ is quasi-isomorphic to the total complex of the above bicomplex of sheaves. Hence by Proposition \ref{alpha}, the hypercohomology $\mathbb{H}^a(M,\mathcal{F}^{\bullet})$ is the $a$-th cohomology of the total complex of the following bicomplex
\begin{center}
$\begin{CD}
@A[\Lambda,-]AA\\
A^{0,0}(M,\wedge^3 T)@>\bar{\partial}>>\cdots\\
@A[\Lambda,-]AA @A[\Lambda,-]AA\\
A^{0,0}(M,\wedge^2 T)@>\bar{\partial}>> A^{0,1}(M,\wedge^2 T)@>\bar{\partial}>>\cdots\\
@A[\Lambda,-]AA @A[\Lambda,-]AA @A[\Lambda,-]AA\\
A^{0,0}(M,T)@>\bar{\partial}>>A^{0,1}(M,T)@>\bar{\partial}>> A^{0,2}(M,T)@>\bar{\partial}>>\cdots\\
@AAA @AAA @AAA @AAA \\
0@>>>0 @>>> 0 @>>> 0@>>> \cdots
\end{CD}$
\end{center}
\end{example}

Now we will consider a \u{C}ech resolution of a complex of sheaves. Let $\mathcal{U}=\{U_{\alpha}:\alpha\in A\}$, for $A\in \mathbb{N}$ be a locally finite open covering of a complex manifold or some open covering of an algebraic scheme $X$. For
\begin{align*}
U_{\alpha_0...\alpha_a}:=U_{\alpha_0}\cap\cdots \cap U_{\alpha_a}\,\,\,\,\,\alpha_0<\alpha_1<\cdots<\alpha_a.
\end{align*}
$\rho$ denotes the open embedding
\begin{align*}
\rho=\rho_{\alpha_0...\alpha_a}:U_{\alpha_0\cdots\alpha_a}\to X,
\end{align*}

To a bounded below complex $\mathcal{F}^{\bullet}$ we associates its \u{C}ech complex $\mathcal{G}^{\bullet}$ such that
\begin{align*}
\mathcal{G}^i :=\bigoplus_{a\geq 0} \mathcal{C}^a(\mathcal{U},\mathcal{F}^{i-a})
\end{align*}
where 
\begin{equation*}
\mathcal{C}^a(\mathcal{U},\mathcal{F}^{i-a})=\Pi_{\alpha_0<\alpha_1<\cdots<\alpha_a} \rho_{*}(\mathcal{F}^{i-1}|_{U_{\alpha_0\cdots \alpha_a}}).
\end{equation*}

The differential $\Delta$ of $\mathcal{G}^{\bullet}$ is defined by
\begin{align*}
\Delta(s)=(-1)^i \delta s + d_{\mathcal{F}^{\bullet}} s \,\,\,\,\, s\in \mathcal{C}^a(\mathcal{U},\mathcal{F}^{i-1}).
\end{align*}
and $d_{\mathcal{F}^{\bullet}}$ is the differential of $\mathcal{F}^{\bullet}$. Then the natural map 
\begin{align*}
\sigma:\mathcal{F}^{\bullet}\to \mathcal{G}^{\bullet}
\end{align*}
defined by
\begin{equation*}
\mathcal{F}^i\xrightarrow{\rho} \Pi_{\alpha\in A} \rho_{*}(\mathcal{F}^i|_{U_{\alpha}})= \mathcal{C}^0(\mathcal{U},\mathcal{F}^i)
\end{equation*}
is a quasi isomorphism.

\begin{example}
Now let $(M,\Lambda)$ be a $n$-dimensional compact holomorphic Poisson manifold. let $\mathcal{F}^{\bullet}$ be the complex of sheaves $0\to \Theta\to \wedge^2 \Theta\to\cdots\to \wedge^n \Theta\to 0$ induced by $[\Lambda,-]$. And assume that for each $U_j \in \mathcal{U}= \{ U_{\alpha}:\alpha\in A\}$, we have
\begin{align*}
U_j=\{z_j\in \mathbb{C}^n| |z_j^{\alpha}|<r_j^{\alpha},\alpha=1,...,n\}
\end{align*}
Then by Proposition \ref{alpha}, the hypercohomology $\mathbb{H}^a(M,\mathcal{F}^{\bullet})$ is the $a$-th cohomology of the total complex of the following bicomplex 
\begin{center}
$\begin{CD}
@A[\Lambda,-]AA\\
C^0(\mathcal{U},\wedge^3 \Theta)@>-\delta>>\cdots\\
@A[\Lambda,-]AA @A[\Lambda,-]AA\\
C^0(\mathcal{U},\wedge^2 \Theta)@>\delta>> C^1(\mathcal{U},\wedge^2 \Theta)@>-\delta>>\cdots\\
@A[\Lambda,-]AA @A[\Lambda,-]AA @A[\Lambda,-]AA\\
C^0(\mathcal{U},\Theta)@>-\delta>>C^1(\mathcal{U},\Theta)@>\delta>>C^2(\mathcal{U},\Theta)@>-\delta>>\cdots\\
@AAA @AAA @AAA @AAA \\
0@>>>0 @>>> 0 @>>> 0@>>> \cdots
\end{CD}$
\end{center}
where $C^i(\mathcal{U},\wedge^j \Theta)=H^0(M,\mathcal{C}^i(\mathcal{U},\wedge^j \Theta))=\bigoplus_{\alpha_0<\cdots <\alpha_i} H^0(\mathcal{U}_{\alpha_0...\alpha_i},\wedge^j \Theta)$.

\end{example}

\begin{example}
Now let $(X,\Lambda)$ be an algebraic Poisson scheme over $k$, where $k$ is an algebraically closed field and $\Lambda\in \Gamma(X,\mathscr{H}om_{\mathcal{O}_X}(\wedge^2\Omega_{X/k}^1,\mathcal{O}_X))$ with $[\Lambda,\Lambda]=0$ $($for the definition, see chapter $\ref{chapter7}$$)$. Let $\mathcal{F}^\bullet$ be the complex of sheaves
\begin{align*}
0\to \mathscr{H}om_{\mathcal{O}_X}(\Omega_{X/k}^1,\mathcal{O}_X)\xrightarrow{[\Lambda,-]} \mathscr{H}om_{\mathcal{O}_X}(\wedge^2 \Omega^1_{X/k},\mathcal{O}_X)\xrightarrow{[\Lambda,-]}\mathscr{H}om_{\mathcal{O}_X}(\wedge^3\Omega^1_{X/k},\mathcal{O}_X)\xrightarrow{[\Lambda,-]}\cdots
\end{align*}
which is the $($truncated$)$ Lichnerowicz-Poisson complex of sheaves of $(X,\Lambda_0)$ $($See Definition \ref{3def}$)$. And assume that for each $U_j\in \mathcal{U}=\{U_{\alpha},\alpha\in A\}$, $U_j$ is affine. Then by Remark $\ref{3remark2}$ and Proposition $\ref{alpha}$, the hypercohomology $\mathbb{H}^a(X,\mathcal{F}^\bullet)$ is the $a$-th cohomology of the total complex of the following bicomplex
\begin{center}
$\begin{CD}
@A[\Lambda,-]AA\\
C^0(\mathcal{U},\mathscr{H}om_{\mathcal{O}_X}(\wedge^3 \Omega_{X/k},\mathcal{O}_X)))@>-\delta>>\cdots\\
@A[\Lambda,-]AA @A[\Lambda,-]AA\\
C^0(\mathcal{U},\mathscr{H}om_{\mathcal{O}_X}(\wedge^2 \Omega_{X/k},\mathcal{O}_X)))@>\delta>> C^1(\mathcal{U},\mathscr{H}om_{\mathcal{O}_X}(\wedge^2 \Omega_{X/k},\mathcal{O}_X)))@>-\delta>>\cdots\\
@A[\Lambda,-]AA @A[\Lambda,-]AA @A[\Lambda,-]AA\\
C^0(\mathcal{U},\mathscr{H}om_{\mathcal{O}_X}(\Omega_{X/k}^1,\mathcal{O}_X)))@>-\delta>>C^1(\mathcal{U},\mathscr{H}om_{\mathcal{O}_X}(\Omega_{X/k}^1,\mathcal{O}_X)))@>\delta>>C^2(\mathcal{U},\mathscr{H}om_{\mathcal{O}_X}(\Omega_{X/k}^1,\mathcal{O}_X)))\\
@AAA @AAA @AAA \\
0@>>>0 @>>> 0 
\end{CD}$
\end{center}
where 
\begin{align*}
C^i(\mathcal{U},\mathscr{H}om_{\mathcal{O}_X}(\wedge^j \Omega_{X/k}^1,\mathcal{O}_X))&=H^0(X,\mathcal{C}^i(\mathcal{U},\mathscr{H}om_{\mathcal{O}_X}(\wedge^j \Omega_{X/k}^1,\mathcal{O}_X)))\\&=\bigoplus_{\alpha_0<\cdots <\alpha_i} H^0(\mathcal{U}_{\alpha_0...\alpha_i},\mathscr{H}om_{\mathcal{O}_X}(\wedge^j \Omega_{X/k}^1,\mathcal{O}_X)).
\end{align*}
\end{example}

\chapter{Differential graded Lie algebra structure on $\bigoplus_{i\geq 0} g^i=\bigoplus_{p+q-1=i, q \geq1} A^{0,p}(M,\wedge^q T_M)$}\label{appendixc}

 In this section, we fix $M$ as a $n$-dimensional complex manifold. Let $F=\bar{T}^*\oplus T$ be the direct sum of antiholomorphic cotangent bundle and holomorphic tangent bundle on $M$. We consider the bundle $\bigoplus_{i\geq 1} F^i=\bigoplus_{i\geq 1} \wedge^i F$ where $F^i=\bigoplus_{p+q=i, p,q\geq 0} \wedge^p \bar{T}^*\otimes \wedge^q T$. We denote by $\mathcal{F}$ the sheaf of germs of $C^{\infty}$ section of $F$, by $\mathcal{F}^i$ by the sheaf of germs of $C^{\infty}$ section of $ \bigoplus_{p+q=i, p,q\geq 0} \wedge^p \bar{T}^*\otimes \wedge^q T$ and  by $\mathscr{A}^{0,p}(\wedge^q T)$ the sheaf of germs of $C^{\infty}$-section of $\wedge^p \bar{T}^*\otimes \wedge^q T$. Then we have $\mathcal{F}=\bigoplus_{i\geq 1} \mathcal{F}^i= \bigoplus_{p+q=i, p,q\geq 0} \mathscr{A}^{0,p}(\wedge^q T)$. If we denote by $A^{0,p}(M,\wedge^q T)$ the global section of $\mathscr{A}^{0,p}(\wedge^q T)$, then the global section of $\mathcal{F}$ is $\bigoplus_{p+q=i, p,q\geq 0} {A}^{0,p}(M,\wedge^q T)$. We set $A^i:=\bigoplus_{p+q=i} A^{0,p}(M,\wedge^q T)$ and $A=:\bigoplus_{i\geq0}  A^i=\bigoplus_{p+q=i, p,q\geq 0} {A}^{0,p}(M,\wedge^q T)$. Then $A$ is a graded vector space over $\mathbb{C}$. In this appendix, we discuss the differential Gerstenhaber algebra structure on $A$ and, by shifting the degree 1, a differential grade Lie algebra structure on $A[1]=\bigoplus_{i\geq 0} g^i=\bigoplus_{p+q-1=i} A^{0,p}(M,\wedge^q T_M)$. Our main references are \cite{Mac05} and \cite{Man04}.

\begin{definition}

A differential Gerstenhaber algebra $(B,[,],\wedge,\bar{\partial})$ is the data of a graded vector space $B=\sum_{i\in \mathbb{Z}} B^i$ equipped with a bilinear bracket $[-,-]:B\times B\to B$, a wedge product $\wedge$  and a linear map $\bar{\partial}:B\to B$ satisfying the following properties $($here $\bar{x}$ is the grading of the homogeneous element $x$.$)$

\begin{align}
& [B^i,B^j]\subset B^{i+j-1}.\\
&  [a,b]=(-1)^{\bar{a}\bar{b}+\bar{a}+\bar{b}} [b,a] \text{for homogeneous elements $a,b\in B$}. \label{y} 
\end{align}

\begin{align}
&  [a,[b,c]]=[[a,b],c]-(-1)^{\bar{a}\bar{b}+\bar{a}+\bar{b}}[b,[a,c]] \text{ for homogeneous elements $a,b,c \in B$}.\\
&  B^i\wedge B^j\subset B^{i+j}.\\
&  a\wedge b=(-1)^{\bar{a}\bar{b}}b\wedge a\\
&  [a\wedge b,c]=a\wedge [b,c]+(-1)^{\bar{a}\bar{b}} b\wedge [a,c] \label{z} \\ 
&  \bar{\partial} B^i\subset B^{i+1}.\\
&  \bar{\partial}\circ \bar{\partial}=0.\\
&  \bar{\partial}[a,b]=[\bar{\partial}a,b]-(-1)^{\bar{a}}[a,\bar{\partial}b]\\
&  \bar{\partial}(a\wedge b)=\bar{\partial}a \wedge b+(-1)^{\bar{a}}a\wedge \bar{\partial}b.\label{x}
\end{align}

\end{definition}

We discuss a differential Gerstenhaber algebra structure on $A=\bigoplus_{p+q=i, p,q\geq 0} A^{0,p}(M,\wedge^q T)$.
\begin{enumerate}
\item Let $f\in A^{0,0}(M,\mathcal{O}_M)$, $g_i\frac{\partial}{\partial z_i}\in A^{0,0}(M,T)$ and $h_jd\bar{z}^j\in A^{0,1}(M,\mathcal{O}_M)$. We define $\bar{\partial}$ in the following way.
\begin{enumerate}
\item  $\bar{\partial}f:=\frac{\partial f}{\partial \bar{z}^i}d\bar{z}^i$
\item  $\bar{\partial} (g_i\frac{\partial}{\partial z_i}):=\frac{\partial g_i}{\partial \bar{z}^k}d\bar{z}^k\wedge \frac{\partial}{\partial z^i}$
\item $\bar{\partial} (h_j d\bar{z}^j):=\frac{\partial h_j}{\partial \bar{z}^k}d\bar{z}^k\wedge d\bar{z}^j$
\end{enumerate}
These definitions are independent of coordinate transformations. By the rule \ref{x}, we extend $\bar{\partial}$ to $A$.
\item We define $[-,-]$ in the following way
\begin{enumerate}
\item $[g_i\frac{\partial}{\partial z_i},f]:= g_i\frac{\partial f}{\partial z_i}=L_{g_i\frac{\partial}{\partial z_i}} f$ (Lie derivative of $f$ in the direction of $g_i\frac{\partial}{\partial z_i}$).
\item $[g_i\frac{\partial}{\partial z_i},a_k\frac{\partial}{\partial z_k}]:=g_i\frac{\partial a_k}{\partial z_i}-a_k\frac{\partial g_i}{\partial z_k}$
\item $[g_i\frac{\partial}{\partial z_i},h_jd\bar{z}_j]:=g_i\frac{\partial h_j}{\partial z_i}d\bar{z}_j$
\item $[h_jd\bar{z}_j,f]:=0$
\item $[h_jd\bar{z}_j,b_ld\bar{z}_l]:=0$
\end{enumerate}
These definitions are independent of coordinate transformations. By the rule \ref{y} and \ref{z}, we extend $[-.-]$ to $A$.
\end{enumerate}
Then $(A,\wedge,\bar{\partial})$ is a differential Gerstenhaber algebra and have the following property: For $v_i=f_{ik}\frac{\partial}{\partial z_k}$ and $w_j=f'_{jk}\frac{\partial}{\partial z_k}$, then we have
\begin{equation*}
[v_1\wedge\cdots\wedge v_n,w_1\wedge \cdots w_m]=\sum_{i=1}^n \sum_{j=1}^m (-1)^{i+j} [v_i,w_j]\wedge v_1\wedge\cdots \wedge \hat{v}_i\wedge\cdots \wedge v_n\wedge w_1\wedge\cdots \wedge \hat{w}_j\wedge\cdots \wedge w_m
\end{equation*}

\begin{remark}
 polyvector fields $A'=\oplus_{i\geq 0} A^{0,0}(M,\wedge^i T)$ is a subalgebra of $L$. When we restrict the bracket $[-,-]$ on $A'$, we get the Schouten-Nijenhuis bracket $[-,-]_{Sch}$.
\end{remark}

More generally, for $\phi_p=g_{pk}d\bar{z}_k$ and $\psi_q=g'_{qk}d\bar{z}_k$, we have
\begin{align*}
&[\phi_1\wedge\cdots \wedge \phi_l\wedge v_1\wedge\cdots\wedge v_n,\psi_1\wedge \cdots \wedge \psi_k\wedge w_1\wedge \cdots \wedge w_m]\\
&=\phi_1\wedge\cdots\wedge \phi_l\wedge [v_1\wedge\cdots,v_n,\psi_1\wedge\cdots\wedge \psi_k]\wedge w_1\wedge\cdots \wedge w_m\\
&+(-1)^{(l+n)(k+m)+(l+n)+(k+m)}\psi_1\wedge\cdots \wedge \psi_k\wedge [w_1\wedge\cdots\wedge w_m,\phi_1\wedge\cdots\wedge \phi_l]\wedge v_1\wedge\cdots\wedge v_n\\
&+(-1)^{k(n+1)}\phi_1\wedge\cdots \wedge \phi_l \wedge \psi_1\wedge \cdots \wedge \psi_k\wedge[v_1\wedge\cdots\wedge v_n,w_1\wedge\cdots \wedge w_m]
\end{align*}
In particular, it is practical to have the following formula

\begin{align*}
[fdz_I\frac{\partial}{\partial z_J},gdz_K\frac{\partial}{\partial z_L}]&=(-1)^{|K|(|J|+1)} dz_I\wedge dz_K [f\frac{\partial}{\partial z_J},g\frac{\partial}{\partial z_L}]_{Sch}\\
[f\frac{\partial}{\partial z_I}d\bar{z}_J,g\frac{\partial}{\partial z_H}d\bar{z}_K]&=(-1)^{|J|(|H|-1)}[f\frac{\partial}{\partial z_I},g\frac{\partial}{\partial z_H}]_{Sch} d\bar{z}_J\wedge d\bar{z}_K
\end{align*}

Now we discuss a differential graded Lie algebra structure on $A[1]$. 
\begin{definition}
A differential graded Lie algebra $(C,[-,-],\bar{\partial})$ is the data of a graded vector space $C=\oplus_{i\in \mathbb{Z}} C^i$ over $\mathbb{C}$ together with a bilinear bracket $[-,-]:C\times C \to C$ and a linear map $\bar{\partial}:C\to C$ satisfying the following properties
\begin{enumerate}
\item $[C^i,C^j]\subset C^{i+j}$
\item $[c,d]=-(-1)^{\bar{c}\bar{d}}[d,c]$ for homogeneous elements $c,d$ $($here $\bar{x}$ is the grading of the homogeneous element $x$$)$
\item $[a,[b,c]]=[[a,b],c]+(-1)^{\bar{a}\bar{b}}[b,[a,c]]$ for homogeneous elements $a,b,c$.
\item $\bar{\partial} C^i\subset C^{i+1}$
\item $\bar{\partial}\circ \bar{\partial}=0$
\item $\bar{\partial}[a,b]=[\bar{\partial}a,b]+(-1)^{\bar{a}}[a,\bar{\partial}b]$
\end{enumerate}

\end{definition}

It is clear that if $L$ is a differential Gerstenhaber algebra, then $L[1]$ \footnote{$L[1]$ is the algebra $L$ but the grading is shifted by $1$. Hence $L[1]^i =L^{i+1}$} is a differential graded lie algebra. Hence $(A[1],[-,-],\bar{\partial})$ is a differential graded lie algebra and so satisfies the following properties. But we have to mention that on $A[1]$, we have other differential graded Lie algebra structures by changing the differential into $L=\bar{\partial}+[\Lambda,-]$, where $\Lambda$ is a holomorphic bivector field such that $[\Lambda,\Lambda]=0$.(i.e $L\Lambda=0$). 

\begin{proposition}\label{d}
$(A[1],L,[-,-])$ is a differential graded Lie algebra.
\end{proposition}

\begin{proof}
Let $a\in A[1]$. We only need to see the properties $(4),(5)$ and $(6)$. $(4)$ is clear by definition, and we have $L \circ L =0$ by simple computation. For $(6)$, note that by plugging $\Lambda$ into $a$ in $(3)$ and combining $(6)$ in the definition, we get $L[a,b]=[La,b]+(-1)^{\bar{a}}[a,Lb]$.
\end{proof} 

In holomorphic Poisson deformations there is no role of structure sheaf $\mathcal{O}_M$, we define the following sub differential graded Lie algebra of $(A[1],L,[-,-])$ on a holomorphic Poisson manifold $(M,\Lambda)$.

\begin{definition}
$(\mathfrak{g}=\bigoplus_{i\geq 0} g^i=\bigoplus_{p+q-1=i, q \geq1} A^{0,p}(M,\wedge^q T_M),L,[-,-])$.
\end{definition}

\begin{definition}
The Maurer-Cartan equation of a differential graded Lie algebra $(C,[-,-],\bar{\partial})$ is
\begin{align*}
\bar{\partial}a+\frac{1}{2}[a,a]=0,\,\,\,\,\, a\in C^1
\end{align*}
The solutions $MC(L)\subset C^1$ of the Maurer-Cartan equation are called the Maurer Cartan elements of the differential graded Lie algebra $C$.
\end{definition}

\chapter{Ellipticity of the operator $\bar{\partial} +[\Lambda,-]$}\label{appendixd}

Let $(M,\Lambda)$ be a holomorphic Poisson manifolds. In this appendix, we discuss the operator $\bar{\partial}+[\Lambda,-]$ on $M$. Our main reference is \cite{Wel08}. Let $\pi:E\to X$ be a differentiable complex vector bundle on a differentiable manifold. For an open set $U\subset X$, we denote $C^{\infty}$ functions on $U$ by $\mathcal{E}(U)$, $C^{\infty}$ section of $E$  on $X$ by $\mathcal{E}(U,E)$ and $E_x:=\pi^{-1}(x)$, $x\in X$ by a $\mathbb{C}$ vector space fiber over $x$.

\begin{definition}
Let $E$ and $F$ be differentiable complex vector bundles over a differentiable manifold $X$. Let $L:\mathcal{E}(X,E)\to \mathcal{E}(X,F)$ be linear. We say that $L$ is a differentiable operator if for any choice of local coordinates and local trivializations, there exists a linear partial differential operator $\tilde{L}$ such that the diagram
\begin{align*}
\begin{CD}
[\mathcal{E}(U)]^p@>\tilde{L}>>[\mathcal{E}(U)]^q\\
@A\cong AA @AA\cong A \\
\mathcal{E}(U,U\times \mathbb{C}^p) @>>> \mathcal{E}(U,U\times \mathbb{C}^q)\\
@A\bigcup AA @AA\bigcup A\\
\mathcal{E}(X,E)|_U @>L>> \mathcal{E}(X,F)|_U
\end{CD}
\end{align*}
commutes. That is, for $f=(f_1,...,f_p)\in [\mathcal{E}(U)]^p$
\begin{align*}
\tilde{L}(f)_i=\sum_{j=1,|\alpha|\leq k}^p a_{\alpha}^{ij}D^{\alpha}f_j,\,\,\,\,\, i=1,...,q.
\end{align*}
A differential operator is said to be of order $k$ if there are no derivatives of order $\geq k+1$ appearing in a local representation. Let Diff$_k(E,F)$ denote the vector space of all differential operators of order $k$ mapping $\mathcal{E}(X,E)$ to $\mathcal{E}(X,F)$. 
\end{definition}

\begin{example}
If $(M,\Lambda_0)$ is a holomorphic Poisson manifold, then 
\begin{center}
$L=\bar{\partial} +[\Lambda_0, -]:A^{0,p-1}(M,T)\oplus \cdots \oplus A^{0,0}(M,\wedge^p T)\to A^{0,p}(M,T)\oplus A^{0,0}(M,\wedge^{p+1} T)$
\end{center}
is a differential operator of order $1$.
\end{example}

\begin{definition}
Let $L:\mathcal{E}(X,E)\to \mathcal{E}(X,F)$ is differential operator of order $k$. Let $T_{\mathbb{R}}^*(X)$ be the real cotangent bundle to a differentiable manifold $X$, let $T_{\mathbb{R}}'(X)$ denote $T_{\mathbb{R}}^*(X)$ with the zero section deleted (the bundle of nonzero cotangent vectors), and let $T_{\mathbb{R}}'(X)\xrightarrow{\pi} X$ denote the projection mapping. Then $\pi^*E$ and $\pi^* F$ denote the pullback of $E$ and $F$ over $T_{\mathbb{R}}'(X)$. We set, for any $k\in \mathbb{Z}$,
\begin{align*}
Smbl_k(E,F):=\{ \sigma \in Hom(\pi^* E,\pi^* F):\sigma(x,\rho v)=\rho^k \sigma(x,v),(x,v) \in T_{\mathbb{R}}'(X),\rho>0\}
\end{align*}
We now define a linear map
\begin{align*}
\sigma_k:\text{Diff}_k(E,F)\to Smbl_k(E,F)
\end{align*}
where $\sigma_k(L)$ is called the $k$-symbol of the differential operator $L$. To define $\sigma_k(L)$, we first note that $\sigma_k(L)(x,v)$ is to be a linear mapping from $E_x$ to $F_x$, where $(x,v)\in T_{\mathbb{R}}'(X)$. Therefore let $(x,v)\in T_{\mathbb{R}}'(X)$ and $e\in E_x$ be given. Find $g\in \mathcal{E}(X)$ and $f\in \mathcal{E}(X,E)$ such that $dg_x=v$, and $f(x)=e$. The we define
\begin{align*}
\sigma_k(L)(x,v)e=L\left(\frac{i^k}{k!}(g-g(x))^kf\right)(x)\in F_x.
\end{align*}
Then this defines a linear mapping
\begin{align*}
\sigma_k(L)(x,v):E_x \to F_x
\end{align*}
which then defines an element of $Symbl_k(E,F)$ and $\sigma_k(L)$ is independent of the choices made. We call $\sigma_k(L)$ the $k$-symbol of $L$.
\end{definition}

\begin{example}
Let $(M,\Lambda)$ be a holomorphic Poisson manifold.
Here we compute a symbol of 
\begin{center}
$L=[\Lambda, -]:A^{0,p}(M,\wedge^q T)\to A^{0,p}(M,\wedge^{q+1} T)$
\end{center}
First we note that $L$ has order $1$. Let $(x,v)\in T_{\mathbb{R}}' M$ and $e\in (\wedge^p \overline{T}^*\otimes \wedge^q T)_x$ be given. Find $g\in \mathcal{E}(M)$ and $f\in A^{0,p}(M,\wedge^q T)$ such that $dg_x=v$, and $f(x)=e$. Let $\Lambda=\sum_{i,j} \Lambda_{ij}(z) \frac{\partial}{\partial z_i}\wedge \frac{\partial}{\partial z_j}$ around $x$ and $f=\sum_{r_a,s_b} f_{r_1...r_q s_1...s_p}(z)\frac{\partial}{\partial z_{r_1}}\wedge \cdots\wedge \frac{\partial}{\partial z_{r_q}}\wedge d\bar{z}_{s_1}\cdots \wedge d\bar{z}_{s_p}$ around $x$. Then

\begin{align*}
&\sigma_1(L)(x,v)e=L(i(g-g(x))f)(x)=iL((g-g(x))f)(x)\\
&=i\sum_{r_a,s_b,i,j}[\Lambda_{ij}(z)\frac{\partial}{\partial z_i}\wedge \frac{\partial}{\partial z_j}, (g-g(x)) f_{r_1...r_q s_1...s_p}(z)\frac{\partial}{\partial z_{r_1}}\wedge \cdots\wedge \frac{\partial}{\partial z_{r_q}}\wedge d\bar{z}_{s_1}\cdots \wedge d\bar{z}_{s_p}](x)\\
&=i\sum_{r_a,s_b,i,j}[\Lambda_{ij}(z)\frac{\partial}{\partial z_i}\wedge \frac{\partial}{\partial z_j}, (g-g(x)) f_{r_1...r_q s_1...s_p}(z)\frac{\partial}{\partial z_{r_1}}\wedge \cdots\wedge \frac{\partial}{\partial z_{r_q}}]\wedge d\bar{z}_{s_1}\cdots \wedge d\bar{z}_{s_p}(x)\\
&=i\sum_{r_a,s_b,i,j}  [\Lambda_{ij}(z)\frac{\partial}{\partial z_i},(g-g(x)) f_{r_1...r_q s_1...s_p}(z) \frac{\partial}{\partial z_{r_1}}]\wedge \frac{\partial}{\partial z_j}\wedge \frac{\partial}{\partial z_{r_2}}\cdots\wedge \frac{\partial}{\partial z_{r_q}}\wedge d\bar{z}_{s_1}\cdots \wedge d\bar{z}_{s_p}(x)\\
&-\Lambda_{ij}(z)[\frac{\partial}{\partial z_j}, (g-g(x)) f_{r_1...r_q s_1...s_p}(z) \frac{\partial}{\partial z_{r_1}}] \wedge \frac{\partial}{\partial z_i}\wedge \frac{\partial}{\partial z_{r_2}}\cdots\wedge \frac{\partial}{\partial z_{r_q}}\wedge d\bar{z}_{s_1}\cdots \wedge d\bar{z}_{s_p}(x)\\
\end{align*}

\begin{align*}
&=i\sum_{r_a,s_b,i,j}\Lambda_{ij}(z)\frac{\partial g}{\partial z_i}  f_{r_1...r_q s_1...s_p}(z) \frac{\partial}{\partial z_{r_1}}\wedge \frac{\partial}{\partial z_j}\wedge \frac{\partial}{\partial z_{r_2}}\cdots\wedge \frac{\partial}{\partial z_{r_q}}\wedge d\bar{z}_{s_1}\cdots \wedge d\bar{z}_{s_p}(x)\\
&-\Lambda_{ij}(z)\frac{\partial g}{\partial z_j}f_{r_1...r_q s_1...s_p}(z) \frac{\partial}{\partial z_{r_1}} \wedge \frac{\partial}{\partial z_i}\wedge \frac{\partial}{\partial z_{r_2}}\cdots\wedge \frac{\partial}{\partial z_{r_q}}\wedge d\bar{z}_{s_1}\cdots \wedge d\bar{z}_{s_p}(x)\\
&=-i\left(\sum_{i,j} \Lambda_{ij}(z)\frac{\partial g}{\partial z_i}\frac{\partial}{\partial z_j}-\Lambda_{ij}(z)\frac{\partial g}{\partial z_j}\frac{\partial}{\partial z_i}\right)\wedge f(x)\\
&=-i\Lambda_x(v^{1,0})\wedge e
\end{align*}
where $v^{1,0}=\sum_{k=1}^n \frac{\partial g}{\partial z_k}dz_k|_x$.
\end{example}

\begin{example}
Let $(M,\Lambda)$ be a holomorphic Poisson manifold.
Here we compute a symbol of 
\begin{center}
$L=\bar{\partial}:A^{0,p}(M,\wedge^q T)\to A^{0,p}(M,\wedge^{q+1} T)$
\end{center}
First we note that $L$ has order $1$. Let $(x,v)\in T_{\mathbb{R}}'(M)$ and $e\in (\wedge^p \overline{T}^*\otimes \wedge^q T)_x$ be given. Find $g\in A^{0,0}(M)$ and $f\in A^{0,p}(M,\wedge^q T)$ such that $dg_x=v$, and $f(x)=e$. Let $f=\sum_{r_a,s_b} f_{r_1...r_q s_1...s_p}(z)d\bar{z}_{s_1}\cdots \wedge d\bar{z}_{s_p} \wedge \frac{\partial}{\partial z_{r_1}}\wedge \cdots\wedge \frac{\partial}{\partial z_{r_q}}$ around $x$. Then
\begin{align*}
\sigma_1(L)(x,v)e=L(i(g-g(x))f)(x)=iL((g-g(x))f)(x)\\
=i\sum_{r_a,s_b,i}\frac{\partial g}{\partial \bar{z_i}}d\bar{z}_i \wedge f(x)\\
=iv^{0,1}\wedge e
\end{align*}
\end{example}

\begin{remark}
Let $(M,\Lambda)$ be a holomorphic Poisson manifold and let
\begin{center}
$L=\bar{\partial} +[\Lambda_0, -]:A^p=A^{0,p-1}(M,T)\oplus \cdots \oplus A^{0,0}(M,\wedge^p T)\to A^{p+1}= A^{0,p}(M,T)\oplus \cdots \oplus A^{0,0}(M,\wedge^{p+1} T)$
\end{center}
Let $(x,v)\in T_{\mathbb{R}}'(M)$ and $g\in \mathcal{E}(M)$ such that $dg_x=v$. Then $\sigma_1(L)(x,v)e=iv^{0,1}\wedge e -i\Lambda_x(v^{1,0})\wedge e $
\end{remark}

\begin{definition}
Let $E_0,...,E_N$ be a sequence of differentiable vector bundles defined over a compact differentiable manifold $X$. Suppose that there is a sequence of differential operators, of fixed order $k$, $L_0,...,L_{N-1}$ mapping as in the following sequence
\begin{align*}
\mathcal{E}(X,E_0)\xrightarrow{L_0} \mathcal{E}(X,E_1)\xrightarrow{L_1} \mathcal{E}(X,E_2)\to \cdots \xrightarrow{L_{N-1}} \mathcal{E}(X,E_N).
\end{align*}
Associated with the sequence is the associated symbol sequence
\begin{align*}
0\to \pi^* E_0\xrightarrow{\sigma(L_0)}\pi^*E_1 \xrightarrow{\sigma(L_1)} \pi^*E_2\to \cdots \xrightarrow{\sigma(L_{N-1})} \pi^*E_N\to 0
\end{align*}
Here we denote by $\sigma(L_i)$ the $k$-symbol of the operator $L_j$.
The sequence of operators and vector bundle $E_i$, the sequence is called a complex if $L_i\circ L_{i-1}=0$,$i=1,...,N-1$. Such a complex is called an elliptic complex if the associated symbol sequence is exact.
\end{definition}

\begin{notation}
Denote by
\begin{align*} 
\mathcal{E}=\oplus_{j=0}^N\mathcal{E}(E_j)
\end{align*}
the graded vector space so obtained with the natural grading. We define a operator $L$ by letting
\begin{align*}
L(\xi)=L(\xi_0+\cdots+\xi_N)=L_0\xi_0+\cdots L_N\xi_N
\end{align*}
where $\xi=\xi_0+\cdots+\xi_N$ is the decomposition of $\xi\in \mathcal{E}(E)$ into homogeneous component corresponding to the above grading.
\end{notation}

\begin{proposition}
Let $(M,\Lambda)$ be a holomorphic Poisson manifold.
Let $A^p=A^{0,p-1}(M,T)\oplus \cdots \oplus A^{0,0}(M,\wedge^p T))$ and $L=\bar{\partial} +[\Lambda,-]$. Then
\begin{align*}
0\xrightarrow{L} A^1\xrightarrow{L}\cdots \to A^n\xrightarrow{L} A^{n+1 }\to0
\end{align*}
is elliptic.
\end{proposition}

\begin{proof}
Let $(x,v)\in T_{\mathbb{R}}'(X)$ and $g\in \mathcal{E}(M)$ such that $dg_x=v$. Then the associated symbol sequence is a total complex of the following bicomplex
\begin{center}
$\begin{CD}
@A-i\Lambda_x(v^{1,0})\wedge AA @A-i\Lambda_x(v^{1,0})\wedge AA @A-i\Lambda_x(v^{1,0})\wedge AA  \\
\wedge^3 T_x  @>iv^{0,1}\wedge>> \bar{T}_x^{*}\otimes \wedge^3 T_x @>iv^{0,1}\wedge>> \wedge^2 \bar{T}_x^*\otimes \wedge^3 T_x @>iv^{0,1}\wedge>>\\
@A-i\Lambda_x(v^{1,0})\wedge AA @A-i\Lambda_x(v^{1,0})\wedge AA @A-i\Lambda_x(v^{1,0})\wedge AA \\
\wedge^2 T_x @>iv^{0,1}\wedge>> \bar{T}_x^*\otimes \wedge^2 T_x @>iv^{0,1}\wedge>> \wedge^2 \bar{T}_x^* \otimes \wedge^2 T_x\ @>iv^{0,1}\wedge>>\\
@A-i\Lambda_x(v^{1,0})\wedge AA @A-i\Lambda_x(v^{1,0})\wedge AA @A-i\Lambda_x(v^{1,0})\wedge AA \\
T_x @> iv^{0,1}\wedge>> \bar{T}_x^*\otimes T_x @>iv^{0,1}\wedge>> \wedge^2 \bar{T}_x^*\otimes T_x @>iv^{0,1}\wedge>>\\
@AAA @AAA @AAA \\
0@>>> 0 @>>> 0 @>>>
\end{CD}$
\end{center}
Indeed, the above diagram is bicomplex since $-i\Lambda_x(v^{1,0})\wedge iv^{0,1}+ iv^{0,1}\wedge(-i\Lambda_x(v^{1,0}))=0$. Note that each row is exact. So the total complex is exact by the following lemma. Hence the sequence is elliptic.
\end{proof}

\begin{lemma}[Acyclic Assembly Lemma]
Let $C$ be a bounded double complex in the category of $\mathbb{C}$-vector spaces. Then the total complex is an acyclic chain complex if $C$ is an upper half-plane complex with exact rows.
\end{lemma}

\begin{proof}
See \cite{Wei94} Lemma 2.7.3 p.59.
\end{proof}

\begin{thm}\label{i}
Let $(\mathcal{E}(E)=\oplus_{j=0}^N\mathcal{E}(E_j),L)$ be an elliptic complex equipped with an inner product. Then
\begin{enumerate}
\item There is an orthogonal decomposition 
\begin{align*}
\mathcal{E}=\mathbb{H}(E)\oplus LL^*G\mathcal{E}(E)\oplus L^*LG\mathcal{E}(E)
\end{align*}\\
\item The following commutation relations are valid:

$(a) I=H+\Box G=H+G\Box$

$(b) HG=GH=H\Box=\Box H=0$

$(c) L\Box=\Box L,L^*\Box=\Box L^*$

$(d) LG=GL, L^*G=GL^*$\\
\item $\dim_{\mathbb{C}} \mathbb{H}(E)<\infty$, there is a canonical isomorphism
\begin{align*}
\mathbb{H}^j=\mathbb{H}(E_j)\cong H^j(E) :\text{$j$-th cohomology group with coeffiient $E$}
\end{align*}
\end{enumerate}
where $\Box=LL^*+L^*L$ and $\mathbb{H}(E)=\oplus \mathbb{H}(E_j)$ is the total space of $ \Box$-harmonic sections. $G$ is a Green operator and $H$ is the orthogonal projection onto the closed subspace $\mathbb{H}(E)$.
\end{thm}

\begin{proof}
See \cite{Wel08}  Theorem 5.2 p.147.
\end{proof}

\begin{proposition}\label{j}
Let $\xi\in \mathcal{E}(E)$. Then $\Box\xi=0$ if and only if $L\xi=L^*\xi=0$; moreover, $LH=HL=L^*H=HL^*=0$.
\end{proposition}

\begin{proof}
See \cite{Wel08} Proposition 5.3 p.147.
\end{proof}

\begin{remark}
Let $(M,\Lambda)$ be a holomorphic Poisson manifold. Let $L=\bar{\partial} +[\Lambda,-]$. And  we have a Hermitian inner product on $A^p$. Since 
\begin{align*}
0\xrightarrow{L} A^1\xrightarrow{L}\cdots \to A^n\xrightarrow{L} A^{n+1 }\to0
\end{align*}
is elliptic, we can apply the above Theorem \ref{i} and Proposition \ref{j} .
\end{remark}

\bibliographystyle{amsalpha}
\bibliography{References-Ret3}

\end{document}